%% file: thesis.tex
\documentclass[12pt]{UOthesis}

\usepackage{fancybox}
\usepackage{shadow}
\usepackage[all]{xy}
\usepackage{url}

\allowdisplaybreaks

\NoResume
\NoDedications
\NoListoftables
\NoListoffigures

\setUOname{ Dmitrij Zelo }   
\setUOcpryear{2008}               
\setUOtitle{Simultaneous Approximation to Real and $p$-adic Numbers}  

\phd                

\setUOabstract{
	In this thesis, we study the problem of simultaneous approximation
	to a fixed family of real and $p$-adic numbers by roots of integer polynomials
	of restricted type.
	The method that we use for this purpose was developed by
	H.~{\sc Davenport} and W.M.~{\sc Schmidt} in their study
	of approximation to real numbers by algebraic integers.
	This method based on Mahler's Duality requires to study the dual problem of 
	approximation to successive powers of these numbers by rational numbers with the same denominators.
	Dirichlet's Box Principle provides estimates for such approximations
	but one can do better.
	In this thesis we establish constraints on how much better one can do
	when dealing with the numbers and their squares.
	We also construct examples showing that at least in some instances 
	these constraints are optimal. 
	Going back to the original problem, we obtain estimates for simultaneous
	approximation to real and $p$-adic numbers by roots of integer polynomials
	of degree $3$ or $4$ with fixed coefficients in degree $\geq3$.
	In the case of a single real number (and no $p$-adic numbers), we
	extend work of D.~{\sc Roy} by showing that the
	square of the golden ratio is the optimal exponent of approximation by algebraic numbers of degree $4$ with bounded denominator and trace.
}

\setUOthanks{
 	I would like to express my gratitude to all professors of the University of Ottawa and Carleton University who taught me during my studies.
 	I am infinitely indebted to my supervisor, professor Damien Roy, for his support,
 	patience and direction of my research.
\\
\\
	I am very thankful to all members of the Ottawa-Carleton Institute of Mathematics and Statistics for recommending me and to the board of the Faculty of Graduate and Postdoctoral studies for admitting me to the PhD program, for giving me financial support, and the opportunity to study and work at the Department of Mathematics and Statistics.
	\\
	\\
	Ottawa, September 2008
	\\
	\\
	Dmitrij Zelo
}

\setUOdedicationsText{
  In memory of my famous bottle of port which was taken away from me
  just after reaching full maturity.
}

\newtheorem{theorem}{Theorem}[section]

\newtheorem{lemma}[theorem]{Lemma}

\newtheorem{proposition}[theorem]{Proposition}
\newtheorem{corollary}[theorem]{Corollary}

\newtheorem{definition}[theorem]{Definition}
\newtheorem{remark}[theorem]{Remark}

\theoremstyle{remark}

\newcommand{\vol}{\mathrm{vol}}
\newcommand{\den}{\mathrm{den}}
\newcommand{\Tr}{\mathrm{Tr}}
\newcommand{\dist}{\mathrm{dist}}
\newcommand{\Mat}{\mathrm{Mat}}
\newcommand{\GL}{\mathrm{GL}}
\newcommand{\Res}{\mathrm{Res}}
\newcommand{\nequiv}{ \equiv { \hspace{-10pt} / } \ }
\newcommand{\bigo}{\mathrm{O}}

\begin{document}
\chapter{Introduction}
         \input{introduction.tex}
\chapter{Simultaneous Approximation to real and {\it p}-adic numbers}
\section{General Setting}
	\input{I-01-01-general-setting.tex}	%
	\input{I-01-02-general-setting.tex}	%
\cleardoublepage
	\section{Inequalities (Case n=2)}
	\input{I-02-inequalities.tex}	%
	\newpage
	\section{Constraints on exponents of approximation}
	\input{I-03-real-padic.tex}	%
	\newpage
	\section{Special approximation sequences}
	\input{I-04-real-fibonacci.tex}	%
	\input{apprx-quadratic-real.tex}

	\input{I-06-constrains-real.tex}	%

	\input{I-05-padic-fibonacci.tex}	%
	\input{I-07-constrains-padic.tex}	%
	\newpage
	\section{Examples}\label{C1:S5:examples}
	\input{I-09-optimal-exponents-real-padic.tex}	%
	\input{I-10-optimal-exponents-padic.tex}	%
\cleardoublepage
\chapter{Duality}
	\input{II-01-duality-real-padic.tex}

	\input{II-02-duality-real.tex}

	\input{II-03-duality-padic.tex}
\cleardoublepage
\chapter{Extremal Real Numbers}
	\input{I-extremal-real-preliminary.tex}
	\input{paper.tex}

	\input{period.tex}

\cleardoublepage
\input{biblio}		
\newpage
\end{document}

%% file: introduction.tex
	\section{Approximation to real numbers}

	Firstly, we consider the problem of approximation to
	transcendental real numbers by elements of a given
	infinite set $\mathcal{A}$ of algebraic numbers.
	To each such set $\mathcal{A}$, we attach an exponent
	of approximation $\tau(\mathcal{A})$
	defined as the supremum of all numbers $\tau \in \mathbb{R}$,
	such that for any transcendental number $\xi \in \mathbb{R}$
	there exist infinitely many numbers $\alpha \in \mathcal{A}$,
	with $|\xi - \alpha|_{\infty} \leq H(\alpha)^{-\tau}$.
	Here $|*|_{\infty}$ denotes the absolute value on $\mathbb{R}$
	and $H(\alpha)$ denotes the \emph{height} of $\alpha$.
	It is defined as the height of the
	minimal polynomial of $\alpha$ over $\mathbb{Z}$, namely
	the largest absolute value of the coefficients
	of this polynomial.

	Let $\gamma = (1 + \sqrt{5})/{2}$ denote the golden ratio.
	Consider the case where $\mathcal{A}=\mathcal{A}_n$
	is the set of all algebraic integers of degree $\leq n$
	over $\mathbb{Q}$ and write $\tau_n = \tau(\mathcal{A}_n)$.
	In $1969$ H.~{\sc Davenport} and W.M.~{\sc Schmidt}
	showed that $\tau_2 = 2$, $\tau_3 \geq \gamma^2$,
	$\tau_4 \geq 3$ and $\tau_n \geq \lbrack (n+1)/2 \rbrack$ for each $n\geq5$
	(see in \cite{DS}).
	To prove this in the case $n=3$, H.~{\sc Davenport} and W.M.~{\sc Schmidt}
	consider the system of inequalities
\begin{equation}\label{DavenportSchmidt:1}
	\max_{0\leq l \leq 2} |x_l|_{\infty} \leq X
	\ \text{ and } \
	\max_{0\leq l \leq 2} |x_l - x_0 \xi^l|_{\infty}
		\leq c X^{-\lambda}.
\end{equation}
	  Choosing $\lambda = 1/\gamma$ and assuming that
	$\xi$ is a non-quadratic real number they show, in Theorem 1a of \cite{DS},
	that there exists a constant $c > 0$ such that the
	inequalities (\ref{DavenportSchmidt:1})
	have no non-zero solution
	${\bf x} = (x_{0}, x_{1}, x_{2}) \in \mathbb{Z}^{3}$
	for arbitrarily large values of $X$.
	  Combining the above result with Mahler's Duality they conclude,
	in Theorem 1 of \cite{DS}, that
	there are infinitely many algebraic integers $\alpha$ of degree $\leq 3$
	satisfying $|\xi - \alpha|_{\infty} \leq c' H(\alpha)^{-\gamma^2}$
	for some constant $c'>0$.
	This means that $\tau_3 \geq \gamma^2$.

	Around $2003$, using additional tools presented in \cite{ARNCAI.1}
	and \cite{ARNCAI.2}, D.~{\sc Roy} showed conversely that
	there exist a transcendental real number $\xi$
	and a constant $c > 0$, such that the
	inequalities (\ref{DavenportSchmidt:1}), with $\lambda = 1/\gamma$,
	have a non-zero solution ${\bf x} \in \mathbb{Z}^3$ for each
	real $X\geq1$. Such a number is called an {\it extremal real number}.
	In \cite{ARNCAI.2} D.~{\sc Roy}
	constructed a special class of extremal real numbers and
	showed that, for each number $\xi$ from this class, there exists
	a constant $c_1 > 0$, such that for any algebraic integer $\alpha$
	of degree at most $3$ over $\mathbb{Q}$, we have
	$|\xi - \alpha|_{\infty} \geq c_1 H(\alpha)^{-\gamma^2}$.
	This means that $\tau_3 \leq \gamma^2$.
	Together with the result of H.~{\sc Davenport}
	and W.M.~{\sc Schmidt}, it gives $\tau_3 = \gamma^2$.

	In Chapter $3$ of this thesis we work with the set
	all algebraic numbers which are roots
	of polynomials of the form $a_0 T^4 + a_1 T^3 + a_2 T^2 + a_3 T + a_4$,
	with $|a_0|+|a_1|\neq0$ bounded by some given number.
	We first show that if $\xi$ is a non-quadratic real number
	then, for any given polynomial $R \in \mathbb{Z}[T]$,
	there are infinitely many algebraic numbers $\alpha$ which are roots
	of polynomials $F \in \mathbb{Z}[T]$ satisfying
\[
	\deg(R - F) \leq 2
	\text{ \ and \ }
	|\xi - \alpha|_{\infty} \leq cH(\alpha)^{- \gamma^2},
\]
	for an appropriate constant $c > 0$ depending only on $\xi$ and $R$.
	Here $\deg(P)$ denotes the \emph{degree} of a polynomial
	$P \in \mathbb{R}[T]$.
	Upon taking $R(T) = T^3$, we recover the result of
	H.~{\sc Davenport} \& W.M.~{\sc Schmidt}
	concerning to approximation to $\xi$ by algebraic integers of degree $\leq3$.

	Our result below extends the main result of D.~{\sc Roy} in \cite{ARNCAI.2}
	to the case of approximation to real numbers by algebraic
	numbers of degree $\leq4$ with bounded denominator and trace.
\begin{theorem}\label{introduction:T:3}
	There exist a transcendental real number $\xi$ and a
	constant $c=c(\xi) > 0$ such that, for any
	algebraic number $\alpha$ of degree $3$ or $4$, we have
\[
	|\xi - \alpha| \geq c D^{-2\gamma^9} H(\alpha)^{- \gamma^2},
\]
	where
\[
	D =
	\left \{
	\begin{matrix}
	&|a_0| + |a_1| & \text{if} \quad \deg(\alpha) = 4,
	\\
	&|a_0| & \text{if} \quad \deg(\alpha) = 3,
	\end{matrix}
	\right .
\]
	and where $a_0, a_1$ denote the first and the second leading coefficients
	of the minimal polynomial of $\alpha$ over $\mathbb{Z}$.
\end{theorem}
	In view of the preceding discussion, this means that for any fixed
	choice of $a_0, a_1 \in \mathbb{Z}$ not
	both $0$, the optimal exponent of approximation to
	non-quadratic real numbers by roots of polynomials of
	the form $a_0 T^4 + a_1 T^3 + a_2 T^2 + a_3 T + a_4$
	is $\gamma^2$.
	In particular, if we fix a real number $B > 0$, then $\gamma^2$
	is the optimal exponent of approximation to
	non-quadratic real numbers by algebraic numbers of degree
	$3$ or $4$ with denominator and trace bounded
	above by $B$ in absolute value.
	The real number $\xi$ that we use in the proof
	of Theorem \ref{introduction:T:3}
	belongs to the specific family of extremal
	real numbers considered by D.~{\sc Roy}
	in Theorem 3.1 of \cite{ARNCAI.2},
	some of which are given explicitly,
	by  Proposition 3.2 of \cite{ARNCAI.2},
	in terms of their continued fraction expansion.

	D.~Roy showed in \cite{ARNCAI.1} that for any extremal real number $\xi$
	there exists an unbounded sequence of primitive points
	${\bf x}_{k} = (x_{k,0}, x_{k,1}, x_{k,2})
	\in \mathbb{Z}^3$ indexed by integers $k \geq 1$, such that
\begin{equation}\label{introduction:extr:1}
	\| {{\bf  x}_{k+1} } \|_{\infty}
	\sim \| {{\bf  x}_k } \|_{\infty}^{\gamma},
	\quad
	\max\{|x_{k,0} \xi - x_{k,1}|_{\infty},|x_{k,0} \xi^2 - x_{k,2}|_{\infty}\}
	\ll \| {{\bf  x}_{k} } \|_{\infty}^{-1},
\end{equation}
	where $\| {\bf  x}_k \|_{\infty}
	 = \max\{|x_{k,0}|_{\infty},|x_{k,1}|_{\infty},|x_{k,2}|_{\infty}\}$.
	For $X,Y \in \mathbb{R}$ the notation $Y \ll X$
	means that $Y \leq c X$ for some constant $c>0$ independent of $X$ and $Y$
	and the notation $X \sim Y$ means that $Y \ll X \ll Y$.
	D.~Roy showed also that there exists a unique
	non-symmetric matrix $M$ with $\det(M) \neq 0$, such that
	for sufficiently large $k \geq 1$ viewing the point ${\bf  x}_k$
	as a symmetric matrix
	$
	\left(
	\begin{matrix}
	x_{k,0}      & x_{k,1}
	\\
	x_{k,1}      &  x_{k,2}
	\end{matrix}
	\right )
	$,
	the point ${\bf x}_{k+1}$ is a rational multiple
	of ${\bf x}_{k} M_{k} {\bf x}_{k-1}$, where
\begin{equation}\label{introduction:extr:2}
	M_k =
	\left \{
	\begin{matrix}
		M \text{ if } k \text{ is even},
		\\
		{}^t M\text{ if } k \text{ is odd}.
	\end{matrix}
	\right .
\end{equation}

	In \S \ref{C1:S4:SS3} we show similarly that
	there exists a number $\lambda_0 \approx 0.611455261\ldots$, so that
	if $\xi$ is non-quadratic and if
	$\lambda \in (\lambda_0,1/\gamma]$ are such that the
	inequalities (\ref{DavenportSchmidt:1}) have a non-zero solution
	${\bf x} \in \mathbb{Z}^3$ for each $X \geq1$, then
	there exist an unbounded sequence $({\bf x}_{k})_{k\geq1}$
	of primitive points in $\mathbb{Z}^3$
	satisfying constraints similar to (\ref{introduction:extr:1})
	and a non-symmetric matrix $M \in \Mat_{2\times2}(\mathbb{Z})$
	with $\det(M) \neq 0$, such that
	for sufficiently large $k\geq3$,
	the point ${\bf x}_{k+1}$ is a rational multiple
	of ${\bf x}_{k} M_{k} {\bf x}_{k-1}$, where $M_k$
	is defined as in (\ref{introduction:extr:2}).
	\ \\

	\section{Approximation to p-adic numbers}

	Now we turn to the problem of approximation to
	p-adic numbers by algebraic integers.
	Let $p$ be a prime number
	and let $|*|_{p}$ denotes the usual absolute value on $\mathbb{Q}_p$
	with $|p|_p=p^{-1}$.
	For each $n\geq2$, we define the exponent
	of approximation $\tau_{n}'$ as the
	supremum of all numbers $\tau \in \mathbb{R}$
	such that, for any transcendental number $\xi_p \in \mathbb{Z}_p$,
	there exist infinitely many algebraic integers $\alpha$ of
	degree $\leq n$, with $|\xi_p - \alpha|_{p} \leq H(\alpha)^{-\tau}$.

		In $2002$, O.~{\sc Teuli\'{e}} transposed the method of
	H.~{\sc Davenport} and W.M.~{\sc Schmidt} to the realm of p-adic numbers
	and showed similarly that $\tau_2' \geq 2$, $\tau_3' \geq \gamma^2$,
	$\tau_4' \geq 3$ and $\tau_n' \geq \lbrack (n+1)/2 \rbrack$ for each $n\geq5$
	(see in \cite{Teu}).
	To prove this in the case $n=3$, O.~{\sc Teuli\'{e}}
	considers the system of inequalities
\begin{equation}\label{Teulie:1}
	\max_{0\leq l \leq 2} |x_l|_{\infty} \leq X
	\ \text{ and } \
	\max_{0\leq l \leq 2} |x_l - x_0 \xi_p^l|_{p}
		\leq c X^{-\lambda}.
\end{equation}
	Choosing $\lambda = \gamma$ and assuming that
	$\xi_p \in \mathbb{Z}_p$ is non-quadratic,
	he shows that there exists a constant $c > 0$ such that
	inequalities (\ref{Teulie:1})
	have no non-zero solution ${\bf x} \in \mathbb{Z}^3$
	for arbitrarily large values of $X$ (see Theorem 2 of \cite{Teu}).
		Combining the above result with Mahler's Duality,
	he deduces that
	there are infinitely many algebraic integers $\alpha$ of degree $\leq 3$
	with $|\xi_p - \alpha|_{p} \leq H(\alpha)^{-\gamma^2}$
	(see Theorem 3 of \cite{Teu}).
	This means that $\tau_3' \geq \gamma^2$.

		Conversely, we show in \S\ref{C1:S5:SS2:p-adic} of 
		this thesis that, for each $\lambda < \gamma$,
	there exist a constant $c > 0$ and a number
	$\xi_p \in \mathbb{Q}_p$ which is non-quadratic,
	such that the inequalities (\ref{Teulie:1})
	have a non-zero solution ${\bf x} \in \mathbb{Z}^3$ for each
	$X \geq 1$.
		Moreover, suppose that $\xi_p \in \mathbb{Q}_p$
	is a non-quadratic and that $\lambda$ is such that the
	system of inequalities (\ref{Teulie:1})
	has a non-zero solution
	${\bf x} \in \mathbb{Z}^3$ for each $X \geq1$.
	Similarly as in the real case,
	we show that there exists some real number
	$\lambda_{p,0}\approx 1.615358873\ldots$, such that
	for each exponent $\lambda \in (\lambda_{p,0},\gamma]$
	there exist an unbounded sequence $({\bf y}_{k})_{k\geq1}$
	of primitive  points in $\mathbb{Z}^3$
	satisfying constraints similar to (\ref{introduction:extr:1})
	and a non-symmetric matrix $M \in \Mat_{2\times2}(\mathbb{Z})$
	with $\det(M) \neq 0$, such that
	for sufficiently large $k\geq3$, each point
	${\bf y}_{k+1}$ is a non-zero rational multiple of
	${\bf y}_{k}M_{k}{\bf y}_{k-1}$, where $M_k$
	is defined as in (\ref{introduction:extr:2}).
	\ \\

	\section{Simultaneous approximation to real and p-adic numbers}

	Now we consider simultaneous approximation
	to real and p-adic numbers by algebraic numbers of bounded degree.
	Our goal is to unify and extend the results
	of H.~{\sc Davenport} and W.M.~{\sc Schmidt} in \cite{DS}
	and those of O.~{\sc Teuli\'{e}} in \cite{Teu}
	concerning the system (\ref{DavenportSchmidt:1})
	or (\ref{Teulie:1}) and the exponents $\tau_3$ or $\tau_3'$.

	For this purpose, we fix a finite set $\mathcal{S}$
	of prime numbers and points
\[
	\bar \xi =
	(\xi_{\infty},(\xi_{p})_{{p} \in \mathcal{S}})
	\in \mathbb{R} \times \prod_{{p} \in \mathcal{S}}\mathbb{Q}_{p}
	\ \text{ and } \
	\bar \lambda = (\lambda_{\infty},(\lambda_{p})_{{p} \in \mathcal{S}})
	\in \mathbb{R}^{|\mathcal{S}|+1}.
\]
	We say that $\bar\lambda$ is an exponent of approximation
	in degree $n\geq1$
	to $\bar\xi$ if there exists a constant
	$c > 0$ such that the inequalities
\begin{equation}\label{introduction:1}
 \begin{aligned}
	&\max_{0\leq l \leq n} |x_l|_{\infty} \leq X,
	\\
	&\max_{0\leq l \leq n} |x_l - x_0 \xi_{\infty}^l|_{\infty}
		\leq c X^{-\lambda_{\infty}},
	\\
	&\max_{0\leq l \leq n} |x_l - x_0 \xi_{p}^l|_{p}
		\leq c X^{-\lambda_{p}} \
		( \forall {p} \in \mathcal{S} ),
 \end{aligned}
\end{equation}
	have a non-zero solution
	${\bf x} = (x_{0}, x_{1},\ldots,x_{n}) \in \mathbb{Z}^{n+1}$
	for each real number $X \geq 1$.

	Based on Minkowski's convex body theorem we
	show in Chapter~1 that $\bar \lambda$
	is an exponent of approximation in degree $n\geq1$
	to $\bar \xi$
	if the following conditions are satisfied
\[
	\lambda_{\infty} \geq -1,
	\quad
	\lambda_{p} \geq 0 \ (p \in \mathcal{S})
	\ \text{ and } \
	\lambda_{\infty} + \sum_{p \in \mathcal{S}}\lambda_{p} \leq 1/n.
\]

	In Chapter~1, we also prove the following statement 
	which provides constraints that a point $\bar\lambda$ must 
	satisfy in order to be an exponent of approximation
	in degree $2$.
\begin{theorem}\label{introduction:T:1}
	Suppose that
\[
	\lambda_{\infty} + \sum_{p \in \mathcal{S}}\lambda_{p} \geq 1/\gamma
\]
	and suppose that one of the
	following conditions is satisfied:
\begin{itemize}
	\item[(i)]
		$[\mathbb{Q}(\xi_{\infty}) \colon \mathbb{Q}] >2$
		and
		$ \lambda_{\infty} > 1/\gamma^2$,
	\item[(ii)]
		$\mathcal{S} = \{p\}$ for some prime number $p$,
		$[\mathbb{Q}(\xi_{p}) \colon \mathbb{Q}] >2$,
	\[
		-1 \leq \lambda_{\infty} <0
				\ \text{ and } \
				\lambda_{p} > 2 - 1/\gamma.
	\]

\end{itemize}
	If $\lambda_{\infty} + \sum_{p \in \mathcal{S}}\lambda_{p} = 1/\gamma$,
	there exists a constant $c > 0$ such that
	for $n=2$ the system of inequalities (\ref{introduction:1})
	have no non-zero solution ${\bf x} \in \mathbb{Z}^{3}$
	for arbitrarily large values of $X$.
	If $\lambda_{\infty} + \sum_{p \in \mathcal{S}}\lambda_{p} > 1/\gamma$,
	then any constant $c>0$ has this property.

\end{theorem}
	Applying this result with $\mathcal{S}=\emptyset$
	and $\lambda_{\infty} = 1/{\gamma}$ we recover
	Theorem 1a of H.~{\sc Davenport} and W.M.~{\sc Schmidt} in \cite{DS}.
	Applying it with $\mathcal{S}=\{p\}$, $\lambda_{\infty} = -1$
	and $\lambda_{p} = {\gamma}$, it gives
	Theorem 2 of O.~{\sc Teuli\'{e}} in \cite{Teu}.
	Combining the above result with Mahler's Duality we obtain the
	following statement.
\begin{theorem}\label{introduction:T:2}
	Suppose that  $\bar \xi$ and $\bar \lambda$
	satisfy the hypothesis of Theorem \ref{introduction:T:1}
	and suppose that
	$\lambda_{\infty} + \sum_{p \in \mathcal{S}}\lambda_{p} = 1/\gamma$.
	Let $R(T)$ be a polynomial in $\mathbb{Z}[T]$.
	Suppose $R(\xi_p) \in \mathbb{Z}_{p}$ for each ${p} \in \mathcal{S}$.
	Then there exist infinitely many polynomials $F(T) \in \mathbb{Z}[T]$ with
	the following properties:
	\\
	$(i) $ \ $\deg(R - F) \leq 2$,
	\\
	$(ii)$  if $\lambda_{\infty}>-1$,
		there exists a real root $\alpha_{\infty}$ of $F$, such that
\[
	| \xi_{\infty} - \alpha_{\infty} |_{\infty}
		\ll H(F)^{-\gamma(\lambda_{\infty}+1)},
\]
	\\
	$(iii)$ for each ${p} \in \mathcal{S}$,
		there exists a root $\alpha_{p}$ of $F$
		in ${\mathbb{Q}}_p$, such that
\[
 	| \xi_p - \alpha_p |_p \ll H(F)^{-\gamma\lambda_{p}}.
\]
	Moreover,
	for each ${p} \in \mathcal{S}$
	such that $\xi_{p} \in \mathbb{Z}_{p}$,
	we can choose $\alpha_{p} \in \mathbb{Z}_{p}$.
\end{theorem}
	Here $H(F)$ stands for the \emph{height}
	of a polynomial $F$, which
	is the maximum of the absolute values of its coefficients.

	Suppose that $\mathcal{S} = \emptyset$.
	Then Theorem \ref{introduction:T:1} implies Theorem 1a  in \cite{DS}
	while Theorem \ref{introduction:T:2} applied with $R(T) = T^3$
	implies Theorem 1 in \cite{DS}, due to H.~{\sc Davenport} and W.M.~{\sc Schmidt}.
	Let $p$ be a prime number and suppose that $\mathcal{S} = \{p\}$.
	If $\bar \lambda = (-1,\lambda_{p})$, then
	Theorem \ref{introduction:T:1} implies
	the case $n = 3$ of Theorem 2 in \cite{Teu}
	while Theorem \ref{introduction:T:2} applied with $R(T) = T^3$
	implies the case $n = 3$ of
	Theorem 3 in \cite{Teu}, due to O.~{\sc Teuli\'{e}}.
\\
\\
	Fix $n\in\mathbb{Z}_{\geq2}$, $R(T) \in \mathbb{Z}[T]$
	and a finite set $\mathcal{S}$ of prime numbers
	and define $\tau_{\mathcal{S},R,n}$ as the supremum of all sums
\[
	\sum_{\nu\in\{\infty\}\cup\mathcal{S}}\tau_{\nu}
\]
	taken over families
	$\tau_{\nu} \in \mathbb{R}$ $(\nu\in\{\infty\}\cup\mathcal{S})$
	such that, for any transcendental numbers
	$\xi_{\infty} \in \mathbb{R}$
	and
	$\xi_p \in \mathbb{Z}_p$ $\ ({p} \in \mathcal{S})$,
	there exist infinitely many polynomials $F(T)\in\mathbb{Z}[T]$
	with $\deg(R-F)\leq n$ having roots $\alpha_{\infty} \in \mathbb{R}$
	and $\alpha_p \in \mathbb{Z}_p$ $\ ({p} \in \mathcal{S})$ such that
\[
	|\xi_{\nu} - \alpha_{\nu}|_{\nu} \leq H(F)^{-\tau_{\nu}}
	\ (\nu\in\{\infty\}\cup\mathcal{S}).
\]
	In this context Theorem \ref{introduction:T:2}
	leads to the conclusion that \ $\tau_{\mathcal{S},R,2}\geq \gamma^2$.
\\

	In \S\ref{C1:S5:examples}, we construct examples showing that the condition 
	$\lambda_{\infty} +
	\sum_{p \in \mathcal{S}}\lambda_{p} \geq 1/\gamma$
	in Theorem \ref{introduction:T:1} cannot be improved.
        We obtain the following statement.

\begin{theorem}\label{introduction:T:4}
	For any $\bar \lambda \in \mathbb{R}_{>0}^{|\mathcal{S}|+1}$
	with
\[
	\sum_{\nu \in \mathcal{S}\cup\{\infty\}}\lambda_{\nu}
	< \frac{1}{\gamma}
\]
	there exist a non-zero point
	$\bar \xi =
	(\xi_{\infty},(\xi_{p})_{{p} \in \mathcal{S}})
	\in \mathbb{R} \times \prod_{{p} \in \mathcal{S}}\mathbb{Q}_{p}$ with
	$[\mathbb{Q}(\xi_{\infty}):\mathbb{Q}]>2$
	such that $\bar \lambda$ is an exponent of approximation
	in degree $2$ to $\bar \xi$.
\end{theorem}

	We also prove the following result which shows that, for any $\lambda_{p}$
	with $\lambda_{p} < \gamma$, the pair
	$\bar \lambda = (-1,\lambda_{p})$ is an exponent of approximation
	in degree $2$ to some point $\bar \xi = (\xi_{\infty},\xi_{p})$,
	with $[\mathbb{Q}(\xi_{p}):\mathbb{Q}]>2$.

\begin{theorem}\label{introduction:T:2_0}
	Let $p$ be some prime number.
	For any real $\lambda_{p} < \gamma$,
	there exists a number $\xi_{p} \in \mathbb{Q}_p$
	with $[\mathbb{Q}(\xi_{p}):\mathbb{Q}]>2$,
	such that the inequalities
\[
	\max_{0\leq l \leq 2} |x_l|_{\infty} \leq X
	\ \text{ and } \
	\max_{0\leq l \leq 2} |x_l - x_0 \xi_{p}^l|_{p}
		\leq c X^{-\lambda_{p}},
\]
	have a non zero solution ${\bf x} \in \mathbb{Z}^3$, for any $X \gg 1$.
\end{theorem}

%% file: I-01-01-general-setting.tex
	Let $n \geq 1$ be an integer, let $\mathcal{S}$ be a finite set of
	prime numbers and let
	$\bar \xi =
	(\xi_{\infty},(\xi_{p})_{{p} \in \mathcal{S}})
	\in \mathbb{R} \times \prod_{{p} \in \mathcal{S}}\mathbb{Q}_{p}$.
	Let ${\nu} \in \mathcal{S}$ or ${\nu} = \infty$.
	For any point ${\bf  x} = (x_{0}, x_{1},\ldots,x_{n}) \in
	\mathbb{Q}_{\nu}^{n+1}$	we define the ${\nu}$-adic norm of ${\bf x}$ by
\begin{equation}\label{I:1:D:0:1}
 \| {\bf  x } \|_{\nu} := \max_{0 \leq i\leq n}\{|x_i|_{\nu}\},
\end{equation}
	and we put
\begin{equation}\label{I:1:D:0:2}
 	L_{\nu}({\bf x }) := \| {\bf  x }  - x_{0} {\bf t }_{\nu} \|_{\nu},
\end{equation}
	where ${\bf t}_{\nu} := (1,\xi_{\nu},\ldots,\xi_{\nu}^n)$.
	We denote by $|\mathcal{S}|$ the number of elements in $\mathcal{S}$.
\begin{definition} \label{I:1:D:1}
	Let $\bar \xi \in \mathbb{R} \times \prod_{{p} \in
	\mathcal{S}}\mathbb{Q}_{p}$
	and
	$\bar \lambda = (\lambda_{\infty},(\lambda_{p})_{{p} \in \mathcal{S}})
	\in \mathbb{R}^{|\mathcal{S}|+1}$.
	We say that $\bar\lambda$ is an exponent of approximation
	in degree $n$
	to $\bar\xi$ if there exists a constant
	$c > 0$ such that the inequalities
\begin{equation}\label{I:1:D:1:1}
 \begin{aligned}
	&\| {\bf x} \|_{\infty} \leq X,
	\\
	&L_{\infty}({\bf x }) \leq c X^{-\lambda_{\infty}},
	\\
	&L_{p}({\bf x }) \leq c X^{-\lambda_{p}} \
	\forall {p} \in \mathcal{S},
 \end{aligned}
\end{equation}
	have a non-zero solution ${\bf x} \in \mathbb{Z}^{n+1}$ for each real
	number $X \geq 1$.
\end{definition}
\subsection{Application of Minkowski's convex body theorem}\label{C1:S1:SS1}

	The following proposition based on Minkowski's convex body theorem
	provides a sufficient condition for
	$\bar \lambda \in \mathbb{R}^{|\mathcal{S}|+1}$
	to be an exponent of approximation in degree $n$.
\begin{proposition} \label{I:1:P:1}
	Suppose that
	$\bar \lambda = (\lambda_{\infty},(\lambda_{p})_{{p} \in \mathcal{S}})
	\in \mathbb{R} \times \mathbb{R}_{\geq0}^{|\mathcal{S}|}$
	satisfies the inequalities
\begin{equation} \label{I:1:P:1:00}
	\lambda_{\infty} \geq -1
	\ \text{ and } \
	\lambda_{\infty} + \sum_{p \in \mathcal{S}}\lambda_{p} \leq 1/n.
\end{equation}
	Then $\bar \lambda$ is an exponent of approximation in degree $n$
	to any $\bar \xi \in \mathbb{R} \times \prod_{{p} \in
	\mathcal{S}}\mathbb{Q}_{p}$.

\end{proposition}
\begin{proof}
	Fix any $\bar \xi \in \mathbb{R} \times \prod_{{p} \in
	\mathcal{S}}\mathbb{Q}_{p}$.
	For each $c>0$ and $X\geq1$ we define the convex body
\[
	C_{c,X} = \{{\bf x} \in \mathbb{R}^{n+1} \ | \
		\| {\bf x} \|_{\infty} \leq X, \
		L_{\infty}({\bf x }) \leq  cX^{-\lambda_{\infty}}
		\}
\]
	 and the lattice
\[
	\Lambda_{c,X} = \{{\bf x} \in \mathbb{Z}^{n+1} \ | \
	L_{p}({\bf x }) \leq cX^{-\lambda_p}
	\ \text{ for each } {p} \in \mathcal{S} \}.
\]
	We claim that there exists a constant $c>0$ such that
	for each $X$ sufficiently large,
	we have $C_{c,X} \cap \Lambda_{c,X} \neq \emptyset$.
	This means that for such a constant $c>0$ the
	inequalities (\ref{I:1:D:1:1}) have a non-zero solution
	${\bf x} \in \mathbb{Z}^{n+1}$ for each $X$ sufficiently large.
	Upon replacing $c$ by a larger constant if necessary, we ensure that
 	the inequalities (\ref{I:1:D:1:1}) have a non-zero solution
	${\bf x} \in \mathbb{Z}^{n+1}$ for {\it each} real
	number $X \geq 1$, which means that $\bar\lambda$ is an exponent
	of approximation to $\bar\xi$.

	To prove the claim, we fix a constant $c>0$
	(all implied constants below depend on $c$).
	For each $X\geq1$, we define a new convex body
\begin{align*}
	C'_{c,X} &= \{{\bf x} \in \mathbb{R}^{n+1} \ | \
		| x_0 |_{\infty} \leq X/M, \
		L_{\infty}({\bf x }) \leq  cX^{-\lambda_{\infty}}
		\}
		\\
		&= \{{\bf x} \in \mathbb{R}^{n+1} \ | \
		\| A{\bf x } \|_{\infty} \leq 1 \},
\end{align*}
	where $M = 2\max\{1,|\xi_{\infty}^n|_{\infty}\}$ and
\[
	A =
	\left (
	\begin{matrix}
	MX^{-1} & 0 & 0 & \ldots &0
	\\
	-\xi_{\infty}c^{-1}X^{\lambda_{\infty}}
	& c^{-1}X^{\lambda_{\infty}} & 0 & \ldots &0
	\\
	\vdots& \vdots& \vdots & \ddots &\vdots
	\\
	-\xi_{\infty}^n c^{-1}X^{\lambda_{\infty}} & 0 & 0 &  \ldots &
	c^{-1}X^{\lambda_{\infty}}
	\end{matrix}
	\right ).
\]
	Case 1. If $\lambda_{\infty}>-1$,
	we have $C_{c,X}' \subseteq C_{c,X}$
	for each $X$ sufficiently large.
	Assuming $X\gg1$, we also construct a lattice $\Lambda_{c,X}'$ as follows.
	Fix any real $X$ sufficiently large.
	For each ${p} \in \mathcal{S}$, we choose
	$n_p \in \mathbb{Z}_{\geq0}$ such that
\begin{equation} \label{I:1:P:1:0}
	p^{-n_p} \leq cX^{-\lambda_p} < p^{-n_p+1}
\end{equation}
	and put $b = \prod_{{p} \in \mathcal{S}}{p}^{n_{p}}$.
	Let $d_0$ be the smallest positive
	integer such that $d_0{\bf t}_{p} \in \mathbb{Z}_p^{n+1}$
	for each $ p \in \mathcal{S}$.
	By the Strong Approximation Theorem (see \cite{CASS-FROH}),
	for each $l=1,\ldots,n$, there exists $d_l \in \mathbb{Z}_{>0}$
	satisfying
\begin{equation} \label{I:1:P:1:1}
	|d_l - d_0 \xi_p^l|_p \leq cX^{-\lambda_p}
	\ \text{ for each } \ p \in \mathcal{S}.
\end{equation}
	Let $({\bf e}_l)_{l=0,\ldots,n}$ denote the
	canonical basis of $\mathbb{Q}^{n+1}$.
	Define another basis $({\bf u}_l)_{l=0,\ldots,n}$ as follows
\begin{align*}
	{\bf u}_0 &= (d_0,d_1,\ldots,d_n),
	\\
	{\bf u}_l &= b{\bf e}_l
	\ \text{for each} \ l=1,\ldots,n,
\end{align*}
	and put
\[
	\Lambda_{c,X}'
	= \langle {\bf u}_0,{\bf u}_1,\ldots,{\bf u}_n \rangle_{\mathbb{Z}}.
\]
	By (\ref{I:1:P:1:0}) and (\ref{I:1:P:1:1}),
	for each $p \in \mathcal{S}$, we have
\begin{align*}
	L_{p}({\bf u}_0)
		& = \max_{l=1,\ldots,n} |d_l - d_0 \xi_p^l|_p
		\leq cX^{-\lambda_p},
	\\
	L_{p}({\bf u}_l) & = |b|_p = p^{-n_p} \leq cX^{-\lambda_p}
	\ \text{for each} \ l=1,\ldots,n.
\end{align*}
	So, we have $\Lambda_{c,X}' \subseteq \Lambda_{c,X}$
	for each $X\gg1$.
	Hence, for each $X\gg1$, we get
\[
	C_{c,X}' \cap \Lambda_{c,X}' \subseteq C_{c,X} \cap \Lambda_{c,X}.
\]
	We make the stronger claim that
	$C_{c,X}' \cap \Lambda_{c,X}' \neq \emptyset$
	for each $X$ sufficiently large.
	By Minkowski's convex body theorem \cite{CASS} (see p.~71),
	it suffices to show that the inequality
\begin{equation} \label{I:1:P:1:2}
	\vol(C_{c,X}') > 2^{n+1}\det(\Lambda_{c,X}')
\end{equation}
	holds for each $X$ sufficiently large.
	Let us find the value of $\vol(C_{c,X}')$
	and an upper bound for $\det(\Lambda_{c,X}')$ in terms of $X$.
	Using (\ref{I:1:P:1:0}), we find that
\[
	\vol(C_{c,X}')
	= \int_{A^{-1}\big([-1,1]^{n+1}\big)}d{\bf x}
	= \int_{[-1,1]^{n+1}}|\det(A)|^{-1}d{\bf y}
	= 2^{n+1} M^{-1} c^{n}X^{1-n\lambda_{\infty}},
\]
	and
 \begin{align*}
	|\det(\Lambda_{c,X}')|
	&= \det[{\bf u}_0,{\bf u}_1,\ldots,{\bf u}_n]
				= d_0 b^{n}
				= d_0 \prod_{{p} \in \mathcal{S}}{p}^{n n_{p}}
				= d_0 \Big(\prod_{{p} \in \mathcal{S}}p^n\Big)
				\prod_{{p} \in \mathcal{S}}{p}^{n (n_{p}-1)}
				\\
				&\leq
				 d_0  \Big ( \prod_{{p} \in \mathcal{S}}p^n \Big )
				 \prod_{{p} \in \mathcal{S}}(c^{-n}X^{n\lambda_p})
				 =
				 d_0 c^{-n} \Big( \prod_{{p} \in \mathcal{S}}p^n \Big )
				 X^{n\sum_{{p} \in \mathcal{S}}\lambda_p}.
 \end{align*}
 	To fulfill (\ref{I:1:P:1:2}), it suffices to require that
\[
	2^{n+1} M^{-1} c^{n}X^{1-n\lambda_{\infty}}
	> 2^{n+1} d_0 c^{-n} \Big( \prod_{{p} \in \mathcal{S}}p^n \Big )
				 X^{n\sum_{{p} \in \mathcal{S}}\lambda_p}.
\]
	This gives
\[
	X^{1-n(\lambda_{\infty} + \sum_{{p} \in \mathcal{S}}\lambda_p)}
	> M d_0 c^{-2n}\prod_{{p} \in \mathcal{S}}p^n.
\]
	Since $\lambda_{\infty} + \sum_{{p} \in \mathcal{S}}\lambda_p \leq 1/n$,
	the above inequality holds for each $X$ sufficiently large,
	provided that $c$ is chosen large enough to ensure that
	$M d_0 c^{-2n}\prod_{{p} \in \mathcal{S}}p^n < 1$.

	Case 2. If $\lambda_{\infty}=-1$, then
	$C_{c,X} = \{{\bf x} \in \mathbb{R}^{n+1} \ | \
		\| {\bf x} \|_{\infty} \leq X \}$
	provided that $c>0$ is large enough. Since $\vol(C_{c,X}) 
        = 2^{n+1} X^{1+n}$, we get $C_{c,X} \cap \Lambda_{c,X}' \neq 
        \emptyset$ if $c$
	is chosen so that
	$d_0 c^{-n}\prod_{{p} \in \mathcal{S}}p^n < 1$.
        Then the conclusion follows as in the previous case.
\end{proof}
\subsection{A covering of ${\mathbb{R}_{\geq1}}$}\label{C1:S1:SS2}

	Throughout this paragraph, we fix a point
$
	\bar \xi \in
	(\mathbb{R}\backslash \mathbb{Q}) \times \prod_{{p} \in
	\mathcal{S}}(\mathbb{Q}_{p}\backslash \mathbb{Q})
$
	and an exponent of approximation
$
	\bar \lambda = (\lambda_{\infty},(\lambda_{p})_{{p} \in \mathcal{S}})
	\in \mathbb{R} \times \mathbb{R}_{\geq0}^{|\mathcal{S}|}
$
	to $\bar \xi$ in degree $n$.
	We also fix a corresponding constant $c>0$ such that
	the inequalities (\ref{I:1:D:1:1})
	have a non-zero solution in $\mathbb{Z}^{n+1}$ for each $X \geq 1$.

	For each $X\geq1$, we denote by
	$\mathcal{C}_{c,X} = \mathcal{C}_{c,X}(\bar \xi,\bar \lambda)$
	the set of all non-zero integer solutions of the system of inequalities
	(\ref{I:1:D:1:1}). We also denote by $\mathbb{Z}_{\mathcal{S}}$ the
	set of all non-zero integers
	of the form $ \pm \prod_{p \in \mathcal{S}}p^{k_{p}} $,
	where $k_{p} \geq 0$ is an integer for each $p \in \mathcal{S}$.

	Let ${\bf v}$ be a primitive point in $\mathbb{Z}^{n+1}$.
	Looking at (\ref{I:1:D:1:1}) we note that if
	$l{\bf v} \in \mathcal{C}_{c,X}$
	for some integer $l \in \mathbb{Z}_{\neq 0}$, then the integer
	$m =\prod_{{p} \in \mathcal{S}}|l|_p^{-1} \in \mathbb{Z}_{\mathcal{S}}$
	also has the property that $m{\bf v} \in \mathcal{C}_{c,X}$.
	We can therefore define a set $I_{c}({\bf v})$ in the following two ways:
\begin{align*}
	I_{c}({\bf v}) 	&=
			\{X\in \mathbb{R}_{\geq1}
			\ | \ \exists l \in \mathbb{Z}
			\ s.t. \ l{\bf v} \in \mathcal{C}_{c,X}\}
			\\
			&=
			\{X\in \mathbb{R}_{\geq1}
			\ | \ \exists m \in \mathbb{Z}_{\mathcal{S}}
			\ s.t. \ m{\bf v} \in \mathcal{C}_{c,X}\}.
\end{align*}
	For any non-empty compact set
	$A \subset \mathbb{R}$ we denote by $\max A$ and $\min A$
	its maximal and minimal elements respectively.
	The next lemma shows that, if the sum of the components of $\bar \lambda$
	is positive, then the sets $I_{c}({\bf v})$ provide
	a covering of ${\mathbb{R}_{\geq1}}$ by compact sets.
\begin{lemma} \label{I:1:L:1}
	Suppose that
	$\lambda := \sum_{\nu \in \{\infty\} \cup \mathcal{S}}\lambda_{\nu} >0$.
	\\
	(i) For each primitive point ${\bf v} \in \mathbb{Z}^{n+1}$,
	the set $I_{c}({\bf v})$ is a compact subset of $\mathbb{R}_{\geq1}$.
	\\
	(ii) $\mathbb{R}_{\geq1}$ is covered by the sets $I_{c}({\bf v})$,
	where ${\bf v}$ runs through all primitive points of $\mathbb{Z}^{n+1}$.
	\\
	(iii) For any $X \geq 1$, there exists
	a primitive point ${\bf w} \in \mathbb{Z}^{n+1}$,
	such that
\[
	X \in I_{c}({\bf w}) \ \text{ and } \ X < \max I_{c}({\bf w}),
\]
	Moreover, if $X >1$,
	there also exists a primitive point ${\bf u} \in \mathbb{Z}^{n+1}$,
	such that
\[
	X \in I_{c}({\bf u}) \ \text{ and } \ X > \min I_{c}({\bf u}).
\]
	(iv) Define
	$\phi(X) := \min \{ \|{\bf w}\|_{\infty} \ | \ X \in I_{c}({\bf w})
	\text{ for some primitive point } {\bf w} \in \mathbb{Z}^{n+1} \}$
	for each $X\geq1$.
	Then $\phi(X) \rightarrow \infty$ as $X  \rightarrow \infty$.
\end{lemma}
\begin{proof}
	For the proof of (i) we suppose that $I_{c}({\bf v})\neq\emptyset$
	and choose any $X \in I_{c}({\bf v})$.
	Then there exists
	some $ m \in \mathbb{Z}_{\mathcal{S}}$, such that
	$m{\bf v} \in \mathcal{C}_{c,X}$, and using the product formula,
	we find
\begin{align*}
	\prod_{{\nu} \in \{\infty\}\cup\mathcal{S}}L_{\nu}({\bf v})
		& = \prod_{{\nu} \in \{\infty\}\cup\mathcal{S}}
			|m|_{\nu}L_{\nu}({\bf v})
		\\
		&= \prod_{{\nu} \in \{\infty\}\cup\mathcal{S}}L_{\nu}(m{\bf v})
		\\
		&\leq c^{|\mathcal{S}|+1}
		\prod_{{\nu} \in \{\infty\}\cup\mathcal{S}}
		X^{-\lambda_{\nu}}
  		\\
		&= c^{|\mathcal{S}|+1}
		X^{-\lambda}.
\end{align*}
	For each ${\nu} \in \{\infty\}\cup\mathcal{S}$, we have
	$\xi_{\nu} \notin \mathbb{Q}$ and so $L_{\nu}({\bf v}) \neq 0$.
	Hence, because of the hypothesis $\lambda>0$, we get
\[
	I_{c}({\bf v}) \subseteq [1,c^{(|\mathcal{S}|+1)/\lambda}
	\prod_{{\nu} \in \{\infty\}\cup\mathcal{S}}L_{\nu}({\bf v})^{-1/\lambda}].
\]
	This shows that $I_{c}({\bf v})$ is a bounded subset of $\mathbb{R}_{\geq1}$.
	Now, we suppose that $X$
	is an accumulation point of $I_{c}({\bf v})$. Then there exists an
	infinite sequence
	$(X_i)_{i\geq1}$ in $I_{c}({\bf v})$ and a sequence
	$(m_i)_{i\geq1}$ in $\mathbb{Z}_{\mathcal{S}}$, such that
	$\lim_{i\rightarrow \infty} X_i = X$ and
	$m_i{\bf v} \in \mathcal{C}_{c,X_i} $ for each $i\geq1$.
	Using the first inequality in (\ref{I:1:D:1:1}), we have
\[
	|m_i|_{\infty}\|{\bf v}\|_{\infty}\leq X_i.
\]
	Since $(X_i)_{i\geq1}$ is bounded, we deduce that the sequence
	$(m_i)_{i\geq1}$ contains only finitely many different elements.
	Hence, there exists an index $i_0 \geq 1$, such that
	$m_{i_0}{\bf v} \in \mathcal{C}_{c,X_i}$, for infinitely many values
	of $i$.
	By continuity we deduce that $m_{i_0}{\bf v} \in \mathcal{C}_{c,X}$,
	which means that $X \in I_{c}({\bf v})$.
	Therefore $I_{c}({\bf v})$ is closed and so it is a compact subset of
	$\mathbb{R}_{\geq1}$.

	For the proof of (ii) we use the assumption that
	for any $X \in \mathbb{R}_{\geq1}$
	the system (\ref{I:1:D:1:1}) has a non-zero solution
	${\bf x} \in \mathbb{Z}^{n+1}$. Writing ${\bf x} = l{\bf v}$
	for some $l\in\mathbb{Z}$ and some primitive point ${\bf v}$,
	we deduce that $X\in I_{c}({\bf v})$.
	This shows that
	$\mathbb{R}_{\geq1}$
	is covered by sets $I_{c}({\bf v})$ where ${\bf v}$ runs through
	all primitive points of $\mathbb{Z}^{n+1}$.

	To show the first part of (iii), we consider the interval $[X,X+1]$.
	Denote by $\mathcal{W}$ the set of primitive points ${\bf w}$ in
	$\mathbb{Z}^{n+1}$ such that
	$[X,X+1]\cap I_{c}({\bf w})\neq\emptyset$.
	By Part (ii), we have
\[
	[X,X+1] \subseteq \cup_{{\bf w} \in \mathcal{W}}I_{c}({\bf w}).
\]
	Moreover, the set $\mathcal{W}$ is finite since for each
	${\bf w} \in \mathcal{W}$, we have $\|{\bf w}\|_{\infty}\leq X+1$.
	Define also
\[
	\mathcal{W}_{X} =
	\{
		{\bf w} \in \mathcal{W}
		\ | \
		X = \max I_{c}({\bf w})
	\}
\]
	so that,
\[
	(X,X+1] \subseteq \cup_{{\bf w}
	\in \mathcal{W}\setminus\mathcal{W}_{X}}I_{c}({\bf w}).
\]
	Since $\mathcal{W}$ is finite, the set
	$\cup_{{\bf w} \in \mathcal{W}\setminus\mathcal{W}_{X}}I_{c}({\bf w})$
	is compact and thus, we have
\[
	[X,X+1] \subseteq \cup_{{\bf w}
	\in \mathcal{W}\setminus\mathcal{W}_{X}}I_{c}({\bf w}).
\]
	This means that there exists a primitive point
	${\bf w}$, such that $X \in I_{c}({\bf w})$
	and $X < \max I_{c}({\bf w})$.

	To show the second part of (iii),
	we consider instead the interval $[1,X]$.
	Denote by $\mathcal{U}$ the set of primitive points
	${\bf u} \in \mathbb{Z}^{n+1}$ such that
	$[1,X]\cap I_{c}({\bf u})\neq\emptyset$.
	By Part (ii), we have
\[
	[1,X] \subseteq \cup_{{\bf u} \in \mathcal{U}}I_{c}({\bf u}).
\]
	Moreover, the set $\mathcal{U}$ is finite since for each
	${\bf u} \in \mathcal{W}$, we have $\|{\bf u}\|_{\infty}\leq X$.
	Define also
\[
	\mathcal{U}_{X} =
	\{
		{\bf u} \in \mathcal{U}
		\ | \
		X = \min I_{c}({\bf u})
	\}
\]
	so that, 
\[
	[1,X) \subseteq \cup_{{\bf u}
	\in \mathcal{U}\setminus\mathcal{U}_{X}}I_{c}({\bf u}).
\]
	Since $\mathcal{U}$ is finite, the set
	$\cup_{{\bf u} \in \mathcal{U}\setminus\mathcal{U}_{X}}I_{c}({\bf u})$
	is compact and thus, we have
\[
	[1,X] \subseteq \cup_{{\bf u}
	\in \mathcal{U}\setminus\mathcal{U}_{X}}I_{c}({\bf u}).
\]
	This means that there exists a primitive point
	${\bf u}$, such that $X \in I_{c}({\bf u})$
	and $X > \min I_{c}({\bf u})$.

	For the proof of (iv), we suppose on the contrary that
	there exists some positive real number $B \geq1$
	and a sequence $(X_i)_{i\geq0}$ such that
	$\lim_{i \rightarrow \infty} X_i = \infty$
	and $\phi(X_i) \leq B$ for all $i\geq0$.
	Then there exists a sequence of primitive points
	$({\bf w}_i)_{i\geq0}$ in $\mathbb{Z}^3$ such that
	$X_i \in I_{c}({\bf w}_i)$  and $ \|{\bf w}_i\|_{\infty} \leq B$,
	for each $i\geq0$. Hence, we have only finitely many different elements
	in the sequence $({\bf w}_i)_{i\geq0}$.
	By Part (i), we have that each $I_{c}({\bf w}_i)$ is compact and thus
	the sequence $(X_i)_{i\geq0}$ is contained in a finite collection
	of compact sets. So, it is bounded,
	which contradicts the assumption that
	$\lim_{i \rightarrow \infty} X_i = \infty$.
\end{proof}
\begin{remark}\label{I:1:R:1}
	Lemma \ref{I:1:L:1}(iii) shows in particular that, for any
	primitive point ${\bf v} \in \mathbb{Z}^{n+1}$ with
	$I_{c}({\bf v}) \neq \emptyset$, there exists a primitive
	point ${\bf w} \in \mathbb{Z}^{n+1}$
	such that ${\bf w} \neq \pm {\bf v}$ and
	$I_{c}({\bf v}) \cap I_{c}({\bf w}) \neq \emptyset$.
\end{remark}
\subsection{A sequence of primitive points}\label{C1:S1:SS3}

	Let the notation and hypotheses be as in the
	previous paragraph. Here, we construct a sequence
	of primitive points $({\bf v}_k)_{k\geq0}$ in $\mathbb{Z}^{n+1}$
	which in our context will play the role of the sequence of
	minimal points of
	H.~{\sc Davenport} and W.M.~{\sc Schmidt} in \cite{DS}.
	Because of the extra complexity of working with several
	places at the same time, we will also need to introduce
	other sequences $({\bf x}_k)_{k\geq0}$ and $({\bf x'}_k)_{k\geq0}$
	in $\mathbb{Z}^{n+1}$ which will be derived from $({\bf v}_k)_{k\geq0}$.

	In order to fulfill the above task, we introduce
	one more piece of notation.
	For each real number $X\geq1$ and each primitive point
	${\bf v} \in \mathbb{Z}^{n+1}$, we define
\begin{align*}
	\mathcal{L}_{c}({\bf v},X) &=
	\{m{\bf v}
	\ | \
		m \in \mathbb{Z}_{\mathcal{S}}
		\text{ and }
		m{\bf v} \in \mathcal{C}_{c,X}
	 \}
	 \\
	 &=
	 \mathbb{Z}_{\mathcal{S}}{\bf v} \cap \mathcal{C}_{c,X}.
\end{align*}
	Note that if $X \in I_{c}({\bf v})$,
	then $\mathcal{L}_{c}({\bf v},X) \neq \emptyset$.
	We first prove a technical lemma.
\begin{lemma} \label{I:1:L:2-0}
	Let ${\bf v}$ be a primitive point in $\mathbb{Z}^{n+1}$ with
	$I_{c}({\bf v}) \neq \emptyset$. Let $X \in I_{c}({\bf v})$
	and let ${\bf x}$ be a point in $\mathcal{L}_{c}({\bf v},X)$ with
	minimal norm. Then for any $Y \in I_{c}({\bf v})$ with $Y \geq X$,
	we have
\begin{equation}\label{I:1:L:2-0:1}
 \begin{gathered}
	 \mathcal{L}_{c}({\bf v},Y) \subseteq \mathbb{Z}_{\mathcal{S}} \ {\bf x},
	 \\
	 L_{\infty}({\bf x}) \leq c Y^{-\lambda_{\infty}}.
 \end{gathered}
\end{equation}
\end{lemma}
\begin{proof}
	Fix $Y \in I_{c}({\bf v})$ with $Y \geq X$ and choose
	a point ${\bf y} \in \mathcal{L}_{c}({\bf v},Y)$.
	There exist $m,n \in \mathbb{Z}_{\mathcal{S}}$, such that
	${\bf x} = m{\bf v}$ and ${\bf y} = n{\bf v}$.
	In order to prove that
	$\mathcal{L}_{c}({\bf v},Y) \subseteq \mathbb{Z}_{\mathcal{S}} \ {\bf x}$,
	we need to show that $m\mid n$.
	Suppose on the contrary that $|m|_q < |n|_q$ for some $q \in \mathcal{S}$.
	Put $l = m|m|_q|n|_q^{-1}$.
	Then $l \in \mathbb{Z}_{\mathcal{S}}$ and it
	satisfies the following relations
\begin{align*}
	&|l|_{\infty} < |m|_{\infty},
	\\
	 &|l|_q  = |n|_q > |m|_q,
	 \\
	 &|l|_p = |m|_p
	  \ \text{ for each } p \in \mathcal{S}\setminus \{q\}.
\end{align*}
	So, using the fact that ${\bf x} \in \mathcal{C}_{c,X}$,
	${\bf y} \in \mathcal{C}_{c,Y}$ and the assumption $\lambda_q\geq0$,
	we have
\begin{align*}
	\|l{\bf v}\|_{\infty}
	&= |l|_{\infty}\|{\bf v}\|_{\infty}
	< |m|_{\infty}\|{\bf v}\|_{\infty}
	= \|{\bf x}\|_{\infty}  \leq X,
	\\
	L_{\infty}(l{\bf v})
	&= |l|_{\infty}L_{\infty}({\bf v})
	< |m|_{\infty}L_{\infty}({\bf v})
	= L_{\infty}({\bf x}) \leq c X^{-\lambda_{\infty}},
	\\
	L_{q}(l{\bf v})
	&= |l|_qL_{q}({\bf v})
	= |n|_qL_{q}({\bf v})
	= L_{q}({\bf y}) \leq c Y^{-\lambda_q} \leq c X^{-\lambda_q},
	\\
	L_{p}(l{\bf v})
	&= |l|_pL_{p}({\bf v})
	= |m|_pL_{p}({\bf v})
	= L_{p}({\bf x}) \leq c X^{-\lambda_p}
	\ \text{ for each } p \in \mathcal{S}\setminus \{q\}.
\end{align*}
	This means that $l{\bf v} \in \mathcal{L}_{c}({\bf v},X)$.
	Since $\|l{\bf v}\|_{\infty} < \|{\bf x}\|_{\infty}$,
	this contradicts the fact that ${\bf x}$ has minimal norm in
	$\mathcal{L}_{c}({\bf v},X)$.
	Hence, $m\mid n$ and so we find that
\[
	L_{\infty}({\bf x}) =
	|m|_{\infty}L_{\infty}({\bf v})
	\leq
	|n|_{\infty}L_{\infty}({\bf v})
	=
	L_{\infty}({\bf y}).
\]
	Finally, since ${\bf y} \in \mathcal{C}_{c,Y}$, we conclude that
\[
	L_{\infty}({\bf x}) \leq L_{\infty}({\bf y})
	\leq c Y^{-\lambda_{\infty}}.
\]
\end{proof}

	We can now state and prove the main result of this paragraph.
\begin{proposition} \label{I:1:L:2}
	Suppose that
	$\lambda := \sum_{\nu \in \{\infty\} \cup \mathcal{S}}\lambda_{\nu} >0$.
	Then, there exist a sequence of primitive points $({\bf v}_k)_{k\geq0}$
	in $\mathbb{Z}^{n+1}$ any two of which are linearly independent,
	two sequences $({\bf x}_k)_{k\geq0}$ and $({\bf x'}_k)_{k\geq0}$
	of non-zero integer points in $\mathbb{Z}^{n+1}$, and an unbounded
	increasing sequence of
	real numbers $(X_k)_{k\geq0}$, such that for each $k\geq0$, we have
\begin{equation}\label{I:1:L:2:1}
 \begin{gathered}
 	{\bf x'}_{k} \in \mathbb{Z}_{\mathcal{S}} \ {\bf x}_{k}
		\subseteq \mathbb{Z}_{\mathcal{S}} \ {\bf v}_{k},
	\\
	{\bf x'}_{k}, {\bf x}_{k+1} \in \mathcal{C}_{c,X_{k+1}},
	\\
	L_{\infty}({\bf x}_{k}) \leq c X_{k+1}^{-\lambda_{\infty}},
	\\
	X_{k} \notin I_{c}({\bf v}_{k+1}).
 \end{gathered}
\end{equation}
	In particular, any two points of
	$({\bf x}_k)_{k\geq0}$ or of $({\bf x'}_k)_{k\geq0}$
	with distinct indexes
	are linearly independent.
\end{proposition}
\begin{proof}
	Choose ${\bf v}_{0}$ to be an integer point in
	$\mathcal{C}_{c,1}$ with the largest $\max I_{c}({\bf v}_{0})$.
	Since $\|{\bf v}_{0}\|_{\infty} = 1$, this point ${\bf v}_{0}$ is primitive.
	Put $X_1 = \max I_{c}({\bf v}_{0})$ and
	consider the following finite set
\[
	\mathcal{V}_1 =
	\{
		{\bf v} \in \mathbb{Z}^{n+1} \ | \
		{\bf v} \text{ is primitive } \text{ with } X_1 \in I_{c}({\bf v})
		\text{ and } X_1 < \max I_{c}({\bf v})
	\}.
\]
	By Part (iii) of Lemma \ref{I:1:L:1}, we have that
	$\mathcal{V}_1 \neq \emptyset$ and
	that ${\bf v} \neq \pm {\bf v}_{0}$
	for each ${\bf v} \in \mathcal{V}_1$.
	Choose a point
	${\bf v}_1 \in \mathcal{V}_1$ such that $\max I_{c}({\bf v}_1)$ is the largest.
	Arguing in this way we construct an
	increasing sequence of real numbers $(X_k)_{k\geq1}$
	and a sequence of primitive points $({\bf v}_{k})_{k\geq0}$
	in $\mathbb{Z}^{n+1}$, any two of each are linearly independent,
	such that for each $k\geq0$, we have
\begin{equation}\label{I:1:L:2:2}
 \begin{gathered}
	X_{k} \in I_{c}({\bf v}_{k})
	\ \text{ and } \
	X_{k} < X_{k+1} = \max I_{c}({\bf v}_{k}),
	\\
	\max_{{\bf v} \in \mathcal{V}_k} \big\{ \max I_{c}({\bf v})  \big\}
	= \max I_{c}({\bf v}_{k}),
 \end{gathered}
\end{equation}
	where
\[
	\mathcal{V}_k =
	\{
		{\bf v} \in \mathbb{Z}^{n+1} \ | \
		{\bf v} \text{ is primitive } \text{ with } X_{k} \in I_{c}({\bf v})
		\text{ and } X_{k} < \max I_{c}({\bf v})
	\}.
\]
	Since, for each $k\geq1$, we have $X_{k} \in I_{c}({\bf v}_{k})$, then
	$\|{\bf v}_{k}\|_{\infty} \leq X_{k}$. Since the sequence
	$({\bf v}_{k})_{k\geq0}$ consists of infinitely many different elements,
	this shows that
	the sequence $(X_k)_{k\geq1}$ is unbounded.
	Also, we note that $X_{k} \notin I_{c}({\bf v}_{k+1})$.
	Indeed, suppose on the contrary that $X_{k} \in I_{c}({\bf v}_{k+1})$.
	Since $X_{k} < X_{k+1} < X_{k+2} =  \max I_{c}({\bf v}_{k+1})$,
	this means that ${\bf v}_{k+1} \in \mathcal{V}_k$.
	So, we have
\[
	X_{k+2} = \max I_{c}({\bf v}_{k+1})
	\leq
	\max_{{\bf v} \in \mathcal{V}_{k}} \big\{ \max I_{c}({\bf v})  \big\}
	= \max I_{c}({\bf v}_{k}) = X_{k+1},
\]
	but this contradicts the second relation in (\ref{I:1:L:2:2})
	with $k$ replaced by $k+1$.

	Now, for each $k\geq1$, we choose a point
	${\bf x}_{k} \in \mathcal{L}_{c}({\bf v}_{k},X_{k})$ with minimal norm
	and a point ${\bf x'}_{k} \in \mathcal{L}_{c}({\bf v}_{k},X_{k+1})$.
	Then, the sequences $({\bf x}_{k})_{k\geq0}$ and $({\bf x'}_{k})_{k\geq0}$
	satisfy the second relation in (\ref{I:1:L:2:1}).
	Moreover, since the points in the sequence $({\bf v}_k)_{k\geq0}$
	are primitive and since ${\bf v}_i \neq \pm{\bf v}_j$
	for $i,j \geq 0$ with $i\neq j$, then any two of them are linearly
	independent. Hence any two different points of $({\bf x}_k)_{k\geq0}$ and
	any two different points of $({\bf x'}_k)_{k\geq0}$ are linearly independent.
	Finally, the first and third relations follow from Lemma \ref{I:1:L:2-0}
	applied to ${\bf v} = {\bf v}_{k}$, ${\bf x} = {\bf x}_{k}$
	and $Y = X_{k+1}$.
\end{proof}
\subsection{A criterion in terms of primitive points}\label{C1:S1:SS4}

	The following proposition provides a criterion
	which interprets the notion of an exponent of approximation
	in degree $n$ in terms of the existence
	of primitive points with certain properties.
	In the case where the set $\mathcal{S}$ consists of just one
	prime number $p$ and where $\lambda_{\infty}\leq-1$,
	this is due to O.~{\sc Teuli\'{e}} \cite{Teu}.
\begin{proposition} \label{I:1:P:3}
	Let $\bar \xi \in
	(\mathbb{R}\backslash \mathbb{Q}) \times \prod_{{p} \in
	\mathcal{S}}(\mathbb{Q}_{p}\backslash \mathbb{Q})$.
	Then $\bar \lambda = (\lambda_{\infty},(\lambda_{p})_{{p} \in \mathcal{S}})
	\in \mathbb{R} \times \mathbb{R}_{\geq0}^{|\mathcal{S}|}$ is
	an exponent of approximation to $\bar \xi$ in degree $n$ iff
	there exists a constant
	$c_1 > 0$ such that the relation
\begin{equation}\label{I:1:P:3:1}
	1 \leq
	\min
	\Big \{
		\frac{X}{\|{\bf v }\|_{\infty}},
		\frac{c_1X^{-\lambda_{\infty}}}{L_{\infty}({\bf v })}
	\Big \}
	\prod_{{p} \in \mathcal{S}}
	\min
	\Big \{
		1,
		\frac{c_1X^{-\lambda_{p}}}{L_{p}({\bf v })}
	\Big \}
\end{equation}
	has a non-zero primitive solution
	${\bf v} \in \mathbb{Z}^{n+1}$ for each real number $X \geq 1$.
\end{proposition}
\begin{proof}
	\ \\
	($ \Rightarrow $ )
	If  $\bar \lambda$ is an exponent of approximation to $\bar \xi$,
	there exists a constant $c > 0$ such that the inequalities
	(\ref{I:1:D:1:1}) have a non-zero solution
	${\bf x} \in \mathbb{Z}^{n+1}$ for each real number $X \geq 1$.
	Fix a real $X \geq 1$.
	According to the comments made in \S \ref{C1:S1:SS2},
	we can choose a solution ${\bf x}$ of the system (\ref{I:1:D:1:1})
	in the form ${\bf x} = m {\bf v}$, where
	${\bf v} \in \mathbb{Z}^{n+1}$ is primitive
	and $m \in \mathbb{Z}_{\mathcal{S}}$.
	Then, we have
\begin{align*}
	&|m|_{\infty}\| {\bf v} \|_{\infty} \leq X,
	\\
	&|m|_{\infty}L_{\infty}({\bf v }) \leq c X^{-\lambda_{\infty}},
	\\
	&|m|_{p}L_{p}({\bf v }) \leq c X^{-\lambda_{p}} \
	\forall {p} \in \mathcal{S}.
\end{align*}
	This is equivalent to the system of inequalities
\begin{align*}
	&|m|_{\infty} \leq \min
	\Big \{
		\frac{X}{\|{\bf v }\|_{\infty}},
		\frac{cX^{-\lambda_{\infty}}}{L_{\infty}({\bf v })}
	\Big \},
	\\
	&
	|m|_{p} \leq \frac{cX^{-\lambda_{p}}}{L_{p}({\bf v })}
	\ \text{ and } \
	|m|_{p} \leq 1
	\quad \forall {p} \in \mathcal{S}.
\end{align*}
	Since $|m|_{\infty}\prod_{{p} \in \mathcal{S}}|m|_{p} = 1$, it follows
	that
\[
	1 = |m|_{\infty}\prod_{{p} \in \mathcal{S}}|m|_{p}
	\leq
	\min
	\Big \{
		\frac{X}{\|{\bf v }\|_{\infty}},
		\frac{cX^{-\lambda_{\infty}}}{L_{\infty}({\bf v })}
	\Big \}
	\prod_{{p} \in \mathcal{S}}
	\min
	\Big \{
		1,
		\frac{cX^{-\lambda_{p}}}{L_{p}({\bf v })}
	\Big \}.
\]
	Choosing $c_1 = c$, we get (\ref{I:1:P:3:1}).

	($ \Leftarrow $ )
	Fix a real $X \geq 1$ and assume that (\ref{I:1:P:3:1}) holds
	for some primitive point
	${\bf v} \in \mathbb{Z}^{n+1}$ and some constant $c_1 > 0$ independent of $X$.
	Let $m$ be the largest positive element in $\mathbb{Z}_{\mathcal{S}}$
	satisfying
\begin{equation}\label{I:1:P:3:2}
	|m|_p \geq \frac{c_1X^{-\lambda_{p}}}{L_{p}({\bf v })}
	\ \text{ for each } \
	 {p} \in \mathcal{S}.
\end{equation}
	By the choice of $m$, we also have
\begin{equation}\label{I:1:P:3:3}
	|m|_p < \frac{pc_1X^{-\lambda_{p}}}{L_{p}({\bf v })}
	\ \text{ for each } \
	 {p} \in \mathcal{S}.
\end{equation}
	By (\ref{I:1:P:3:1}) and (\ref{I:1:P:3:2}), we get
\begin{align*}
	1 & \leq
	\min
	\Big \{
		\frac{X}{\|{\bf v }\|_{\infty}},
		\frac{c_1X^{-\lambda_{\infty}}}{L_{\infty}({\bf v })}
	\Big \}
	\prod_{{p} \in \mathcal{S}}
	\min
	\Big \{
		1,
		\frac{c_1X^{-\lambda_{p}}}{L_{p}({\bf v })}
	\Big \}
	\\
	&\leq
		\min
	\Big \{
		\frac{X}{\|{\bf v }\|_{\infty}},
		\frac{c_1X^{-\lambda_{\infty}}}{L_{\infty}({\bf v })}
	\Big \}
	\prod_{{p} \in \mathcal{S}}
	\min
	\Big \{ 1,|m|_p \Big \}
	\\
	&\leq
		\min
	\Big \{
		\frac{X}{\|{\bf v }\|_{\infty}},
		\frac{c_1X^{-\lambda_{\infty}}}{L_{\infty}({\bf v })}
	\Big \}
	\prod_{{p} \in \mathcal{S}} |m|_p.
\end{align*}
	Since $|m|_{\infty}\prod_{{p} \in \mathcal{S}}|m|_{p} = 1$, it follows
	that
\[
	|m|_{\infty} \leq
		\min
	\Big \{
		\frac{X}{\|{\bf v }\|_{\infty}},
		\frac{c_1X^{-\lambda_{\infty}}}{L_{\infty}({\bf v })}
	\Big \}
\]
	and therefore, we have
\begin{align*}
	&|m|_{\infty}\|{\bf v }\|_{\infty} \leq X,
	\\
	&|m|_{\infty}L_{\infty}({\bf v }) \leq  c_1X^{-\lambda_{\infty}}.
\end{align*}
	From this and (\ref{I:1:P:3:3}) it follows that
	the point ${\bf x } = m {\bf v }$ is in $\mathcal{C}_{c,X}$,
	with $c = c_1 \max_{{p} \in \mathcal{S}} p$.
	Thus $\bar \lambda$ is an exponent of approximation to $\bar \xi$,
	if (\ref{I:1:P:3:1}) has a non-zero primitive solution
	${\bf v} \in \mathbb{Z}^{n+1}$ for each $X \geq 1$.
\end{proof}
\subsection{Another covering of ${\mathbb{R}_{\geq1}}$}\label{C1:S1:SS5}

	Let $\bar \xi$, $\bar \lambda$ and $c$
	be as in \S \ref{C1:S1:SS2}.
	For each $c_1 > 0$ and each primitive point
	${\bf v} \in \mathbb{Z}^{n+1}$, we define the set
\[
	J_{c_1}({\bf v})
			=
			\{ X\in \mathbb{R}_{\geq1} \ | \
			X \text{ satisfies } (\ref{I:1:P:3:1}) \}.
\]
	This set is a closed interval because it can be presented as the
	set of all solutions of a system of inequalities of the form
	$a_1 \leq X^{\alpha_1},\ldots,a_s \leq X^{\alpha_s}$,
	where $a_1,\ldots,a_s \in \mathbb{R}_{>0}$
	and $\alpha_1,\ldots,\alpha_s \in \mathbb{R}$.
	The proof of Proposition \ref{I:1:P:3}
	provides moreover the following connection between
	$J$-sets and $I$-sets.
\begin{lemma} \label{I:1:L:3}
 	For each primitive point ${\bf v} \in \mathbb{Z}^{n+1}$,
	the set $J_{c}({\bf v})$ is a closed interval and, we have
	\[
		I_{c}({\bf v})
		\subseteq
		J_{c}({\bf v})
		\subseteq
		I_{c'}({\bf v}),
	\]
	with $c' = c \max_{{p} \in \mathcal{S}} p$.
\end{lemma}
\begin{proof}
	Fix a primitive point ${\bf v} \in \mathbb{Z}^{n+1}$ such that
	$I_{c}({\bf v}) \neq \emptyset$ and choose any $X \in I_{c}({\bf v})$.
	The first part of the proof of Proposition \ref{I:1:P:3}
	shows that $X \in J_{c}({\bf v})$.
	Now, suppose that $J_{c}({\bf v}) \neq \emptyset$
	and choose any $X \in J_{c}({\bf v})$.
	The second part of the proof of Proposition \ref{I:1:P:3}
	shows that $X \in I_{c'}({\bf v})$, with
	$c' = c \max_{{p} \in \mathcal{S}} p$.
\end{proof}

	By combining the above lemma with Lemma \ref{I:1:L:1},
	we obtain the following.
\begin{lemma}\label{I:1:L:3-1}
	Suppose that
	$\lambda := \sum_{\nu \in \{\infty\} \cup \mathcal{S}}\lambda_{\nu} >0$.
	\\
	(i) For each primitive point ${\bf v} \in \mathbb{Z}^{n+1}$,
	the set $J_{c}({\bf v})$ is a compact sub-interval of $\mathbb{R}_{\geq1}$.
	\\
	(ii) $\mathbb{R}_{\geq1}$ is covered by the sets $J_{c}({\bf v})$,
	where ${\bf v}$ runs through all primitive points of $\mathbb{Z}^{n+1}$.
	\\
	(iii) For any $X \geq 1$, there exists
	a primitive point ${\bf w} \in \mathbb{Z}^{n+1}$,
	such that
\[
	X \in J_{c}({\bf w}) \ \text{ and } \ X < \max J_{c}({\bf w}),
\]
	Moreover, if $X >1$,
	there also exists a primitive point ${\bf u} \in \mathbb{Z}^{n+1}$,
	such that
\[
	X \in J_{c}({\bf u}) \ \text{ and } \ X > \min J_{c}({\bf u}).
\]
\\
	(iv) Define
	$\psi(X) := \min \{ \|{\bf w}\|_{\infty} \ | \ X \in J_{c}({\bf w})
	\text{ for some primitive point } {\bf w} \in \mathbb{Z}^{n+1} \}$
	for each $X\geq1$.
	Then $ \psi(X) \rightarrow \infty$ as $X  \rightarrow \infty$.
\end{lemma}

	The following result is an analogue
	of Proposition \ref{I:1:L:2} in terms of $J$-sets.

\begin{proposition} \label{I:1:P:5}
	Suppose that
	$\lambda := \sum_{\nu \in \{\infty\} \cup \mathcal{S}}\lambda_{\nu} >0$.
	There exists a sequence of primitive points $({\bf v}_k)_{k\geq0}$
	in $\mathbb{Z}^{n+1}$, any two of which are linearly independent
	and satisfy the following relations
\begin{equation}\label{I:1:P:5:1}
	\max J_{c}({\bf v}_{k})
	<
	\min J_{c}({\bf v}_{k+2})
	\leq
	\max J_{c}({\bf v}_{k+1})
	\ \text{ for each } \  k\geq0,
\end{equation}
	and the sequence $(\max J_{c}({\bf v}_{k}))_{k\geq1}$ is unbounded.
	Moreover, there exist sequences
	$({\bf x}_k)_{k\geq0}$ and $({\bf x'}_k)_{k\geq0}$
	of non-zero integer points in $\mathbb{Z}^{n+1}$
	such that, for each $k\geq0$,
	upon putting $X_{k+1}= \max J_{c}({\bf v}_{k})$, we have
\begin{equation}\label{I:1:P:5:1:1}
 \begin{gathered}
 	{\bf x'}_{k} \in \mathbb{Z}_{\mathcal{S}} \ {\bf x}_{k}
		\subseteq \mathbb{Z}_{\mathcal{S}} \ {\bf v}_{k},
	\\
	{\bf x'}_{k}, {\bf x}_{k+1} \in \mathcal{C}_{c',X_{k+1}},
	\\
	L_{\infty}({\bf x}_{k}) \leq c' X_{k+1}^{-\lambda_{\infty}},
 \end{gathered}
\end{equation}
	where $c' = c \max_{{p} \in \mathcal{S}} p$.
	Finally, for each $k\geq0$,
	the points ${\bf x'}_k$ and ${\bf x}_{k+1}$
	are linearly independent.
\end{proposition}
\begin{proof}
	Choose ${\bf v}_{0}$ to be a primitive point in $\mathbb{Z}^{n+1}$
	satisfying the inequality (\ref{I:1:P:3:1}) with $X =1$ and $c_1 = c$,
	with the largest $\max J_{c}({\bf v}_{0})$.
	Consider the following finite set
\begin{align*}
	\mathcal{V}_1 =
	\{
		{\bf v} \in \mathbb{Z}^{n+1} \ | \
		{\bf v} \text{ is primitive }
		&\text{ with }
		\max J_{c}({\bf v}_{0}) \in J_{c}({\bf v})
		\\
		&\text{ and } \max J_{c}({\bf v}_{0}) < \max J_{c}({\bf v})
	\}.
\end{align*}
	By Part (iii) of Lemma \ref{I:1:L:3-1}, we have that
	$\mathcal{V}_1 \neq \emptyset$ and
	that ${\bf v} \neq \pm {\bf v}_{0}$
	for each ${\bf v} \in \mathcal{V}_1$.
	Choose a point
	${\bf v}_1 \in \mathcal{V}_1$ such that $\max J_{c}({\bf v}_1)$
	is largest.
	Arguing in this way we construct recursively
	a sequence of primitive points $({\bf v}_{k})_{k\geq0}$
	in $\mathbb{Z}^{n+1}$, any two of each are linearly independent,
	such that for each $k\geq0$, we have
\begin{equation}\label{I:1:P:5:2}
 \begin{gathered}
	\max J_{c}({\bf v}_{k}) \in J_{c}({\bf v}_{k+1}),
	\\
	\max J_{c}({\bf v}_{k}) < \max J_{c}({\bf v}_{k+1}),
	\\
	\max J_{c}({\bf v}_{k}) =
	\max_{{\bf v} \in \mathcal{V}_k} \big\{ \max J_{c}({\bf v})  \big\},
 \end{gathered}
\end{equation}
	where
\begin{align*}
	\mathcal{V}_k =
	\{
		{\bf v} \in \mathbb{Z}^{n+1} \ | \
		{\bf v} \text{ is primitive }
		&\text{ with }
		\max J_{c}({\bf v}_{k-1}) \in J_{c}({\bf v})
		\\
		&\text{ and }
		\max J_{c}({\bf v}_{k-1}) < \max J_{c}({\bf v})
	\}.
\end{align*}
	Since, for each $k\geq0$, we have
	$\max J_{c}({\bf v}_{k}) \in J_{c}({\bf v}_{k+1})$, then
	$\|{\bf v}_{k+1}\|_{\infty} \leq \max J_{c}({\bf v}_{k})$.
	Since the sequence
	$({\bf v}_{k})_{k\geq0}$ consists of infinitely many different elements,
	this shows that the sequence $(\max J_{c}({\bf v}_{k}))_{k\geq1}$ is unbounded.

	Using the first relation in (\ref{I:1:P:5:2}) with $k$ replaced by $k+1$,
	we deduce that $\min J_{c}({\bf v}_{k+2}) \leq \max J_{c}({\bf v}_{k+1})$
	for each $k\geq0$. We claim that
	$\max J_{c}({\bf v}_{k})<\min J_{c}({\bf v}_{k+2})$ for each
	$k\geq0$. Fix any $k\geq0$ and suppose on the contrary that
	$\max J_{c}({\bf v}_{k}) \geq \min J_{c}({\bf v}_{k+2})$.
	By the second relation in (\ref{I:1:P:5:2}) this means that
	we have $\max J_{c}({\bf v}_{k}) \in J_{c}({\bf v}_{k+2})$
	and
	$\max J_{c}({\bf v}_{k})
	< \max J_{c}({\bf v}_{k+2})$. So, it follows that
	${\bf v}_{k+2} \in \mathcal{V}_{k+1}$.
	Hence, we have
\[
	\max J_{c}({\bf v}_{k+2})
	\leq
	\max_{{\bf v} \in \mathcal{V}_{k+1}} \big\{ \max J_{c}({\bf v})  \big\}
	= \max J_{c}({\bf v}_{k+1}),
\]
	but this contradicts the second relation in (\ref{I:1:P:5:2})
	with $k$ replaced by $k+1$.

	Moreover, put $X_{k+1}= \max J_{c}({\bf v}_{k})$ for each $k\geq0$.
	By construction, we have
	$X_{k},X_{k+1} \in J_{c}({\bf v}_{k})$ for each $k\geq0$.
	Since Lemma \ref{I:1:L:3} gives
	$J_{c}({\bf v}_{k})\subseteq I_{c'}({\bf v}_{k})$,
	we have	$\mathcal{L}_{c'}({\bf v}_{k},X_{k}) \neq \emptyset$
	and $\mathcal{L}_{c'}({\bf v}_{k},X_{k+1}) \neq \emptyset$
	for each $k\geq0$.
	Now, choose a point
	${\bf x}_{k} \in \mathcal{L}_{c'}({\bf v}_{k},X_{k})$ with minimal norm
	and a point ${\bf x'}_{k} \in \mathcal{L}_{c'}({\bf v}_{k},X_{k+1})$.
	Then, the sequences $({\bf x}_{k})_{k\geq0}$ and $({\bf x'}_{k})_{k\geq0}$
	satisfy the second relation in (\ref{I:1:P:5:1:1}).
	Moreover, since the points in the sequence $({\bf v}_k)_{k\geq0}$
	are primitive and since ${\bf v}_i \neq \pm{\bf v}_j$
	for $i,j \geq 0$ with $i\neq j$, then any two of them are linearly
	independent. Hence any two different points of $({\bf x}_k)_{k\geq0}$ and
	any two different points of $({\bf x'}_k)_{k\geq0}$ are linearly independent.
	Finally, the first and third relations in (\ref{I:1:P:5:1:1})
	follow from Lemma \ref{I:1:L:2-0}
	applied to ${\bf v} = {\bf v}_{k}$, ${\bf x} = {\bf x}_{k}$
	and $Y = X_{k+1}$, with $c$ replaced by $c'$.
\end{proof}

	In the next two paragraphs, we show how the above proposition
	allows one to recover the construction of minimal points
	by  H.~{\sc Davenport} and W.M.~{\sc Schmidt} in \cite{DS}
	and by O.~{\sc Teuli\'{e}} in \cite{Teu}.
	

%% file: I-01-02-general-setting.tex
\subsection{Approximation to real numbers}

	In the case where $\mathcal{S} = \emptyset$,
	the Definition \ref{I:1:D:1} takes the following form.
\begin{definition} \label{I:1:D:3}
	Let $\xi_{\infty} \in \mathbb{R}$
	and $\lambda_{\infty} \in \mathbb{R}$.
	We say that $\lambda_{\infty}$ is an exponent of approximation
	to $\xi_{\infty}$ in degree $n$ if there exists a
	constant $c > 0$ such that the inequalities
\begin{equation}\label{I:1:D:3:1}
 \begin{aligned}
	& \| {\bf x} \|_{\infty} \leq X,
	\\
	& L_{\infty}({\bf x }) \leq c X^{-\lambda_{\infty}},
 \end{aligned}
\end{equation}
	have a non-zero solution ${\bf x} \in \mathbb{Z}^{n+1}$,
	for any real number $X \geq 1$.
\end{definition}

	In this context, Proposition \ref{I:1:P:5}
	leads to the following statement.
\begin{lemma} \label{I:1:L:4}
	Let $\lambda_{\infty}\in \mathbb{R}_{>0}$ be an exponent of approximation
	in degree $n$ to $\xi_{\infty} \in \mathbb{R}\setminus \mathbb{Q}$.
	There exists a sequence of non-zero primitive points
	$({\bf v}_k)_{k \geq 0} \subseteq \mathbb{Z}^{n+1}$ such that
	for each $k \geq 0$, we have
\begin{equation}\label{I:1:L:4:1}
 \begin{gathered}
	\|{\bf v}_k\|_{\infty} < \|{\bf v}_{k+1}\|_{\infty},
	\\
	\|{\bf v}_{k+3}\|_{\infty}^{-\lambda_{\infty}}
	\ll
	L_{\infty}({\bf v}_{k+1})
	< L_{\infty}({\bf v}_k)
	\ll \|{\bf v}_{k+1}\|_{\infty}^{-\lambda_{\infty}}.
 \end{gathered}
\end{equation}
\end{lemma}
\begin{proof}
	By Proposition \ref{I:1:P:5} there exists
	a sequence of primitive points $({\bf v}_k)_{k\geq0}$
	in $\mathbb{Z}^{n+1}$, any two of which are linearly independent,
	satisfying the relations (\ref{I:1:P:5:1})
	for some constant $c>0$.
	Since $\mathcal{S} = \emptyset$,
	then for each $k\geq0$, we have
\begin{align*}
	J_c({\bf v }_k)
	&=
		\Big \{
		X\in \mathbb{R}_{\geq1} \ | \
				X \text{ satisfies }
		1 \leq
		\min
		\big \{
			\frac{X}{\|{\bf v }_k\|_{\infty}},
			\frac{c X^{-\lambda_{\infty}}}{L_{\infty}({\bf v }_k)}
		\big \}
		\Big \}
	\\
	&=
	\Big [
		\|{\bf v }_k\|_{\infty},
		(c /L_{\infty}({\bf v }_k) )^{1/\lambda_{\infty}}
	\Big ]
\end{align*}
	and then the relations (\ref{I:1:P:5:1}) can be written in the form
\[
	(c /L_{\infty}({\bf v }_{k}) )^{1/\lambda_{\infty}}
	<
	\|{\bf v }_{k+2}\|_{\infty}
	\leq
	(c /L_{\infty}({\bf v }_{k+1}) )^{1/\lambda_{\infty}}
	\ \text{ for each } \  k\geq0.
\]
	So, the sequence $({\bf v}_k)_{k\geq0}$
	satisfies the inequalities (\ref{I:1:L:4:1}).
\end{proof}
\subsection{Approximation to p-adic numbers}

	Let $p$ be a prime number.
	In the case where $\mathcal{S} = \{p\}$ and
	$\lambda_{\infty}  = -1$, the condition
	that $(\lambda_{\infty},\lambda_p)$ is an exponent of approximation
	in degree $n$ to a point $(\xi_{\infty},\xi_p) \in \mathbb{R} \times \mathbb{Q}_p$
	is independent of the choice of $\xi_{\infty}$.
	This justifies the following definition.
\begin{definition} \label{I:1:D:4}
	Let $\xi_p \in \mathbb{Q}_p\setminus \mathbb{Q}$ and
	$\lambda_p \in \mathbb{R}$.
	We say that $\lambda_p$ is an exponent of approximation
	in degree $n$ to $\xi_p$ if there exists a constant
	$c > 0$ such that the inequalities
\begin{equation}\label{I:1:D:4:1}
	\| {\bf x} \|_{\infty} \leq X, \quad
	L_{p}({\bf x }) \leq c X^{-\lambda_p},
\end{equation}
	have a non-zero solution ${\bf x} \in \mathbb{Z}^{n+1}$ for any real
	number $X \geq 1$.
\end{definition}
\begin{remark} \label{I:1:2:R:1}
	The criterion presented in Proposition \ref{I:1:P:3} shows
	that $\lambda_p \in \mathbb{R}$
	is an exponent of approximation
	to  $\xi_p \in \mathbb{Q}_p\setminus \mathbb{Q}$
	in degree $n$ 
	if there exists a constant $c > 0$ such that the inequalities
\begin{equation}\label{I:1:2:R:1:1}
	\| {\bf x} \|_{\infty} \leq X, \quad
	\| {\bf x} \|_{\infty}L_{p}({\bf x }) \leq c X^{1-\lambda_p},
\end{equation}
	have a non-zero solution ${\bf x} \in \mathbb{Z}^{n+1}$
	for any real number $X \geq 1$.
\end{remark}

	In this context, Proposition \ref{I:1:P:5}
	leads to the following statement.
\begin{lemma} \label{I:1:L:6}
	Let $\lambda_{p}\in \mathbb{R}_{>0}$ be an exponent of approximation
	in degree $n$ to $\xi_p \in \mathbb{Q}_{p}\setminus \mathbb{Q}$
	and suppose that $\lambda_{p} > 1$.
	There exists a sequence of primitive points
	$({\bf v}_k)_{k \geq 0}$ in $\mathbb{Z}^{n+1}$ such that,
	for each $k\geq0$, we have
\begin{equation}\label{I:1:L:6:1}
 \begin{gathered}
	\|{\bf v}_k\|_{\infty} < \|{\bf v}_{k+1}\|_{\infty}
	\\
	\|{\bf v}_{k+3}\|_{\infty}^{1-\lambda_p}
	\ll
	\|{\bf v}_{k+1}\|_{\infty}L_{p}({\bf v}_{k+1})
	<
	\|{\bf v}_{k}\|_{\infty}L_{p}({\bf v}_{k})
	\ll
	\|{\bf v}_{k+1}\|_{\infty}^{1-\lambda_p}.
 \end{gathered}
\end{equation}
\end{lemma}
\begin{proof}
	As mentioned before Definition \ref{I:1:D:4}, we choose
	any number $\xi_{\infty} \in \mathbb{R} \setminus \mathbb{Q}$ and
	put $\lambda_{\infty} = -1$.
	Then $(\lambda_{\infty},\lambda_{p})$ is an exponent of approximation
	in degree $n$ to $(\xi_{\infty},\xi_{p})$ and by Proposition \ref{I:1:P:5}
	there exists
	a sequence of primitive points $({\bf v}_k)_{k\geq0}$
	in $\mathbb{Z}^{n+1}$, any two of which are linearly independent,
	satisfying the relations (\ref{I:1:P:5:1})
	for some constant $c>0$.
	Also, for each $k\geq0$, we have
\begin{align*}
	J_c({\bf v }_k)
	&=
		\Big \{
		X\in \mathbb{R}_{\geq1} \ | \
				X \text{ satisfies }
		1 \leq
		\min
		\big \{
			\frac{X}{\|{\bf v }_k\|_{\infty}},
			\frac{c X}{L_{\infty}({\bf v }_k)}
		\big \}
		\min
		\big \{
			1,
			\frac{c X^{-\lambda_{p}}}{L_{p}({\bf v }_k)}
		\big \}
		\Big \}.
\end{align*}
	Assuming that the constant $c>0$ is sufficiently large, so that
	the inequality $L_{\infty}({\bf v }_k) \leq c \|{\bf v }_k\|_{\infty}$
	holds for each $k\geq0$, we
	get $\min
		\big \{
			X/\|{\bf v }_k\|_{\infty},
			c X/L_{\infty}({\bf v }_k)
		\big \} = X/\|{\bf v }_k\|_{\infty}$.
	Then, since $\lambda_p>1$, we obtain
\begin{align*}
	J_c({\bf v }_k)
	&=
		\Big \{
		X\in \mathbb{R}_{\geq1} \ | \
				X \text{ satisfies }
		1 \leq
		\frac{X}{\|{\bf v }_k\|_{\infty}}
		\min
		\big \{
			1,
			\frac{c X^{-\lambda_{p}}}{L_{p}({\bf v }_k)}
		\big \}
		\Big \}
	\\
	&=
	\Big [
		\|{\bf v }_k\|_{\infty},
		\Big(\frac{c}{
		\|{\bf v }_k\|_{\infty}
		L_{p}({\bf v }_k)} \Big)^{1/(\lambda_{p}-1)}
	\Big ]
\end{align*}
	and the relations (\ref{I:1:P:5:1}) become
\[
	\Big(\frac{c}{
		\|{\bf v }_k\|_{\infty}
		L_{p}({\bf v }_k)} \Big)^{1/(\lambda_{p}-1)}
	<
	\|{\bf v }_{k+2}\|_{\infty}
	\leq
	\Big(\frac{c}{
		\|{\bf v }_{k+1}\|_{\infty}
		L_{p}({\bf v }_{k+1})} \Big)^{1/(\lambda_{p}-1)}
	\ \text{ for each } \  k\geq0.
\]
	So, the sequence $({\bf v}_k)_{k\geq0}$
	satisfies the inequalities (\ref{I:1:L:6:1}).
\end{proof}
	Note that if
	$\bar \lambda = (\lambda_{\infty},(\lambda_{p})_{{p} \in \mathcal{S}})
	\in \mathbb{R}^{|\mathcal{S}|+1}$ is an exponent of approximation
	in degree $n$
	to $\bar \xi \in
	\mathbb{R} \times \prod_{{p} \in
	\mathcal{S}} \mathbb{Q}_{p}$, then
	$\lambda_{\infty}$ is an exponent of approximation
	to $\xi_{\infty}$ and
	$\lambda_p$ is an exponent of approximation
	to $\xi_p$ for each ${p} \in \mathcal{S}$ in the same degree $n$.

%% file: I-02-inequalities.tex
	From now on, we assume that $n = 2$. An exponent
	of approximation to a point
	$\bar \xi = (\xi_{\infty},(\xi_{p})_{{p} \in \mathcal{S}})
	\in \mathbb{R} \times \prod_{{p} \in \mathcal{S}}\mathbb{Q}_{p}$
	means simply an exponent of approximation in degree $2$
	to this point.

	Any triple ${\bf  x} = (x_{0}, x_{1}, x_{2}) \in \mathbb{Z}^3$
	can be identified with a symmetric matrix
	$
	\left (
	\begin{matrix}
	x_{0}      &  x_{1}
	\\
	x_{1}      &  x_{2}
	\end{matrix}
	\right )
	$ with determinant $\det({\bf  x}) := x_0x_2 - x_1^2$.
	Following \cite{ARNCAI.1}, for points
	${\bf x},{\bf y},{\bf z}\in \mathbb{Z}^3$
	viewed as symmetric matrices,
	we also define
\[
	[{\bf x},{\bf y},{\bf z}] := -{\bf x}J{\bf z}J{\bf y},
	\ \text{ where } \
 J =
		\left (
		\begin{matrix}
			0 &  1
			\\
			-1 &  0
		\end{matrix}
		\right ).
\]
	We recall from \cite{ARNCAI.1} that $[{\bf x},{\bf y},{\bf z}]$
	is also a symmetric matrix if ${\bf x},{\bf y}$ and ${\bf z}$
	are linearly independent over $\mathbb{Q}$.
	It then corresponds to a new point ${\bf w}\in \mathbb{Z}^3$.
	The next lemma provides most of the estimates
	used throughout the thesis.
\begin{lemma} \label{I:2:L:1}
	Let ${\bf  x, y, z} \in \mathbb{Z}^3$
	and $\nu \in \mathcal{S}\cup\{\infty\}$.
\begin{itemize} 
     \item[(i)] For the determinants $\det({\bf x})$ and
	$\det({\bf x}, {\bf y}, {\bf z})$ we have the following estimates
\begin{gather*}
	| \det({\bf x}) |_{\nu} \ll \| {\bf x} \|_{\nu}L_{\nu}({\bf x }),
	\\
	| \det({\bf x}, {\bf y}, {\bf z}) |_{\nu} \leq
			\| {\bf x} \|_{\nu} L_{\nu}({\bf y }) L_{\nu}({\bf z })+
			\| {\bf y} \|_{\nu} L_{\nu}({\bf x }) L_{\nu}({\bf z })+
			\| {\bf z} \|_{\nu} L_{\nu}({\bf x }) L_{\nu}({\bf y }).
\end{gather*}
    \item[(ii)] Let ${\bf w} = [{\bf x},{\bf x},{\bf y}]$, then
\begin{gather*}
	\| {\bf w} \|_{\nu} \ll
	\max\{ \| {\bf y} \|_{\nu}L_{\nu}({\bf x })^2,\| {\bf x} \|_{\nu}^2L_{\nu}({\bf y }) \},
	\\
	L_{\nu}({\bf w }) \ll
	L_{\nu}({\bf x })\max\{\| {\bf y} \|_{\nu}L_{\nu}({\bf x }),\| {\bf x} \|_{\nu}L_{\nu}({\bf y }) \}.
\end{gather*}
  \item[(iii)] Let $V$ be a subspace of $\mathbb{Q}^2$ and let ${\bf x},{\bf y} \in \mathbb{Z}^3$
 be a basis of $V$ over $\mathbb{Q}$.
 Then its height $H(V)$, satisfies
\[
	H(V) \ll
	\max\{\| {\bf x} \|_{\infty} L_{\infty}({\bf y}),
		\|{\bf y} \|_{\infty} L_{\infty}({\bf x}) \}
	\prod_{{q} \in \mathcal{S}}
	\max\{L_{q}({\bf x }),L_{q}({\bf y })\}.
\]
  \end{itemize} 
\end{lemma}
\begin{proof}
	(i): Since
\[
	|x_{2} - x_{1} \xi_{\nu}|_{\nu} \leq
	|x_{2} - x_{0}\xi_{\nu}^2|_{\nu} + 
	|x_{1} - x_{0}\xi_{\nu}|_{\nu} |\xi_{\nu}|_{\nu}
	\ll L_{\nu}({\bf x }),
\]
	we get by multilinearity of determinants,
\begin{align*}
	| \det({\bf x}) |_{\nu} &=
		\left \|
		\begin{matrix}
		x_{0}      &  x_{1} - x_{0} \xi_{\nu}
		\\
		x_{1}      &  x_{2} - x_{1} \xi_{\nu}
		\end{matrix}
		\right \|_{\nu} \leq
		|x_{0}|_{\nu} |x_{2} - x_{1} \xi_{\nu}|_{\nu}+
		|x_{1}|_{\nu} |x_{1} - x_{0} \xi_{\nu}|_{\nu}
		\ll  \| {\bf x} \|_{\nu} L_{\nu}({\bf x }).
\end{align*}
	Similarly we obtain following the estimates
\begin{align*}
	| \det({\bf x}, {\bf y}, {\bf z}) |_{\nu}
	&=
		\left \|
		\begin{matrix}
		x_{0}      &  x_{1} - x_{0} \xi_{\nu} &  x_{2} - x_{0} \xi_{\nu}^2
		\\
		y_{0}      &  y_{1} - y_{0} \xi_{\nu} &  y_{2} - y_{0} \xi_{\nu}^2
		\\
		z_{0}      &  z_{1} - z_{0} \xi_{\nu} &  z_{2} - z_{0} \xi_{\nu}^2
		\end{matrix}
		\right \|_{\nu}
		\\
		&\leq
			\| {\bf x} \|_{\nu} L_{\nu}({\bf y }) L_{\nu}({\bf z }) +
			\| {\bf y} \|_{\nu} L_{\nu}({\bf x }) L_{\nu}({\bf z }) +
			\| {\bf z} \|_{\nu} L_{\nu}({\bf x }) L_{\nu}({\bf y }).
\end{align*}
	(ii): By the computations in \cite{ARNCAI.1}, p.~45, we have
\[
	{\bf w} = [{\bf x},{\bf x},{\bf y}] = -{\bf x}J{\bf y}J{\bf x} =
		\left (
		\begin{matrix}
			\left |
			\begin{matrix}
				x_{0} &  x_{1}
				\\
				\left |
				\begin{matrix}
				x_{0} &  x_{1}
				\\
				y_{0} &  y_{1}
				\end{matrix}
				\right |
				&
				\left |
				\begin{matrix}
				x_{0} &  x_{1}
				\\
				y_{1} &  y_{2}
				\end{matrix}
				\right |
			\end{matrix}
			\right |
		&
			\left |
			\begin{matrix}
				x_{0} &  x_{1}
				\\
				\left |
				\begin{matrix}
				x_{1} &  x_{2}
				\\
				y_{0} &  y_{1}
				\end{matrix}
				\right |
				&
				\left |
				\begin{matrix}
				x_{1} &  x_{2}
				\\
				y_{1} &  y_{2}
				\end{matrix}
				\right |
			\end{matrix}
			\right |
		\\\\
			\left |
			\begin{matrix}
				x_{1} &  x_{2}
				\\
				\left |
				\begin{matrix}
				x_{0} &  x_{1}
				\\
				y_{0} &  y_{1}
				\end{matrix}
				\right |
				&
				\left |
				\begin{matrix}
				x_{0} &  x_{1}
				\\
				y_{1} &  y_{2}
				\end{matrix}
				\right |
			\end{matrix}
			\right |
		&
			\left |
			\begin{matrix}
				x_{1} &  x_{2}
				\\
				\left |
				\begin{matrix}
				x_{1} &  x_{2}
				\\
				y_{0} &  y_{1}
				\end{matrix}
				\right |
				&
				\left |
				\begin{matrix}
				x_{1} &  x_{2}
				\\
				y_{1} &  y_{2}
				\end{matrix}
				\right |
			\end{matrix}
			\right |
		\end{matrix}
		\right ).
\]
	By multilinearity of determinants, ${\bf w}$ can be presented in the following form
\[
		\left (
		\begin{matrix}
			\left |
			\begin{matrix}
				x_{0} &  x_{1} - x_{0}\xi_{\nu}
				\\
				\left |
				\begin{matrix}
				x_{0} &  x_{1} - x_{0}\xi_{\nu}
				\\
				y_{0} &  y_{1} - y_{0}\xi_{\nu}
				\end{matrix}
				\right |
				&
				\left |
				\begin{matrix}
				x_{0} &  x_{1}
				\\
				y_{1} - y_{0}\xi_{\nu} &  y_{2} - y_{1}\xi_{\nu}
				\end{matrix}
				\right |
			\end{matrix}
			\right |
		&
			\left |
			\begin{matrix}
				x_{0} &  x_{1} - x_{0}\xi_{\nu}
				\\
				\left |
				\begin{matrix}
				x_{1} &  x_{2} - x_{1}\xi_{\nu}
				\\
				y_{0} &  y_{1} - y_{0}\xi_{\nu}
				\end{matrix}
				\right |
				&
				\left |
				\begin{matrix}
				x_{1} &  x_{2}
				\\
				y_{1} - y_{0}\xi_{\nu} &  y_{2} - y_{1}\xi_{\nu}
				\end{matrix}
				\right |
			\end{matrix}
			\right |
		\\\\
			\left |
			\begin{matrix}
				x_{1} &  x_{2} - x_{1}\xi_{\nu}
				\\
				\left |
				\begin{matrix}
				x_{0} &  x_{1} - x_{0}\xi_{\nu}
				\\
				y_{0} &  y_{1} - y_{0}\xi_{\nu}
				\end{matrix}
				\right |
				&
				\left |
				\begin{matrix}
				x_{0} &  x_{1}
				\\
				y_{1} - y_{0}\xi_{\nu} &  y_{2} - y_{1}\xi_{\nu}
				\end{matrix}
				\right |
			\end{matrix}
			\right |
		&
			\left |
			\begin{matrix}
				x_{1} &  x_{2} - x_{1}\xi_{\nu}
				\\
				\left |
				\begin{matrix}
				x_{1} &  x_{2} - x_{1}\xi_{\nu}
				\\
				y_{0} &  y_{1} - y_{0}\xi_{\nu}
				\end{matrix}
				\right |
				&
				\left |
				\begin{matrix}
				x_{1} &  x_{2}
				\\
				y_{1} - y_{0}\xi_{\nu} &  y_{2} - y_{1}\xi_{\nu}
				\end{matrix}
				\right |
			\end{matrix}
			\right |
		\end{matrix}
		\right ).
\]
	Hence, since
	$|x_{l+1} - x_{l} \xi_{\nu}|_{\nu} \ll L_{\nu}({\bf x })$ and
	$|y_{l+1} - y_{l} \xi_{\nu}|_{\nu} \ll L_{\nu}({\bf y })$, for $l = 0,1$, we deduce that
\[
	\| {\bf w} \|_{\nu} \ll
	\max\{ \| {\bf y} \|_{\nu}L_{\nu}({\bf x })^2,\| {\bf x} \|_{\nu}^2L_{\nu}({\bf y }) \}.
\]
	We now find the upper bound for $|w_1 - w_0\xi_{\nu}|_{\nu}$.
	Using the above presentation of ${\bf w}$, we find that
	$w_1 - w_0\xi_{\nu}$ can be written in the form
\begin{align*}
		&\left |
			\begin{matrix}
				x_{0} &  x_{1} - x_{0}\xi_{\nu}
				\\
				\left |
				\begin{matrix}
				x_{1} &  x_{2} - x_{1}\xi_{\nu}
				\\
				y_{0} &  y_{1} - y_{0}\xi_{\nu}
				\end{matrix}
				\right |
				&
				\left |
				\begin{matrix}
				x_{1} &  x_{2}
				\\
				y_{1} - y_{0}\xi_{\nu} &  y_{2} - y_{1}\xi_{\nu}
				\end{matrix}
				\right |
			\end{matrix}
			\right |
		-
			\left |
			\begin{matrix}
				x_{0} &  x_{1} - x_{0}\xi_{\nu}
				\\
				\left |
				\begin{matrix}
				x_{0} &  x_{1} - x_{0}\xi_{\nu}
				\\
				y_{0} &  y_{1} - y_{0}\xi_{\nu}
				\end{matrix}
				\right |
				&
				\left |
				\begin{matrix}
				x_{0} &  x_{1}
				\\
				y_{1} - y_{0}\xi_{\nu} &  y_{2} - y_{1}\xi_{\nu}
				\end{matrix}
				\right |
			\end{matrix}
			\right | \xi_{\nu} \quad
		\\
			&
			= \left |
			\begin{matrix}
				x_{0} &  x_{1} - x_{0}\xi_{\nu}
				\\
				\left |
				\begin{matrix}
				x_{1} - x_{0}\xi_{\nu} &  (x_{2} - x_{1}\xi_{\nu}) - (x_{1} - x_{0}\xi_{\nu})\xi_{\nu}
				\\
				y_{0} &  y_{1} - y_{0}\xi_{\nu}
				\end{matrix}
				\right |
				&
				\left |
				\begin{matrix}
				x_{1} - x_{0}\xi_{\nu} &  x_{2} - x_{1}\xi_{\nu}
				\\
				y_{1} - y_{0}\xi_{\nu} &  y_{2} - y_{1}\xi_{\nu}
				\end{matrix}
				\right |
			\end{matrix}
			\right |,
\end{align*}
	and so
\[
	| w_1 - w_0\xi_{\nu} |_{\nu} \ll \max\{\| {\bf y} \|_{\nu}L_{\nu}({\bf x })^2,\| {\bf x} \|_{\nu}L_{\nu}({\bf x })L_{\nu}({\bf y }) \}.
\]
	Similarly, the same upper bound holds for $|w_2 - w_0\xi_{\nu}^2|_{\nu}$,
	and therefore
\[
	L_{\nu}({\bf w }) \ll L_{\nu}({\bf x })\max\{\| {\bf y} \|_{\nu}L_{\nu}({\bf x }),\| {\bf x} \|_{\nu}L_{\nu}({\bf y }) \}.
\]
	(iii): Recall that 
\[
	H(V) = \prod_{{\nu} } \|{\bf x }\wedge{\bf y }\|_{\nu},
\]
	where $\nu$ runs through all prime numbers and $\infty$.
        Thus 
\[
        H(V) \leq \| {\bf x} \wedge {\bf y} \|_{\infty} 
	          \prod_{{q} \in \mathcal{S}}\| {\bf x} \wedge {\bf y} \|_q,
\]
	and so we simply need upper bounds for $\| {\bf x} \wedge {\bf y} \|_{\infty}$
	and $\| {\bf x} \wedge {\bf y} \|_q \ ({q} \in \mathcal{S})$. Using the presentation
\begin{align*}
	{\bf x} \wedge {\bf y} &= \Big (
				\left |
				\begin{matrix}
				x_{1} &  x_{2}
				\\
				y_{1} &  y_{2}
				\end{matrix}
				\right |
				,
				-\left |
				\begin{matrix}
				x_{0} &  x_{2}
				\\
				y_{0} &  y_{2}
				\end{matrix}
				\right |
				,
				\left |
				\begin{matrix}
				x_{0} &  x_{1}
				\\
				y_{0} &  y_{1}
				\end{matrix}
				\right |
				\Big )
			\\
			&=
				\Big (
				\left |
				\begin{matrix}
				x_{1} &  x_{2} - x_{1}\xi_{\nu}
				\\
				y_{1} &  y_{2} - x_{1}\xi_{\nu}
				\end{matrix}
				\right |
				,
				-\left |
				\begin{matrix}
				x_{0} &  x_{2} - x_{0}\xi_{\nu}^2
				\\
				y_{0} &  y_{2} - y_{0}\xi_{\nu}^2
				\end{matrix}
				\right |
				,
				\left |
				\begin{matrix}
				x_{0} &  x_{1} - x_{0}\xi_{\nu}
				\\
				y_{0} &  y_{1} - y_{0}\xi_{\nu}
				\end{matrix}
				\right |
				\Big ),
\end{align*}
	we find the estimates
\begin{align*}
	& \| {\bf x} \wedge {\bf y} \|_{\infty} \ll
	\max\{\| {\bf x} \|_{\infty} L_{\infty}({\bf y}),
		\|{\bf y} \|_{\infty} L_{\infty}({\bf x}) \},
	\\
	& \| {\bf x} \wedge {\bf y} \|_q \ll
		\max\{ L_{q}({\bf x }),L_{q}({\bf y })\} \ ({q} \in \mathcal{S}),
\end{align*}
	
	and thus
\[
	H(V) \ll \max\{\| {\bf x} \|_{\infty} L_{\infty}({\bf y}),
		\|{\bf y} \|_{\infty} L_{\infty}({\bf x}) \}
	\prod_{{q} \in \mathcal{S}} \max\{L_{q}({\bf x }),L_{q}({\bf y })\}.
\]

\end{proof}
 

%% file: I-03-real-padic.tex

	Let $\mathcal{S}$ be a finite set of prime numbers.
	Here we consider the problem of simultaneous approximation
	to real and p-adic numbers in degree $n=2$.
	We find constraints on
	$\bar \lambda \in \mathbb{R}^{|\mathcal{S}|+1}$
	and $\bar \xi \in \mathbb{R}\times \prod_{{p} \in
	\mathcal{S}}\mathbb{Q}_{p}$
	which ensure that, for some constant $c>0$, the inequalities
\[
	\| {\bf x} \|_{\infty} \leq X,
	\
	L_{\infty}({\bf x }) \leq c X^{-\lambda_{\infty}},
	\
	L_{p}({\bf x }) \leq c X^{-\lambda_{p}} \
	\forall {p} \in \mathcal{S},
\]
 	have no non-zero solution ${\bf x} \in \mathbb{Z}^3$
	for arbitrarily large values of $X$.
\subsection{Simultaneous case}

	Throughout this paragraph, we fix
	a finite set $\mathcal{S}$ of prime numbers,
	a point
\[
	\bar \xi = \big(\xi_{\infty},(\xi_{p})_{{p} \in \mathcal{S}}\big)
	\in (\mathbb{R}\setminus\mathbb{Q}) \times \prod_{{p} \in
	\mathcal{S}}(\mathbb{Q}_{p}\setminus\mathbb{Q}),
\]
	and a point
\[
	\bar \lambda  = \big(\lambda_{\infty},(\lambda_{p})_{{p} \in
	\mathcal{S}}\big) \in
	[-1,\infty) \times \mathbb{R}_{\geq0}^{|\mathcal{S}|}.
\]
	We define $\mathcal{S'}$ to be the set (possibly empty) of all
	$\nu \in \{\infty\} \cup \mathcal{S}$ such that
	$[\mathbb{Q}(\xi_{\nu})\colon\mathbb{Q}] >2$.
	We also define
\[
	\lambda := \sum_{\nu \in \{\infty\}\cup\mathcal{S}}\lambda_{\nu}.
\]
\begin{proposition} \label{I:3:T:1}
	Suppose that
\begin{equation}\label{I:3:T:1:1-1}
	\lambda > 0
	\ \text{ and } \
	\lambda
	+ \sum_{\nu \in \mathcal{S}'}\lambda_{\nu}  >
	\left \{
	\begin{matrix}
		0 \ \text{ if } \ \infty \in \mathcal{S}',
		\\
		1 \ \text{ if } \ \infty \notin \mathcal{S}'.
	\end{matrix}
	\right.
\end{equation}
	Suppose also that there exists a constant $c>0$ for which
	the inequalities
\begin{equation}\label{I:3:T:1:0}
 \begin{aligned}
	&\| {\bf x} \|_{\infty} \leq X,
	\\
	&L_{\infty}({\bf x }) \leq c X^{-\lambda_{\infty}},
	\\
	&L_{p}({\bf x }) \leq c X^{-\lambda_{p}} \
	\forall {p} \in \mathcal{S},
 \end{aligned}
\end{equation}
	have a non-zero solution ${\bf x} = (x_0,x_1,x_2) \in \mathbb{Z}^3$
	for each $X$ sufficiently large.
	Suppose finally that, for each $X$ sufficiently large,
	any such solution has
\begin{equation}\label{I:3:T:1:1}
	\det({\bf x}) = x_0x_2 - x_1^2 \neq  0.
\end{equation}
	Then, we have $\lambda \leq 1/{\gamma}$.
	Moreover, if $\lambda = 1/{\gamma}$, then
	$c$ is bounded from below by a positive constant
	depending only on ${\bar \xi }$.
\end{proposition}
\begin{proof}
	WLOG, we may assume that $0<c\leq1$.
	The hypotheses imply that $\bar\lambda$
	is an exponent of approximation to $\bar\xi$ in degree $2$,
	with corresponding constant $c$.
	Proposition \ref{I:1:P:5} applies to this situation
	with $n=2$, as the main condition $\lambda>0$ is fulfilled.

	Consider the sequences $({\bf v}_k)_{k\geq0}$, $({\bf x}_k)_{k\geq0}$,
	$({\bf x'}_k)_{k\geq0}$ and $(X_k)_{k\geq1}$
	given by Proposition \ref{I:1:P:5}.
	For all $k$ sufficiently large, the assumption
	(\ref{I:3:T:1:1}) implies that $\det({\bf x}_k) \neq  0$.
	Using Part (i) of Lemma \ref{I:2:L:1}
	and the first relation in (\ref{I:1:P:5:1:1}), we deduce that
\begin{align*}
 	1 &\leq |\det({\bf x}_k)|_{\infty}
		\prod_{p \in \mathcal{S}}|\det({\bf x}_k)|_{p}
		\\
		&\ll
		X_k L_{\infty}({\bf x}_k)
		\prod_{p \in \mathcal{S}}L_{p}({\bf x}_k)
		=
		X_k L_{\infty}({\bf x'}_k)
		\prod_{p \in \mathcal{S}}L_{p}({\bf x'}_k),
\end{align*}
	for all these values of $k$,
	with implied constants depending only on
	$\bar\xi$ and not on $c$ (same through all the proof).
	Using this and the second relations in (\ref{I:1:P:5:1:1}),
	we get
\begin{equation}\label{I:3:T:1:2}
 \begin{aligned}
 	1 &\ll	X_k L_{\infty}({\bf x'}_k)\prod_{p \in \mathcal{S}}L_{p}({\bf x'}_k)
		\\
		&\ll
		X_k c^{|\mathcal{S}|+1}
		X_{k+1}^{-\lambda_{\infty}}
		\prod_{p \in \mathcal{S}}X_{k+1}^{-\lambda_{p}}
		=
		c^{|\mathcal{S}|+1}X_kX_{k+1}^{-\lambda},
 \end{aligned}
\end{equation}
	for all $k$ sufficiently large.

	Now, we claim that for infinitely many $k\geq1$, the points
	${\bf v}_{k-1}, {\bf v}_{k}, {\bf v}_{k+1}$ are linearly independent
	over $\mathbb{Q}$.
	To prove this, we argue like in \cite{DS} and \cite{Teu}, assuming
	on the contrary that
	${\bf v}_{k-1}, {\bf v}_{k}, {\bf v}_{k+1}$ are linearly dependent
	for each $k$ sufficiently large.
	Since any two different points of $({\bf v}_k)_{k\geq0}$
	are linearly independent, it follows that
\[
	\langle{\bf v}_{k-1},{\bf v}_{k}\rangle_{\mathbb{Q}}
	=
	\langle{\bf v}_{k-1}, {\bf v}_{k}, {\bf v}_{k+1}\rangle_{\mathbb{Q}}
	=
	\langle{\bf v}_k,{\bf v}_{k+1}\rangle_{\mathbb{Q}},
\]
	for each $k$ sufficiently large.
	Hence, there is a two dimensional subspace
	$V$ in $\mathbb{Q}^3$ such that
	$V = \langle{\bf v}_k,{\bf v}_{k+1}\rangle_{\mathbb{Q}}$
	for each $k$ sufficiently large.
	There exist integers $r,s,t \in \mathbb{Z}$, not all zero,
	such that
\[
	V = 	\{
			{\bf x} = (x_0,x_1,x_2) \in \mathbb{Q}^3 \ | \
			r x_0 + s x_1 + t x_2 = 0
		\}.
\]
	Since ${\bf v}_k \in V$ for each $k$ sufficiently large, then
	for these values of $k$, we get
\begin{equation}\label{I:3:T:1:2-0}
	r v_{k,0} + s v_{k,1} + t v_{k,2} = 0.
\end{equation}
	Fix $\nu \in \mathcal{S}'$.
	Since $[\mathbb{Q}(\xi_{\nu})\colon\mathbb{Q}] >2$, we have
\begin{equation}\label{I:3:T:1:2-1}
	r  + s \xi_{\nu} + t \xi_{\nu}^2 \neq 0.
\end{equation}
	Using this and (\ref{I:3:T:1:2-0}),
	we find that, for those values of $k$,
\[
	|v_{k,0}|_{\nu} \ll | s(v_{k,1} - \xi_{\nu}{v_{k,0}}) +
	t(v_{k,2} - \xi_{\nu}^2{v_{k,0}}) |_{\nu}
	\ll L_{\nu}({\bf v}_k)
\]
	and therefore
\[
	\|{\bf v}_k\|_{\nu} \ll L_{\nu}({\bf v}_k).
\]
	Finally, since ${\bf x'}_k$ and ${\bf x}_{k+1}$
	are integer multiples of ${\bf v}_k$ and ${\bf v}_{k+1}$
	and since they both belong to $\mathcal{C}_{c',X_{k+1}}$,
	where $c'=c\max_{p \in \mathcal{S}}p\leq\max_{p \in \mathcal{S}}p$,
	we deduce that
\begin{equation}\label{I:3:T:1:2-2:0}
	\|{\bf x'}_k\|_{\nu} \ll L_{\nu}({\bf x'}_k)
		\ll X_{k+1}^{-\lambda_{\nu}}
	\quad 
	\text{and}
	\quad
	\|{\bf x}_{k+1}\|_{\nu} \ll L_{\nu}({\bf x}_{k+1})
		\ll X_{k+1}^{-\lambda_{\nu}},
\end{equation}
	for each $\nu \in \mathcal{S}'$.

	Fix any index $k\geq0$.
	By Proposition \ref{I:1:P:5}, the points ${\bf x'}_k$
	and ${\bf x}_{k+1}$ are linearly independent.
	This means that the matrix
\[
	\left (
	\begin{matrix}
		x_{k,0}'	&  x_{k,1}'	&  x_{k,2}'
		\\
		x_{k+1,0}	&  x_{k+1,1}	&  x_{k+1,2}
	\end{matrix}
	\right ) {,}
\]
	has rank 2. So, there exist $i,j \in \{0,1,2\}$
	with $i < j$, such that
\[
	\left |
	\begin{matrix}
		x_{k,i}' &  x_{k,j}'
		\\
		x_{k+1,i}	&  x_{k+1,j}
	\end{matrix}
	\right | \neq 0.
\]
	Using the product formula and the fact that
	${\bf x'}_k$ and ${\bf x}_{k+1}$ both belong to
	$\mathcal{C}_{c',X_{k+1}}$, we find that
\begin{equation}\label{I:3:T:1:2-3}
 \begin{aligned}
	1 & \leq
		\left \|
			\begin{matrix}
			x_{k,i}' &  x_{k,j}'
			\\
			x_{k+1,i}&  x_{k+1,j}
			\end{matrix}
		\right\|_{\infty}
		\prod_{p \in \mathcal{S}}
		\left \|
			\begin{matrix}
			x_{k,i}' &  x_{k,j}'
			\\
			x_{k+1,i}&  x_{k+1,j}
			\end{matrix}
		\right\|_{p}
	\\
	& \leq
		\Big(\|{\bf x'}_k\|_{\infty}L_{\infty}({\bf x}_{k+1}) +
		\|{\bf x}_{k+1}\|_{\infty}L_{\infty}({\bf x'}_k)\Big)
		\\
		& \quad \quad \prod_{p \in \mathcal{S}}
		\max\{\|{\bf x'}_k\|_{p}L_{p}({\bf x}_{k+1}),
		\|{\bf x}_{k+1}\|_{p}L_{p}({\bf x'}_k)\}
	\\
	& \ll
		X_{k+1}^{-\lambda_{\infty}}\big(\|{\bf x'}_k\|_{\infty} +
		\|{\bf x}_{k+1}\|_{\infty}\big)
		\prod_{p \in \mathcal{S}}
		X_{k+1}^{-\lambda_{p}}\max\{\|{\bf x'}_k\|_{p},
		\|{\bf x}_{k+1}\|_{p}\}
		\\
	& =
		X_{k+1}^{-\lambda}\big(\|{\bf x'}_k\|_{\infty} +
		\|{\bf x}_{k+1}\|_{\infty}\big)
		\prod_{p \in \mathcal{S}}\max\{\|{\bf x'}_k\|_{p},
		\|{\bf x}_{k+1}\|_{p}\}.
 \end{aligned}
\end{equation}

	If $\infty \in \mathcal{S}'$,
	this estimate combined with (\ref{I:3:T:1:2-2:0}) gives
 \[
	1 \ll
		X_{k+1}^{-\lambda}X_{k+1}^{-\lambda_{\infty}}
		\prod_{p \in \mathcal{S}'\setminus\{\infty\}}
		X_{k+1}^{-\lambda_{p}}
	 	=
		X_{k+1}^{-\lambda
		- \sum_{\nu \in \mathcal{S}'}\lambda_{\nu}}.
 \]
	Since $\bar \lambda$ satisfies inequalities (\ref{I:3:T:1:1-1}),
	this is impossible for $k$ large enough.

	If $\infty \notin \mathcal{S}'$, we find instead that
 \[
	1 \ll
		X_{k+1}^{-\lambda}X_{k+1}
		\prod_{p \in \mathcal{S}'}
		X_{k+1}^{-\lambda_{p}}
	 =
		X_{k+1}^{-\lambda + 1
		- \sum_{p \in \mathcal{S}'}\lambda_{p}},
 \]
 	using the fact that
	$\|{\bf x'}_{k}\|_{\infty},\|{\bf x}_{k+1}\|_{\infty} \leq X_{k+1}$.
	By (\ref{I:3:T:1:1-1}), this leads again to a contradiction
	for $k$ large enough.
	So, we proved the claim.

	Thus, for infinitely
	many values of $k$, we have
	$\det({\bf x'}_{k-1},{\bf x'}_{k},{\bf x}_{k+1}) \neq 0$.
	Combining the product formula with Lemma \ref{I:2:L:1}(i), we get
	for those values of $k$ the following
 \begin{align*}
	1
	&\leq | \det({\bf x'}_{k-1},{\bf x'}_{k},{\bf x}_{k+1}) |_{\infty}
	\prod_{p \in \mathcal{S}}|
	\det({\bf x'}_{k-1},{\bf x'}_{k},{\bf x}_{k+1}) |_p
	\\
	&\ll
	\big(	\|{\bf x'}_{k-1}\|_{\infty}
		L_{\infty}({\bf x'}_{k})
		L_{\infty}({\bf x}_{k+1})
		+
		\|{\bf x'}_{k}\|_{\infty}
		L_{\infty}({\bf x'}_{k-1})
		L_{\infty}({\bf x}_{k+1})
		+
		\|{\bf x}_{k+1}\|_{\infty}
		L_{\infty}({\bf x'}_{k-1})
		L_{\infty}({\bf x'}_{k})
	\big)
	\\
	& \quad \quad
	\prod_{p \in \mathcal{S}}
	\max
	\{
		L_{p}({\bf x'}_{k}) L_{p}({\bf x}_{k+1}),
		L_{p}({\bf x'}_{k-1}) L_{p}({\bf x}_{k+1}),
		L_{p}({\bf x'}_{k-1}) L_{p}({\bf x'}_{k})
	\}.
 \end{align*}
	Since ${\bf x'}_{k-1} \in \mathcal{C}_{c',X_{k}}$
	while  ${\bf x'}_{k}, {\bf x}_{k+1} \in \mathcal{C}_{c',X_{k+1}}$,
	and since $c'=c\max_{p \in \mathcal{S}}p$, this gives
\begin{equation}\label{I:3:T:1:3-0}
 \begin{aligned}
	1
	&\ll
	c^{2(|\mathcal{S}|+1)}
	(
		X_{k}X_{k+1}^{-\lambda_{\infty}} X_{k+1}^{-\lambda_{\infty}} +
		X_{k+1}X_{k}^{-\lambda_{\infty}} X_{k+1}^{-\lambda_{\infty}} +
		X_{k+1}X_{k}^{-\lambda_{\infty}} X_{k+1}^{-\lambda_{\infty}}
	)
	\\
	& \quad \quad
	\prod_{p \in \mathcal{S}}
	\max
	\{
		X_{k+1}^{-\lambda_{p}} X_{k+1}^{-\lambda_{p}},
		X_{k}^{-\lambda_{p}} X_{k+1}^{-\lambda_{p}},
		X_{k}^{-\lambda_{p}} X_{k+1}^{-\lambda_{p}}
	\}.
 \end{aligned}
\end{equation}
	Note that, since $\lambda_{\infty}\geq-1$, we have
	$X_{k}^{1+\lambda_{\infty}} \leq X_{k+1}^{1+\lambda_{\infty}}$
	and so
	$X_{k}X_{k+1}^{-\lambda_{\infty}} X_{k+1}^{-\lambda_{\infty}}
	\leq
	X_{k+1}X_{k}^{-\lambda_{\infty}} X_{k+1}^{-\lambda_{\infty}}$
	for each $k\geq0$.
	Combining this with (\ref{I:3:T:1:3-0})
	and recalling that $\lambda_p\geq0$
	for each $p \in \mathcal{S}$, we find that
\begin{equation}\label{I:3:T:1:3}
	1 \leq
	c^{2(|\mathcal{S}|+1)}
	X_{k}^{-\lambda_{\infty}} X_{k+1}^{1-\lambda_{\infty}}
	\prod_{p \in \mathcal{S}}X_{k}^{-\lambda_{p}} X_{k+1}^{-\lambda_{p}}
	=
	c^{2(|\mathcal{S}|+1)}
	X_{k}^{-\lambda} X_{k+1}^{1-\lambda}.
\end{equation}
	Since $\lambda>0$, multiplying
	(\ref{I:3:T:1:3}) by (\ref{I:3:T:1:2}) raised to the
	power $\lambda$, we get
\[
	c^{(2 + \lambda)(|\mathcal{S}|+1)}
	X_{k+1}^{-\lambda^2 + 1 - \lambda} \gg 1.
\]
	So, we conclude that $-\lambda^2 + 1 - \lambda \geq 0$,
	which means that $\lambda \leq 1/{\gamma}$.
	Moreover, if $-\lambda^2 + 1 - \lambda = 0$, which means that
	$\lambda = 1/{\gamma}$, this gives $c\gg1$.
\end{proof}
\begin{remark}\label{I:3:R:1}
	The above proof shows that,
	under hypotheses of Proposition \ref{I:3:T:1},
	the points ${\bf v}_{k-1}, {\bf v}_{k}, {\bf v}_{k+1}$,
	of the sequence $({\bf v}_k)_{k\geq0}$ given by Proposition \ref{I:1:P:5},
	are linearly independent over $\mathbb{Q}$
	for infinitely many $k\geq1$.
\end{remark}
	In the case where $\mathcal{S} = \emptyset$,
	the following statement implies Lemma 2 of
	{\sc H.~Davenport \& W.M.~Schmidt} in \cite{DS}.
\begin{proposition} \label{I:3:P:1}
	Assume that
\[
	\lambda_{\infty} \geq 0,
	\
	\lambda> 0
	\ \text{ and }
	\lambda + \lambda_{\infty} > 1.
\]
	Suppose also that, for some $c>0$, the system
\begin{equation}\label{I:3:P:1:1}
 \begin{aligned}
	&\| {\bf x} \|_{\infty} \leq X,
	\\
	&L_{\infty}({\bf x }) \leq c X^{-\lambda_{\infty}},
	\\
	&L_{p}({\bf x }) \leq c X^{-\lambda_{p}} \
	\forall {p} \in \mathcal{S},
 \end{aligned}
\end{equation}
	has a non-zero solution ${\bf x} \in \mathbb{Z}^3$,
	for each $X\geq1$.
	Then, for each $X$ sufficiently large,
	any such solution ${\bf x}$
	satisfies $\det({\bf x}) \neq 0$.
\end{proposition}
	As we indicated above, if $\mathcal{S} = \emptyset$,
	the condition $\lambda + \lambda_{\infty} > 1 $
	becomes $\lambda_{\infty} > 1/2 $
	and we recover Lemma 2 of \cite{DS}.
\\
\begin{proof}
	Suppose that ${\bf x}\in \mathbb{Z}^3$ is a non-zero solution
	of (\ref{I:3:P:1:1}) for some large real number $X$.
	Let ${\bf v}$ be a primitive point of $\mathbb{Z}^3$
	of which ${\bf x}$ is a multiple.
	By  of Lemma \ref{I:1:L:1}(iv), we have that
	$\|{\bf v}\|_{\infty}$ tends to infinity with $X$.
	In particular, we have $\|{\bf v}\|_{\infty} > 1$
	if $X$ is sufficiently large.
	Assuming, as we may, that this is the case,
	put $Y = \min I({\bf v})$ and choose a point
	${\bf x'} \in \mathcal{L}_c({\bf v},Y)$.
	This choice means in particularly that
	${\bf x'}$ is a non-zero solution of the system (\ref{I:3:P:1:1})
	with $X$ replaced by $Y$. Since $\lambda_{\infty}\geq0$,
	the point ${\bf x'}$ is also a solution of
	the system (\ref{I:3:P:1:1}) with $X$ replaced by $\|{\bf x'}\|_{\infty}$,
	and then $\|{\bf x'}\|_{\infty} \in I_c({\bf v})$.
	Since $\|{\bf x'}\|_{\infty}\leq Y$ and $Y = \min I_c({\bf v})$,
	we conclude that $\|{\bf x'}\|_{\infty}=Y$.
	Since $Y \geq \|{\bf v}\|_{\infty} > 1$,
	Lemma \ref{I:1:L:1}(iii) ensures
	the existence of a primitive point ${\bf u} \in \mathbb{Z}^3$
	such that
	$Y\in I_c({\bf u})$ and $\min I_c({\bf u})<Y$.
	Choose a point ${\bf y} \in \mathcal{L}_c({\bf u},Y)$.
	Since $\min I({\bf u})<Y= \min I({\bf v})$, we have
	${\bf u} \neq \pm {\bf v}$, so ${\bf u}$ and ${\bf v}$
	are linearly independent over $\mathbb{Q}$ and thus
	${\bf x'}$ and ${\bf y}$ are linearly independent
	points of $\mathcal{C}_{c,Y}$.
	
	Now, suppose that $\det({\bf x}) = 0$.
	Then we get $\det({\bf v}) = 0$, and so
\[
	{\bf v} = \pm(a^2,ab,b^2),
\]
	for some non-zero $(a,b) \in \mathbb{Z}^2$.
	Since ${\bf v}$ is primitive,
	the point ${\bf s} = (a,b)$ is also primitive,
	and therefore $\|{\bf s}\|_{p} = 1$
	for each ${p} \in \mathcal{S}$.
	Writing ${\bf x'} = m{\bf v}$ with $m\in\mathbb{Z}$,
	we find that
\begin{equation}\label{I:3:P:1:1:1}
 \begin{aligned}
	L_{\infty}({\bf x'}) &= |m|_{\infty}L_{\infty}({\bf v})
		\\
		&\sim |m|_{\infty}
			\max\{|ab-a^2\xi_{\infty}|_{\infty},
			|b^2-ab\xi_{\infty}|_{\infty}\}
		\\
		&= |m|_{\infty}|b-a\xi_{\infty}|_{\infty} \|{\bf s}\|_{\infty},
	\\
	L_{p}({\bf x'}) &= |m|_{p}L_{p}({\bf v})
		\\
		&\sim |m|_{p}
			\max\{|ab-a^2\xi_{p}|_{p},
			|b^2-ab\xi_{p}|_{p}\}
		\\
		&= |m|_{p}|b-a\xi_{p}|_{p} \|{\bf s}\|_{p}
		\\
		&= |m|_{p}|b-a\xi_{p}|_{p},
		\quad {p} \in \mathcal{S},
	\\
	|b-a\xi_p|_{p} &\ll \|{\bf s}\|_{p} = 1
	\quad \text{ for } p \in \mathcal{S}.
 \end{aligned}
\end{equation}
	Since $\|{\bf v}\|_{\infty} = \|{\bf s}\|_{\infty}^2$ and
	$\|{\bf x'}\|_{\infty} = |m|_{\infty}\|{\bf v}\|_{\infty}$,
	we find that
\begin{equation}\label{I:3:P:1:1:2}
	|a|_{\infty} \leq \|{\bf s}\|_{\infty}
	= \|{\bf x'}\|_{\infty}^{1/2}|m|_{\infty}^{-1/2}
	= Y^{1/2}|m|_{\infty}^{-1/2}.
\end{equation}

	Recall that the points ${\bf x'}$ and ${\bf y}$ are linearly independent.
	Therefore the matrix
\[
	\left (
	\begin{matrix}
		y_{0}	&  y_{1}	&  y_{2}
		\\
		x'_{0}	&  x'_{1}	&  x'_{2}
	\end{matrix}
	\right ) {,}
\]
	has rank 2.
	We claim that $x_1'\neq0$.
	Otherwise, since ${\bf x'}$ is a multiple of ${\bf v}$,
	we would have $ab=0$, then ${\bf s} =\pm(1,0)$
	or ${\bf s} =\pm(0,1)$ and so, $\|{\bf v}\|_{\infty} = 1$
	against our assumption that $\|{\bf v}\|_{\infty} > 1$.
	Since $x'_{1} \neq0$, at least one of the
	determinants
\[
	\left |
	\begin{matrix}
		y_{0} &  y_{1}
		\\
		x'_{0}   &  x'_{1}
	\end{matrix}
	\right | {, \ }
	\left |
	\begin{matrix}
		y_{1} &  y_{2}
		\\
		x'_{1} &  x'_{2}
	\end{matrix}
	\right |
\]
	is not zero. So, there exists $i \in \{0,1\}$ such that
\[
	\left |
	\begin{matrix}
		y_{i} &  y_{i+1}
		\\
		a   &  b
	\end{matrix}
	\right | \neq 0.
\]
	For such a choice of $i$, we find,
	using the estimates (\ref{I:3:P:1:1:1})
	together with those of Lemma \ref{I:2:L:1}, that
 \begin{align*}
	1 & \leq
		\left \|
			\begin{matrix}
			y_i & y_{i+1}
			\\
			a & b
			\end{matrix}
		\right\|_{\infty}
		\prod_{p \in \mathcal{S}}
		\left \|
			\begin{matrix}
			y_i & y_{i+1}
			\\
			a & b
			\end{matrix}
		\right\|_{p}
	\\
	& \leq
		\big(\|{\bf y}\|_{\infty}|b-a\xi_{\infty}|_{\infty} +
		|a|_{\infty}L_{\infty}({\bf y})\big)
		\prod_{p \in \mathcal{S}}\max\{|b-a\xi_p|_{p},L_{p}({\bf y})\}
	\\
	& \ll
		\big( \|{\bf y}\|_{\infty} L_{\infty}({\bf x'})
		|m|_{\infty}^{-1}\|{\bf s}\|_{\infty}^{-1} +
			|a|_{\infty}L_{\infty}({\bf y}) \big)
		\prod_{p \in \mathcal{S}}\max\{\min\{L_{p}({\bf x'})
		|m|_{p}^{-1},1\},L_{p}({\bf y})\}.
 \end{align*}
	Using (\ref{I:3:P:1:1:2}) and the fact that
	${\bf x'}$ and ${\bf y}$
	belong to $\mathcal{C}_{c,Y}$, we deduce that
\begin{align*}
	1 & \ll
		\big( Y Y^{-\lambda_{\infty}}|m|_{\infty}^{-1}
		Y^{-1/2}|m|_{\infty}^{1/2} +
		Y^{1/2} |m|_{\infty}^{-1/2}
		Y^{-\lambda_{\infty}}\big)
		\\
		&
		\quad \quad
		\prod_{p \in \mathcal{S}}
		\max\{\min\{Y^{-\lambda_{p}}|m|_{p}^{-1},1\},
		Y^{-\lambda_{p}}\}
	\\
	  & \ll Y^{1/2-\lambda_{\infty}-
	  \sum_{p \in \mathcal{S}}\lambda_{p}}
	  	|m|_{\infty}^{-1/2}
		\prod_{p \in \mathcal{S}}
		\max\{
		\min\{|m|_{p}^{-1},Y^{\lambda_{p}}\},1\}.
\end{align*}
	Since $|m|_{\infty}\prod_{p \in \mathcal{S}}|m|_{p} \geq 1$,
	this can be rewritten as
\[
	1 \ll
		Y^{1/2-\lambda_{\infty}-
	  	\sum_{p \in \mathcal{S}}\lambda_{p}}
		\prod_{p \in \mathcal{S}}|m|_{p}^{1/2}\min\{|m|_{p}^{-1},
		Y^{\lambda_{p}}\}.
\]
	Fix $p \in \mathcal{S}$.
	If we suppose that $|m|_{p}^{-1} \leq Y^{\lambda_{p}}$,
	we get
\[
	|m|_{p}^{1/2}\min\{|m|_{p}^{-1},Y^{\lambda_{p}}\}
		=
		|m|_{p}^{-1/2}
		\\
		\leq
		Y^{\lambda_{p}/2}.
\]
	On the other hand, if we suppose that
	$|m|_{p}^{-1} \geq Y^{\lambda_{p}}$,
	we get
\[
	|m|_{p}^{1/2}\min\{|m|_{p}^{-1},Y^{\lambda_{p}}\}
		=
		|m|_{p}^{1/2}Y^{\lambda_{p}}
		\\
		\leq
		Y^{\lambda_{p}/2}.
\]
	So, in both cases, we obtain the same upper bound.
	Combining these relations for all $p \in \mathcal{S}$, we find that
\begin{equation}\label{I:3:P:1:1:4}
	\prod_{p \in \mathcal{S}}|m|_{p}^{1/2}\min\{|m|_{p}^{-1},
		Y^{\lambda_{p}}\}
		\ll
		Y^{(\sum_{p \in \mathcal{S}}\lambda_{p})/2}.
\end{equation}
	So, we obtain
\begin{align*}
	1 \ll
	Y^{1/2-
	\lambda_{\infty}- (\sum_{p \in \mathcal{S}}\lambda_{p})/2}.
\end{align*}
	Since $\lambda + \lambda_{\infty} > 1 $, we have
	$ \lambda_{\infty} + (\sum_{p \in \mathcal{S}}\lambda_{p})/2 > 1/2$
	and this means that $Y \ll 1$, thus $\|{\bf v}\|_{\infty}\ll 1$.
	Therefore, if $X$ is large enough, we have $\det({\bf x})\neq0$.
\end{proof}

	The previous proposition assumed that $\lambda_{\infty}\geq0$.
	The next proposition deals with the case when $\lambda_{\infty} < 0$
	and the set $\mathcal{S}$ consists of just one prime number $p$.
	So that, in this case, we have
	$\bar \lambda  = (\lambda_{\infty},\lambda_p)$
	and  $\bar \xi  = (\xi_{\infty},\xi_p)$.
\begin{proposition} \label{I:3:P:1-1}
	Assume that
\begin{equation}\label{I:3:P:1-1:0}
	-1\leq \lambda_{\infty} <0
	\ \text{ and } \
 	\lambda_{\infty}/2 + \lambda_{p} > 1.
\end{equation}
	Suppose also that, for some $c>0$, the system
\begin{equation}\label{I:3:P:1-1:1}
 \begin{aligned}
	&\| {\bf x} \|_{\infty} \leq X,
	\\
	&L_{\infty}({\bf x }) \leq c X^{-\lambda_{\infty}},
	\\
	&L_{p}({\bf x }) \leq c X^{-\lambda_{p}},
 \end{aligned}
\end{equation}
	has a non-zero solution ${\bf x} \in \mathbb{Z}^3$,
	for each $X\geq1$.
	Then, for each $X$ sufficiently large,
	any such solution ${\bf x}$
	satisfies $\det({\bf x}) \neq 0$.
\end{proposition}
	In particular, if $\lambda_{\infty} = -1$,
	the condition $\lambda_{\infty}/2 + \lambda_{p} > 1$
	becomes $\lambda_{p} > 3/2 $
	and we recover the p-adic analog of Lemma 2 of
	{\sc H.~Davenport \& W.M.~Schmidt} in \cite{DS},
	given by {\sc O.~Teuli\'{e}} in \cite{Teu}.
\\
\begin{proof}
		Let ${\bf x}$ be a non-zero solution of the system
	of inequalities (\ref{I:3:P:1-1:1}) for some large $X\geq1$.
	Write ${\bf x} = m{\bf v}$ for a
	non-zero integer $m$ and a primitive point
	${\bf v} \in \mathbb{Z}^3$.
		We note that $\lambda_{\infty} + \lambda_{p} > 0$,
	so the requirements of Lemma \ref{I:1:L:1} are satisfied.
	Put
\[
	X' = \min I_c({\bf v})
	\ \text{and } \
	Y = \max \{
		\|{\bf v}\|_{\infty},
		c^{1/\lambda_{\infty}}L_{\infty}({\bf v })^{-1/\lambda_{\infty}}
		\}.
\]
	By Lemma \ref{I:1:L:1}(iv), we have that
	$\|{\bf v}\|_{\infty}$ tends to infinity with $X$,
	and so $Y$ tends to infinity with $X$.
	So, assuming that $X$ is sufficiently large,
	we have $Y \geq \|{\bf v}\|_{\infty} > 1$.
	Then, by Lemma \ref{I:1:L:1}(iii),
	there exists a primitive point
	${\bf u} \in \mathbb{Z}^3$ such that $Y \in I_c({\bf u})$
	and $\min I_c({\bf u}) < Y$.
		Since $X' = \min I_c({\bf v})$, we have
	$ \| {\bf v} \|_{\infty} \leq X' $
	and
	$ L_{\infty}({\bf v }) \leq c (X')^{-\lambda_{\infty}} $
	and so, $Y \leq X'$.
	This means that $\min I_c({\bf u }) < Y \leq X' = \min I_c({\bf v})$
	and therefore the points ${\bf u}$ and ${\bf v}$ are linearly independent.
	Choose any points ${\bf y} \in \mathcal{L}_c({\bf u},Y)$
	and ${\bf x'} \in \mathcal{L}_c({\bf v},X')$.

	Now, suppose that $\det({\bf x}) = 0$.
	So, we get $\det({\bf v}) = 0$.
	Arguing as in \cite{DS}, we find that
\[
	{\bf v} = \pm (a^2,ab,b^2),
\]
	for some non-zero point ${\bf s} = (a,b) \in \mathbb{Z}^2$.
	Since ${\bf v}$ is primitive,
	the point ${\bf s}$ is primitive, and therefore $\|{\bf s}\|_{p} = 1$.
	Since $\|{\bf v}\|_{\infty} = \|{\bf s}\|_{\infty}^2$,
	we find that
\begin{equation}\label{I:3:P:1-1:1:2}
	|a|_{\infty} \leq \|{\bf s}\|_{\infty}
	= \|{\bf v}\|_{\infty}^{1/2}.
\end{equation}
	Also, since ${\bf x'} = p^l {\bf v} \in \mathcal{L}_c({\bf v},X')$
	for some integer $l\geq0$, since
	$\lambda = \lambda_{\infty} + \lambda_p > 0$ and since
	$\lambda_p  > 1 - \lambda_{\infty}/2 > 1$, we have
\begin{equation}\label{I:3:P:1-1:1:3}
 \begin{aligned}
	L_{\infty}({\bf v })^{\lambda_{p}}L_{p}({\bf v })^{-\lambda_{\infty}}
	&\leq
	p^{l(\lambda_{p}+\lambda_{\infty})}L_{\infty}({\bf v })^{\lambda_{p}}L_{p}({\bf v })^{-\lambda_{\infty}}
	\\
	&=
	L_{\infty}({\bf x' })^{\lambda_{p}}L_{p}({\bf x' })^{-\lambda_{\infty}}
	\ll 1,
	\\
	\|{\bf v}\|_{\infty}L_{p}({\bf v })
	&\ll (X')^{1-\lambda_{p}}
	\ll \|{\bf v}\|_{\infty}^{1-\lambda_{p}},
 \end{aligned}
\end{equation}
	Since the points ${\bf x'}$ and ${\bf y}$ are linearly independent,
	the matrix
\[
	\left (
	\begin{matrix}
		y_{0}	&  y_{1}	&  y_{2}
		\\
		x'_{0}	&  x'_{1}	&  x'_{2}
	\end{matrix}
	\right ) {,}
\]
	has rank 2.
	As in the proof of the previous Proposition \ref{I:3:P:1},
	we have $x'_{1}\neq0$ and so, at least one of the determinants
\[
	\left |
	\begin{matrix}
		y_{0} &  y_{1}
		\\
		x'_{0}   &  x'_{1}
	\end{matrix}
	\right | {, \ }
	\left |
	\begin{matrix}
		y_{1} &  y_{2}
		\\
		x'_{1} &  x'_{2}
	\end{matrix}
	\right |
\]
	is not zero. So, there exists $i \in \{0,1\}$, such that
\[
	\left |
	\begin{matrix}
		y_{i} &  y_{i+1}
		\\
		a   &  b
	\end{matrix}
	\right | \neq 0.
\]
	For such a choice of $i$, we have
 \begin{align*}
	1 & \leq
		\left \|
			\begin{matrix}
			y_i & y_{i+1}
			\\
			a & b
			\end{matrix}
		\right\|_{\infty}
		\left \|
			\begin{matrix}
			y_i & y_{i+1}
			\\
			a & b
			\end{matrix}
		\right\|_{p}
	\\
	& \leq
		\big(\|{\bf y}\|_{\infty}|b-a\xi_{\infty}|_{\infty} +
		|a|_{\infty}L_{\infty}({\bf y})\big)
		\max\{|b-a\xi_p|_{p},L_{p}({\bf y})\}
	\\
	& \ll
		\big( \|{\bf y}\|_{\infty}
		\|{\bf v}\|_{\infty}^{-1/2}L_{\infty}({\bf v})
			+ \|{\bf v}\|_{\infty}^{1/2}L_{\infty}({\bf y}) \big)
		\max\{L_{p}({\bf v}),L_{p}({\bf y})\}.
\end{align*}

	Firstly, if
	$\|{\bf v}\|_{\infty}
	\geq
	c^{1/\lambda_{\infty}}L_{\infty}({\bf v })^{-1/\lambda_{\infty}}$,
	we have $Y = \|{\bf v}\|_{\infty}$.
	By this assumption and by the last relation in (\ref{I:3:P:1-1:1:3}),
	we have that
	$L_{\infty}({\bf v}) \ll \|{\bf v}\|_{\infty}^{-\lambda_{\infty}}$
	and
	$L_{p}({\bf v}) \ll \|{\bf v}\|_{\infty}^{-\lambda_{p}}$.
	Then, using the fact that ${\bf y} \in \mathcal{C}_{c,Y}$, we get
\begin{align*}
	1 & \ll
		\big( \|{\bf v}\|_{\infty}
		\|{\bf v}\|_{\infty}^{-1/2}
		\|{\bf v}\|_{\infty}^{-\lambda_{\infty}} +
		\|{\bf v}\|_{\infty}^{1/2}
		\|{\bf v}\|_{\infty}^{-\lambda_{\infty}}\big)
		\max\{\|{\bf v}\|_{\infty}^{-\lambda_{p}},
		\|{\bf v}\|_{\infty}^{-\lambda_{p}}\}
	\\
	  & \ll \|{\bf v}\|_{\infty}^{1/2-\lambda_{\infty} - \lambda_{p}}.
\end{align*}
	Since  $\lambda_{\infty}/2 + \lambda_{p} > 1$
	and $\lambda_{\infty} \geq -1$, we find that
	$ \lambda_{\infty} + \lambda_{p} =
	\lambda_{\infty}/2 + \lambda_{p} + \lambda_{\infty}/2 >
	1 - 1/2 = 1/2$.
	So, we have $\|{\bf v}\|_{\infty} \ll 1$.
	This is impossible if $X$ is sufficiently large.

	Secondly, if $c^{1/\lambda_{\infty}}
	L_{\infty}({\bf v })^{-1/\lambda_{\infty}}
	> \|{\bf v}\|_{\infty}$,
	we have
	$Y = c^{1/\lambda_{\infty}} L_{\infty}({\bf v })^{-1/\lambda_{\infty}}$
	and
	${\bf y} \in \mathcal{C}_{c,Y}$.
		By the first relation in (\ref{I:3:P:1-1:1:3}),
	we obtain $L_{p}({\bf v }) \ll \|{\bf v}\|_{\infty}^{-\lambda_p}$.
	Using this, the assumption
	$L_{\infty}({\bf v })^{-1/\lambda_{\infty}} \gg \|{\bf v}\|_{\infty}$
	and the trivial estimate
	$L_{\infty}({\bf v }) \ll \|{\bf v}\|_{\infty}$, we get
\begin{align*}
	1 & \ll
		\big(L_{\infty}({\bf v })^{-1/\lambda_{\infty}}
		\|{\bf v}\|_{\infty}^{-1/2}
		L_{\infty}({\bf v }) +
		\|{\bf v}\|_{\infty}^{1/2}
		L_{\infty}({\bf v })\big)
		\max\{L_{\infty}({\bf v })^{\lambda_{p}/\lambda_{\infty}},
		L_{\infty}({\bf v })^{\lambda_{p}/\lambda_{\infty}}\}
	\\
	  & \ll
		\big(L_{\infty}({\bf v })^{-1/\lambda_{\infty} + 1/2} +
		L_{\infty}({\bf v })^{1-1/(2\lambda_{\infty})}\big)
		L_{\infty}({\bf v })^{\lambda_{p}/\lambda_{\infty}}
\end{align*}
	Since $c^{1/\lambda_{\infty}}
	L_{\infty}({\bf v })^{-1/\lambda_{\infty}}
	>
	\|{\bf v}\|_{\infty}$ and since $\|{\bf v}\|_{\infty}$ is large if
	$X$ is large, we deduce that $L_{\infty}({\bf v })$ is large
	for large $X$.
		Hence, by this and since $-1 \leq \lambda_{\infty} < 0$,
		we find that
	$L_{\infty}({\bf v })^{-1/\lambda_{\infty} + 1/2} \geq
		L_{\infty}({\bf v })^{1-1/(2\lambda_{\infty})}$.
	Using this, we have
\[
	1 \ll
		L_{\infty}({\bf v })^{-1/\lambda_{\infty} +
		1/2 +\lambda_{p}/\lambda_{\infty}}.
\]
	Since $ \lambda_{\infty}/2 + \lambda_{p} > 1$
	and since  $ -1 \leq \lambda_{\infty} < 0$, we get
	$L_{\infty}({\bf v }) \ll 1$. Thus $Y\ll 1$, which
	is impossible if $X$ is sufficiently large.
\end{proof}

	The following is the main result of this paragraph.
	It is essential for the next chapter where we apply
	Mahler's Duality Theorem to find a measure of simultaneous
	approximation to real and p-adic numbers.
\begin{proposition} \label{I:3:P:2}
	Suppose that
\begin{equation}\label{I:3:P:2:0}
	\lambda >0
	\ \text{ and } \
	\lambda
	+ \sum_{\nu \in \mathcal{S}'}\lambda_{\nu}  >
	\left \{
	\begin{matrix}
		0 \ \text{ if } \ \infty \in \mathcal{S}',
		\\
		1 \ \text{ if } \ \infty \notin \mathcal{S}'.
	\end{matrix}
	\right.
\end{equation}
	Suppose also that one of the
	following conditions is satisfied:
\begin{itemize}
	\item[(i)]
		$
			\lambda_{\infty} \geq0
			\ \text{ and } \
			\lambda + \lambda_{\infty} >1,
		$
	\item[(ii)]
		$\mathcal{S} = \{p\}$ for some prime number $p$,
		$
			-1 \leq \lambda_{\infty} <0
			\ \text{ and } \
			\lambda + \lambda_{p} > 2.
		$
\end{itemize}
	If $\lambda \geq 1/{\gamma}$,
	then there exists a constant $c > 0$ such that
	the inequalities
\begin{equation}\label{I:3:P:2:1}
 \begin{aligned}
	&\| {\bf x} \|_{\infty} \leq X,
	\\
	&L_{\infty}({\bf x }) \leq c X^{-\lambda_{\infty}},
	\\
	&L_{p}({\bf x }) \leq c X^{-\lambda_{p}} \
	\forall {p} \in \mathcal{S},
 \end{aligned}
\end{equation}
	have no non-zero solution ${\bf x} \in \mathbb{Z}^3$ for
	arbitrarily large values of $X$.
	Moreover, if $\lambda > 1/{\gamma}$, then any constant $c>0$
	has this property.
\end{proposition}
\begin{proof}
	Suppose on the contrary that, for each constant $c>0$
	and each $X$ sufficiently large (with a lower bound depending on $c$),
	the inequalities (\ref{I:3:P:2:1}) have a non-zero solution
	in $\mathbb{Z}^3$.
	Since one of the conditions (i)-(ii) is satisfied,
	by Proposition \ref{I:3:P:1} or by Proposition \ref{I:3:P:1-1},
	we have that the requirement (\ref{I:3:T:1:1}) of Proposition \ref{I:3:T:1}
	is fulfilled. Together with the condition (\ref{I:3:P:2:0})
	this fulfills all the requirements of Proposition \ref{I:3:T:1}
	and so $\lambda < 1/\gamma$, which is a contradiction.
\end{proof}
\subsection{Special cases}\label{C1:S3:SS2:special}

	Applying Proposition \ref{I:3:P:2} with $\mathcal{S} = \emptyset$
	or with $\mathcal{S} = \{p\}$ and $\lambda_{\infty} = -1$, we obtain
	the following results of {\sc H.~Davenport \& W.M.~Schmidt} in \cite{DS}
	or of {\sc O.~Teuli\'{e}} in \cite{Teu}, respectively.
\begin{corollary} \label{I:3:L:9}
( \ {\sc H.~Davenport \& W.M.~Schmidt}, \cite{DS}\ )
	Let $\xi_{\infty} \in \mathbb{R}\setminus\mathbb{Q}$.
	Assume that $[\mathbb{Q}(\xi_{\infty}):\mathbb{Q}]>2$.
	Then there exists a constant $c > 0$ such that
	the inequalities
 \[
	\| {\bf x} \|_{\infty} \leq X,
	\quad
	L_{\infty}({\bf x }) \leq c X^{-1/\gamma},
 \]
	have no non-zero solution ${\bf x} \in \mathbb{Z}^3$ for
	arbitrarily large values of $X$.
\end{corollary}
\begin{corollary} \label{I:3:L:8} ( \ {\sc O.~Teuli\'{e}}, \cite{Teu}\ )
	Let $p$ be a prime number and
	let $\xi_p \in \mathbb{Q}_p\setminus\mathbb{Q}$.
	Assume that $[\mathbb{Q}(\xi_{p}):\mathbb{Q}]>2$.
	Then there exists a constant $c > 0$ such that
	the inequalities
\[
	\| {\bf x} \|_{\infty} \leq X,
	\quad
	L_{p}({\bf x }) \leq c X^{-\gamma},
\]
	have no non-zero solution ${\bf x} \in \mathbb{Z}^3$ for
	arbitrarily large values of $X$.
\end{corollary}
%

%% file: I-04-real-fibonacci.tex

	Let $\mathcal{S}$ be a finite set of prime numbers
	and let $\bar \lambda =
	(\lambda_{\infty},(\lambda_{p})_{{p} \in \mathcal{S}}) \in
	\mathbb{R}^{|\mathcal{S}|+1}$ be an exponent of approximation
	in degree $2$ to $\bar \xi = (\xi_{\infty},(\xi_{p})_{{p} \in \mathcal{S}})
	\in (\mathbb{R}\setminus\mathbb{Q})
	\times \prod_{{p} \in \mathcal{S}}(\mathbb{Q}_{p}\setminus\mathbb{Q})$.
	Note that $\lambda_{\nu}$ is an exponent of approximation in degree $2$
	to $\xi_{\nu}$ for each ${\nu} \in \mathcal{S}\cup\{\infty\}$
	in the sense of Definition \ref{I:1:D:3}
	or Definition \ref{I:1:D:4}. Under this hypothesis we study the cases 
	where $\lambda_{\infty}$ or $\lambda_{p}$ for some $p \in \mathcal{S}$
	belongs to a given interval.
\subsection{Growth conditions for an approximation sequence (real case)}\label{C1:S4:SS1-Real}

	Recall that $\gamma = (1+\sqrt{5})/2$ is the golden ratio.
	Assuming that $\lambda_{\infty} \in \mathbb{R}_{>0}$
	is an exponent of
	approximation in degree $2$ to a non-quadratic real number $\xi_{\infty}$,
	we show the existence of a sequence of primitive 
	points in $\mathbb{Z}^3$ satisfying a certain growth condition.
\begin{proposition} \label{RP5:P:1}
	Let $\xi_{\infty} \in \mathbb{R}$ be such that
	$[\mathbb{Q}(\xi_{\infty}):\mathbb{Q}]>2$ and
	let $\lambda_{\infty} \in \mathbb{R}_{>0}$ be an exponent of
	approximation to $\xi_{\infty}$ in degree $2$.
	Assume that $1/2 < \lambda_{\infty} \leq 1/\gamma$ and define
\[
	\theta =  \lambda_{\infty}/(1-\lambda_{\infty}).
\]
	Then there exists a sequence $({\bf y}_k)_{k \geq 1}$ of primitive points
	in $\mathbb{Z}^3$ such that upon putting
	$Y_k = \|{\bf y}_k\|_{\infty}$, for each $k \geq1$, we have
\begin{equation}\label{RP5:P:1:1}
 \begin{aligned}
	&Y_k^{\theta}
	\ll
	Y_{k+1}
	\ll Y_k^{1/(\theta - 1)},
	\\
	&
	Y_k^{-1}
	\ll L_{\infty}({\bf y}_k)
	\ll Y_k^{-\theta^2/(\theta+1)},
	\\
	&1 \leq \det({\bf y}_k)\leq Y_k^{1-\theta^2/(\theta+1)},
	\\
	&1 \leq \det({\bf y}_{k},{\bf y}_{k+1},{\bf y}_{k+2})\leq
	Y_{k}^{1/(\theta-1)^2 -\theta^2}.
 \end{aligned}
\end{equation}
\end{proposition}
\begin{proof}
	We consider the sequence $({\bf v}_i)_{i \geq 0}$ of primitive points 
	of $\mathbb{Z}^3$ constructed in Lemma \ref{I:1:L:4}
	and put $X_i = \|{\bf v}_i\|_{\infty}$ for each $i\geq0$.
	By Proposition \ref{I:3:P:1} (with $\mathcal{S}\neq\emptyset$),
	there exists an index $i_0\geq2$ such that $\det({\bf v}_i)\neq0$
	for each $i\geq i_0$.
	Define $I$ to be the set of indexes $i\geq i_0$ for which
	${\bf v}_{i-1}, {\bf v}_{i}, {\bf v}_{i+1}$ are linearly 
	independent over $\mathbb{Q}$.
	According to Remark \ref{I:3:R:1}, the set $I$ is infinite 
	since $\lambda_{\infty}>1/2$.
	
	Using (\ref{I:1:L:4:1}) and the estimates of 
	of Lemma \ref{I:2:L:1}(i), we find that, for each $i\in I$,
\begin{equation}\label{RP5:P:1:2}
	1 \leq |\det({\bf v}_i)|_{\infty}
		=
		\left \|
		\begin{matrix}
		v_{i,0} 	& v_{i,1} -  v_{i,0}\xi_{\infty}
		\\
		v_{i,1}		& v_{i,2}  -  v_{i,1}\xi_{\infty}
		\end{matrix}
		\right \|_{\infty}
		\ll
		X_i L_{\infty}({\bf v}_i)
		\ll
		X_i X_{i+1}^{-\lambda_{\infty}}
\end{equation}
	and
\begin{equation}\label{RP5:P:1:2:1}
 \begin{aligned}
	1 &\leq | \det({\bf v}_{i-1},{\bf v}_{i},{\bf v}_{i+1}) |_{\infty}
	\\
	&=
	\left \|
	\begin{matrix}
	v_{i-1,0} & v_{i-1,1} -  v_{i,0}\xi_{\infty} & v_{i-1,2} -  v_{i-1,0}\xi_{\infty}^2
	\\
	v_{i,0}	  & v_{i,1}  -  v_{i,0}\xi_{\infty}  & v_{i,2}  -  v_{i,0}\xi_{\infty}^2
	\\
	v_{i+1,0} & v_{i+1,1} - v_{i+1,0}\xi_{\infty}  & v_{i+1,2}  -  v_{i+1,0}\xi_{\infty}^2
	\end{matrix}
	\right \|_{\infty}
	\\
	&\ll
	X_{i}^{-\lambda_{\infty}}X_{i+1}^{1-\lambda_{\infty}}.
 \end{aligned}
\end{equation}
	Combining these estimates upon noting that 
	$\lambda_{\infty} \leq 1/\gamma < 1$, we deduce that
\begin{equation}\label{RP5:P:1:3}
 \begin{gathered}
	X_{i}^{\lambda_{\infty}/(1-\lambda_{\infty})}
	\ll X_{i+1}
	\ll X_i^{1/\lambda_{\infty}},
	\\
	X_i^{-1} \ll L_{\infty}({\bf v}_i) \ll X_{i+1}^{-\lambda_{\infty}}
	\ll X_{i}^{-\lambda_{\infty}^2/(1-\lambda_{\infty})}.
 \end{gathered}
\end{equation}
	Now, fix $i\in I$ and let $j$ to the largest integer such that
	${\bf v}_i$, ${\bf v}_{i+1}$, \ldots,
	${\bf v}_{j}$ $\in \langle{\bf v}_i,{\bf v}_{i+1}\rangle_{\mathbb{Q}}$.
	Since any two consecutive points of the sequence
	$({\bf v}_{i})_{i\geq1}$ are linearly
	independent over $\mathbb{Q}$, we have
	$\langle{\bf v}_i,{\bf v}_{i+1}\rangle_{\mathbb{Q}} =
	\langle{\bf v}_{j-1},{\bf v}_{j}\rangle_{\mathbb{Q}}$.
	Since ${\bf v}_{j+1} \notin \langle{\bf v}_i,
	{\bf v}_{i+1}\rangle_{\mathbb{Q}}$,
	the points ${\bf v}_{j-1}, {\bf v}_{j}, {\bf v}_{j+1}$ are linearly
	independent
	over $\mathbb{Q}$, and we deduce that $j \in I$. Since
	${\bf v}_i$, ${\bf v}_{i+1}$, \ldots, ${\bf v}_{j}$
	$\in \langle{\bf v}_i,{\bf v}_{i+1}\rangle_{\mathbb{Q}}$,
	then any three of these points are linearly dependent, and therefore $j$
	is the smallest element of $I$ with $j \geq i + 1$.
	Put
\[
	V_i := \langle{\bf v}_i,{\bf v}_{i+1}\rangle_{\mathbb{Q}} \ \text{and} \
	V_j := \langle{\bf v}_{j},{\bf v}_{j+1}\rangle_{\mathbb{Q}}.
\]
	To proceed further, we need estimates for the
	heights of the subspaces $V_i$, $V_j$, $V_i \cap V_j$ and $V_i + V_j$.
	Since $V_i \cap V_j = \langle{\bf v}_{j}\rangle_{\mathbb{Q}}$
	and $V_i + V_j = \mathbb{Q}^3$, we have (see \cite{WS}, p.~10)
 \begin{align*}
	&	H(V_i \cap V_j)  = H(\langle{\bf v}_{j}\rangle_{\mathbb{Q}}) \sim X_{j},
	\\
	&	H(V_i + V_j)  = H(\mathbb{Q}^3) = 1.
 \end{align*}
	By Lemma \ref{I:2:L:1}(iii) and (\ref{I:1:L:4:1}), we have
\[
	H(V_i) \leq \|{\bf v}_i \wedge {\bf v}_{i+1} \|_{\infty}
	\ll X_{i+1} L_{\infty}({\bf v}_i)
	\ll X_{i+1}^{1-\lambda_{\infty}}.
\]
	Similarly, we have
\[
	H(V_j) \ll X_{j+1}^{1-\lambda_{\infty}}.
\]
	Applying W.M.~Schmidt's inequality (see \cite{WS}, Lemma 8A, p.~28)
\[
	H(V_i \cap V_j)H(V_i + V_j) \ll H(V_i)H(V_j),
\]
	we conclude that
\[
	X_{j} \ll X_{i+1}^{1-\lambda_{\infty}} X_{j+1}^{1-\lambda_{\infty}}.
\]
	Since $i,j \in I$, then by (\ref{RP5:P:1:3}), we have
	$X_{i+1} \ll X_i^{1/\lambda_{\infty}}$ and
	$X_{j+1} \ll X_j^{1/\lambda_{\infty}}$.
	So, it follows that
\[
	X_{j}^{\lambda_{\infty}}
	\ll
	X_{i}^{1-\lambda_{\infty}} X_{j}^{1-\lambda_{\infty}},
\]
	and therefore, we get
\[
	X_{j} \ll X_{i}^{(1-\lambda_{\infty})/(2\lambda_{\infty}-1)}.
\]
	To establish a lower bound for $X_j$,
	recall that $i+1 \leq j$. So, by (\ref{RP5:P:1:3}),
	we find that
\[
	X_{i}^{\lambda_{\infty}/(1-\lambda_{\infty})}
	\ll X_{i+1} \leq X_j.
\]
	Combining the above two estimates, we obtain
\begin{equation}\label{RP5:P:1:4}
	X_{i}^{\lambda_{\infty}/(1-\lambda_{\infty})}
	\ll
	X_{j}
	\ll X_{i}^{(1-\lambda_{\infty})/(2\lambda_{\infty}-1)}.
\end{equation}
	Now, if we write all the elements of $I$ in increasing order,
	we obtain a sequence $\{i_1, i_2,\ldots,i_k,\ldots\}$.
	Then $i_k = i$ for some index $k\geq1$, and by the minimality of $j$
	we deduce that $i_{k+1} = j$.
	Let us define ${\bf y}_{k} := {\bf v}_{i_k}$ and
	$Y_{k} := {X}_{i_k} = \|{\bf y}_{k}\|_{\infty}$
	for each $k\geq1$.
	Then by (\ref{RP5:P:1:3}) and (\ref{RP5:P:1:4}), we have
\begin{align*}
	&Y_k^{\lambda_{\infty}/(1-\lambda_{\infty})}
	\ll
	Y_{k+1}
	\ll Y_k^{(1-\lambda_{\infty})/(2\lambda_{\infty}-1)},
	\\
	&Y_k^{-1}
	\ll L_{\infty}({\bf y}_k)
	\ll Y_k^{-\lambda_{\infty}^2/(1-\lambda_{\infty})}.
\end{align*}
	These are the first two estimates in (\ref{RP5:P:1:1}).
	Furthermore, since any three of the points
	${\bf v}_i$, ${\bf v}_{i+1}$, \ldots, ${\bf v}_{j}$
	are linearly dependent over ${\mathbb{Q}}$, then
	${\bf v}_{j-1} \in
	\langle{\bf v}_i,{\bf v}_{j}\rangle_{\mathbb{Q}} =
	\langle{\bf y}_k,{\bf y}_{k+1}\rangle_{\mathbb{Q}}$.
	Going one step further, we obtain the point
	${\bf y}_{k+2} = {\bf v}_{h}$,
	for some $h \geq j+1$, such that any three of the points
	${\bf v}_j$, ${\bf v}_{j+1}$, \ldots, ${\bf v}_{h}$ are linearly
	dependent over ${\mathbb{Q}}$,
	and therefore ${\bf v}_{j+1} \in
	\langle{\bf v}_j,{\bf v}_{h}\rangle_{\mathbb{Q}} =
	\langle{\bf y}_{k+1},{\bf y}_{k+2}\rangle_{\mathbb{Q}}$.
	It follows that $\langle{\bf y}_{k},{\bf y}_{k+1},
	{\bf y}_{k+2}\rangle_{\mathbb{Q}}$ contains the
	linearly independent points ${\bf v}_{j-1},{\bf v}_{j},{\bf v}_{j+1}$,
	and therefore
	${\bf y}_{k},{\bf y}_{k+1},{\bf y}_{k+2}$ are also linearly independent.

	To prove the third estimate in (\ref{RP5:P:1:1}),
	we proceed as in (\ref{RP5:P:1:2}). For each $k\geq1$, we find that
\[
	1 \leq |\det({\bf y}_{k})|_{\infty}
		\ll
		Y_k L_{\infty}({\bf y}_{k}).
\]
	Using the second estimates in (\ref{RP5:P:1:1}),
	this gives
\[
	1 \leq |\det({\bf y}_{k})|_{\infty}
		\ll Y_k^{1-\theta^2/(\theta+1)}.
\]
	To prove the last estimate in (\ref{RP5:P:1:1}), we use
	the fact that ${\bf y}_{k},{\bf y}_{k+1},{\bf y}_{k+2}$
	are linearly independent for each $k\geq1$.
	Proceeding as in (\ref{RP5:P:1:2:1}), we find that
 \[
	1 \leq | \det({\bf y}_{k},{\bf y}_{k+1},{\bf y}_{k+2}) |_{\infty}
	\ll
	Y_{k+2}L_{\infty}({\bf y}_{k})L_{\infty}({\bf y}_{k+1}).
 \]
 	Using the first two estimates in
	(\ref{RP5:P:1:1}), this leads to
\begin{align*}
	1 &\leq
	| \det({\bf y}_{k},{\bf y}_{k+1},{\bf y}_{k+2}) |_{\infty}
	\ll
	Y_{k+2}Y_{k}^{-\theta^2/(\theta+1)}Y_{k+1}^{-\theta^2/(\theta+1)}
	\\
	&\ll
	Y_{k}^{1/(\theta-1)^2 -\theta^2/(\theta+1) -\theta^3/(\theta+1)}
	=
	Y_{k}^{1/(\theta-1)^2 -\theta^2}.
\end{align*}
\end{proof}

%% file: apprx-quadratic-real.tex
	In \cite{ARNCAI.1} D.~Roy showed that there exist
	a transcendental real number $\xi$, such that
	for an appropriate constant $c = c(\xi) > 0$,
	the inequalities
\[
	\max_{0\leq l \leq 2}|x_l|_{\infty} \leq X,
	\quad
	\max_{0\leq l \leq 2} |x_l - x_0 \xi^l|_{\infty} \leq c X^{- 1/{\gamma}},
\]
	have a nonzero solution ${\bf x} \in \mathbb{Z}^3$ for any real
	number $X \geq 1$. Such real numbers are called \emph{extremal}.

	Suppose that $\xi$ is extremal, then
	applying Proposition \ref{RP5:P:1} with
	$\xi_{\infty} = \xi$ and
	$\lambda_{\infty} = 1/\gamma$,
	we obtain an unbounded sequence of positive integers
	$(Y_k)_{k \geq 1}$ and a sequence of
	points $( {\bf y}_{k})_{k \geq 1}$ in $\mathbb{Z}^3$ with
\[
\begin{matrix}
	Y_{k+1} \sim Y_{k}^{\gamma},
	\quad \| {{\bf  y}_k }\|_{\infty} \sim Y_{k}, \quad
	L_{\infty}({\bf y}_k) \ll Y_{k}^{-1},
	\\
	|\det({{\bf  y}_k })|_{\infty} \sim 1
	\quad \text{and} \quad
	|\det({{\bf  y}_k },{{\bf  y}_{k+1} },{{\bf  y}_{k+2} })|_{\infty} \sim 1,
\end{matrix}
\]
	This is the second part of the statement of Theorem 5.1 in \cite{ARNCAI.1},
	which is a criterion characterizing an extremal real number.
\subsection{Approximation by quadratic algebraic numbers}\label{C1:S4:SS2}

	Let $n\geq1$ be an integer and $\xi$ be a real number.
	Recall that the classical exponent of approximation $w_n(\xi)$,
	introduced by Mahler in \cite{Mahl}, is defined as the supremum of the real
	numbers $w$ for which the inequality
\[
	0<|P(\xi)|_{\infty}\leq H(P)^{-w}
\]
	holds for infinitely many polynomials $P(T) \in \mathbb{Z}[T]$
	of degree at most $n$.

	The main result of this paragraph
	is that for any extremal real number $\xi$,
	we have $w_2(\xi) = \gamma^3$.
	Y.~{\sc Bugeaud} and M.~{\sc Laurent}
	computed in \cite{BuLa.3} the exponent of approximation $w_2(\xi)$
	for any Sturmian continued fraction $\xi$
	and our result agrees with their formula
	in the case where $\xi$ is a Fibonacci continued fraction.
	However, our result below applies to all extremal real numbers
	instead of just the Fibonacci continued fractions,
	and it is more precise.
\begin{theorem} \label{at1}
	Let $\xi$ be an extremal real number.
	There exist constants $c_1,c_2>0$ with the following properties:
\begin{itemize}
	\item[(i)]
		there exists infinitely many polynomials
		$P(T) \in \mathbb{Z}[T]_{\leq 2}$, such that
\begin{equation}\label{a0}
	|P(\xi)|_{\infty} \leq c_1H(P)^{-\gamma^3},
\end{equation}
	\item[(ii)]
		for any polynomial $P(T) \in \mathbb{Z}[T]_{\leq 2}$, we have
\begin{equation}\label{a1}
	|P(\xi)|_{\infty} > c_2H(P)^{-\gamma^3}.
\end{equation}
\end{itemize}
\end{theorem}
\begin{proof}
	We know from \cite{DASD}, Theorem 7.2, p.282 that
	there exist a sequence of irreducible polynomials $(Q_k)_{k \geq 1}$ 
	of degree $2$ in $\mathbb{Z}[T]$ and a constant $c \geq 1$, such that
	for each $k \geq 1$, we have
\begin{equation}\label{a4}
\left\{
	\begin{aligned}
		&c^{-1}H(Q_k)^{-\gamma^3} \leq |Q_k(\xi)|_{\infty}
		\leq {c}H(Q_k)^{-\gamma^3},
		\\
		&H(Q_{k+1})  \leq {c}H(Q_k)^{\gamma},
		\\
		&1 \leq H(Q_1) < H(Q_2) < \ldots < H(Q_k) < \ldots.
	\end{aligned}
	\right.
\end{equation}
	The first relation in (\ref{a4}) with $c_1 = c$ proves
	the part $(i)$ of the theorem.

	For the proof of part $(ii)$, it suffices to consider a polynomial
	$P(T) \in \mathbb{Z}[T]_{\leq 2}$ with $\gcd(P,Q_k) = 1$
	for all $k \geq 1$.
	To the polynomials $P(T)$ and $Q_k(T)$ we apply the following
	inequality for the resultant (see \cite{ARNCAI.3}, Lemma 2, p.98.)
\[
 |\Res(P,Q_k)|_{\infty}
 	\leq 6H(P)^2H(Q_k)^2
	\left( \cfrac{|P(\xi)|_{\infty}}{H(P)}
	+ \cfrac{|Q_k(\xi)|_{\infty}}{H(Q_k)} \right).
\]
	Since $\gcd(P,Q_k) = 1$, we have $|\Res(P,Q_k)| \geq 1$ and
	the above inequality implies
\begin{equation}\label{a3}
	1 \leq 6H(P)^{2}H(Q_k)^{2}\left(\cfrac{|P(\xi)|_{\infty}}{H(P)} +
	\frac{|Q_k(\xi)|_{\infty}}{H(Q_k)}\right),  \quad \forall k \geq 1.
\end{equation}
	Choose a real number $\epsilon$ with
\begin{equation}\label{a5}
	0 < \epsilon \leq (2 \ c)^{-{3}/{2}}
	\text{ and } \epsilon \leq H(Q_{1})^{-1}.
\end{equation}
	Then there exists an index $k = k(\epsilon,P) \geq 1$ such that
\begin{equation}\label{a6}
	\epsilon H(Q_k) \leq H(P) < \epsilon H(Q_{k+1}). 
\end{equation}
	By (\ref{a4}) and (\ref{a6}), we have
\[
	H(P)^{-2} > \epsilon^{-2} H(Q_{k+1})^{-2}
	\geq \epsilon^{-2} c^{-2} H(Q_{k})^{-2\gamma}
\]
	and so (\ref{a3}) leads to
\[
	\frac{1}{6} \epsilon^{-2} c^{-2} H(Q_{k})^{-2\gamma^2}
		< \frac{|P(\xi)|_{\infty}}{H(P)} + \frac{|Q_k(\xi)|_{\infty}}{H(Q_k)}
		\leq \frac{|P(\xi)|_{\infty}}{H(P)} + c H(Q_k)^{-\gamma^3-1}.
\]
	Since $\gamma^3 + 1 = 2\gamma^2$, we deduce that
\[
	\Big(\frac{1}{6} \epsilon^{-2} c^{-2} - c\Big)H(Q_{k})^{-2\gamma^2}
		< \cfrac{|P(\xi)|_{\infty}}{H(P)} \ .
\]
	By (\ref{a5}) and (\ref{a6}) this gives
\[
	\frac{1}{3}c \epsilon^{2\gamma^2} H(P)^{-2\gamma^2+1} < |P(\xi)|_{\infty}.
\]
	This shows that for each $P \in \mathbb{Z}[T]_{\leq 2}$
	such that $\gcd(P,Q_k) = 1$ $( k \geq 1 )$, we have
\begin{equation}\label{a10}
	c_1 H(P)^{-\gamma^3} < |P(\xi)|_{\infty},
	\text{ where } c_1 = \frac{1}{3}c \epsilon^{2\gamma^2}.
\end{equation}
	Since $c_1 < c^{-1}$, we deduce from (\ref{a4})
	that the estimate (\ref{a10}) holds for all
	$P \in \mathbb{Z}[T]_{\leq 2}$.
\end{proof}

	We recall also that the classical
	exponent of approximation $w^*_n(\xi)$ introduced
	by Koksma in \cite{Koks} is the supremum of the real
	numbers $w$ for which the inequality
\[
	|\xi - \alpha|_{\infty} \leq H(\alpha)^{-w-1}
\]
	holds for infinitely many algebraic numbers $\alpha$
	of degree at most $n$.

	Part $(i)$ of the previous proposition follows from
	the first part of Theorem 1.4 of
	\cite{ARNCAI.1} which states the existence of infinitely
	many algebraic numbers $\alpha$ of degree at most
	$2$ over $\mathbb{Q}$, such that
  \begin{equation}\label{a11:0}
	|\xi - \alpha|_{\infty} \leq c_3 H(\alpha)^{-2\gamma^2},
  \end{equation}
  	for some constant $c_3>0$.
	Similarly, Part $(ii)$ of the previous proposition implies the
	following result, which is the second part of Theorem 1.4 of \cite{ARNCAI.1}.
\begin{corollary}
	Let $\xi$ be an extremal real number.
	There exists a constant $c_4>0$, such that
	for any algebraic number $\alpha$ of degree at most $2$
	over $\mathbb{Q}$, we have
  \begin{equation}\label{a11}
	|\xi - \alpha|_{\infty} \geq c_4 H(\alpha)^{-2\gamma^2}.
  \end{equation}
\end{corollary}
\begin{proof}
	Suppose that $\alpha$ is an algebraic number of degree at most
	$2$ over $\mathbb{Q}$ and that
	$P(T) \in \mathbb{Z}[T]_{\leq 2}$ is its minimal polynomial.
	Since $H(\alpha) = H(P)$, the result follows from
	Theorem \ref{at1} applied to $P$, combined with the inequality
\[
	|P(\xi)|_{\infty}
	\leq 2(|\xi|_{\infty}+1)H(P)|\xi - \alpha|_{\infty}.
\]
\end{proof}
	In particular, this means that we have
	$w^*_2(\xi) = \gamma^3$ for any extremal real number $\xi$.

%% file: I-06-constrains-real.tex
\subsection{Constraints on the exponents
	of approximation and a recurrence relation among points of 
	an approximation sequence (real case)}
		\label{C1:S4:SS3}

	Let $\xi_{\infty} \in \mathbb{R}$ with
	$[\mathbb{Q}(\xi_{\infty}):\mathbb{Q}] > 2$.
	Suppose that $\lambda_{\infty} \in (1/2,1/\gamma]$ is an exponent of
	approximation in degree $2$ to
	$\xi_{\infty}$ in the sense of Definition \ref{I:1:D:3}.
	Let $({\bf y}_{k})_{k \geq 1}$ and $(Y_{k})_{k \geq 1}$ be
	the sequences corresponding to $\xi_{\infty}$, constructed
	in Proposition \ref{RP5:P:1}.
	Define monotone increasing functions on the interval $(0,1)$
	by the following formulas
\begin{equation}\label{RP6:2}
\begin{aligned}
	\theta(\lambda) &= \frac{\lambda}{1-\lambda},
	\\
	\delta(\lambda) &= \frac{\theta^2}{\theta+1}
	= \frac{\lambda^2}{1-\lambda} = \lambda\theta,
	\\
	\phi(\lambda) &=
	\frac{\theta^2 -1}{\theta^2+1},
	\\
	\psi(\lambda) &= \frac{\theta -1}{\theta+1} = 2\lambda-1,
\end{aligned}
\end{equation}
	and monotone decreasing functions on the interval $(1/2,1)$
	by the formulas
\begin{equation}\label{RP6:3}
\begin{aligned}
	f(\lambda) &=
	\frac{1}{\lambda(\theta-1)} - \theta -1,
	\\
	g(\lambda) &= 1-\delta\theta(\theta-1).
\end{aligned}
\end{equation}
	We note that $\psi(1/2) = \phi(1/2) = 0$ and
	$g(1/\gamma) = f(1/\gamma) = 0$, and that
\begin{equation}\label{RP6:3:1}
\begin{aligned}
	0 &< g(\lambda) \leq f(\lambda) \quad \forall \lambda \in (1/2,1/\gamma],
	\\
	0 &< \psi(\lambda) < \phi(\lambda) \quad \forall \lambda \in (1/2,1).
\end{aligned}
\end{equation}
	We also note that the functions $f$ and $\phi$
	map the interval $\big(1/2,1/\gamma\big)$ respectively
	onto the intervals $(0,\infty)$ and $\big(0,\gamma/(2+\gamma)\big)$.
	Since the function $f-\phi$ is continuous and changes its sign
	on the interval $\big(1/2,1/\gamma\big)$, there exists a number
	$\lambda_{\infty,0} \in (1/2,1/\gamma)$
	such that $\phi(\lambda_{\infty,0})=f(\lambda_{\infty,0})$ and
\begin{equation}\label{RP6:4}
	0 < g(\lambda) \leq f(\lambda)< \phi(\lambda)
	\quad \forall \lambda \in (\lambda_{\infty,0},1/\gamma].
\end{equation}
	Its numerical value is $\lambda_{\infty,0}\approx 0.60842266\ldots$.
	Furthermore, the function $\psi$ maps the interval $\big(1/2,1/\gamma\big)$
	onto the interval $(0,1/\gamma^3)$.
	By the second relation in (\ref{RP6:3:1}) and the fact that the function
	$f-\psi$ is continuous and changes its sign on the interval
	$\big(1/2,1/\gamma\big)$, there exists
	a number $\lambda_{\infty,1} \in (\lambda_{\infty,0},1/\gamma)$
	such that $\psi(\lambda_{\infty,1})=f(\lambda_{\infty,1})$ and
\begin{equation}\label{RP6:5}
	0 < g(\lambda) \leq f(\lambda)< \psi(\lambda)< \phi(\lambda)
	 \quad \forall \lambda \in (\lambda_{\infty,1},1/\gamma].
\end{equation}
	Its numerical value is $\lambda_{\infty,1}\approx 0.61263521\ldots$.
\begin{proposition} \label{RP6:P:1}
	Take $\epsilon \in \big(0,\gamma/(2+\gamma)\big)$.
	\\
	(i) Suppose that $\phi(\lambda_{\infty}) > \epsilon$. Then for $k\gg1$,
	we have
\begin{equation}\label{RP6:P:1:1}
	Y_k^{1+\epsilon} < Y_{k+2}^{1-\epsilon}.
\end{equation}
	(ii) Furthermore, suppose that $f(\lambda_{\infty}) < \epsilon$.
	Then for each $k\gg1$ and each real number $X\geq1$ with
\begin{equation}\label{RP6:P:1:2}
	X \in [Y_k^{1+\epsilon},Y_{k+2}^{1-\epsilon}],
\end{equation}
	and any non-zero integer point ${\bf x} \in \mathbb{Z}^3$ with
\begin{equation}\label{RP6:P:1:2:1}
	\| {\bf x} \|_{\infty} \leq X,
	\quad
	L_{\infty}({\bf x }) \leq c X^{-\lambda_{\infty}},
\end{equation}
	we have
\begin{equation}\label{RP6:P:1:3}
	{\bf x} \in
	\langle{\bf y}_{k},{\bf y}_{k+1}\rangle_{\mathbb{Q}}.
\end{equation}
\end{proposition}
\begin{proof}
	For the proof of (i) we rewrite the condition
	$\phi(\lambda_{\infty}) > \epsilon$ in the form
\[
	1+\epsilon < (1-\epsilon)\theta^2,
\]
	where $\theta=\theta(\lambda_{\infty})$ is defined in (\ref{RP6:2}).
	Also, by the estimates (\ref{RP5:P:1:1}),
	we have
\[
	Y_k^{\theta^2} \ll Y_{k+2},
\]
	for each $k\geq1$.
	Combining these inequalities, we find that (\ref{RP6:P:1:1}) holds
	for $k\gg1$.

	For the proof of (ii), we use part (i). Fix
	a real number $X$ satisfying (\ref{RP6:P:1:2}).
	Choose a non-zero integer solution ${\bf x}$ of (\ref{RP6:P:1:2:1})
	corresponding to this $X$.
	In order to prove (\ref{RP6:P:1:3}), it suffices to show that
	the determinant $\det({\bf y}_{k},{\bf y}_{k+1},{\bf x})$
	is zero when $k\gg1$. Using  Lemma \ref{I:2:L:1}(i)
	and the fact that $L_{\infty}({\bf y}_{k+1})<L_{\infty}({\bf y}_{k})$,
	we have
\begin{align*}
	| \det({\bf y}_{k},{\bf y}_{k+1},{\bf x}) |_{\infty}
	&\ll
	Y_{k}L_{\infty}({\bf y}_{k+1})L_{\infty}({\bf x})+
	Y_{k+1}L_{\infty}({\bf y}_{k})L_{\infty}({\bf x})+
	X L_{\infty}({\bf y}_{k+1})L_{\infty}({\bf y}_{k})
	\\
	&\ll
	Y_{k+1}L_{\infty}({\bf y}_{k})L_{\infty}({\bf x})+
	X L_{\infty}({\bf y}_{k+1})L_{\infty}({\bf y}_{k})
\end{align*}
	for each $k\geq1$.
	By (\ref{RP5:P:1:1}) and (\ref{RP6:P:1:2}), we find that
\begin{equation}\label{RP6:P:1:4}
\begin{aligned}
	| \det({\bf y}_{k},{\bf y}_{k+1},{\bf x}) |_{\infty}
	&\ll
	Y_{k+1}Y_{k}^{-\delta}X^{-\lambda_{\infty}}+
	X Y_{k+1}^{-\delta}Y_{k}^{-\delta}
	\\
	&\ll
	Y_{k+1}Y_{k}^{-\delta}Y_{k}^{-\lambda_{\infty}(1+\epsilon)} +
	Y_{k+2}^{1-\epsilon} Y_{k+1}^{-\delta}Y_{k}^{-\delta}
	\\
	&\ll
	Y_{k}^{1/(\theta-1) - \delta -\lambda_{\infty}(1+\epsilon)}
	+
	Y_{k+2}^{1-\epsilon-\delta(\theta-1)-\delta(\theta-1)^2}
	\\
	&=
	Y_{k}^{\lambda_{\infty}(f(\lambda_{\infty}) -\epsilon)}
	+
	Y_{k+2}^{g(\lambda_{\infty})-\epsilon},
\end{aligned}
\end{equation}
	where $\delta = \delta(\lambda_{\infty})$ and
	$f(\lambda_{\infty})$, $g(\lambda_{\infty})$ are as defined in (\ref{RP6:2}) and (\ref{RP6:3}).
	Since $\lambda_{\infty} \in (1/2,1/\gamma]$ and
	$f(\lambda_{\infty}) < \epsilon$, it follows from (\ref{RP6:3:1}) that
	$g(\lambda_{\infty}) < \epsilon$, and so
\[
	| \det({\bf y}_{k},{\bf y}_{k+1},{\bf x}) |_{\infty} = o(1).
\]
	Since the determinant is an integer, we conclude that it is zero,
	and therefore the points ${\bf y}_{k}, {\bf y}_{k+1}$ and ${\bf x}$
	are linearly dependent. Hence (\ref{RP6:P:1:3}) holds.
\end{proof}
\begin{proposition} \label{RP6:P:2}
	Take $\epsilon \in (0,1/\gamma^3)$.
	\\
	(i) Suppose that $\psi(\lambda_{\infty}) > \epsilon$. Then for $k\gg1$,
	we have
\begin{equation}\label{RP6:P:2:1}
	Y_k^{1+\epsilon} < Y_{k+1}^{1-\epsilon}.
\end{equation}
	(ii) Furthermore, suppose that $f(\lambda_{\infty}) < \epsilon$.
	Then for each $k\gg1$ and each real number $X\geq1$ with
\begin{equation}\label{RP6:P:2:2}
	X \in [Y_k^{1+\epsilon},Y_{k+1}^{1-\epsilon}],
\end{equation}
	any non-zero integer point ${\bf x} \in \mathbb{Z}^3$ with
\begin{equation}\label{RP6:P:2:2-0}
 \begin{aligned}
	& \| {\bf x} \|_{\infty} \leq X,
	\\
	& L_{\infty}({\bf x }) \leq c X^{-\lambda_{\infty}},
 \end{aligned}
\end{equation}
	is a rational multiple of ${\bf y}_{k}$.
\end{proposition}
\begin{proof}
	For the proof of (i) we write the condition
	$\psi(\lambda_{\infty}) > \epsilon$ in the form
\[
	1+\epsilon < (1-\epsilon)\theta(\lambda_{\infty}).
\]
	Also, by the estimates (\ref{RP5:P:1:1}), we have
\[
	Y_k^{\theta(\lambda_{\infty})} \ll Y_{k+1},
\]
	for each $k\geq1$.
	Combining these relations, we find that (\ref{RP6:P:2:1}) holds for $k\gg1$.
	For the proof of (ii) we use part (i) and Proposition \ref{RP6:P:1}.
	By part (i), we have that
\[
	[Y_{k-1}^{1+\epsilon},Y_{k+1}^{1-\epsilon}]
	\cap
	[Y_{k}^{1+\epsilon},Y_{k+2}^{1-\epsilon}]
	=
	[Y_{k}^{1+\epsilon},Y_{k+1}^{1-\epsilon}] \neq \emptyset,
\]
	for each $k\gg1$.
	Also, by (\ref{RP6:3:1}),  we have
	$\phi(\lambda_{\infty}) > \psi(\lambda_{\infty}) > \epsilon$.
	Hence, by Proposition \ref{RP6:P:1},
	for each $k\gg1$ and each real number
$
	X \in [Y_{k}^{1+\epsilon},Y_{k+1}^{1-\epsilon}],
$
	any non-zero integer solution ${\bf x}$ of (\ref{RP6:P:2:2-0})
	satisfies
\[
	{\bf x} \in
	\langle{\bf y}_{k-1},{\bf y}_{k}\rangle_{\mathbb{Q}}
	\cap
	\langle{\bf y}_{k},{\bf y}_{k+1}\rangle_{\mathbb{Q}}
	=
	\langle{\bf y}_{k}\rangle_{\mathbb{Q}}.
\]
\end{proof}
\begin{corollary} \label{RP6:C:1}
	Let $\bar \lambda =
	(\lambda_{\infty},(\lambda_{p})_{{p} \in \mathcal{S}}) \in
	\mathbb{R}_{>0}^{|\mathcal{S}|+1}$ be an exponent of approximation
	in degree $2$ to $\bar \xi = (\xi_{\infty},(\xi_{p})_{{p} \in \mathcal{S}})
	\in \mathbb{R}\setminus\mathbb{Q}
	\times \prod_{{p} \in \mathcal{S}}(\mathbb{Q}_{p}\setminus\mathbb{Q})$,
	with $[\mathbb{Q}(\xi_{\infty}):\mathbb{Q}] > 2$
	and $\lambda_{\infty} \in (1/2,1/\gamma]$.
	Suppose that $f(\lambda_{\infty}) < \psi(\lambda_{\infty})$.
	Then, we have
\[
	\sum_{{p} \in \mathcal{S}}\lambda_{p}
	\leq
	1 - \frac{\delta(\lambda_{\infty})}{1+f(\lambda_{\infty})}.
\]
\end{corollary}
\begin{proof}
	Choose $\epsilon \in \mathbb{R}$ such that
	$f(\lambda_{\infty}) <\epsilon< \psi(\lambda_{\infty})$.
	Since $\lambda_{\infty} \in (1/2,1/\gamma]$, we have
	$f(\lambda_{\infty}) \geq0$ and $\psi(\lambda_{\infty})\leq1/\gamma^3$,
	so that $\epsilon \in (0,1/\gamma^3)$
	and we can apply Proposition \ref{RP6:P:2}.
		To do this, we note that for each $X\geq1$ a
		solution of (\ref{I:1:D:1:1})
	is also a solution of (\ref{I:1:D:3:1}). Hence, by Proposition \ref{RP6:P:2},
	for each $k\gg1$ and each real number $X\geq1$ with
	$X \in [Y_k^{1+\epsilon},Y_{k+1}^{1-\epsilon}]$,
	any non-zero integer solution ${\bf x}$ of (\ref{I:1:D:1:1}) is
	of the form $m {\bf y}_{k}$, for some non-zero integer $m$,
	where $({\bf y}_{k})_{k \geq 1}$ and $(Y_{k})_{k \geq 1}$ are
	the sequences corresponding to $\xi_{\infty}$, as
	in Proposition \ref{RP5:P:1}.

	Choosing $X = Y_{k}^{1+\epsilon}$, we obtain
\begin{align*}
	|m|_{\infty}Y_{k} =|m|_{\infty} \| {\bf y}_{k} \|_{\infty}
	&\leq X = Y_{k}^{1+\epsilon},
	\\
	|m|_{\infty} L_{\infty}({\bf y}_{k})
	&\ll X^{-\lambda_{\infty}}
	= Y_{k}^{-\lambda_{\infty}(1+\epsilon)},
	\\
	|m|_{p}L_{p}({\bf y}_{k})
	&\ll X^{-\lambda_{p}}
	= Y_{k}^{-\lambda_{p}(1+\epsilon)}\
	\forall {p} \in \mathcal{S}.
\end{align*}
 	By the third relation in (\ref{RP5:P:1:1})
	the determinant $\det({\bf y}_{k})$
	is a non-zero integer for each $k\geq1$. So, we find that
\begin{align*}
	1 & \leq |\det({\bf y}_{k})|_{\infty}
	\prod_{{p} \in \mathcal{S}}|\det({\bf y}_{k})|_{p}
	\\
	&\ll
	Y_{k}L_{\infty}({\bf y}_{k})
	\prod_{{p} \in \mathcal{S}}L_{p}({\bf y}_{k})
	\\
	&\leq
	|m|_{\infty}Y_{k}L_{\infty}({\bf y}_{k})
	\prod_{{p} \in \mathcal{S}}|m|_{p}L_{p}({\bf y}_{k})
	\\
	&\ll
	Y_{k}^{1+\epsilon}L_{\infty}({\bf y}_{k})
	\prod_{{p} \in \mathcal{S}}Y_{k}^{-\lambda_{p}(1+\epsilon)}.
\end{align*}
	By the second relation in (\ref{RP5:P:1:1}), we have
	$L_{\infty}({\bf y}_{k}) \ll Y_{k}^{-\delta}$, and hence
\[
	1 \ll
	Y_{k}^{1+\epsilon}Y_{k}^{-\delta}
	Y_{k}^{-(1+\epsilon)\sum_{{p} \in \mathcal{S}}\lambda_{p}}
	=
	Y_{k}^{1+\epsilon -\delta -(1+\epsilon)\sum_{{p} \in \mathcal{S}}\lambda_{p}}.
\]
	It follows that
\[
	1+\epsilon -\delta -(1+\epsilon)\sum_{{p} \in \mathcal{S}}\lambda_{p} \geq 0,
\]
	whence we get
\[
	\sum_{{p} \in \mathcal{S}}\lambda_{p}
	\leq
	1 - \frac{\delta(\lambda_{\infty})}{1+\epsilon}.
\]
	The conclusion follows by letting $\epsilon \rightarrow f(\lambda_{\infty})$.
\end{proof}
\begin{corollary} \label{RP6:C:2}
	There exists a number $\lambda_{\infty,2} = 0.61455261\ldots$,
	such that if $\lambda_{\infty} \in (\lambda_{\infty,2},1/\gamma]$, then
	for each sufficiently large $k\geq3$, the point ${\bf y}_{k+1}$ is
	a non-zero rational multiple of $[{\bf y}_{k},{\bf y}_{k},{\bf y}_{k-2}]$.

	Moreover, there exists a non-symmetric matrix $M$, such that
	for each sufficiently large $k\geq3$, the point ${\bf y}_{k+1}$ is
	a non-zero rational multiple of ${\bf y}_{k} M_{k} {\bf y}_{k-1}$,
	where
\[
	M_k =
	\left \{
	\begin{matrix}
		M \text{ if } k \text{ is even},
		\\
		{}^t M\text{ if } k \text{ is odd}.
	\end{matrix}
	\right .
\]
\end{corollary}
\begin{proof}
		Here we follow the proof of Corollary 5.2 on p.~50 of \cite{ARNCAI.1}.
	Let $k\geq0$ be an integer and put
	${\bf w} := [{\bf y}_{k},{\bf y}_{k},{\bf y}_{k+1}]$.
	By Lemma 2.1(i) of \cite{ARNCAI.1}, we have
\[
	\det({\bf w}) = \det({\bf y}_{k})^2 \det({\bf y}_{k+1}).
\]
	By the third relation in (\ref{RP5:P:1:1}), we have
	$\det({\bf w}) \neq 0$ and then $0 \neq {\bf w} \in \mathbb{Z}^3$.
	We claim that ${\bf w}$ is a rational multiple of ${\bf y}_{k-2}$.
	If we take this claim for granted and use
	the identity
\[
	[{\bf y}_{k},{\bf y}_{k},{\bf w}] = \det({\bf y}_{k})^2 {\bf y}_{k+1}
\]
	given by Lemma 2.1(ii) of \cite{ARNCAI.1}, we
	deduce that ${\bf y}_{k+1}$ is a non-zero rational multiple of
	$[{\bf y}_{k},{\bf y}_{k},{\bf y}_{k-2}]$.

	Fix any $\lambda_{\infty} \in (\lambda_{\infty,2},1/\gamma]$.
	Using the first estimate in Lemma 3.1(iii) of \cite{ARNCAI.1},
	we find that
\begin{equation}\label{RP6:C:2:1}
 \begin{aligned}
	\| {\bf w} \|_{\infty}
	&\ll
	Y_k^2 L_{\infty}({\bf y}_{k+1}) + Y_{k+1} L_{\infty}({\bf y}_{k})^2
	\\
 	L_{\infty}({\bf w })
		&\ll
		\big(Y_k L_{\infty}({\bf y}_{k+1})+
		Y_{k+1} L_{\infty}({\bf y}_{k}) \big) L_{\infty}({\bf y}_{k}).
 \end{aligned}
\end{equation}
	Using the estimates (\ref{RP5:P:1:1}), we also obtain
\begin{align*}
 	& Y_k^2 L_{\infty}({\bf y}_{k+1})
		\ll Y_k^{a(\lambda_{\infty})},
	\\
	& Y_{k+1} L_{\infty}({\bf y}_{k})^2
		\ll
		Y_k^{b(\lambda_{\infty})},
	\\
	&Y_k L_{\infty}({\bf y}_{k+1})L_{\infty}({\bf y}_{k})
		\ll
		Y_k^{c(\lambda_{\infty})},
\end{align*}
	where
\[
	a(\lambda_{\infty}) = 2-\frac{\theta^3}{\theta+1},
	\quad
	b(\lambda_{\infty}) = \frac{1}{\theta-1}-\frac{2\theta^2}{\theta+1},
	\quad
	c(\lambda_{\infty}) = 1-\frac{\theta^2}{\theta+1}
				-\frac{\theta^3}{\theta+1}.
\]
	Note that since $\lambda_{\infty} \in (\lambda_{\infty,2},1/\gamma]$,
	we have
\[
	c(\lambda_{\infty}) < b(\lambda_{\infty})
	< 0 < a(\lambda_{\infty}).
\]
	So, from (\ref{RP6:C:2:1}) it follows that
\begin{equation}\label{RP6:C:2:5}
 \begin{aligned}
	& \| {\bf w} \|_{\infty}
		\ll Y_k^{a(\lambda_{\infty})}
		\ll Y_{k-1}^{a(\lambda_{\infty})/(\theta-1)},
	\\
	& L_{\infty}({\bf w })
		\ll
		Y_k^{b(\lambda_{\infty})}
		\ll Y_{k-2}^{\theta^2 b(\lambda_{\infty})}.
 \end{aligned}
\end{equation}
	Since $\lambda_{\infty,2} \geq \lambda_{\infty,1} = 0.61263521\ldots$,
	then by (\ref{RP6:5}), we have
	$f(\lambda_{\infty}) < \psi(\lambda_{\infty})\leq 1/{\gamma^3}$.
 	So, by Proposition \ref{RP6:P:2},
	applied with the index $k$ replaced by $k-2$, it suffices to show that
	${\bf w }$ satisfies the inequalities
\begin{equation}\label{RP6:C:2:3}
 \begin{aligned}
	& \| {\bf w} \|_{\infty} \leq X,
	\\
	& L_{\infty}({\bf w }) \leq c X^{-\lambda_{\infty}},
 \end{aligned}
\end{equation}
 	for some $X \in
	[Y_{k-2}^{1+\epsilon},Y_{k-1}^{1-\epsilon}]$,
	and some
	$\epsilon \in \big(f(\lambda_{\infty}),\psi(\lambda_{\infty})\big)$.
	Since the conditions (\ref{RP6:C:2:3}) are equivalent to
\[
	X \in \big[\| {\bf w} \|_{\infty},
	c^{1/\lambda_{\infty}}L_{\infty}({\bf w })^{-1/\lambda_{\infty}}\big],
\]
	then for such $X$ to exist it suffices that
\[
	[Y_{k-2}^{1+\epsilon},Y_{k-1}^{1-\epsilon}]
	\cap
	\big[\| {\bf w} \|_{\infty},
	c^{1/\lambda_{\infty}}L_{\infty}({\bf w })^{-1/\lambda_{\infty}}\big]
	\neq \emptyset.
\]
	This is possible if the inequalities
\begin{align*}
	& \| {\bf w} \|_{\infty} \leq Y_{k-1}^{1-\epsilon},
	\\
	& L_{\infty}({\bf w }) \leq c Y_{k-2}^{-(1+\epsilon)\lambda_{\infty}}
\end{align*}
	hold and, by (\ref{RP6:C:2:5}), this is the case if
\[
	\frac{a(\lambda_{\infty})}{\theta-1} < 1- \epsilon
	\quad \text{ and } \quad
	\theta^2 b(\lambda_{\infty})  < -(1 + \epsilon)\lambda_{\infty}.
\]
	So, the constraints on $\epsilon$ become simply
\[
	f(\lambda_{\infty})< \epsilon <
	\min\{
		\psi(\lambda_{\infty}),
		1- \frac{a(\lambda_{\infty})}{\theta-1},
		- 1 -\frac{\theta^2 b(\lambda_{\infty})}{\lambda_{\infty}}
	\}
\]
	and this interval is not empty for each
	$\lambda_{\infty} \in (\lambda_{\infty,2},1/\gamma]$.
	So, by Proposition \ref{RP6:P:2}, we conclude that
	${\bf w}$ is proportional to ${\bf y}_{k-2}$
	for each $\lambda_{\infty} \in (\lambda_{\infty,2},1/\gamma]$.

	From now on, we use the notation $a \propto b$ to express that
	$a$ is a rational multiple of $b$.
	To show the last part of the corollary, we first
	choose $k_0$ to be the smallest even index such that
	${\bf y}_{k+1} \propto [{\bf y}_{k},{\bf y}_{k},{\bf y}_{k-2}]$
	holds for each $k\geq k_0$.
	Recall that
\begin{equation}\label{RP6:C:2:4}
	[{\bf y}_{k},{\bf y}_{k},{\bf y}_{k-2}]
	= -{\bf y}_{k}J{\bf y}_{k-2}J{\bf y}_{k},
	\text{ where }
	 J =
	\left(
	\begin{matrix}
	0      &  1
	\\
	-1      &  0
	\end{matrix}
	\right )
\end{equation}
	and that
\begin{equation}\label{RP6:C:2:6}
	{\bf w}J{\bf w}J = J{\bf w}J{\bf w} = -\det({\bf w})I,
	\text{ where }
	 I =
	\left(
	\begin{matrix}
	1      &  0
	\\
	0      &  1
	\end{matrix}
	\right ),
\end{equation}
	for any $2\times2$ symmetric matrix ${\bf w}$ (see pp.45,46 of \cite{ARNCAI.1}).
	Put
\[
	M = J{\bf y}_{k_0-2}J{\bf y}_{k_0}{\bf y}_{k_0-1}^{-1}.
\]
	We define
\[
	M_k =
	\left \{
	\begin{matrix}
		M \text{ if } k \text{ is even},
		\\
		{}^t M\text{ if } k \text{ is odd},
	\end{matrix}
	\right .
\]
	and show by induction that
	${\bf y}_{k+1} \propto {\bf y}_{k} M_{k} {\bf y}_{k-1}$ for each $k\geq k_0$.
	Clearly this holds for $k = k_0$.
	Assume that ${\bf y}_{k+1} \propto {\bf y}_{k} M_{k} {\bf y}_{k-1}$
	holds for each index $k$ with $k_0\leq k<m$.
	We need to show that this holds for $k=m$.
	Since ${\bf y}_{m}$
	is a symmetric matrix and that $M_{m} = {}^t M_{m-1}$,
	the induction hypothesis gives 
	${\bf y}_{m} \propto {\bf y}_{m-2} M_{m} {\bf y}_{m-1}$.
	By the first part of the corollary, together with 
	the identities (\ref{RP6:C:2:4}) and (\ref{RP6:C:2:6}),
	we deduce that that
\begin{align*}
	{\bf y}_{m+1} &\propto {\bf y}_{m}J{\bf y}_{m-2}J{\bf y}_{m}
		\propto {\bf y}_{m}J{\bf y}_{m-2}J
		{\bf y}_{m-2} M_{m} {\bf y}_{m-1}
		\\
		& = - {\bf y}_{m}\det({\bf y}_{m-2}) M_{m} {\bf y}_{m-1}
		\propto {\bf y}_{m}M_{m} {\bf y}_{m-1}.
\end{align*}
	It remains only to show that the matrix $M$ is non-symmetric. 
	Suppose on the contrary that $M$ is symmetric.
	Then $M_k$ is symmetric for each $k\geq k_0$.
	It is shown in \cite{ARNCAI.1} that
	for any $2\times2$ symmetric matrices ${\bf x},{\bf y},{\bf z}$, we have
\[
	\Tr(J{\bf x}J{\bf y}J{\bf z}) = \det({\bf x},{\bf y},{\bf z}),
\]
	For each $k\geq k_0$, we have
	${\bf y}_{k+1} \propto {\bf y}_{k} M_{k} {\bf y}_{k-1}$
	and
	$J{\bf y}_{k}J{\bf y}_{k} = - \det({\bf y}_{k})I$, thus
\[
	J{\bf y}_{k-1}J{\bf y}_{k}J{\bf y}_{k+1}
	\\
	\propto
	J{\bf y}_{k-1}J{\bf y}_{k}J{\bf y}_{k} M_{k} {\bf y}_{k-1}
	\propto
	J{\bf y}_{k-1}M_{k} {\bf y}_{k-1}.
\]
	Since $M_k$ is symmetric, then
	${\bf y}_{k-1}M_{k} {\bf y}_{k-1}$ is also symmetric,
	and so $\Tr(J{\bf y}_{k-1}M_{k} {\bf y}_{k-1}) = 0$.
	Hence, we have $\det({\bf y}_{k-1},{\bf y}_{k},{\bf y}_{k+1}) = 0$
	for each $k\geq k_0$, which contradicts the last relation in (\ref{RP5:P:1:1})
	of Proposition \ref{RP5:P:1}.
\end{proof}

%% file: I-05-padic-fibonacci.tex
\subsection{Growth conditions for an aproximation sequence ($p$-adic case)}

	We now turn to a $p$-adic analog of the study done in \S\ref{C1:S4:SS1-Real}.
\begin{proposition} \label{PIII:P:1}
	Let $p$ be a prime number and
	let $\xi_p \in \mathbb{Q}_p$ be with
	$[\mathbb{Q}(\xi_{p}):\mathbb{Q}]>2$.
	Let $\lambda_p \in \mathbb{R}_{>0}$ be an
	exponent of approximation to $\xi_p$ in degree $2$.
	Assume that $3/2 <\lambda_p \leq \gamma$ and define
\[
	\theta = (\lambda_p-1)/(2-\lambda_p).
\]
	Then there exists a sequence $({\bf y}_k)_{k \geq 1}$ of primitive points
	in $\mathbb{Z}^3$ such that upon putting
	$Y_k = \|{\bf y}_k\|_{\infty}$ for each $k \geq1$, we have
\begin{equation}\label{PIII:P:1:1}
 \begin{aligned}
	& Y_{k}^{\theta} \ll Y_{k+1} \ll Y_{k}^{1/(\theta-1)},
	\\
	&
	Y_{k}^{-2} \ll L_{p}({\bf y}_k) \ll Y_{k}^{-(1+\theta^2/(\theta+1))} ,
	\\
	& 1 \leq |\det({\bf y}_k)|_{\infty}
	|\det({\bf y}_k)|_{p} \ll Y_k^{1-\theta^2/(\theta+1)},
	\\
	&
	1 \leq  | \det({\bf y}_{k},{\bf y}_{k+1},{\bf y}_{k+2}) |_{\infty}
	| \det({\bf y}_{k},{\bf y}_{k+1},{\bf y}_{k+2}) |_p
	\ll
	Y_{k}^{1/(\theta-1)^2-\theta^2}.
 \end{aligned}
\end{equation}
\end{proposition}
\begin{proof}
	We consider the sequence $({\bf v}_i)_{i \geq 0}$ of primitive points 
	of $\mathbb{Z}^3$ constructed in Lemma \ref{I:1:L:6}
	and put $X_i = \|{\bf v}_i\|_{\infty}$ for each $i\geq0$.
	By Proposition \ref{I:3:P:1-1} 
	(with $\mathcal{S} = \{p\}$ and $\lambda_{\infty}=-1$),
	there exists an index $i_0\geq2$ such that $\det({\bf v}_i)\neq0$
	for each $i\geq i_0$.
	Define $I$ to be the set of indexes $i\geq i_0$ for which
	${\bf v}_{i-1}, {\bf v}_{i}, {\bf v}_{i+1}$ are linearly 
	independent over $\mathbb{Q}$.
	According to Remark \ref{I:3:R:1}, the set $I$ is infinite 
	since $\lambda_{p}>3/2$.
	Using (\ref{I:1:L:6:1}) and the estimates of 
	of Lemma \ref{I:2:L:1}(i), we find that, for each $i\in I$,
\[
	1 	\leq |\det({\bf v}_i)|_{\infty}
			|\det({\bf v}_i)|_{p}
		\ll \|{\bf v}_{i}\|_{\infty}^2 L_{p}({\bf v}_i)
		\ll X_i X_{i+1}^{1-\lambda_p}.
\]
	From this, we get
\begin{equation}\label{PIII:P:1:4}
	X_{i+1} \ll X_{i}^{1/(\lambda_p-1)}
	\ \text{and} \ X_i^{-2} \ll  L_{p}({\bf v}_i).
\end{equation}
	Note that, by (\ref{I:1:L:6:1}), we have
\[
	L_p({\bf v}_{i+1})
	< \frac{\|{\bf v}_{i}\|_{\infty}}{\|{\bf v}_{i+1}\|_{\infty}}
	L_p({\bf v}_{i})
	<
	L_p({\bf v}_{i}),
\]
	for each $i\geq0$.
	So, applying Lemma \ref{I:2:L:1}(i) to the non-zero integer 
	$\det({\bf v}_{i-1},{\bf v}_{i},{\bf v}_{i+1})$
	and using (\ref{I:1:L:6:1}), we find,
	for each $i\in I$,
\begin{align*}
	1 \leq & | \det({\bf v}_{i-1},{\bf v}_{i},{\bf v}_{i+1}) |_{\infty}
	| \det({\bf v}_{i-1},{\bf v}_{i},{\bf v}_{i+1}) |_p
	\\
	&\ll
	\|{\bf v}_{i-1}\|_{\infty}
	\|{\bf v}_{i}\|_{\infty}
	\|{\bf v}_{i+1}\|_{\infty}
	L_p({\bf v}_{i-1})L_p({\bf v}_{i})
	\\
	&\ll
	X_{i}^{1-\lambda_p} X_{i+1}^{2-\lambda_p}.
\end{align*}
	From this, we get
\begin{equation}\label{PIII:P:1:5}
	X_{i}^{(\lambda_p-1)/(2-\lambda_p)} \ll X_{i + 1}.
\end{equation}
	Combining (\ref{PIII:P:1:4}), (\ref{PIII:P:1:5})
	and (\ref{I:1:L:6:1})
	upon noting that $3/2<\lambda_p\leq\gamma$, we have,
	for each $i \in I$,
\begin{equation}\label{PIII:P:1:6}
 \begin{gathered}
	X_{i}^{(\lambda_p-1)/(2-\lambda_p)}
	\ll X_{i + 1}
	\ll X_{i}^{1/(\lambda_p-1)},
	\\
	X_i^{-2}
	\ll L_{p}({\bf v}_i)
	\ll X_{i}^{-1} X_{i+1}^{1-\lambda_p}
	\ll  X_{i}^{-1} X_{i}^{-(\lambda_p-1)^2/(2-\lambda_p)}
	\\
	\quad \quad \quad \quad \quad \quad \quad
	\quad \quad \quad \quad \quad
	= X_{i}^{-1-(\lambda_p-1)^2/(2-\lambda_p)}.
 \end{gathered}
\end{equation}
	Now, fix $i \in I$ and let $j$ be the largest integer such that
	${\bf v}_i$, ${\bf v}_{i+1}$, \ldots,
	${\bf v}_{j}$ $\in \langle{\bf v}_i,{\bf v}_{i+1}\rangle_{\mathbb{Q}}$.
	Since any two consecutive points of the sequence $({\bf v}_{i})_{i\geq1}$ are linearly
	independent over $\mathbb{Q}$, we have
	$\langle{\bf v}_i,{\bf v}_{i+1}\rangle_{\mathbb{Q}} = \langle{\bf v}_{j-1},{\bf v}_{j}\rangle_{\mathbb{Q}}$.
	Since ${\bf v}_{j+1} \notin \langle{\bf v}_i,{\bf v}_{i+1}\rangle_{\mathbb{Q}}$,
	the points ${\bf v}_{j-1}, {\bf v}_{j}, {\bf v}_{j+1}$ are linearly independent
	over $\mathbb{Q}$, and we deduce that $j \in I$. Since
	${\bf v}_i$, ${\bf v}_{i+1}$, \ldots, ${\bf v}_{j}$
	$\in \langle{\bf v}_i,{\bf v}_{i+1}\rangle_{\mathbb{Q}}$,
	then any three of these points are linearly dependent, and therefore
	$j$ is the smallest element of $I$
	with $j \geq i + 1$. 
	Put
\[
	V_i := \langle{\bf v}_i,{\bf v}_{i+1}\rangle_{\mathbb{Q}} \ \text{and} \
	V_j := \langle{\bf v}_{j},{\bf v}_{j+1}\rangle_{\mathbb{Q}}.
\]
	To proceed further, we need estimates for the
	heights of the subspaces $V_i$, $V_j$, $V_i \cap V_j$ and $V_i + V_j$.
	Since $V_i \cap V_j = \langle{\bf v}_{j}\rangle_{\mathbb{Q}}$
	and $V_i + V_j = \mathbb{Q}^3$, we have (see \cite{WS}, p.~10)
\[
 \begin{aligned}
	&	H(V_i \cap V_j)  = H(\langle{\bf v}_{j}\rangle_{\mathbb{Q}}) \sim X_{j},
	\\
	&	H(V_i + V_j)  = H(\mathbb{Q}^3) = 1.
 \end{aligned}
\]
	By Lemma \ref{I:2:L:1}(iii),
	the estimates (\ref{I:1:L:6:1}) and (\ref{PIII:P:1:6}), we have
\[
 \begin{aligned}
	H(V_i) &\ll 	\|{\bf v}_i\|_{\infty}
			\|{\bf v}_{i+1}\|_{\infty}
			\max\{ L_{p}({\bf v}_i),L_{p}({\bf v}_{i+1}) \}
		\\
		&\ll 	\|{\bf v}_i\|_{\infty}
			\|{\bf v}_{i+1}\|_{\infty}
			L_{p}({\bf v}_i)
		\ll
		X_{i+1}^{2-\lambda_p}
		\ll
		X_i^{(2-\lambda_p)/(\lambda_p-1)}.
 \end{aligned}
\]
 	Similarly, we get
\[
	H(V_j) \ll X_j^{(2-\lambda_p)/(\lambda_p-1)},
\]
	Applying W.M.~Schmidt's inequality (see \cite{WS}, Lemma 8A, p.~28)
\[
	H(V_i \cap V_j)H(V_i + V_j) \ll H(V_i)H(V_j),
\]
	we conclude that
\[
	X_{j} 
	\ll 
	X_i^{(2-\lambda_p)/(\lambda_p-1)} X_j^{(2-\lambda_p)/(\lambda_p-1)},
\]
	and hence, we have
\[
	X_{j} \ll X_i^{(2-\lambda_p)/(2\lambda_p-3)}.
\]
	For the reverse estimate recall that $i+1 \leq j$, and then
\begin{equation}\label{PIII:P:1:8:0}
	X_i^{(\lambda_p-1)/(2-\lambda_p)} \ll X_{i+1} \leq X_j \ll X_i^{(2-\lambda_p)/(2\lambda_p-3)}.
\end{equation}
	Now, if we write all the elements of $I$ in increasing order,
	we obtain a sequence $\{i_1, i_2,\ldots,i_k,\ldots\}$.
	Then $i_k = i$ for some index $k\geq1$, and by the minimality of $j$ we deduce that $i_{k+1} = j$.
	Let us denote ${\bf y}_{k} := {\bf v}_{i_k}$ and
	${Y}_{k} := {X}_{i_k} = \|{\bf y}_{k}\|_{\infty}$
	for each $k\geq1$.
	Then by (\ref{PIII:P:1:6}) and (\ref{PIII:P:1:8:0}),
	we have
\begin{equation}\label{PIII:P:1:8}
 \begin{aligned}
	& Y_{k}^{\theta} \ll Y_{k+1} \ll Y_{k}^{1/(\theta-1)},
	\\
	&
	Y_{k}^{-2} \ll L_{p}({\bf y}_k) \ll Y_{k}^{-(1+\theta^2/(\theta+1))}.
 \end{aligned}
\end{equation}
	These are the first two estimates in (\ref{PIII:P:1:1}).
	Furthermore, since any three of the points ${\bf v}_i$, ${\bf v}_{i+1}$, \ldots, ${\bf v}_{j}$
	are linearly dependent over ${\mathbb{Q}}$, then
	${\bf v}_{j-1} \in
	\langle{\bf v}_i,{\bf v}_{j}\rangle_{\mathbb{Q}} = \langle{\bf y}_k,{\bf y}_{k+1}\rangle_{\mathbb{Q}}$.
	Going one step further, we obtain the point ${\bf y}_{k+2}={\bf v}_{h}$
	for some $h \geq j+1$, such that any three of the points
	${\bf v}_j$, ${\bf v}_{j+1}$, \ldots, ${\bf v}_{h}$ are
	linearly dependent over ${\mathbb{Q}}$,
	and therefore ${\bf v}_{j+1} \in
	\langle{\bf v}_j,{\bf v}_{h}\rangle_{\mathbb{Q}} = \langle{\bf y}_{k+1},{\bf y}_{k+2}\rangle_{\mathbb{Q}}$.
	It follows that
	$\langle{\bf y}_{k},{\bf y}_{k+1},{\bf y}_{k+2}\rangle_{\mathbb{Q}}$
	contains the
	linearly independent points ${\bf v}_{j-1},{\bf v}_{j},{\bf v}_{j+1}$,
	and therefore
	${\bf y}_{k},{\bf y}_{k+1},{\bf y}_{k+2}$ are also linearly independent.

	To prove the third estimate in (\ref{PIII:P:1:1}),
	we use the fact that $\det({\bf y}_{k})$ is a non-zero integer,
	for each $k\geq1$. So, using the second estimates in (\ref{PIII:P:1:1}),
	for each $k\geq1$, we find that
\[
	1 \leq |\det({\bf y}_{k})|_{\infty}|\det({\bf y}_{k})|_{p}
		\ll Y_k^2 L_{p}({\bf y}_k)
		\ll Y_k^{1-\theta^2/(\theta+1)}.
\]
	To prove the last estimate in (\ref{PIII:P:1:1}), we use
	the fact that ${\bf y}_{k},{\bf y}_{k+1},{\bf y}_{k+2}$
	are linearly independent for each $k\geq1$.
	We also note that the sequence $(L_{p}({\bf y}_{k}))_{k\geq0}$
	is decreasing since $(L_{p}({\bf v}_{i}))_{i\geq0}$
	is decreasing (as we observed above).
	So, using these facts, Lemma \ref{I:2:L:1}(i)
	and the first two estimates in (\ref{PIII:P:1:1}),
	for each $k\geq1$, we find that
\begin{align*}
	1 &\leq  | \det({\bf y}_{k},{\bf y}_{k+1},{\bf y}_{k+2}) |_{\infty}
	| \det({\bf y}_{k},{\bf y}_{k+1},{\bf y}_{k+2}) |_p
	\\
	&\leq
	\|{\bf y}_{k}\|_{\infty}\|{\bf y}_{k+1}\|_{\infty}\|{\bf y}_{k+2}\|_{\infty}
	L_{p}({\bf y}_{k})L_{p}({\bf y}_{k+1})
	\\
	&\ll
	Y_{k}Y_{k+1}Y_{k+2}
	Y_{k}^{-(1+\theta^2/(\theta+1))}Y_{k+1}^{-(1+\theta^2/(\theta+1))}
	\\
	&=
	Y_{k+2}
	Y_{k}^{-\theta^2/(\theta+1)}Y_{k+1}^{-\theta^2/(\theta+1)}
	\\	
	&\ll
	Y_{k}^{1/(\theta-1)^2-\theta^2}.
\end{align*}
\end{proof}
\subsection{Extremal $p$-adic numbers.}

	Here we introduce a notion of extremal p-adic numbers
	and present a criterion which is a p-adic analog
	of Theorem 5.1 of \cite{ARNCAI.1}.
	
\begin{definition} \label{I:5:D:3}
	A number $\xi_p \in \mathbb{Q}_p$ is called extremal if it is not rational
	nor quadratic irrational and if $\gamma$ is an exponent of approximation
	to $\xi_p$ in degree $2$.
\end{definition}
\begin{remark} \label{I:5:R:1}
	By the Remark \ref{I:1:2:R:1}, we can say that $\xi_p \in \mathbb{Q}_p$
	is extremal if it is not rational
	nor quadratic irrational and
	if there exists a constant $c > 0$ such that the inequalities
\begin{equation}\label{I:5:R:1:1}
	\| {\bf x} \|_{\infty} \leq X, \quad
	\| {\bf x} \|_{\infty}L_{p}({\bf x }) \leq c X^{-1/\gamma},
\end{equation}
	have a non-zero solution ${\bf x} \in \mathbb{Z}^3$
	for any real number $X \geq 1$.
\end{remark}
\begin{theorem} \label{I:5:T:1}
	A number $\xi_p \in \mathbb{Q}_p$ is extremal if and only if there exist an increasing sequence
	of positive integers $(X_k)_{k \geq 1}$ and a sequence of
	primitive points $({\bf x}_k)_{k \geq 1}$
	in $\mathbb{Z}^3$ such that, for all $k \geq 1$, we have
\begin{equation}\label{I:5:T:1:1}
 \begin{aligned}
	& X_{k+1} \sim X_k^{\gamma}, \quad
	\| {\bf x}_k \|_{\infty} \sim  X_k, \quad
	L_{p}({\bf x}_k) \sim X_k^{-2}  \sim | \det({\bf x}_k) |_p,
	\\
	& | \det({\bf x}_k) |_{\infty} | \det({\bf x}_k) |_p \sim 1,
	\\
	&| \det({\bf x}_{k},{\bf x}_{k+1},{\bf x}_{k+2}) |_{\infty} | \det({\bf x}_{k},{\bf x}_{k+1},{\bf x}_{k+2}) |_p \sim 1.
 \end{aligned}
\end{equation}
\end{theorem}
\begin{proof}
	$(\Leftarrow)$
	Suppose that for a given number $\xi_p \in \mathbb{Q}_p$ there exist sequences
	$(X_k)_{k \geq 1}$ and $({\bf x}_k)_{k \geq 1}$ satisfying (\ref{I:5:T:1:1}).
	If $\xi_p$ is rational or quadratic irrational, then
	there exists $r,t,s \in \mathbb{Z}$
	not all zero such that $r + t \xi_p + s \xi_p^2 = 0$
	and hence, for all $k\geq1$, we have
\begin{align*}
	| r x_{k,0} + t x_{k,1} + s x_{k,2} |_{\infty}
		&| r x_{k,0} + t x_{k,1} + s x_{k,2} |_p
		\\
		& \ll
		X_k | t (x_{k,1} - x_{k,0}\xi_p) + s (x_{k,2} - x_{k,0}\xi_p^2) |_p
		\\
		& \ll
		X_k L_{p}({\bf x}_k) \ll X_k X_k^{-2} \ll X_k^{-1}.
\end{align*}
	Since $r x_{k,0} + t x_{k,1} + s x_{k,2}$
	is an integer, this implies that $r x_{k,0} + t x_{k,1} + s x_{k,2} = 0$ for
	all $k$ sufficiently large. So, for each
	$k$ sufficiently large, the point ${\bf x}_k$
	belongs to some fixed 2-dimensional subspace of $\mathbb{Q}^3$,
	but this contradicts the fact that the determinant of
	three consecutive points is non-zero.
	Therefore $\xi_p$ is not rational nor quadratic irrational.
	Moreover, for any sufficiently large real number $X$ there exists an
	index $k \geq 1$
	such that $\| {\bf x}_k \|_{\infty} \leq X < \| {\bf x}_{k+1} \|_{\infty}$
	and hence the point ${\bf x}_k$ satisfies
\[
	\| {\bf x}_k \|_{\infty} \leq  X, \quad
	\| {\bf x}_k \|_{\infty} L_{p}({\bf x}_k)
	\ll X_k^{-1}
	\ll X_{k+1}^{-1/\gamma}
	\ll \| {\bf x}_{k+1} \|_{\infty}^{-1/\gamma}
	\leq X^{-1/\gamma}.
\]
	So, $\xi_p$ is extremal.
	\\
	$(\Rightarrow)$
		This follows from Proposition \ref{PIII:P:1},
		with $\lambda_p = \gamma$.
\end{proof}

%

%% file: I-07-constrains-padic.tex
\subsection{Constraints on the exponents
	of approximation and a recurrence relation among points of 
	an approximation sequence ($p$-adic case)}

	Let $p$ be some prime number in $\mathcal{S}$ and let
	$\xi_{p} \in \mathbb{Q}_p$
	with $[\mathbb{Q}(\xi_{p}):\mathbb{Q}] > 2$.
	Suppose that $\lambda_{p} \in (3/2,\gamma]$ is an exponent of
	approximation to
	$\xi_{p}$ in degree $2$ in the sense of Definition \ref{I:1:D:4}.
	Let $({\bf y}_{k})_{k \geq 1}$ and $(Y_{k})_{k \geq 1}$ be
	the sequences corresponding to $\xi_{p}$, constructed
	in Proposition \ref{PIII:P:1}.
	Define monotone increasing functions on the interval $(1,2)$,
	by the following formulas
\begin{equation}\label{RP7:2}
\begin{aligned}
	\theta(\lambda) &= \frac{\lambda-1}{2-\lambda},
	\\
	\delta(\lambda) &= \frac{\theta^2}{\theta+1},
	\\
	\phi(\lambda) &=
	\frac{\theta^2 -1}{\theta^2+1},
	\\
	\psi(\lambda) &= \frac{\theta -1}{\theta+1} = 2\lambda-3,
\end{aligned}
\end{equation}
	and monotone decreasing functions on the interval $(3/2,2)$,
	by the formulas
\begin{equation}\label{RP7:3}
\begin{aligned}
	f(\lambda) &=
	\frac{1}{(\theta-1)(\lambda-1)} - \frac{\delta}{\lambda-1}-1
	=
	\frac{1}{(\theta-1)(\lambda-1)} - \theta-1,
	\\
	g(\lambda) &= 1-\delta\theta(\theta-1).
\end{aligned}
\end{equation}
	We note that $\psi(3/2) = \phi(3/2) = 0$ and
	$g(\gamma) = f(\gamma) = 0$, and that
\begin{equation}\label{RP7:3:1}
\begin{aligned}
	0 &< g(\lambda) < f(\lambda) \quad \forall \lambda \in (3/2,\gamma),
	\\
	0 &< \psi(\lambda) < \phi(\lambda) \quad \forall \lambda \in (3/2,2).
\end{aligned}
\end{equation}
	Also, we note that functions $f$ and $\phi$
	map the interval $\big(3/2,\gamma\big)$ respectively
	onto the intervals $(0,\infty)$ and $\big(0,\gamma/(2+\gamma)\big)$.
	Since the function $f-\phi$ is continuous and changes its sign
	on the interval $\big(3/2,\gamma\big)$, there exists a number
	$\lambda_{p,0} \in (3/2,\gamma)$,
	such that $\phi(\lambda_{p,0})=f(\lambda_{p,0})$ and
\begin{equation}\label{RP7:4}
	0 < g(\lambda) \leq f(\lambda)< \phi(\lambda)
	\quad \forall \lambda \in (\lambda_{p,0},\gamma],
\end{equation}
	and $\lambda_{p,0}\approx 1.60842266\ldots$.
	Furthermore, the function $\psi$ maps the interval $(3/2,\gamma)$
	onto the interval $(0,1/\gamma^3)$.
	By the second relation in (\ref{RP7:3:1}) and since the function
	$f-\psi$ is continuous and changes its sign on the interval
	$\big(3/2,\gamma\big)$, there exists
	a number $\lambda_{p,1} \in (\lambda_{p,0},\gamma)$,
	such that $\psi(\lambda_{p,1})=f(\lambda_{p,1})$ and
\begin{equation}\label{RP7:5}
	0 < g(\lambda) \leq f(\lambda)< \psi(\lambda)< \phi(\lambda)
	 \quad \forall \lambda \in (\lambda_{p,1},\gamma],
\end{equation}
	and $\lambda_{p,1}\approx 1.61263521\ldots$.
\begin{proposition} \label{RP7:P:1}
	Take $\epsilon \in \big(0,\gamma/(2+\gamma)\big)$.
	\\
	(i) Suppose that $\phi(\lambda_{p}) > \epsilon$.
	Then for $k\gg1$, we have
\begin{equation}\label{RP7:P:1:1}
	Y_k^{1+\epsilon} < Y_{k+2}^{1-\epsilon}.
\end{equation}
	(ii) Furthermore, suppose that $f(\lambda_{p}) < \epsilon$.
	Then for each $k\gg1$ and each real number $X\geq1$ with
\begin{equation}\label{RP7:P:1:2}
	X \in [Y_k^{1+\epsilon},Y_{k+2}^{1-\epsilon}],
\end{equation}
	any non-zero integer solution ${\bf x}$ of (\ref{I:1:D:4:1}) satisfies
\begin{equation}\label{RP7:P:1:3}
	{\bf x} \in \langle{\bf y}_{k},{\bf y}_{k+1}\rangle_{\mathbb{Q}}.
\end{equation}
\end{proposition}
\begin{proof}
	For the proof of (i) we rewrite the condition
	$\phi(\lambda_{p}) > \epsilon$ in the form
\[
	1+\epsilon < (1-\epsilon)\theta^2,
\]
	where $\theta=\theta(\lambda_{p})$ is defined in (\ref{RP7:3}).
	Also, by the estimates Proposition \ref{PIII:P:1},
	we have
\[
	Y_k^{\theta^2} \ll Y_{k+2},
\]
	for each $k\geq1$.
	Combining these inequalities, we find that (\ref{RP7:P:1:1}) holds,
	for $k\gg1$.
	For the proof of (ii) we use Part (i). Fix
	a real number $X$ satisfying (\ref{RP7:P:1:2}).
	Choose a non-zero integer solution ${\bf x}$ of (\ref{I:1:D:4:1}),
	corresponding to this $X$.
	In order to prove (\ref{RP7:P:1:3}), it suffices to show that
	the determinant $\det({\bf y}_{k},{\bf y}_{k+1},{\bf x})$
	is zero. Using Lemma \ref{I:2:L:1}(i)
	and the fact that $L_p({\bf y}_{k+1})<L_p({\bf y}_{k})$ for each $k\geq0$,
	which follows from Lemma \ref{I:1:L:6}, we have
\begin{align*}
	| \det({\bf y}_{k},{\bf y}_{k+1},{\bf x}) |_{\infty}
	&\ll
	Y_{k}Y_{k+1}X,
	\\
	| \det({\bf y}_{k},{\bf y}_{k+1},{\bf x}) |_{p} \
	&\ll
	\max\{
		L_{p}({\bf y}_{k+1})L_{p}({\bf x}),
		L_{p}({\bf y}_{k})L_{p}({\bf x}),
		L_{p}({\bf y}_{k+1})L_{p}({\bf y}_{k})\}
		\\
	&=
	\max\{
		L_{p}({\bf y}_{k})L_{p}({\bf x}),
		L_{p}({\bf y}_{k+1})L_{p}({\bf y}_{k})\}.
\end{align*}
	By Proposition \ref{PIII:P:1}, we find that
\begin{align*}
	| \det({\bf y}_{k},{\bf y}_{k+1},{\bf x}) |_{p}
	&\ll
	\max\{
		Y_{k}^{-1-\delta}X^{-\lambda_{p}},
		Y_{k+1}^{-1-\delta}Y_{k}^{-1-\delta}\}
	\\
	&\ll
	Y_{k}^{-1-\delta}
	\max\{X^{-\lambda_{p}},Y_{k+1}^{-1-\delta}\},
\end{align*}
	where $\delta = \delta(\lambda_{p})$ is as defined in (\ref{RP7:2}).
	To find an upper bound for the product of these two norms we use
	Proposition \ref{PIII:P:1}, the hypothesis (\ref{RP7:P:1:2})
	and the fact that $\lambda_{p} > 1$. So, we have
\begin{align*}
	| \det({\bf y}_{k},{\bf y}_{k+1},{\bf x}) |_{\infty}
	| \det({\bf y}_{k},{\bf y}_{k+1},{\bf x}) |_{p}
	&\ll
	Y_{k+1}XY_{k}^{-\delta}
	\max\{X^{-\lambda_{p}},Y_{k+1}^{-1-\delta}\}
	\\
	&\ll
	\max\{
	Y_{k+1}Y_{k}^{-\delta}X^{-(\lambda_{p}-1)},
	XY_{k+1}^{-\delta}Y_{k}^{-\delta}
	\}
	\\
	&\ll
	\max\{
	Y_{k}^{1/(\theta-1)-\delta-(\lambda_{p}-1)(1+\epsilon)},
	Y_{k+2}^{1-\epsilon-\delta(\theta-1)-\delta(\theta-1)^2}
	\}
	\\
	&=
	\max\{
	Y_{k}^{(\lambda_{p}-1)(f(\lambda_{p})-\epsilon)},
	Y_{k+2}^{g(\lambda_{p})-\epsilon}
	\},
\end{align*}
	where $\theta = \theta(\lambda_{p})$ and
	$f(\lambda_{p})$, $g(\lambda_{p})$ are defined in (\ref{RP7:3}).
	Since $\lambda_{p} \in (3/2,\gamma]$ and
	$f(\lambda_{p}) < \epsilon$, by (\ref{RP7:3:1}), we also get
	$g(\lambda_{p}) < \epsilon$, and so
\[
	| \det({\bf y}_{k},{\bf y}_{k+1},{\bf x}) |_{\infty}
	| \det({\bf y}_{k},{\bf y}_{k+1},{\bf x}) |_{p} = o(1).
\]
	Since the determinant is an integer, we conclude that it is zero,
	and therefore the points ${\bf y}_{k}, {\bf y}_{k+1}$ and ${\bf x}$
	are linearly dependent. Hence, (\ref{RP7:P:1:3}) holds.

\end{proof}
\begin{proposition} \label{RP7:P:2}
	Take $\epsilon \in (0,1/\gamma^3)$.
	\\
	(i) Suppose that $\psi(\lambda_{p}) > \epsilon$. Then for $k\gg1$,
	we have
\begin{equation}\label{RP7:P:2:1}
	Y_k^{1+\epsilon} < Y_{k+1}^{1-\epsilon}.
\end{equation}
	(ii) Furthermore, suppose that $f(\lambda_{p}) < \epsilon$.
	Then for each $k\gg1$ and each real number $X\geq1$ with
\begin{equation}\label{RP7:P:2:2}
	X \in [Y_k^{1+\epsilon},Y_{k+1}^{1-\epsilon}],
\end{equation}
	any non-zero integer solution ${\bf x}$ of (\ref{I:1:D:4:1}) is
	a rational multiple of ${\bf y}_{k}$.
\end{proposition}
\begin{proof}
	For the proof of (i) we write the condition
	$\psi(\lambda_{p}) > \epsilon$ in the form
\[
	1+\epsilon < (1-\epsilon)\theta(\lambda_{p}).
\]
	Also, by Proposition \ref{PIII:P:1}, we have
\[
	Y_k^{\theta(\lambda_{p})} \ll Y_{k+1},
\]
	for each $k\geq1$.
	Combining these relations, we find that (\ref{RP7:P:2:1}) holds for each $k\gg1$.
	For the proof of (ii) we use Part (i) and Proposition \ref{RP7:P:1}.
	By Part (i), we have that
\[
	[Y_{k-1}^{1+\epsilon},Y_{k+1}^{1-\epsilon}]
	\cap
	[Y_{k}^{1+\epsilon},Y_{k+2}^{1-\epsilon}]
	=
	[Y_{k}^{1+\epsilon},Y_{k+1}^{1-\epsilon}] \neq \emptyset,
\]
	for each $k\gg1$.
	Also, by (\ref{RP7:3:1}),  we have
	$\phi(\lambda_{p}) > \psi(\lambda_{p}) >\epsilon$.
	Hence, by Proposition \ref{RP7:P:1},
	for each $k\gg1$ and each real number
$
	X \in 
	[Y_{k}^{1+\epsilon},Y_{k+1}^{1-\epsilon}],
$
	any non-zero integer solution ${\bf x}$ of (\ref{I:1:D:4:1})
	satisfies
\[
	{\bf x} \in
	\langle{\bf y}_{k-1},{\bf y}_{k}\rangle_{\mathbb{Q}}
	\cap
	\langle{\bf y}_{k},{\bf y}_{k+1}\rangle_{\mathbb{Q}}
	=
	\langle{\bf y}_{k}\rangle_{\mathbb{Q}}.
\]
\end{proof}
	Let $p \in \mathcal{S}$ and
	let
	$\bar \lambda =
	(\lambda_{\infty},\lambda_{p},
	(\lambda_{\nu})_{{\nu} \in \mathcal{S}\setminus \{p\}}) \in
	\mathbb{R}_{>0}^{|\mathcal{S}|+1}$ be an exponent of approximation
	to $\bar \xi = (\xi_{\infty},\xi_{p},
	(\xi_{\nu})_{{\nu} \in \mathcal{S}\setminus \{p\}})
	\in \mathbb{R} \times
	\mathbb{Q}_{p} \times
	\prod_{{\nu} \in \mathcal{S}\setminus \{p\}}\mathbb{Q}_{\nu}$
	in degree $2$,
	with $[\mathbb{Q}(\xi_{p}):\mathbb{Q}] > 2$ and
	$\lambda_{p} \in (3/2,\gamma]$.
	Let $({\bf y}_{k})_{k \geq 1}$ and $(Y_{k})_{k \geq 1}$ be
	the sequences corresponding to $\xi_{p}$, as
	in Proposition \ref{PIII:P:1}.
\begin{corollary} \label{RP7:C:1}
	Suppose that $
	f(\lambda_{p}) < \psi(\lambda_{p})$.
	Then, we have
\[
	\lambda_{\infty}+
	\sum_{{\nu} \in \mathcal{S}\setminus \{p\}}\lambda_{\nu}
	\leq
	\frac{-\delta(\lambda_{p})}{1+f(\lambda_{p})}.
\]
\end{corollary}
\begin{proof}
	Similarly as in the proof of Corollary \ref{RP6:C:1},
	we choose $\epsilon \in \mathbb{R}$ such that
	$f(\lambda_{p}) <\epsilon< \psi(\lambda_{p})$.
	Since $\lambda_{p} \in (3/2,\gamma]$, we have
	$f(\lambda_{p}) \geq0$ and $\psi(\lambda_{p})\leq1/\gamma^3$,
	so that $\epsilon \in (0,1/\gamma^3)$
	and we can apply Proposition \ref{RP7:P:2}.
	We note that for each $X\geq1$ a solution of (\ref{I:1:D:1:1})
	is a solution of (\ref{I:1:D:4:1}). Hence, by Proposition \ref{RP7:P:2},
	for each $k\gg1$ and each real number $X\geq1$ with
	$X \in [Y_k^{1+\epsilon},Y_{k+1}^{1-\epsilon}]$,
	any non-zero integer solution ${\bf x}$ of (\ref{I:1:D:1:1}) is
	of the form $m {\bf y}_{k}$, for some non-zero integer $m$.
	Putting $X = Y_{k}^{1+\epsilon}$, we have
\begin{align*}
	|m|_{\infty}Y_{k} =|m|_{\infty} \| {\bf y}_{k} \|_{\infty}
	&\leq X = Y_{k}^{1+\epsilon},
	\\
	|m|_{p}L_{p}({\bf y}_{k})
	&\ll X^{-\lambda_{p}}
	= Y_{k}^{-\lambda_{p}(1+\epsilon)},
	\\
	|m|_{\infty} L_{\infty}({\bf y}_{k})
	&\ll X^{-\lambda_{\infty}}
	= Y_{k}^{-\lambda_{\infty}(1+\epsilon)},
	\\
	|m|_{\nu}L_{\nu}({\bf y}_{k})
	&\ll X^{-\lambda_{\nu}}
	= Y_{k}^{-\lambda_{\nu}(1+\epsilon)} \
	\forall {\nu} \in \mathcal{S}\setminus \{p\}.
\end{align*}
 	By the third relation in (\ref{PIII:P:1:1}), we have that
	the determinant $\det({\bf y}_{k})$ is non-zero,
	for each $k\geq1$. So, we find that
\begin{align*}
	1 & \leq |\det({\bf y}_{k})|_{\infty}
		|\det({\bf y}_{k})|_{p}
		\prod_{{\nu} \in \mathcal{S}\setminus \{p\}}
		|\det({\bf y}_{k})|_{\nu}
	\\
	&\ll
	Y_{k}|m|_{\infty}L_{\infty}({\bf y}_{k})
	L_{p}({\bf y}_{k})
	\prod_{{\nu} \in \mathcal{S}\setminus \{p\}}
		|m|_{\nu}L_{\nu}({\bf y}_{k})
	\\
	&\ll
	Y_{k}^{1-\big(\lambda_{\infty}+
	\sum_{{\nu} \in \mathcal{S}\setminus \{p\}}\lambda_{\nu}\big)(1+\epsilon)}
	L_{p}({\bf y}_{k}),
\end{align*}
	for each $k\gg1$.
	By Proposition \ref{PIII:P:1},
	for each $k\geq1$, we have
	$L_{p}({\bf y}_{k}) \ll Y_{k}^{-1-\delta(\lambda_{p})}$, and then the
	relation
\begin{align*}
	1 & \ll
	Y_{k}^{1-\big(\lambda_{\infty}+
	\sum_{{\nu} \in \mathcal{S}\setminus \{p\}}\lambda_{\nu}\big)(1+\epsilon)}
	Y_{k}^{-1-\delta(\lambda_{p})}
	\\
	& =
	Y_{k}^{-(1+\epsilon)\big(\lambda_{\infty}+
	\sum_{{\nu} \in \mathcal{S}\setminus \{p\}}
	\lambda_{\nu}\big)-\delta(\lambda_{p})},
\end{align*}
	holds for each $k\gg1$.
	So, it follows that
\[
	(1+\epsilon)\big(\lambda_{\infty}+
	\sum_{{\nu} \in \mathcal{S}\setminus \{p\}}\lambda_{\nu}\big)
	+\delta(\lambda_{p}) \leq 0,
\]
	whence we get
\[
	\lambda_{\infty}+
	\sum_{{\nu} \in \mathcal{S}\setminus \{p\}}\lambda_{\nu}
	\leq
	\frac{-\delta(\lambda_{p})}{1+\epsilon}.
\]
	Finally, the conclusion follows by letting
	$\epsilon \rightarrow f(\lambda_{p})$.
\end{proof}
\begin{proposition} \label{PII:C:1}
	Let $p$ be a prime number and
	let $\xi_p \in \mathbb{Q}_p$ be with
	$[\mathbb{Q}(\xi_{p}):\mathbb{Q}]>2$.
	Let $\lambda_p \in \mathbb{R}_{>0}$ be an exponent of approximation to $\xi_p$
	in degree $2$.
	Let $({\bf y}_k)_{k \geq 1}$ and $(Y_k)_{k \geq 1}$ be as in the
	statement of Proposition \ref{PIII:P:1}.
	There exists a number $\lambda_{p,0} \approx 1.615358873\ldots$, such that
	if $\lambda_{p} \in (\lambda_{p,0},\gamma]$, then for each $k\geq3$
	sufficiently large, the point ${\bf y}_{k+1}$ is
	a non-zero rational multiple of $[{\bf y}_{k},{\bf y}_{k},{\bf y}_{k-2}]$.

	Moreover, there exists a non-symmetric matrix $M$, such that
	for each sufficiently large $k\geq3$, the point ${\bf y}_{k+1}$ is
	a non-zero rational multiple of ${\bf y}_{k} M_{k} {\bf y}_{k-1}$,
	where
\[
	M_k =
	\left \{
	\begin{matrix}
		M \text{ if } k \text{ is even},
		\\
		{}^t M\text{ if } k \text{ is odd}.
	\end{matrix}
	\right .
\]
\end{proposition}
\begin{proof}
	For the proof we follow the arguments of
	\cite{ARNCAI.1} (see proof of Corollary 5.1, p.~50),
	using estimates with the {\it p}-adic norm.
	Let $k\geq4$ be an integer and put
	${\bf w} := [{\bf y}_{k},{\bf y}_{k},{\bf y}_{k+1}]$.
	By Lemma 2.1(i) of \cite{ARNCAI.1}, we have
\[
	\det({\bf w}) = \det({\bf y}_{k})^2 \det({\bf y}_{k+1}).
\]
	Since $\det({\bf y}_k) \neq 0$ for each $k\gg1$, then
	$\det({\bf w}) \neq 0$ and so ${\bf w} \neq 0$.
	By Lemma \ref{I:2:L:1}(ii) and Lemma \ref{I:1:L:6}, we get
\begin{align*}
	\| {\bf w} \|_{\infty} &\ll Y_{k}^2 Y_{k+1},
	\\
	\| {\bf w} \|_{p} &\ll
		\max\{ L_{p}({\bf y}_{k})^2,L_{p}({\bf y}_{k+1}) \},
	\\
	L_{p}({\bf w }) &\ll
		L_{p}({\bf y}_{k})\max\{ L_{p}({\bf y}_{k}),L_{p}({\bf y}_{k+1}) \}.
\end{align*}
	Using these estimates, Lemma \ref{I:2:L:1}(i)
	and Proposition \ref{PIII:P:1}, we find that
\begin{align*}
	| \det({\bf w},{\bf y}_{k-3},{\bf y}_{k-2}) |_{\infty} &\ll
	\| {\bf w} \|_{\infty} \| {\bf y}_{k-3} \|_{\infty} \| {\bf y}_{k-2} \|_{\infty}
	\\
	&\ll
	Y_{k}^2 Y_{k+1} Y_{k-3} Y_{k-2}
	\\
	&\ll
	Y_{k}^{2 + 1/(\theta-1) + 1/\theta^2 + 1/\theta^3 },
	\\
	| \det({\bf w},{\bf y}_{k-3},{\bf y}_{k-2}) |_{p}
		&\ll
		\max\{
		\| {\bf w} \|_{p} L_{p}({\bf y}_{k-3}) L_{p}({\bf y}_{k-2}),
		L_{p}({\bf w }) L_{p}({\bf y}_{k-2}),
		L_{p}({\bf w }) L_{p}({\bf y}_{k-3}) \}
	\\
		&\ll
		L_{p}({\bf y}_{k-3}) \max\{
		\| {\bf w} \|_{p} L_{p}({\bf y}_{k-2}),
		L_{p}({\bf w })\}
	\\
		&\ll
		L_{p}({\bf y}_{k-3}) \max\{
		L_{p}({\bf y}_{k})^2L_{p}({\bf y}_{k-2}),
		L_{p}({\bf y}_{k+1})L_{p}({\bf y}_{k-2}),
	\\
		&
		\quad
		\quad\quad
		\quad\quad
		\quad\quad
		\quad\quad
		\quad\quad
		\quad\quad
		\quad\quad
		L_{p}({\bf y}_{k})^2,
		L_{p}({\bf y}_{k})L_{p}({\bf y}_{k+1})\}
	\\
		&\ll
		L_{p}({\bf y}_{k-3})
		\max\{
		L_{p}({\bf y}_{k})^2,
		L_{p}({\bf y}_{k+1})L_{p}({\bf y}_{k-2})
		\}.
\end{align*}
	Since $3/2 < \lambda_p \leq \gamma$, we have $1 < \theta \leq \gamma$.
	So, $\theta+(\theta-1)^2 \leq 2$ and we find that
\begin{align*}
	| \det({\bf w},{\bf y}_{k-3},{\bf y}_{k-2}) |_{p}
		&\ll
		Y_{k-3}^{-(1+\theta^2/(\theta+1))}
		\max\{ Y_{k}^{-2(1+\theta^2/(\theta+1))},
		Y_{k}^{-\theta(1+\theta^2/(\theta+1))}
			Y_{k-2}^{-(1+\theta^2/(\theta+1))}\}
	\\
		&\ll
		Y_{k}^{-(\theta-1)^3(1+\theta^2/(\theta+1))}
		\max\{Y_{k}^{-2},
		Y_{k}^{-\theta-(\theta-1)^2}\}^{(1+\theta^2/(\theta+1))}.
	\\
		&\ll
		Y_{k}^{-( \theta + (\theta-1)^2 + (\theta-1)^3 )
		(1+\theta^2/(\theta+1))}.
\end{align*}
	Combining these estimates, we get
\begin{equation} \label{PII:C:1:1}
 \begin{aligned}
	| \det({\bf w},{\bf y}_{k-3},{\bf y}_{k-2}) |_{\infty}
	| \det({\bf w},{\bf y}_{k-3},{\bf y}_{k-2}) |_{p}
		&\ll
		Y_{k}^{f(\theta)},
 \end{aligned}
\end{equation}
	where
\[
	f(\theta) = 2 + (\theta-1)^{-1} + \theta^{-2} + \theta^{-3}
		-\big( \theta + (\theta-1)^2 + (\theta-1)^3 \big)
		\big(1+\theta^2/(\theta+1)\big).
\]
	Similarly, we find that
\begin{align*}
	| \det({\bf w},{\bf y}_{k-2},{\bf y}_{k-1}) |_{\infty} &\ll
	\| {\bf w} \|_{\infty} \| {\bf y}_{k-2} \|_{\infty} \| {\bf y}_{k-1} \|_{\infty}
	\\
	&\ll
	Y_{k}^2 Y_{k+1} Y_{k-2} Y_{k-1}
	\\
	&\ll
	Y_{k}^{2 + 1/(\theta-1) + 1/\theta + 1/\theta^2 },
	\\
	| \det({\bf w},{\bf y}_{k-2},{\bf y}_{k-1}) |_{p}
		&\ll
		\max\{
		\| {\bf w} \|_{p} L_{p}({\bf y}_{k-2}) L_{p}({\bf y}_{k-1}),
		L_{p}({\bf w }) L_{p}({\bf y}_{k-1}),
		L_{p}({\bf w }) L_{p}({\bf y}_{k-2}) \}
		\\
		&\ll
		L_{p}({\bf y}_{k-2}) \max\{
		\| {\bf w} \|_{p} L_{p}({\bf y}_{k-1}),
		L_{p}({\bf w })\}
			\\
		&\ll
		L_{p}({\bf y}_{k-2}) \max\{
		L_{p}({\bf y}_{k})^2L_{p}({\bf y}_{k-1}),
		L_{p}({\bf y}_{k+1})L_{p}({\bf y}_{k-1}),
	\\
		&
		\quad
		\quad\quad
		\quad\quad
		\quad\quad
		\quad\quad
		\quad\quad
		\quad\quad
		\quad\quad
		L_{p}({\bf y}_{k})^2,
		L_{p}({\bf y}_{k})L_{p}({\bf y}_{k+1})\}
	\\
		&\ll
		L_{p}({\bf y}_{k-2})
		\max\{
		L_{p}({\bf y}_{k})^2,
		L_{p}({\bf y}_{k+1})L_{p}({\bf y}_{k-1})
		\\
		&\ll
		\Big(
		Y_{k-2}^{-1}
		\max\{
			Y_{k}^{-2},
			Y_{k+1}^{-1}
			Y_{k-1}^{-1}
			 \}
		\Big)^{(1+\theta^2/(\theta+1))}
		\\
		&\ll
		\Big(
		Y_{k}^{-(\theta-1)^2}
		\max\{Y_{k}^{-2},Y_{k}^{-2\theta+1} \}
		\Big)^{(1+\theta^2/(\theta+1))}.
\end{align*}
	Since $\lambda_p \in [8/5,\gamma]$, then $\theta(\lambda_p) \in [3/2,\gamma]$ 
	and so, $-2\theta+1 \leq -2$. So, we have
\[
	| \det({\bf w},{\bf y}_{k-2},{\bf y}_{k-1}) |_{p}
	\ll
	Y_{k}^{-( 2 + (\theta-1)^2 )(1+\theta^2/(\theta+1))}.
\]
	Combning these estimates, we get
\begin{equation} \label{PII:C:1:2}
\begin{aligned}
	| \det({\bf w},{\bf y}_{k-2},{\bf y}_{k-1}) |_{\infty}
	| \det({\bf w},{\bf y}_{k-2},{\bf y}_{k-1}) |_{p}
		&\ll
		Y_{k}^{g(\theta)},
\end{aligned}
\end{equation}
	where
\[
	g(\theta) =
	2 + (\theta-1)^{-1} + \theta^{-1} + \theta^{-2}
	-( 2 + (\theta-1)^2 )(1+\theta^2(\theta+1)^{-1}).
\]
	Using MAPLE we find that  $\lambda_{p,0} \approx 1.615358873\ldots$ is the smallest
	number from the interval $[8/5,\gamma]$, with the property that
 	$g(\theta(\lambda))<0$ and $f(\theta(\lambda))<0$ for each $\lambda \in (\lambda_{p,0},\gamma]$.
 	Since $\lambda_p \in (\lambda_{p,0},\gamma]$, it follows that
\begin{gather*}
	| \det({\bf w},{\bf y}_{k-3},{\bf y}_{k-2}) |_{\infty}
	| \det({\bf w},{\bf y}_{k-3},{\bf y}_{k-2}) |_{p} = o(1),
	\\
	| \det({\bf w},{\bf y}_{k-2},{\bf y}_{k-1}) |_{\infty}
	| \det({\bf w},{\bf y}_{k-2},{\bf y}_{k-1}) |_{p} = o(1).
\end{gather*}
	Since
	$\det({\bf w},{\bf y}_{k-3},{\bf y}_{k-2})$ and
	$\det({\bf w},{\bf y}_{k-2},{\bf y}_{k-1})$ are integers,
	then by the above estimates, they are zero for each $k$ sufficiently large.
	Since three consecutive points ${\bf y}_{k-3}, {\bf y}_{k-2}, {\bf y}_{k-1}$
	are linearly independent over $\mathbb{Q}$, this implies that
	${\bf w}$ is a rational multiple of ${\bf y}_{k-2}$.
	By Lemma 2.1 (iii) of \cite{ARNCAI.1}, we have
\[
	[{\bf y}_{k},{\bf y}_{k},{\bf w}] = \det({\bf y}_{k})^2 {\bf y}_{k+1},
\]
	and deduce that ${\bf y}_{k+1}$ is a non-zero rational multiple of
	$[{\bf y}_{k},{\bf y}_{k},{\bf y}_{k-2}]$.

	The proof of the second part of the corollary repeats exactly 
	the proof of the second part of Corollary \ref{RP6:C:2}.
\end{proof}

%% file: I-09-optimal-exponents-real-padic.tex
	First we recall some notations introduced in \cite{DROY.2}.
	Put $\mathcal{M} = \Mat_{2\times2}(\mathbb{Z}) \cap \GL_2(\mathbb{C})$.
	A sequence $({\bf w}_i)_{i\geq0}$ in $\mathcal{M}$ is a
	Fibonacci sequence if ${\bf w}_{i+2} = {\bf w}_{i+1}{\bf w}_i$
	for each $i\geq0$.
 	We call a Fibonacci sequence $({\bf w}_i)_{i\geq0}$ in
	$\mathcal{M}$ {\it{admissible}}
	if there exists a matrix $N \in \mathcal{M}$ such that the sequence
	$({\bf y}_i)_{i\geq0}$, given by ${\bf y}_i = {\bf w}_i N_i$, where
\[
	N_i =
	\left \{
	\begin{aligned}
		N        & \quad \text{ if } i \text{ is even},
		\\
		{}^t N   & \quad \text{ if } i \text{ is odd},
	\end{aligned}
	\right.
\]
	consists of symmetric matrices. This new sequence satisfies the following
	recurrence relation
\[
	{\bf y}_{i} = {\bf w}_{i}N_{i} = {\bf w}_{i-1}N_{i-1}N_{i-1}^{-1}{\bf w}_{i-2}N_{i-2} =
	{\bf y}_{i-1}N_{i-1}^{-1}{\bf y}_{i-2}.
\]
	Let $\mathcal{S}$ be a finite set of prime numbers.
	Define $\mathfrak{M}_{\mathbb{Q}}$ to be the set of all prime numbers of
	$\mathbb{Q}$ together with the infinite prime $\infty$.
\subsection{Simultaneous case.}


	In this paragraph we will show
	that for any $\bar \lambda \in \mathbb{R}_{>0}^{|\mathcal{S}|+1}$,
	with the sum of its components $<\frac{1}{\gamma}$,
	there exists a point $\bar \xi \in
	\mathbb{R} \times \prod_{p \in \mathcal{S}}\mathbb{Q}_{p}$
	such that
	$\bar \lambda$ is an exponent of approximation to $\bar \xi$
	in degree $2$ and $[\mathbb{Q}(\xi_{\infty}):\mathbb{Q}]>2$.
	First, we prove several auxiliary results.
\begin{lemma} \label{RP2:L:2}
	Let $({\bf w}_i)_{i\geq0}$ be an unbounded admissible Fibonacci sequence
	in $\mathcal{M}$ and let $({\bf y}_i)_{i\geq0}$ be a corresponding sequence
	of symmetric matrices. Then for any $\nu \in \mathfrak{M}_{\mathbb{Q}}$,
	we have
\begin{equation}\label{RP2:L:2:1}
 \begin{gathered}
	|\det({\bf w}_{i+1})|_{\nu} \sim |\det({\bf w}_{i})|_{\nu}^{\gamma},
	\\
	\|{\bf y}_i\|_{\nu} \sim \|{\bf w}_i\|_{\nu}, \quad |\det({\bf y}_i)|_{\nu} \sim |\det({\bf w}_i)|_{\nu},
 \end{gathered}
\end{equation}
\end{lemma}
\begin{proof}
	Since $\det({\bf w}_{i+1}) = \det({\bf w}_{i})\det({\bf w}_{i-1})$
	for each $i\geq1$,
	Lemma 5.2 of \cite{DROY.2} provides the required estimates
	for the determinants.
	All other relations follow from the relation ${\bf y}_i = {\bf w}_i N_i$,
	since $N_i$ is $N$ or $^tN$.
\end{proof}
	For any $\nu \in \mathfrak{M}_{\mathbb{Q}}$ and any
	${\bf u} \in \mathbb{Q}_{\nu}^3$,
	we denote by $[{\bf u}]$ the point
	of $\mathbb{P}^2(\mathbb{Q}_{\nu})$ having ${\bf u}$ as a set of homogeneous
	coordinates. For any pair of non-zero points ${\bf u}, {\bf v} \in \mathbb{Q}_{\nu}^3$,
	we define the projective distance in $\mathbb{P}^2(\mathbb{Q}_{\nu})$, between the
	corresponding points
	$[{\bf u}]$ and $[{\bf v}]$ of $\mathbb{P}^2(\mathbb{Q}_{\nu})$, by
\[
	\dist_{\nu}([{\bf u}],[{\bf v}]) = \dist_{\nu}({\bf u},{\bf v}) =
	\frac{\|{\bf u} \wedge {\bf v} \|_{\nu}}{\|{\bf u}\|_{\nu}\|{\bf v}\|_{\nu}}.
\]
	The next lemma shows that $\dist_{p}$ is a metric on
	$\mathbb{P}^2(\mathbb{Q}_{p})$ for each prime number $p$.
\begin{lemma} \label{RP2:L:1}
	Let $p$ be a prime number and let
	${\bf u}, {\bf v}, {\bf w} \in \mathbb{Q}_{p}^3$.
	We have
\begin{gather}
	\| \langle {\bf u},{\bf w} \rangle {\bf v} -
	\langle {\bf u},{\bf v} \rangle {\bf w} \|_{p} \leq
	\|{\bf u}\|_{p} \|{\bf v} \wedge {\bf w}\|_{p}, \label{RP2:L:1:1}
	\\
	\|{\bf v}\|_{p} \|{\bf u}\wedge{\bf w}\|_{p} \leq
	\max\{\|{\bf w}\|_{p}
	\|{\bf u}\wedge{\bf v}\|_{p},
	\|{\bf u}\|_{p}\|{\bf v}\wedge{\bf w}\|_{p}\}. \label{RP2:L:1:2}
\end{gather}
	Moreover, if ${\bf u}, {\bf v}, {\bf w}$ are non-zero then
\begin{equation}\label{RP2:L:1:3}
	\dist_{p}([{\bf u}],[{\bf w}]) \leq
	\max\{\dist_{p}([{\bf u}],[{\bf v}]),\dist_{p}([{\bf v}],[{\bf w}])\}.
\end{equation}
\end{lemma}
\begin{proof}
	Upon writing ${\bf v} = (v_0,v_1,v_2)$, ${\bf w} = (w_0,w_1,w_2)$,
	we find that for $i = 0,1,2$,
\begin{equation}\label{RP2:L:1:4}
	\| v_i{\bf w} - w_i{\bf v} \|_{p} \leq \|{\bf v}\wedge{\bf w}\|_{p},
\end{equation}
	and so
\begin{gather*}
	| \langle {\bf u},{\bf w} \rangle v_i -
	\langle {\bf u},{\bf v} \rangle w_i |_{p} =
	| \langle {\bf u},v_i{\bf w} - w_i{\bf v} \rangle|_{p} \leq
	\|{\bf u}\|_{p} \|{\bf v} \wedge {\bf w}\|_{p},
\end{gather*}
	which implies (\ref{RP2:L:1:1}).
	Combining (\ref{RP2:L:1:4}) with the identity
\[
	v_i({\bf u}\wedge{\bf w}) = w_i({\bf u}\wedge{\bf v})
	+ {\bf u}\wedge( v_i {\bf w} - w_i {\bf v}),
\]
	we also get
\[
	|v_i|_{p} \|{\bf u}\wedge{\bf w}\|_{p} \leq
	\max\{|w_i|_{p}
	\|{\bf u}\wedge{\bf v}\|_{p},\|{\bf u}\|_{p}\|{\bf v}\wedge{\bf w}\|_{p}\}
\]
	for $i=0,1,2$, and this gives (\ref{RP2:L:1:2}).
	Dividing both sides of (\ref{RP2:L:1:2}) by
	$\|{\bf u}\|_{p}\|{\bf v}\|_{p}\|{\bf w}\|_{p}$,
	we obtain (\ref{RP2:L:1:3}).
\end{proof}

	The next proposition extends the result obtained by
	D.~{\sc Roy} in \cite{DROY.2} to the p-adic case.
	Here we use the fact that $\mathbb{P}^2(\mathbb{Q}_{\nu})$
	is complete with respect to $\dist_{\nu}$ for each
	$\nu \in \mathfrak{M}_{\mathbb{Q}}$.
\begin{proposition} \label{RP2:P:1}
	Let $\nu \in \mathfrak{M}_{\mathbb{Q}}$,
	let $({\bf w}_i)_{i\geq0}$ be an unbounded admissible Fibonacci sequence in
	$\mathcal{M}$ such that
\begin{equation}\label{RP2:P:1:0}
	\|{\bf w}_{i+1}\|_{\nu} \sim \|{\bf w}_{i}\|_{\nu}^{\gamma},
\end{equation}
	let $({\bf y}_i)_{i\geq0}$ be a corresponding sequence of symmetric matrices
	viewed as points in $\mathbb{Z}^3$
	and let $(\delta_{\nu,i})_{i\geq1}$ be the sequence defined by
\begin{equation}\label{RP2:P:1:1}
	\delta_{\nu,i} := \frac{|\det({\bf w}_i)|_{\nu}}{\|{\bf w}_{i}\|_{\nu}}.
\end{equation}
	Assume that $\det({\bf y}_0,{\bf y}_1,{\bf y}_2) \neq 0$
	and $\delta_{\nu,i} = o(\|{\bf w}_{i}\|_{\nu})$.
	Then there exists a non-zero point
	${\bf y}_{\nu} = (y_{\nu,0},y_{\nu,1},y_{\nu,2}) \in \mathbb{Q}_{\nu}^3$
	with $\det({\bf y}_{\nu})=0$ such that
\begin{equation}\label{RP2:P:1:2}
 \begin{gathered}
 	\|{\bf y}_i \wedge {\bf y}_{i+1} \|_{\nu} \sim \delta_{\nu,i} \|{\bf w}_{i+1}\|_{\nu},
	\\
	\|{\bf y}_i \wedge {\bf y}_{\nu} \|_{\nu} \sim \delta_{\nu,i},
	\\
	| \langle {\bf y}_i \wedge {\bf y}_{i+1},{\bf y}_{\nu} \rangle |_{\nu} \sim
	\delta_{\nu,i+2}.
 \end{gathered}
\end{equation}
\end{proposition}
\begin{proof}
	The proof in the case $\nu = \infty$ is given in \cite{DROY.2}.
	Assume that $\nu \in \mathfrak{M}_{\mathbb{Q}}\setminus\{\infty\}$.
	We note that the identity
\begin{equation}\label{RP2:P:1:5}
	J {\bf w} J {}^t {\bf w} = -\det({\bf w})I
\end{equation}
	holds for any matrix ${\bf w} \in \Mat_{2\times2}(\mathbb{Q}_{\nu})$.
	Using this, the recurrence relation
	${\bf y}_{i+1} = {\bf y}_{i}N_i^{-1}{\bf y}_{i-1}$ and the
	fact that the matrix ${\bf y}_i$ is symmetric, we get
\begin{equation}\label{RP2:P:1:5:1}
	J {\bf y}_i J {\bf y}_{i+1} =
	J {\bf y}_i J {\bf y}_{i} N_i^{-1}{\bf y}_{i-1} =
	-\det({\bf y}_{i})N_i^{-1}{\bf y}_{i-1}.
\end{equation}
	Since
\begin{gather*}
	{\bf y}_i \wedge {\bf y}_{i+1} =
				\Big (
				\left |
				\begin{matrix}
				y_{i,1} &  y_{i,2}
				\\
				y_{i+1,1} &  y_{i+1,2}
				\end{matrix}
				\right |
				,
				-\left |
				\begin{matrix}
				y_{i,0} &  y_{i,2}
				\\
				y_{i+1,0} &  y_{i+1,2}
				\end{matrix}
				\right |
				,
				\left |
				\begin{matrix}
				y_{i,0} &  y_{i,1}
				\\
				y_{i+1,0} &  y_{i+1,1}
				\end{matrix}
				\right |
				\Big ),
	\\
			J {\bf y}_i J {\bf y}_{i+1}
			= \left (
			\begin{matrix}
				y_{i,1}y_{i+1,1}-y_{i,2}y_{i+1,0} &
					\left |
					\begin{matrix}
					y_{i,1} &  y_{i,2}
					\\
					y_{i+1,1} &  y_{i+1,2}
					\end{matrix}
					\right |
				\\
				-\left |
				\begin{matrix}
				y_{i,0} &  y_{i,1}
				\\
				y_{i+1,0} &  y_{i+1,1}
				\end{matrix}
				\right |
				&
				y_{i,1}y_{i+1,1}-y_{i,0}y_{i+1,2}
			\end{matrix}
			\right ),
\end{gather*}
	and since
\[
	\left |
	\begin{matrix}
	y_{i,0} &  y_{i,2}
	\\
	y_{i+1,0} &  y_{i+1,2}
	\end{matrix}
	\right | = (y_{i,1}y_{i+1,1}-y_{i,2}y_{i+1,0}) - (y_{i,1}y_{i+1,1}-y_{i,0}y_{i+1,2}),
\]
	is the difference of the elements of the diagonal of
	$J {\bf y}_i J {\bf y}_{i+1}$, we find that
\[
	\|{\bf y}_i \wedge {\bf y}_{i+1} \|_{\nu} \leq \|J{\bf y}_iJ{\bf y}_{i+1} \|_{\nu}.
\]
	Combining this with (\ref{RP2:P:1:5:1}), we obtain
\[
	\|{\bf y}_i \wedge {\bf y}_{i+1} \|_{\nu} \leq
	|\det({\bf y}_i)|_{\nu} \|N_{i}^{-1}{\bf y}_{i-1} \|_{\nu} \sim
	|\det({\bf y}_i)|_{\nu} \|{\bf y}_{i-1} \|_{\nu}.
\]
	By Lemma \ref{RP2:L:2} and the hypothesis (\ref{RP2:P:1:0}),
	we conclude that
\begin{equation}\label{RP2:P:1:6:0}
	\|{\bf y}_i \wedge {\bf y}_{i+1} \|_{\nu} \ll
	|\det({\bf w}_i)|_{\nu} \|{\bf w}_{i-1} \|_{\nu} \sim
	\delta_{\nu,i} \|{\bf w}_{i+1}\|_{\nu},
\end{equation}
	 and so
\[
	\dist_{\nu}([{\bf y}_i],[{\bf y}_{i+1}]) =
	\frac{\|{\bf y}_i \wedge {\bf y}_{i+1} \|_{\nu}}{\|{\bf y}_i\|_{\nu}\|{\bf y}_{i+1}\|_{\nu}} \ll
	\frac{\delta_{\nu,i}}{\|{\bf w}_i\|_{\nu}}.
\]
	By Lemma \ref{RP2:L:2} and (\ref{RP2:P:1:0}),
	we also note that $\delta_{\nu,i+1} \sim \delta_{\nu,i}^{\gamma}$.
	So, $\delta_{\nu,i}/\|{\bf w}_i\|_{\nu}$
	is decreasing for all $i$ sufficiently large, and by Lemma \ref{RP2:L:1},
	we deduce that,
	for any $i$ and $j$ with $1 \ll i < j$, we have
\begin{equation}\label{RP2:P:1:6}
 \begin{aligned}
	\dist_{\nu}([{\bf y}_i],[{\bf y}_{j}]) &\leq
	\max_{k=0,\ldots,j-i-1}\{\dist_{\nu}([{\bf y}_{i+k}],[{\bf y}_{i+k+1}])\}
	\\
	&\ll
	\max_{k=0,\ldots,j-i-1} \frac{\delta_{\nu,i+k}}{\|{\bf w}_{i+k}\|_{\nu}}
	\leq \frac{\delta_{\nu,i}}{\|{\bf w}_i\|_{\nu}},
 \end{aligned}
\end{equation}
	Therefore $([{\bf y}_i])_{i\geq0}$ is a Cauchy sequence in
	$\mathbb{P}^2(\mathbb{Q}_{\nu})$ and so it converges to a point
	$[{\bf y}_{\nu}]$ for some non-zero ${\bf y}_{\nu} \in \mathbb{Q}_{\nu}^3$.
	By Lemma \ref{RP2:L:2}, we have
\[
	\frac{|\det({\bf y}_i)|_{\nu}}{\|{\bf y}_{i}\|_{\nu}^2} \sim
	\frac{\delta_{\nu,i}}{\|{\bf w}_i\|_{\nu}}.
\]
	Since the left hand side of this inequality depends only on the class
	$[{\bf y}_i]$ of ${\bf y}_i$ in $\mathbb{P}^2(\mathbb{Q}_{\nu})$ and since
	$\delta_{\nu,i} = o(\|{\bf w}_i\|_{\nu})$, we get by continuity that
	$|\det({\bf y}_{\nu})|_{\nu}/\|{\bf y}_{\nu}\|_{\nu}^2 = 0$, and hence
	$\det({\bf y}_{\nu}) = 0$. Moreover, by continuity, it follows
	from (\ref{RP2:P:1:6}) that
	$\dist_{\nu}([{\bf y}_i],[{\bf y}_{\nu}]) \ll \delta_{\nu,i}/\|{\bf w}_i\|_{\nu}$, 
	which implies that
\begin{equation}\label{RP2:P:1:7}
	\|{\bf y}_i\wedge{\bf y}_{\nu}\|_{\nu} \ll \delta_{\nu,i}.
\end{equation}
	Combining (\ref{RP2:L:1:1}), (\ref{RP2:P:1:6:0}) and (\ref{RP2:P:1:7}),
	we obtain
\[
 \begin{aligned}
	\| \langle {\bf y}_i \wedge {\bf y}_{i+1},{\bf y}_{\nu} \rangle {\bf y}_{i+2} -
	\langle {\bf y}_i \wedge {\bf y}_{i+1},{\bf y}_{i+2} \rangle{\bf y}_{\nu} \|_{\nu}
	& \leq  \|{\bf y}_i \wedge {\bf y}_{i+1}\|_{\nu} \|{\bf y}_{i+2} \wedge {\bf y}_{\nu}\|_{\nu}
	\\
	& \ll \delta_{\nu,i} \|{\bf w}_{i+1}\|_{\nu}\delta_{\nu,i+2}
	\\
	& \sim \frac{\delta_{\nu,i}}{\|{\bf w}_i\|_{\nu}} |\det({\bf w}_{i+2})|_{\nu}.
 \end{aligned}
\]
	Since
	$
	\langle {\bf y}_i \wedge {\bf y}_{i+1},{\bf y}_{i+2} \rangle =
	\det({\bf y}_{i},{\bf y}_{i+1},{\bf y}_{i+2})
	$, we find by Proposition 4.1(d) of \cite{DROY.2}, that
\begin{equation}\label{RP2:P:1:8}
 \begin{aligned}
	\| \langle {\bf y}_i \wedge {\bf y}_{i+1},{\bf y}_{i+2} \rangle {\bf y}_{\nu} \|_{\nu}
	&=
	\frac{|\det({\bf y}_{0},{\bf y}_{1},{\bf y}_{2})|_{\nu}}{|\det({\bf w}_2)|_{\nu}} \| {\bf y}_{\nu} \|_{\nu}| \det({\bf w}_{i+2}) |_{\nu}
	\\
	&\sim
	| \det({\bf w}_{i+2}) |_{\nu}.
 \end{aligned}
\end{equation}
	Since $\delta_{\nu,i}= o(\|{\bf w}_i\|_{\nu})$,
	the above two estimates imply that
\[
	\| \langle {\bf y}_i \wedge {\bf y}_{i+1},{\bf y}_{\nu} \rangle {\bf y}_{i+2} \|_{\nu}
	\sim | \det({\bf w}_{i+2}) |_{\nu}
\]
	which leads to the last estimate in (\ref{RP2:P:1:2})
\[
	| \langle {\bf y}_i \wedge {\bf y}_{i+1},{\bf y}_{\nu} \rangle |_{\nu} \sim
	\frac{| \det({\bf w}_{i+2}) |_{\nu}}{\|{\bf w}_{i+2}\|_{\nu}}
	\sim \delta_{\nu,i+2}.
\]
	In turn, this implies
\begin{align*}
	\| \langle {\bf y}_{i+1} \wedge {\bf y}_{i+2},{\bf y}_{\nu}
	\rangle {\bf y}_{i} \|_{\nu}
	\sim \delta_{\nu,i+3} \|{\bf w}_{i}\|_{\nu}
	&= \frac{| \det({\bf w}_{i+2}) |_{\nu}
	| \det({\bf w}_{i+1}) |_{\nu}\|{\bf w}_{i}\|_{\nu}}{\|{\bf w}_{i+3}\|_{\nu}}
	\\
	& \sim \frac{\delta_{\nu,i+1}}{\|{\bf w}_{i+1}\|_{\nu}}
	|\det({\bf w}_{i+2})|_{\nu}.
\end{align*}
	Since
	$\langle {\bf y}_{i+1} \wedge {\bf y}_{i+2},{\bf y}_i \rangle =
	\det({\bf y}_{i},{\bf y}_{i+1},{\bf y}_{i+2}) =
	\langle {\bf y}_{i} \wedge {\bf y}_{i+1},{\bf y}_{i+2} \rangle$, then
	the estimate (\ref{RP2:P:1:8}) can be rewritten in the form
\[
	\| \langle {\bf y}_{i+1} \wedge {\bf y}_{i+2},{\bf y}_{i} \rangle {\bf y}_{\nu} \|_{\nu} \sim
	| \det({\bf w}_{i+2}) |_{\nu}.
\]
	Now, applying (\ref{RP2:L:1:1}) and using the preceding two estimates,
	we find that
\[
	\|{\bf y}_{i+1} \wedge {\bf y}_{i+2}\|_{\nu} \|{\bf y}_i \wedge {\bf y}_{\nu}\|_{\nu} \geq
	\| \langle {\bf y}_{i+1} \wedge {\bf y}_{i+2},{\bf y}_{\nu} \rangle {\bf y}_i -
	\langle {\bf y}_{i+1} \wedge {\bf y}_{i+2},{\bf y}_i \rangle {\bf y}_{\nu} \|_{\nu} \gg
	| \det({\bf w}_{i+2}) |_{\nu}.
\]
	This together with (\ref{RP2:P:1:6:0}) and (\ref{RP2:P:1:7})
	implies that
\[
	| \det({\bf w}_{i+2}) |_{\nu}
	\ll
	\|{\bf y}_{i+1} \wedge {\bf y}_{i+2}\|_{\nu}
	\|{\bf y}_{i} \wedge {\bf y}_{\nu}\|_{\nu}
	\ll
	\delta_{\nu,i+1}\|{\bf w}_{i+1}\|_{\nu}\delta_{\nu,i}
	\ll
	| \det({\bf w}_{i+2}) |_{\nu}.
\]
	Thus, we have
\[
	\|{\bf y}_{i+1} \wedge {\bf y}_{i+2}\|_{\nu}
	\sim
	\delta_{\nu,i+1}\|{\bf w}_{i+1}\|_{\nu}
	\ \text{ and } \
	\|{\bf y}_{i} \wedge {\bf y}_{\nu}\|_{\nu}
	\sim
	\delta_{\nu,i},
\]
	which prove the first two estimates in (\ref{RP2:P:1:2}).
\end{proof}
\begin{proposition} \label{RP2:P:1-1}
	Let $({\bf w}_i)_{i\geq0}$ be an unbounded admissible
	Fibonacci sequence in $\mathcal{M}$ satisfying
	(\ref{RP2:P:1:0}), with a corresponding sequence of symmetric
	matrices $({\bf y}_i)_{i\geq0}$ satisfying
	$\det({\bf y}_0,{\bf y}_1,{\bf y}_2)\neq0$,
	and let  $\mathcal{S}'$ be a finite subset
	of $\mathfrak{M}_{\mathbb{Q}}$.
	Suppose that, for each $\nu \in \mathcal{S}'$,
	the numbers $\delta_{i,\nu}$ defined by (\ref{RP2:P:1:1})
	satisfy $\delta_{\nu,i} = o(\|{\bf w}_{i}\|_{\nu})$
	as $i\rightarrow\infty$.
	Suppose also that
\begin{equation}\label{RP2:P:1-1:1}
	\prod_{p \in \mathcal{S}'\setminus\{\infty\}}\delta_{p,i}
	=
	\left \{
		\begin{matrix}
			o(\delta_{\infty,i}^{-1}) & \text{ if }
			\infty \in \mathcal{S}',
			\\
			o(\|{\bf w}_i\|_{\infty}^{-1}) & \text{ if }
			\infty \notin \mathcal{S}'.
		\end{matrix}
	\right .
\end{equation}
	Finally, for each $\nu \in \mathcal{S}'$, let ${\bf y}_{\nu}$
	be a non-zero point of $\mathbb{Q}_{\nu}^3$ with $\det({\bf y}_{\nu})=0$
	satisfying (\ref{RP2:P:1:2}),
	as given by Proposition \ref{RP2:P:1}.
	Then the points
	${\bf t'}_{l} = (y_{\nu,l})_{\nu \in \mathcal{S}'} \in
	\prod_{\nu \in \mathcal{S}'}\mathbb{Q}_{\nu}$ ($l=0,1,2$)
	are linearly independent over $\mathbb{Q}$.
\end{proposition}
\begin{proof}
		Suppose on the contrary that the points
		${\bf t'}_{0},{\bf t'}_{1},{\bf t'}_{2}$ are linearly dependent
		over $\mathbb{Q}$.
		This means that there exists a non-zero point
		${\bf u} \in \mathbb{Z}^3$, such that
	$\langle {\bf u},{\bf y}_{\nu} \rangle = 0$
	for each $\nu \in \mathcal{S}'$.
	So, by Lemma \ref{RP2:L:1} in the case where $\nu\neq\infty$
	or by Lemma 2.2 of \cite{DROY.2} otherwise,
	we have for each $\nu \in \mathcal{S}'$ and
	for each $i\geq0$,
\begin{equation}\label{RP2:P:1-1:2}
	|\langle {\bf u},{\bf y}_i \rangle |_{\nu} \|{\bf y}_{\nu} \|_{\nu} =
	\| \langle {\bf u},{\bf y}_{\nu} \rangle {\bf y}_i -
	\langle {\bf u},{\bf y}_i \rangle {\bf y}_{\nu} \|_{\nu} \ll
	\|{\bf u}\|_{\nu} \|{\bf y}_i \wedge {\bf y}_{\nu}\|_{\nu}.
\end{equation}
	Assume that $\langle {\bf u},{\bf y}_i \rangle \neq 0$
	for some $i\geq1$.
	Since $\langle {\bf u},{\bf y}_i \rangle$ is an integer, then
\[
	|\langle {\bf u},{\bf y}_i \rangle|_{\infty}
	\prod_{p \in \mathcal{S}'\setminus\{\infty\}}
	|\langle {\bf u},{\bf y}_i \rangle|_{p} \geq1.
\]
	Using the inequality (\ref{RP2:P:1-1:2})
	and the estimate
	$|\langle {\bf u},{\bf y}_i \rangle|_{\infty}
	\ll \|{\bf y}_i\|_{\infty} \sim \|{\bf w}_i\|_{\infty}$,
	this becomes
\[
	1 \ll
	\Big(
	\prod_{p \in \mathcal{S}'\setminus\{\infty\}}
		\|{\bf y}_i \wedge {\bf y}_{\infty}\|_{p}
	\Big)
	\left \{
		\begin{matrix}
			\|{\bf y}_i \wedge {\bf y}_{\infty}\|_{\infty} & \text{ if }
			\infty \in \mathcal{S}',
			\\
			\|{\bf y}_i\|_{\infty} & \text{ if }
			\infty \notin \mathcal{S}'.
		\end{matrix}
	\right .
\]
	By the second relation in (\ref{RP2:P:1:2}), it follows that
\[
	1 \ll
	\Big(
	\prod_{p \in \mathcal{S}'\setminus\{\infty\}}
				\delta_{p,i}
	\Big)
	\left \{
		\begin{matrix}
			\delta_{\infty,i} & \text{ if }
			\infty \in \mathcal{S}',
			\\
			\|{\bf w}_i\|_{\infty} & \text{ if }
			\infty \notin \mathcal{S}',
		\end{matrix}
	\right .
\]
	which contradicts the hypothesis (\ref{RP2:P:1-1:1})
	if $i$ is large.
	So, we conclude that $\langle {\bf u},{\bf y}_i \rangle = 0$
	for each $i$ sufficiently large.
	By Proposition 4.1(d) of \cite{DROY.2} and the assumption
	that $\det({\bf y}_0,{\bf y}_1,{\bf y}_2) \neq 0$, we know that
	any three consecutive points
	of the sequence $({\bf y}_i)_{i\geq0}$ are linearly independent
	over $\mathbb{Q}$. This is a contradiction.
\end{proof}
\begin{remark} \label{RP2:R:1}
	In particular, if $\mathcal{S}' = \{\nu\}$ for some
	$\nu \in \mathfrak{M}_{\mathbb{Q}}$,
	then all components of the point ${\bf y}_{\nu}$ are non-zero and
	after dividing
	${\bf y}_{\nu}$ by its first coordinate,
	we deduce from the condition $\det({\bf y}_{\nu}) = 0$,
	that ${\bf y}_{\nu}$ can be written in the form
	${\bf y}_{\nu} = (1,\xi_{\nu},\xi_{\nu}^2)$,
	for some $\xi_{\nu} \in \mathbb{Q}_{\nu}$
	with $[\mathbb{Q}(\xi_{\nu}):\mathbb{Q}]>2$.
	So, in this case, we have
	$L_{\nu}({\bf x}) \sim \|{\bf x} \wedge {\bf y}_{\nu}\|_{\nu}$.
\end{remark}
\begin{corollary} \label{RP2:C:1}
	Let $\mathcal{S}$ be a finite set of prime numbers.
	Let $(\alpha_{p})_{p \in \mathcal{S}}$ be a sequence of positive
	real numbers indexed by $\mathcal{S}$ and $\beta$ be a real number such that
\begin{equation}\label{RP2:C:1:0}
	\sum_{p \in \mathcal{S}}\alpha_{p}
	\leq \beta < 2.
\end{equation}
	Let $({\bf w}_i)_{i\geq0}$ be an unbounded admissible Fibonacci sequence in
	$\mathcal{M}$ such that
\begin{equation} \label{RP2:C:1:1}
 \begin{gathered}
	\|{\bf w}_{i+1}\|_{\infty} \sim \|{\bf w}_i\|_{\infty}^{\gamma},
	\quad
	\|{\bf w}_i\|_{p} \sim 1
	\ (\forall p \in \mathcal{S}),
	\\
	|\det({\bf w}_i)|_{\infty} \ll
	\|{\bf w}_i\|_{\infty}^{\beta},
	\quad
	|\det({\bf w}_i)|_{p} \ll \|{\bf w}_i\|_{\infty}^{-\alpha_{p}}
	\ (\forall p \in \mathcal{S}).
 \end{gathered}
\end{equation}
 	Let $({\bf y}_i)_{i\geq0}$ be a corresponding sequence
	of symmetric matrices.
	Assume that $\det({\bf y}_0,{\bf y}_1,{\bf y}_2) \neq 0$ and define
\[
	\mu =
	\left \{
	\begin{aligned}
	&(1-\beta)/\gamma &\text{ if }& \ \beta < 1,
	\\
	&1-\beta &\text{ if }& \ 1\leq \beta < 2.
	\end{aligned}
	\right .
\]
	There exist non-zero points
	${\bf y}_{\infty}
	= (y_{\infty,0},y_{\infty,1},y_{\infty,2})\in \mathbb{R}^3$
	and
	${\bf y}_{p}
	= (y_{p,0},y_{p,1},y_{p,2}) \in \mathbb{Q}_{p}^3$ ($p \in \mathcal{S}$),
	with $\det({\bf y}_{\infty}) = 0$
	and $\det({\bf y}_{p}) = 0$ ($p \in \mathcal{S}$),
	such that  the inequalities
\begin{equation}\label{RP2:C:1:2}
	\|{\bf x}\|_{\infty} \leq X,
	\quad
	\|{\bf x} \wedge {\bf y}_{\infty}\|_{\infty} \ll X^{-\mu},
	\quad
	\|{\bf x} \wedge {\bf y}_{p}\|_{p} \ll X^{-\alpha_{p}/\gamma}
	\text{ for } p \in \mathcal{S},
\end{equation}
	have a non-zero solution ${\bf x}$ in $\mathbb{Z}^3$, for every $X\gg1$.

	Moreover,
\begin{itemize}
	\item[(i)] if $\beta <1$, then the point ${\bf y}_{\infty}$
		can be written in the form
		${\bf y}_{\infty} = (1,\xi_{\infty},\xi_{\infty}^2)$,
		for some $\xi_{\infty} \in \mathbb{R}$
		with $[\mathbb{Q}(\xi_{\infty}):\mathbb{Q}]>2$.
		So, in (\ref{RP2:C:1:2}), we have
		$L_{\infty}({\bf x}) \sim \|{\bf x} \wedge {\bf y}_{\infty}\|_{\infty}$.
	\\
	\item[(ii)] if $\sum_{p \in \mathcal{S}}\alpha_{p} > 1$,
		then the points
		${\bf t'}_{l} = (y_{p,l})_{p \in \mathcal{S}} \in
		\prod_{p \in \mathcal{S}}\mathbb{Q}_{p}$
		($l=0,1,2$) are linearly independent
		over $\mathbb{Q}$.
		\\
	\item[(iii)] if $\beta < 1+\sum_{p \in \mathcal{S}}\alpha_{p}$,
		then the points
		${\bf t'}_{l} =
		\big( y_{\infty,l},(y_{\nu,l})_{p \in \mathcal{S}}\big)
		\in \mathbb{R}\times\prod_{p \in \mathcal{S}}\mathbb{Q}_{p}$
		($l=0,1,2$) are linearly independent
		over $\mathbb{Q}$.
\end{itemize}
\end{corollary}
\begin{proof}
	For each $i\geq0$ we have $\det({\bf w}_i) \neq 0$ and so,
	by (\ref{RP2:C:1:1}),
\[
	1 \leq \prod_{\nu\in \{\infty\}\cup\mathcal{S}}
	|\det({\bf w}_i)|_{\nu}
	\ll
	\|{\bf w}_i\|_{\infty}^{\beta - \sum_{p \in \mathcal{S}}\alpha_{p}}.
\]
	So, we necessarily have
	$\sum_{p \in \mathcal{S}}\alpha_{p}
	\leq \beta$.

	To show the first statement of the corollary we
	apply Proposition \ref{RP2:P:1}.
	To this end, we need only to check that
	$\delta_{\nu,i} = o(\|{\bf w}_i\|_{\nu})$
	for each $\nu \in \{\infty\}\cup\mathcal{S}$.
	By (\ref{RP2:C:1:1}), we find that
\begin{gather*}
	\frac{\delta_{\infty,i}}{\|{\bf w}_i\|_{\infty}} \sim
	\frac{|\det({\bf w}_i)|_{\infty}}{\|{\bf w}_i\|_{\infty}^2} \ll
	\|{\bf w}_i\|_{\infty}^{\beta-2} = o(1),
\\
	\frac{\delta_{p,i}}{\|{\bf w}_i\|_{p}} \sim
	|\det({\bf w}_i)|_{p} \ll
	\|{\bf w}_i\|_{\infty}^{-\alpha_{p}} = o(1)
	\ \text{ for } p \in \mathcal{S}.
\end{gather*}
	So, Proposition \ref{RP2:P:1} provides non-zero points
	${\bf y}_{\infty} \in \mathbb{R}^3$ and
	${\bf y}_{p} \in \mathbb{Q}_{p}^3$ ($p \in \mathcal{S}$)
	with $\det({\bf y}_{\nu}) = 0$
	for each $\nu \in \mathcal{S}\cup\{\infty\}$,
	satisfying the estimates (\ref{RP2:P:1:2}).

	Fix a real number $X$. If $X$ is sufficiently large,
	there exists an index $i\geq0$ such that
\begin{equation}\label{RP2:C:1:4}
	\|{\bf w}_{i}\|_{\infty}
	\leq X < \|{\bf w}_{i+1}\|_{\infty}
	\sim \|{\bf w}_{i}\|_{\infty}^{\gamma}.
\end{equation}
	By (\ref{RP2:P:1:2}) and (\ref{RP2:C:1:1}),
	we have
\begin{gather*}
 	\|{\bf y}_i \wedge {\bf y}_{\infty}\|_{\infty} \sim
	\frac{|\det({\bf w}_i)|_{\infty}}{\|{\bf w}_i\|_{\infty}} \ll
	\|{\bf w}_i\|_{\infty}^{-1+\beta},
	\\
	\|{\bf y}_i \wedge {\bf y}_{p}\|_{p} \sim
	\frac{|\det({\bf w}_i)|_{p}}{\|{\bf w}_i\|_{p}} \sim
	|\det({\bf w}_i)|_{p} \ll
	\|{\bf w}_i\|_{\infty}^{-\alpha_{p}}
	\ \text{ for } p \in \mathcal{S}.
\end{gather*}
	Combining this with (\ref{RP2:C:1:4}) we obtain
\[
	\|{\bf y}_i \wedge {\bf y}_{\infty}\|_{\infty} \ll X^{-\mu}
	\ \text{ and } \
	\|{\bf y}_i \wedge {\bf y}_{p}\|_{p} \ll
	X^{-\alpha_{p}/\gamma}
	\text{ for } p \in \mathcal{S}.
\]
	Thus the point ${\bf x} = {\bf y}_i$ satisfies (\ref{RP2:C:1:2}).

	To prove Part (i) in the statement of the corollary,
	we apply Proposition \ref{RP2:P:1-1}
	with $\mathcal{S}' = \{\infty\}$.
	Here we assume that $\beta<1$ and
	we have only to check that the condition (\ref{RP2:P:1-1:1}) holds.
	Indeed, we find
\[
	\delta_{\infty,i}
	=
	\frac{|\det({\bf w}_i)|_{\infty}}{\|{\bf w}_i\|_{\infty}} \ll
	\|{\bf w}_i\|_{\infty}^{\beta-1} = o(1).
\]
	Applying Remark \ref{RP2:R:1} completes the proof of Part (i).

	Similarly, to show Part (ii),
	we apply Proposition \ref{RP2:P:1-1}
	with $\mathcal{S}' = \mathcal{S}$.
	Since $\sum_{p \in \mathcal{S}}\alpha_{p} > 1$,
	we find by (\ref{RP2:C:1:1}) that
\begin{align*}
	\|{\bf w}_i\|_{\infty}
	\prod_{p\in \mathcal{S}}\delta_{p,i}
	&\sim
	\|{\bf w}_i\|_{\infty}
	\prod_{p\in \mathcal{S}}|\det({\bf w}_i)|_{p}
	\ll
	\|{\bf w}_i\|_{\infty}^{1 - \sum_{p \in \mathcal{S}}\alpha_{p}} = o(1),
\end{align*}
	which shows that the condition (\ref{RP2:P:1-1:1}) holds.

	Finally, to show Part (iii), we use Proposition \ref{RP2:P:1-1}
	with $\mathcal{S}' = \{\infty\}\cup\mathcal{S}$.
	Again, we need only to check that the condition (\ref{RP2:P:1-1:1}) holds.
	Since $\beta < 1 + \sum_{p \in \mathcal{S}}\alpha_{p}$,
	we find by (\ref{RP2:C:1:1}) that
\[
	\prod_{\nu\in \{\infty\}\cup\mathcal{S}}\delta_{\nu,i}
	\sim
	\prod_{\nu\in \{\infty\}\cup\mathcal{S}}
	\frac{|\det({\bf w}_i)|_{\nu}}{\|{\bf w}_i\|_{\nu}}
	\ll
	\|{\bf w}_i\|_{\infty}^{\beta - 1 - \sum_{p \in \mathcal{S}}\alpha_{p}}
\]
	and thus (\ref{RP2:P:1-1:1}) holds.
\end{proof}
\begin{theorem} \label{RP2:C:2}
	Let $\mathcal{S}$ be a finite set of prime numbers.
	For any
	$\bar \lambda = \big( \lambda_{\infty},
	(\lambda_{p})_{p \in \mathcal{S}} \big ) \in \mathbb{R}_{>0}^{|\mathcal{S}|+1}$
	with
\begin{equation}\label{RP2:C:2:0}
	\sum_{\nu \in \mathcal{S}\cup\{\infty\}}\lambda_{\nu}
	< \frac{1}{\gamma}
\end{equation}
	there exists a non-zero point
	$\bar \xi = \big( \xi_{\infty},
	(\xi_{p})_{p \in \mathcal{S}} \big ) \in
	\mathbb{R} \times \prod_{p \in \mathcal{S}}\mathbb{Q}_{p}$
	with $[\mathbb{Q}(\xi_{\infty}):\mathbb{Q}]>2$,
	such that
	$\bar \lambda$ is an exponent of approximation in degree $2$
	to $\bar \xi$.
\end{theorem}
\begin{proof}
	For the proof we use the construction of Example 3.3 in \cite{DROY.2}.
	Fix integers $a, b, c$ with $a \geq 2$ and $c > b \geq 1$.
	We consider the Fibonacci sequence $({\bf w}_i)_{i\geq0}$ in $\mathcal{M}$
	with initial matrices ${\bf w}_0$, ${\bf w}_1$ given by
\[
	{\bf w}_0 = \left (
	\begin{matrix}
	1 &  b
	\\
	a &  a(b+1)
	\end{matrix}
	\right ),
	\
	{\bf w}_1 = \left (
	\begin{matrix}
	1 &  c
	\\
	a &  a(c+1)
	\end{matrix}
	\right ).
\]
	Put
\[
	N = \left (
	\begin{matrix}
	-1 + a(b+1)(c+1)&  -a(b+1)
	\\
	-a(c+1) &  a
	\end{matrix}
	\right ).
\]
	Since the matrices
\begin{align*}
	{\bf y}_0 &= {\bf w}_0 N =
	\left (
	\begin{matrix}
	-1 + a(c+1)&  -a
	\\
	-a &  0
	\end{matrix}
	\right ),
	\\
	{\bf y}_1 &= {\bf w}_1 {}^tN =
	\left (
	\begin{matrix}
	-1 + a(b+1)&  -a
	\\
	-a &  0
	\end{matrix}
	\right ),
	\\
	{\bf y}_2 &= {\bf w}_1{\bf w}_0 N =
	\left (
	\begin{matrix}
	-1 + a&  -a
	\\
	-a &  -a^2
	\end{matrix}
	\right )
\end{align*}
	are symmetric, it follows from Proposition 3.1 of \cite{DROY.2}
	that the Fibonacci sequence $({\bf w}_i)_{i\geq0}$ is admissible
	with $\det({\bf y}_0,{\bf y}_1,{\bf y}_2) = a^4(c-b)$
	and $\det({\bf w}_0) = \det({\bf w}_1) = -\det(N) = a$.
	By Lemma 5.1 of \cite{DROY.2}, we have
\[
	\|{\bf w}_{i}\|_{\infty}\|{\bf w}_{i+1}\|_{\infty}
	<
	\|{\bf w}_{i+2}\|_{\infty}
	\leq
	2\|{\bf w}_{i}\|_{\infty}\|{\bf w}_{i+1}\|_{\infty}
\]
	for each $i\geq0$.
	By Lemma 5.2 of \cite{DROY.2}, this implies that
	the sequence $({\bf w}_i)_{i\geq0}$ is unbounded with
\begin{equation}\label{RP2:C:2:2:0}
	\|{\bf w}_{i+1}\|_{\infty} \sim \|{\bf w}_i\|_{\infty}^{\gamma}.
\end{equation}
	Since
\[
	{\bf w}_0, {\bf w}_1 \equiv
	\left (
	\begin{matrix}
	1 &  *
	\\
	0 &  0
	\end{matrix}
	\right )
	\mod a
	\ \text{ and } \
	N \equiv
	\left (
	\begin{matrix}
	-1 &  0
	\\
	0 &  0
	\end{matrix}
	\right )
	\mod a,
\]
	we deduce that
\[
	{\bf w}_i \equiv
	\left (
	\begin{matrix}
	1 &  *
	\\
	0 &  0
	\end{matrix}
	\right )
	\mod a
      \ \text{ and } \
	{\bf y}_i \equiv
	\left (
	\begin{matrix}
	- 1 &  0
	\\
	0 &  0
	\end{matrix}
	\right )
	\mod a
	\ \text{ for any } i\geq0.
\]
	Using the hypothesis (\ref{RP2:C:2:0}), we can choose numbers
	$(\lambda_{\nu}')_{\nu \in \mathcal{S}\cup\{\infty\}}$,
	such that
\begin{equation}\label{RP2:C:2:2:1}
	\lambda_{\nu}' > \lambda_{\nu}
	\ ( \nu \in  \{\infty\}\cup\mathcal{S} )
	\quad \text{ and }
	\sum_{\nu \in \{\infty\}\cup\mathcal{S}}\lambda_{\nu}' = \frac{1}{\gamma}.
\end{equation}
	Fix an arbitrarily large integer $N>0$
	and define integers $a, b, c$ by
\begin{equation}\label{RP2:C:2:2:2}
	a = \prod_{p \in \mathcal{S}}
		p^
		{ \big [
			\gamma\lambda_{p}'N / \log p
		\big ] },
	\quad
	b = 	\exp \big [ \gamma\lambda_{\infty}'N \big ]
	\quad \text{ and }
	c = b + 1.
\end{equation}
	This choice of $a$
	implies that, for each $i\geq0$, we have
\begin{equation}\label{RP2:C:2:2:0:1}
	\|{\bf w}_i\|_{p} = \|{\bf y}_i\|_{p} = 1
	\quad \text{ for each } p \in \mathcal{S}.
\end{equation}
	So, together with (\ref{RP2:C:2:2:0}) this
	means that the first two conditions in (\ref{RP2:C:1:1})
	of Corollary \ref{RP2:C:1} are satisfied.

	Define
\begin{equation}\label{RP2:C:2:2:3}
	\beta = \frac{\log(a)}{\log(a(b+1))}
	\ \text{ and } \
	\alpha_{p} = \frac{\log(|a|_{p}^{-1})}{\log(2a(c+1))}
	\ \text{ for each } \ p \in \mathcal{S}.
\end{equation}
	We have that $0< \beta < 1$ and
	$0<\alpha_{p}<1$ for each $\ p \in \mathcal{S}$.
	So, the inequalities
\begin{equation}\label{RP2:C:2:2:4}
	|\det({\bf w}_i)|_{\infty} \leq \|{\bf w}_i\|_{\infty}^{\beta}
	\ \text{ and } \
	|\det({\bf w}_i)|_{p} \leq (2\|{\bf w}_i\|_{\infty})^{-\alpha_{p}}
	\ (p \in \mathcal{S})
\end{equation}
	hold for $i = 0,1$.
	We claim that these inequalities (\ref{RP2:C:2:2:4}) hold for any $i\geq0$.
	To prove this we proceed by induction on $i$.
	Assuming that $i\geq2$ and that the inequalities hold for $i<2$,
	we get
\begin{align*}
	|\det({\bf w}_i)|_{\infty}
		&=
		|\det({\bf w}_{i-1})|_{\infty}
		|\det({\bf w}_{i-2})|_{\infty}
		\\
		&\leq
		\|{\bf w}_{i-1}\|_{\infty}^{\beta}
		\|{\bf w}_{i-2}\|_{\infty}^{\beta}
		\leq
		\|{\bf w}_i\|_{\infty}^{\beta},
	\\
	|\det({\bf w}_i)|_{p}
		&=
		|\det({\bf w}_{i-1})|_{p}
		|\det({\bf w}_{i-2})|_{p}
		\\
		&\leq
		(4\|{\bf w}_{i-1}\|_{\infty}\|{\bf w}_{i-2}\|_{\infty})^{-\alpha_{p}}
		\leq
		(2\|{\bf w}_i\|_{\infty})^{-\alpha_{p}}
		\quad (p \in \mathcal{S}),
\end{align*}
	which completes the induction step.
	Together with (\ref{RP2:C:2:2:0}), (\ref{RP2:C:2:2:0:1}) and
	these estimates, we have that conditions (\ref{RP2:C:1:1})
 	of Corollary \ref{RP2:C:1} are satisfied.
	We also note that the numbers $(\alpha_{p})_{p \in \mathcal{S}}$
	and $\beta$ satisfy
\[
	\sum_{p \in \mathcal{S}}\alpha_{p} =
	\frac{\log(\prod_{p \in \mathcal{S}}|a|_{p}^{-1})}{\log(2a(c+1))} =
	\frac{\log(a)}{\log(2a(c+1))} <
	\frac{\log(a)}{\log(a(b+1))} = \beta < 1.
\]
	So, by Corollary \ref{RP2:C:1} and moreover, by its Part (i),
	there exist a non-zero point
	$\xi_{\infty} \in \mathbb{R}$ with
	$[\mathbb{Q}(\xi_{\infty}):\mathbb{Q}]>2$
	and non-zero points
	${\bf y}_{p} \in \mathbb{Q}_{p}^3$
	with $\det({\bf y}_{p}) = 0$
	for each $p \in \mathcal{S}$,
	such that inequalities
\begin{equation}\label{RP2:C:2:2:5}
	\|{\bf x}\|_{\infty} \leq X,
	\quad
	L_{\infty}({\bf x}) \ll X^{-(1-\beta)/{\gamma}},
	\quad
	\|{\bf x} \wedge {\bf y}_{p}\|_{p} \ll X^{-\alpha_{p}/\gamma}
	\text{ for each } p \in \mathcal{S},
\end{equation}
	have a non-zero solution ${\bf x} \in \mathbb{Z}^3$, for each $X \geq 1$.
	Moreover, since ${\bf y}_{k} \equiv (-1,0,0) \mod a$
	for each $k\geq0$, the first component of ${\bf y}_{p}$ is non-zero
	for each $p \in \mathcal{S}$.
	So, for each $p \in \mathcal{S}$,
	we deduce from the relation $\det({\bf y}_{p}) = 0$,
	that ${\bf y}_{p}$ is a rational multiple of $(1,\xi_p,\xi_p^2)$
	for some $\xi_p \in \mathbb{Q}_p$. 
	Then, we have
	$L_{p}({\bf x}) \sim \|{\bf x} \wedge {\bf y}_{p}\|_{p}$, and
	(\ref{RP2:C:2:2:5}) can be rewritten as
\begin{equation}\label{RP2:C:2:2:6}
	\|{\bf x}\|_{\infty} \leq X,
	\quad
	L_{\infty}({\bf x}) \ll X^{-(1-\beta)/{\gamma}},
	\quad
	L_{p}({\bf x}) \ll X^{-\alpha_{p}/\gamma}
	\text{ for } p \in \mathcal{S}.
\end{equation}
	By (\ref{RP2:C:2:2:1}), (\ref{RP2:C:2:2:2}) and (\ref{RP2:C:2:2:3}),
	we find that $\beta$ and $\alpha_p \ (p\in\mathcal{S})$
	converge respectively to
\[
	\frac{\sum_{q \in \mathcal{S}}\lambda_{q}'}
		{\sum_{q \in \mathcal{S}}\lambda_{q}'
		+ \lambda_{\infty}'}
	=
	\gamma\sum_{q \in \mathcal{S}}\lambda_{q}'
	=
	1-\gamma\lambda_{\infty}'
	\quad \text{and} \quad
	\frac{\lambda_{p}'}
	{\sum_{q \in \mathcal{S}}\lambda_{q}' + \lambda_{\infty}'}
	=
	\gamma\lambda_{p}' \ (p \in \mathcal{S}).
\]
	Thanks to (\ref{RP2:C:2:2:1}), this means that
	we can choose $N$ large enough so that
\[
	\frac{1-\beta}{\gamma}
	> \lambda_{\infty}
	\ \text{ and } \
	\frac{\alpha_{p}}{\gamma} > \lambda_{p}
	\ (p \in \mathcal{S}).
\]
	Then, according to (\ref{RP2:C:2:2:6}), the inequalities
\[
	\|{\bf x}\|_{\infty} \leq X,
	\quad
	L_{\infty}({\bf x}) \ll X^{-\lambda_{\infty}},
	\quad
	L_{p}({\bf x}) \ll X^{-\lambda_{p}}
	\ (p \in \mathcal{S}),
\]
	have a non-zero solution ${\bf x} \in \mathbb{Z}^3$ for each $X \geq 1$.
	This completes the proof.

\end{proof}

%% file: I-10-optimal-exponents-padic.tex
\subsection{$P$-adic case.}\label{C1:S5:SS2:p-adic}

	Here we present a p-adic version of the results of the
	previous paragraph.
\begin{corollary}\label{III:C:1}
	Let $p$ be a prime number and
	let $({\bf w}_i)_{i\geq0}$ be an unbounded admissible Fibonacci sequence in
	$\mathcal{M}$ such that
\begin{equation}\label{III:C:1:1}
	\|{\bf w}_i\|_{p} \sim 1, \quad
	|\det({\bf w}_i)|_{\infty} |\det({\bf w}_i)|_{p} \sim 1
	\ \text{ and } \
	\|{\bf w}_i\|_{\infty}^{\alpha_{p}} \ll |\det({\bf w}_i)|_{\infty},
\end{equation}
	for some real $\alpha_{p}>1$.
	Suppose that the corresponding sequence $({\bf y}_i)_{i\geq0}$
	satisfies $\det({\bf y}_0,{\bf y}_1,{\bf y}_2) \neq 0$.
	Then there exists a non-zero number $\xi_{p} \in \mathbb{Q}_p$
	with $[\mathbb{Q}(\xi_{p}):\mathbb{Q}]>2$ and a constant $c_1>0$
	such that inequalities
\begin{equation}\label{III:C:1:2}
	\|{\bf x}\|_{\infty} \leq X, \quad
	\|{\bf x}\|_{\infty} L_{p}({\bf x}) \leq c_1 X^{-(\alpha_{p}-1)/\gamma},
\end{equation}
	have a non-zero solution ${\bf x}$ in $\mathbb{Z}^3$ for every $X\gg1$.
\end{corollary}
\begin{proof}
	Proposition \ref{RP2:P:1-1} applies to the present situation with
	$\mathcal{S}' = \{p\}$ since by (\ref{III:C:1:1}), we have
\[
	\frac{\delta_{p,i}}{\|{\bf w}_i\|_{p}}
	=
	\frac{|\det({\bf w}_i)|_{p}}{\|{\bf w}_i\|_{p}^2}
	\sim
	|\det({\bf w}_i)|_{p}
	\ll
	\|{\bf w}_i\|_{\infty}^{-\alpha_{p}} = o(1)
\]
	and
\[
	\|{\bf w}_i\|_{\infty}\delta_{p,i}
	=
	\frac{\|{\bf w}_i\|_{\infty}|\det({\bf w}_i)|_{p}}{\|{\bf w}_i\|_{p}}
	\sim
	\frac{\|{\bf w}_i\|_{\infty}}{|\det({\bf w}_i)|_{\infty}}
	\ll
	\|{\bf w}_i\|_{\infty}^{1-\alpha_p} = o(1).
\]
	Hence, by Remark \ref{RP2:R:1}, there exists
	$\xi_{p} \in \mathbb{Q}_p$ with $[\mathbb{Q}(\xi_{p}):\mathbb{Q}]>2$,
	such that the point
	${\bf y}_{p} = (1,\xi_{p},\xi_{p}^2) \in \mathbb{Q}_{p}^3$
	satisfies the estimates (\ref{RP2:P:1:2}).
	Using (\ref{III:C:1:1}), this means that
\begin{equation}\label{RP2:C:1:5_1}
	L_{p}({\bf y}_i)
	= \|{\bf y}_i \wedge {\bf y}_{p}\|_{p} \sim
	\frac{|\det({\bf w}_i)|_{p}}{\|{\bf w}_i\|_{p}} \sim
	|\det({\bf w}_i)|_{p} \sim
	|\det({\bf w}_i)|_{\infty}^{-1}.
\end{equation}
	Now, we fix a real number $Y\geq1$.
	If $Y$ is sufficiently large,
	there exists an index $i\geq0$ such that
\[
	|\det({\bf w}_{i})|_{\infty}
	\leq Y
	< |\det({\bf w}_{i+1})|_{\infty}
	\sim
	|\det({\bf w}_{i})|_{\infty}^{\gamma}.
\]
	By the hypothesis, there also exists a constant $c>0$
	independent of $i$ such that
\begin{equation}\label{RP2:C:1:5_2}
	c^{-1} \|{\bf y}_i\|_{\infty} \leq
	\|{\bf w}_i\|_{\infty} \leq
	c|\det({\bf w}_{i})|_{\infty}^{1/\alpha_{p}}.
\end{equation}
	Combining (\ref{RP2:C:1:5_1}) with the previous two inequalities, we get
\begin{align*}
	\|{\bf y}_i\|_{\infty} L_{p}({\bf y}_i)
	&\ll
	|\det({\bf w}_i)|_{\infty}^{1/\alpha_{p}} |\det({\bf w}_i)|_{\infty}^{-1}
	=
	|\det({\bf w}_i)|_{\infty}^{(1-\alpha_{p})/\alpha_{p}}
	\\
	&\sim
	|\det({\bf w}_{i+1})|_{\infty}^{-(\alpha_{p}-1)/(\alpha_{p}\gamma)}
	\ll
	Y^{-(\alpha_{p}-1)/(\alpha_{p}\gamma)}.
\end{align*}
	So, putting $X = c^2Y^{1/\alpha_{p}}$, 
	we find that the point ${\bf y}_{i}$ satisfies
\[
	\|{\bf y}_i\|_{\infty} \leq X, \quad
	\|{\bf y}_i\|_{\infty} L_{p}({\bf y}_i) \ll X^{-(\alpha_{p}-1)/\gamma}.
\]
	Since $X$ is a continuous increasing function of $Y$,
	the conclusion follows.
\end{proof}
\begin{corollary} \label{III:C:2}
	Let $p$ be a prime number and let $\epsilon>0$.
 	There exist a $p$-adic number $\xi_{p} \in \mathbb{Q}_p$
	with $[\mathbb{Q}(\xi_{p}):\mathbb{Q}]>2$
	and a constant $c_1>0$, such that inequalities
\begin{equation}\label{III:C:2:1}
	\| {\bf x} \|_{\infty} \leq X, \quad
	\| {\bf x} \|_{\infty} L_{p}({\bf x }) \leq c_1X^{-1/\gamma+\epsilon},
\end{equation}
	have a non-zero solution ${\bf x} \in \mathbb{Z}^3$ for any $X \geq 1$.
\end{corollary}
\begin{proof}
	Let $m\geq1$ be an integer to be determined later.
	Fix a real $\epsilon>0$ and consider the Fibonacci sequence
	$({\bf w}_i)_{i\geq0}$ of $\mathcal{M}$ generated by the matrices
\[
	{\bf w}_0 = \left (
		\begin{matrix}
			1 &  p
			\\
			p &  0
		\end{matrix}
		\right ), \quad
	{\bf w}_1 = \left (
		\begin{matrix}
			1 &  p^m
			\\
			-p^m &  0
		\end{matrix}
		\right ).
\]
	Since
\[
	\| {\bf w}_{i+1} \|_{\infty} \leq
	2\| {\bf w}_{i} \|_{\infty}\| {\bf w}_{i-1} \|_{\infty},
\]
	we claim that the following estimates are satisfied for each $i\geq0$
\begin{equation}\label{III:C:2:2}
 \begin{aligned}
 	&\| {\bf w}_{i} \|_{\infty} \leq
	2^{f_{i+1} - 1} \|{\bf w}_0\|_{\infty}^{f_{i-1}}
	\| {\bf w}_1 \|_{\infty}^{f_i},
	\\
	&\det({\bf w}_{i}) = \det({\bf w}_0)^{f_{i-1}} \det({\bf w}_1)^{f_i},
 \end{aligned}
\end{equation}
	where $(f_i)_{i\geq-1}$ is the Fibonacci sequence defined by
	the conditions $f_{-1} = 1$ and $f_{0} = 0$,
	and the recurrence formula $f_{i+1} = f_{i} + f_{i-1}$
	for each $i\geq0$.
	Clearly the relations (\ref{III:C:2:2}) hold for $i=0$.
	Suppose that they hold for some index $i\geq0$. We find
\begin{align*}
	\| {\bf w}_{i+1} \|_{\infty} & \leq
	2\| {\bf w}_{i} \|_{\infty}\| {\bf w}_{i-1} \|_{\infty}
	\\
	&\leq
	2
	\big (
	2^{f_{i+1} - 1} \|{\bf w}_0\|_{\infty}^{f_{i-1}}
	\| {\bf w}_1 \|_{\infty}^{f_i}
	\big )
	\big (
	2^{f_{i} - 1} \|{\bf w}_0\|_{\infty}^{f_{i-2}}
	\| {\bf w}_1 \|_{\infty}^{f_{i-1}}
	\big )
	\\
	&=
	2^{f_{i+2}-1} \|{\bf w}_0\|_{\infty}^{f_{i}}
	\| {\bf w}_1 \|_{\infty}^{f_{i+1}}
\end{align*}
	and
\begin{align*}
	\det({\bf w}_{i+1}) & = \det({\bf w}_{i})\det({\bf w}_{i-1})
	\\
	&=
	\big (\det({\bf w}_0)^{f_{i-1}} \det({\bf w}_1)^{f_i}\big )
	\big (\det({\bf w}_0)^{f_{i-2}} \det({\bf w}_1)^{f_{i-1}}\big )
	\\
	&=
	\det({\bf w}_0)^{f_{i}} \det({\bf w}_1)^{f_{i+1}}.
\end{align*}
	Thus (\ref{III:C:2:2}) holds for each $i\geq0$.
	Since $\| {\bf w}_0 \|_{\infty} = p$
	and $\| {\bf w}_1 \|_{\infty} = p^m$, this means that
\[
	\| {\bf w}_{i} \|_{\infty} \leq
	2^{f_{i+1}-1} p^{m f_i + f_{i-1}} \leq
	2^{f_{i+1}} p^{m f_i + f_{i-1}}.
\]
	Since $2 \leq p$ and $f_{i+1} + f_{i-1} \leq 3f_{i}$
	for each $i\geq1$, this gives
\begin{equation}\label{III:C:2:3}
	\| {\bf w}_{i} \|_{\infty} \leq
	p^{m f_i + f_{i-1} + f_{i+1}} \leq
	p^{f_i(m+3)}
	\ \text{ if } i\geq1.
\end{equation}
	Put $\alpha_{p} = 2m/(m+3)$. Since
	$|\det({\bf w}_0)|_{\infty} = p^2$ and
	$|\det({\bf w}_1)|_{\infty} = p^{2m}$,
	the estimates (\ref{III:C:2:2}) also give
	$|\det({\bf w}_{i})|_{\infty} = p^{2(mf_i+f_{i-1})}$ and,
	from (\ref{III:C:2:3}), it follows that
\[
	\| {\bf w}_{i} \|_{\infty}^{\alpha_{p}} \leq
	p^{2mf_i} \leq
	p^{2(mf_i+f_{i-1})} = |\det({\bf w}_{i})|_{\infty}
	\ \text{ if } i\geq1.
\]
	We claim that the Fibonacci sequence
	$({\bf w}_{i})_{i\geq0}$ satisfies all the requirements
	of Corollary \ref{III:C:1}.
	Indeed, we find that it is admissible with the corresponding matrix
\[
    N = \left (
	\begin{matrix}
		p(p^{m + 1} + p^{2 m}) & -p (p + p^m - 2 p^{2 m + 1})
			\\
		-p^{m + 1} - 2 p^{2 m + 2} - p^{2 m} &  p + p^{m + 2} + p^m  - p^{2 m + 1}
	\end{matrix}
	\right )
\]
	and that the determinant of the first three consecutive points
	of the corresponding sequence $({\bf y}_i)_{i\geq0}$ is
\[
	\det({\bf y}_0,{\bf y}_1,{\bf y}_2) =
	p^{8 m + 4}(16 p^4 + 8 p^2 + 1) - p^{6 m + 6}2(4 p^2 + 1) + p^{4 m + 8} > 0.
\]
	The sequence $({\bf w}_i)_{i\geq0}$ is unbounded
	as $|\det({\bf w}_i)|_{\infty}$ tends to infinity with $i$.
	 We also have
\[
	|\det({\bf w}_{i})|_{\infty}|\det({\bf w}_{i})|_p = 1
\]
	and
\[
	{\bf w}_{i} \equiv
	\left (
	\begin{matrix}
		1 & 0
		\\
		0 &  0
	\end{matrix}
	\right )
	\mod p,
\]
	for each $i\geq0$.
	The latter relation implies that $\|{\bf w}_{i}\|_p = 1$.
	So, all the requirements of Corollary \ref{III:C:1} are satisfied
	with our choice of $\alpha_p$.
	Choose $m \geq 3(2-\epsilon\gamma)/(\epsilon\gamma)$,
	so that $(\alpha_{p}-1)/\gamma \geq 1/\gamma - \epsilon$.
	Then the number $\xi_{p} \in \mathbb{Q}_{p}$ provided by
	Corollary \ref{III:C:1} has all the required properties.
\end{proof}
 

%% file: II-01-duality-real-padic.tex
	Let $n \geq 1$ be an integer and
	let $\mathcal{S}$ be a finite set of prime numbers.
	Fix $\bar \xi = (\xi_{\infty},(\xi_{p})_{{p} \in \mathcal{S}})
	\in \mathbb{R} \times \prod_{{p} \in \mathcal{S}}\mathbb{Q}_{p}$.
	Recall that for any point ${\bf  x} = (x_{0}, x_{1},\ldots,x_{n}) \in
	\mathbb{Q}_{\nu}^{n+1}$	we define the ${\nu}$-adic norm of ${\bf x}$ by
\[
 	\| {\bf  x } \|_{\nu} := \max_{0 \leq i\leq n}\{|x_i|_{\nu}\},
\]
	and we put
\[
 	L_{\nu}({\bf x }) := \| {\bf  x }  - x_{0} {\bf t }_{\nu} \|_{\nu},
\]
	where ${\bf t}_{\nu} := (1,\xi_{\nu},\ldots,\xi_{\nu}^n)$.
	We denote by $|\mathcal{S}|$ the number of elements in $\mathcal{S}$.

	In this chapter we extend the method of
	H.~{\sc Davenport} and W.M.~{\sc Schmidt} in \cite{DS}
	to the study of simultaneous approximation to the 
	real and p-adic components of $\bar\xi$ by algebraic numbers of
	a restricted type.
	To do this, we assume that for a given
	$\bar \lambda = (\lambda_{\infty},(\lambda_{p})_{{p} \in \mathcal{S}})
	\in \mathbb{R}^{|\mathcal{S}|+1}$
	and some constant $c>0$ the inequalities
 \begin{align*}
	&\| {\bf  x } \|_{\infty} \leq X,
	\\
	&L_{\infty}({\bf x })
		\leq c X^{-\lambda_{\infty}},
	\\
	&L_{p}({\bf x })
		\leq c X^{-\lambda_{p}} \
		( \forall {p} \in \mathcal{S} ),
 \end{align*}
	have no non-zero solution ${\bf x} \in \mathbb{Z}^{n+1}$
	for arbitrarily large values of $X$.
	We can reformulate this by saying that the convex body
\begin{equation}\label{dual:1}
	C = \{{\bf x} \in \mathbb{R}^{n+1} \ | \
		\| {\bf x} \|_{\infty} \leq X, \
		L_{\infty}({\bf x }) \leq  c X^{-\lambda_{\infty}}
		\}
\end{equation}
	contains no non-zero points of the lattice
\begin{equation}\label{dual:2}
	\Lambda = \{{\bf x} \in \mathbb{Z}^{n+1} |
	L_{p}({\bf x }) \leq  c X^{-\lambda_{p}}
	\ \forall {p} \in \mathcal{S} \},
\end{equation}
	for arbitrarily large values of $X$.
	Therefore, for these values of $X$
	the first minimum $\tau_1$ of $\Lambda$
	with respect to $C$ is $> 1$.
	By Mahler's Duality, we have that $ \tau_1 \tau_{n+1}^* \leq (n+1)!$~,
	where $\tau_{n+1}^*$ is the last minimum of the dual lattice $\Lambda^*$
	with respect to the dual convex body $C^*$.
	So, for these values of $X$,
	we get $\tau_{n+1}^*  \leq (n+1)!$ and thus there exist
	$n+1$ linearly independent points in $\Lambda^* \cap (n+1)!~C^*$.
	This translates into the existence of $n+1$ linearly independent
	polynomials of $\mathbb{Z}[T]$ of degree $\leq n$
	taking simultaneously small values at the points $\xi_{\nu}$
	with $\nu \in \mathcal{S}\cup\{\infty\}$.
	Using this, we show that
	for any polynomial $R(T) \in \mathbb{Z}[T]$ satisfying
	mild assumptions, there exist infinitely many polynomials
	$F(T) \in \mathbb{Z}[T]$ with
	the following properties:
	\\
	$(i)$ \ $\deg(R - F) \leq n$,
	\\
	$(ii)$ there exists a real root $\alpha_{\infty}$ of $F$ such that
\[
	| \xi_{\infty} - \alpha_{\infty} |_{\infty}
		\ll H(F)^{-(\lambda_{\infty}+1)/\lambda},
\]
	\\
	$(iii)$ for each ${p} \in \mathcal{S}$,
		there exists a root $\alpha_{p}$ of $F$
		in ${\mathbb{Q}}_p$ such that
\[
 	| \xi_p - \alpha_p |_p \ll H(F)^{-\lambda_{p}/\lambda},
\]
	where $\lambda= \lambda_{\infty} + \sum_{{p} \in \mathcal{S}}\lambda_{p}$.

\section{Simultaneous case.}

	We start by proving an explicit version of the
	Strong Approximation Theorem over $\mathbb{Q}$ (see \cite{CASS-FROH}).
\begin{lemma} \label{RP3:L:1}
	For any
	$\bar \epsilon = (\epsilon_{\infty},(\epsilon_{p})_{{p} \in \mathcal{S}})
	\in \mathbb{R}_{>0}^{|\mathcal{S}|+1}$ satisfying the inequality
\begin{equation}\label{RP3:L:1:1}
	\epsilon_{\infty} \geq
	\frac{1}{2} \prod_{{p} \in \mathcal{S}} \epsilon_{p}^{-1} p,
\end{equation}
	there exists a rational number
	$r \in \mathbb{Q}$ such that
\begin{equation}\label{RP3:L:1:2}
 \begin{aligned}
	&| r - \xi_{\infty} |_{\infty} \leq \epsilon_{\infty},
	\\
	&| r - \xi_{p} |_{p} \leq \epsilon_{p}
	\quad \forall {p} \in \mathcal{S},
	\\
	&| r |_{q} \leq 1
	\quad \forall {q} \notin \mathcal{S}.
 \end{aligned}
\end{equation}
\end{lemma}
\begin{proof}
	For each ${p} \in \mathcal{S}$ there exists $n_{p} \in \mathbb{Z}$,
	such that
\begin{equation}\label{RP3:L:1:3}
	p^{-n_{p}-1} \leq \epsilon_{p} < p^{-n_{p}},
\end{equation}
	Define $M = \prod_{{p} \in \mathcal{S}}p^{n_p+1}$.
	By (\ref{RP3:L:1:3}) we find that
\begin{gather*}
	|M|_{p} = p^{-n_p-1} \leq \epsilon_{p}
	\quad \forall {p} \in \mathcal{S},
	\\
	|M|_{q} = 1
	\quad \forall {q} \notin \mathcal{S}.
\end{gather*}
	By the Strong Approximation Theorem (see \cite{CASS-FROH}),
	there exists $\hat r \in \mathbb{Q}$ such that
\begin{equation}\label{RP3:L:1:4}
 \begin{aligned}
	&| \hat r - \xi_{p} |_{p} \leq \epsilon_{p}
	\quad \forall {p} \in \mathcal{S},
	\\
	&| \hat r |_{q} \leq 1
	\quad \forall {q} \notin \mathcal{S}.
 \end{aligned}
\end{equation}
	Note that $\hat r + k M$ satisfies the inequalities (\ref{RP3:L:1:4}),
	for any $k\in \mathbb{Z}$.
	We can choose $k\in \mathbb{Z}$ such that $r = \hat r + k M$ satisfies
	furthermore
\[
	| r - \xi_{\infty} |_{\infty} \leq \frac{1}{2}M.
\]
	By (\ref{RP3:L:1:3}), we find that
\[
	M \leq \prod_{{p} \in \mathcal{S}}\epsilon_{p}^{-1}p,
\]
	and from the assumption (\ref{RP3:L:1:1}), it follows
	that $r$ satisfies (\ref{RP3:L:1:2}).
\end{proof}
	The next lemma studies the dual lattice $\Lambda^*$
	attached to a lattice $\Lambda$ of the form (\ref{dual:2}).
\begin{lemma} \label{RP3:L:2}
	Fix $\bar \delta =
	(\delta_{p})_{{p} \in \mathcal{S}} \in \mathbb{R}_{>0}^{|\mathcal{S}|}$,
	with $\delta_{p} \leq 1$ for every ${p} \in \mathcal{S}$.
	Let $\Lambda$ be the lattice of $\mathbb{R}^{n+1}$ defined as follows
\begin{equation}\label{RP3:L:2:1}
	\Lambda = \{{\bf x} \in \mathbb{Z}^{n+1} |
	L_{p}({\bf x }) \leq {\delta}_{p}
	\ \forall {p} \in \mathcal{S} \},
\end{equation}
	with its dual lattice $\Lambda^*$ defined by
\[
	\Lambda^* = \{ {\bf y} \in \mathbb{Q}^{n+1} |
	\langle{\bf y},{\bf x}\rangle \in \mathbb{Z} \
	\forall {\bf x} \in \Lambda \}.
\]
	Then, there exists an integer $a = a(\bar \xi,\mathcal{S}) > 0$ such that
\begin{equation}\label{RP3:L:2:2}
	\Lambda^* \subseteq
	\{ {\bf y} \in \mathbb{Q}^{n+1} | \quad
	| \langle {\bf y},{\bf t}_{p} \rangle |_{p}  \leq |a|_{p}^{-1}
	\quad \forall {p} \in \mathcal{S} \}.
\end{equation}
	Moreover,
\begin{equation}\label{RP3:L:2:4}
	ab \Lambda^*  \subseteq \mathbb{Z}^{n+1},
\end{equation}
	for some integer $b > 0$ with
\begin{equation}\label{RP3:L:2:3}
	|b|_{\infty} \leq \prod_{{p} \in \mathcal{S}}p\delta_{p}^{-1},
	\quad |b|_{p} < \delta_{p}
	\quad \forall {p} \in \mathcal{S}.
\end{equation}
\end{lemma}
\begin{proof}
	Choose $a \in \mathbb{Z}_{>0}$ such that
	$a{\bf t}_{p} \in \mathbb{Z}_p^{n+1}$ for each $ p \in \mathcal{S}$.
	Also, for each $ p \in \mathcal{S}$ we choose $k_p \in \mathbb{Z}_{\geq0}$
	such that
\begin{equation}\label{RP3:L:2:9}
	{p}^{k_{p}} \leq \delta_{p}^{-1} < {p}^{k_{p}+1}
\end{equation}
	and put $b = \prod_{{p} \in \mathcal{S}}{p}^{k_{p}+1}$.
	We claim that the numbers $a$ and $b$ have the required properties.
	First of all, from (\ref{RP3:L:2:9}) and the construction of the number $b$,
	we find that
\begin{gather*}
	|b|_{p} = {p}^{-k_{p}-1} < \delta_{p}
	\quad \forall {p} \in \mathcal{S},
	\\
	|b|_{\infty} =
	\prod_{{p} \in \mathcal{S}}{p}^{k_{p}+1} \leq
	\prod_{{p} \in \mathcal{S}}{p}\delta_{p}^{-1},
\end{gather*}
	which gives (\ref{RP3:L:2:3}).

	Now, we fix ${\bf y} = (y_{0}, y_{1},\ldots,y_{n}) \in \Lambda^*$.
	It remains to show that
	$| \langle {\bf y},{\bf t}_{p} \rangle |_{p}  \leq |a|_{p}^{-1}$
	and that $ab {\bf y}  \in \mathbb{Z}^{n+1}$ for each ${p} \in \mathcal{S}$.
	Choose $\epsilon>0$ such that
\begin{align*}
	\epsilon &< \min_{{p}\in\mathcal{S}}\{\|{\bf t}_{p}\|_p\},
	\\
	\epsilon &< \min_{{p}\in\mathcal{S}}
		\{ |a|_{p}^{-1}\| {\bf y} \|_{p}^{-1} \},
	\\
	\epsilon &< \min_{{p}\in\mathcal{S}}
		\{ |a|_{p}^{-1}\delta_{p} \}.
\end{align*}
	By the Strong Approximation Theorem (see \cite{CASS-FROH})
	or by Lemma \ref{RP3:L:1} above,
	for each integer $l$ with $1 \leq l \leq n$,
	there exists $r_{l} \in \mathbb{Q}$ such that
 \begin{gather*}
	|r_l - \xi_{p}^l|_{p}<\epsilon
	\quad \forall {p} \in \mathcal{S},
	\\
	|r_l|_{q} \leq 1
	\quad \forall {q} \notin \mathcal{S},
 \end{gather*}
	Putting ${\bf r} = (1,r_{1},\ldots,r_{n})$,
	this becomes
\begin{equation}\label{RP3:L:2:7}
 \begin{gathered}
	\|{\bf r} - {\bf t}_{p}\|_{p}<\epsilon
	\quad \forall {p} \in \mathcal{S},
	\\
	\|{\bf r}\|_{q} \leq 1
	\quad \forall {q} \notin \mathcal{S}.
 \end{gathered}
\end{equation}
	Since $\epsilon < \min_{{p} \in \mathcal{S}}\{\|{\bf t}_{p}\|_{p}\}$,
	this gives $\|{\bf r}\|_{p} = \|{\bf t}_{p}\|_{p} $
	for each ${p} \in \mathcal{S}$. Then, since
	$a{\bf t}_{p} \in \mathbb{Z}_p^{n+1}$ for each ${p} \in \mathcal{S}$,
	we deduce that $a{\bf r} \in \mathbb{Z}_p^{n+1}$ for each ${p} \in \mathcal{S}$.
	Combining this with (\ref{RP3:L:2:7}) it follows that $a{\bf r} \in \mathbb{Z}^{n+1}$.
	Since
\[
	L_{p}(a {\bf r}) =
	\| a {\bf r} - a {\bf t}_{p} \|_{p}
	< |a|_{p} \epsilon
	\leq \delta_{p}
	\quad \forall {p} \in \mathcal{S},
\]
	we conclude that $a {\bf r} \in \Lambda$. So, we have
	$\langle {\bf y},a {\bf r} \rangle \in \mathbb{Z}$
	and thus $|\langle {\bf y},a {\bf r} \rangle|_p \leq 1$
	for each ${p} \in \mathcal{S}$.
	By virtue of the choice of $a$, $\epsilon$ and ${\bf r}$,
	this leads to
\begin{equation}\label{RP3:L:2:8}
 \begin{aligned}
	| \langle {\bf y},{\bf t}_{p} \rangle |_{p}
	&=
	|a|_{p}^{-1}|\langle {\bf y},a{\bf t}_{p} - a {\bf r} \rangle + \langle {\bf y},a {\bf r} \rangle |_{p}
	\\
	&\leq
	|a|_{p}^{-1}\max\{ \| {\bf y} \|_{p} \| a {\bf r} - a{\bf t}_{p} \|_{p},|\langle {\bf y},
	a {\bf r} \rangle |_{p} \}
	\\
	&\leq
	|a|_{p}^{-1}\max\{ \| {\bf y} \|_{p} L_{p}(a {\bf r}),1 \}
	\\
	&\leq
	|a|_{p}^{-1}\max\{ \| {\bf y} \|_{p} |a|_{p}\epsilon,1 \}
	\\
	&\leq |a|_{p}^{-1},
 \end{aligned}
\end{equation}
	which proves the first part of the claim.
	To show that $ab{\bf y}\in \mathbb{Z}^{n+1}$, denote by
	$\{ {\bf e}_0,\ldots,{\bf e}_{n} \}$ the canonical basis
	of $\mathbb{Q}^{n+1}$ and let $j$ be any integer with $1\leq j\leq n$.
	Since $L_p(b {\bf e}_j) = |b|_p < \delta_p$ for each ${p} \in \mathcal{S}$,
	we have $b {\bf e}_j \in \Lambda $ and therefore
\[
	b y_j =
	\langle{\bf y},b {\bf e}_j \rangle \in \mathbb{Z}.
\]
	Finally, $\langle {\bf y},a{\bf r} \rangle \in \mathbb{Z}$,
	we deduce that $a b y_0 \in \mathbb{Z}$.
	By this we conclude that
	$a b {\bf y} \in \mathbb{Z}^{n+1}$.
\end{proof}
	We now study the dual convex body $C^*$
	attached to a convex body $C$ of the form (\ref{dual:1}).
\begin{lemma} \label{RP3:L:4}
	Let $\delta_{\infty}$ and $X$ be any numbers with
	$0 < \delta_{\infty} \leq X$.
	Define a convex body
\[
	C = \{{\bf x} \in \mathbb{R}^{n+1} \ | \
		\| {\bf x} \|_{\infty} \leq X, \
		L_{\infty}({\bf x }) \leq  \delta_{\infty}
		\}.
\]
	Then the dual convex body
\[
	C^* = \{ {\bf y} \in \mathbb{R}^{n+1} \ | \
			\langle{\bf y},{\bf x}\rangle \leq 1
			\ \forall {\bf x} \in C
			 \}
\]
	satisfies the following inclusion
\begin{equation}\label{RP3:L:4:1}
	C^* \subseteq
	\{
		{\bf y} \in \mathbb{R}^{n+1} \ | \
			\|{\bf y}\|_{\infty} \
			\leq c_1\delta_{\infty}^{-1},
			|\langle{\bf y},{\bf t}_{\infty}\rangle|_{\infty}
			\leq c_1 X^{-1}
	\},
\end{equation}
	where $c_1$ depends only on $n$ and $\xi_{\infty}$.
\end{lemma}
\begin{proof}
	We have the following inclusion
\begin{align*}
	C \subseteq
	\widetilde{C} & =
		\{{\bf x} \in \mathbb{R}^{n+1} \ | \
		| x_0 |_{\infty} \leq X, \
		L_{\infty}({\bf x }) \leq  \delta_{\infty}
		\}
		\\
		& =
		\{{\bf x} \in \mathbb{R}^{n+1} \ | \
		\| A{\bf x} \|_{\infty} \leq 1\},
\end{align*}
	where
\[
	A =
	\left (
	\begin{matrix}
	X^{-1} & 0 & 0 & \ldots &0
	\\
	-\xi_{\infty}\delta_{\infty}^{-1} & \delta_{\infty}^{-1} & 0 & \ldots &0
	\\
	\vdots& \vdots& \vdots & \ddots &\vdots
	\\
	-\xi_{\infty}^n \delta_{\infty}^{-1} & 0 & 0 &  \ldots & \delta_{\infty}^{-1}
	\end{matrix}
	\right ).
\]
	Take a point ${\bf x } \in \widetilde{C}$.
	Since $| x_0 |_{\infty} \leq X$
	and $L_{\infty}({\bf x }) \leq  \delta_{\infty} \leq X$,
	we have $\| {\bf x} \|_{\infty} \leq c_2X$, for some $c_2(n,\xi_{\infty})>0$.
	So, we get
\[
	\widetilde{C}  \subseteq c_2 C.
\]
	Taking the dual of both sides, we obtain
\[
	c_2^{-1} C^*  \subseteq ( \widetilde{C} )^*
		=
				\{ {\bf y} \in \mathbb{R}^{n+1} \ | \
				\langle{\bf y},{\bf x}\rangle \leq 1
				\quad
				\forall {\bf x} \in \widetilde{C}
				\}.
\]
	Fix any ${\bf y} \in \mathbb{R}^{n+1}$.
	Define $A^* = {^tA}^{-1}$.
	Since $ \langle{\bf y},{\bf x}\rangle
	= \langle A^*{\bf y},A{\bf x}\rangle $
	for any ${\bf x} \in \mathbb{R}^{n+1}$, we have
\begin{align*}
	{\bf y} \in ( \widetilde{C} )^*
	&\iff \langle A^*{\bf y},A{\bf x}\rangle \leq 1
	\text{ for each } {\bf x} \in \mathbb{R}^{n+1}
	\text{ such that }
	\|A{\bf x}\|_{\infty}\leq1
	\\
	&\iff \langle A^*{\bf y},z\rangle \leq 1
	\text{ for each } {\bf z} \in \mathbb{R}^{n+1}
	\text{ such that }
	\|z\|_{\infty}\leq1
	\\
	&\iff \| A^*{\bf y}\|_{1} \leq 1.
\end{align*}
	This shows that
\[
	( \widetilde{C} )^*= \{
		{\bf y} \in \mathbb{R}^{n+1} \ | \
		\| A^*{\bf y} \|_{1} \leq 1
		\},
\]
	where
\[
	A^{*} = {^tA}^{-1} =
	\left (
	\begin{matrix}
	X & \xi_{\infty}X & \xi_{\infty}^2 X & \ldots & \xi_{\infty}^n X
	\\
	0 & \delta_{\infty}& 0 & \ldots  & 0
	\\
	\vdots& \vdots& \vdots & \ddots &\vdots
	\\
	0 & 0 & 0 & \ldots & \delta_{\infty}
	\end{matrix}
	\right ).
\]
	Since $\| * \|_{\infty} \leq \| * \|_{1}$, we also get
\begin{equation}\label{RP3:L:4:2}
	(\widetilde{C})^* \subseteq D
		= \{{\bf y} \in \mathbb{R}^{n+1}
		\ | \ \| A^*{\bf y} \|_{\infty} \leq 1\}.
\end{equation}
	Also, there exists a constant $c_3 = c_3(n,\xi_{\infty})$ such that
\[
	D \subseteq \widetilde{D}
	= \{{\bf y} \in \mathbb{R}^{n+1}
	\ | \
	| \langle{\bf y},{\bf t}_{\infty}\rangle|_{\infty} \leq X^{-1}, \
	\|{\bf y}\|_{\infty} \leq c_3 \delta_{\infty}^{-1}\}.
\]
	Putting all together, we conclude that
\[
	c_2^{-1} C^* \subseteq \widetilde{D}
\]
	and therefore (\ref{RP3:L:4:1}) holds with
	$c_1 = c_2\max\{1,c_3\}$.
\end{proof}

	From now on, we fix
	$\bar \lambda = (\lambda_{\infty},(\lambda_{p})_{{p} \in \mathcal{S}})
	\in [-1,\infty) \times \mathbb{R}_{>0}^{|\mathcal{S}|}$
	and put
\begin{equation}\label{RP3:lambda}
	\lambda= \lambda_{\infty} + \sum_{{p} \in \mathcal{S}}\lambda_{p}.
\end{equation}
	We also assume that $\lambda>0$.
	In the following proposition we apply Mahler's Duality Theorem
	to complete the first step of the programme outlined in the
	introduction to this chapter.
\begin{proposition} \label{RP3:P:1}
	Let $c$, $X$  be positive real numbers.
	Suppose that the inequalities
\begin{equation}\label{RP3:P:1:1}
 \begin{aligned}
	&\| {\bf x} \|_{\infty} \leq X,
	\\
	&L_{\infty}({\bf x }) \leq c X^{-\lambda_{\infty}},
	\\
	&L_{p}({\bf x }) \leq c X^{-\lambda_{p}} \
	\forall {p} \in \mathcal{S},
 \end{aligned}
\end{equation}
	have no non-zero solution ${\bf x} \in \mathbb{Z}^{n+1}$.
	Then there exist $n+1$ linearly independent points
	${\bf x}_1,\ldots,{\bf x}_{n+1} \in \mathbb{Z}^{n+1}$ satisfying
\begin{equation}\label{RP3:P:1:2}
 \begin{aligned}
	&\| {\bf x}_i \|_{\infty}
	\ll c^{-|\mathcal{S}|-1}
	X^{\lambda}
	\\
	&| \langle {\bf x}_i,{\bf t}_{\infty} \rangle |_{\infty}
	\ll c^{-|\mathcal{S}|}
	X^{\lambda - \lambda_{\infty} - 1},
	\\
	&| \langle {\bf x}_i,{\bf t}_{p} \rangle |_{p}  \leq c X^{-\lambda_{p}}
	\quad \forall {p} \in \mathcal{S},
 \end{aligned}
\end{equation}
	for $i = 1,\ldots,n+1$,
	and the implied constants depend only on $\bar \xi$, $n$ and
	$\mathcal{S}$ (in particular, they do not depend on $c$).
\end{proposition}
\begin{proof}
	Define a lattice $\Lambda$ of
	$\mathbb{R}^{n+1}$ by
\[
	\Lambda = \{{\bf x} \in \mathbb{Z}^{n+1} \ | \
	L_{p}({\bf x }) \leq c X^{-\lambda_{p}}
	\ \forall {p} \in \mathcal{S} \}
\]
	and consider the convex body
\[
	C = \{{\bf x} \in \mathbb{R}^{n+1} \ | \ \| {\bf x} \|_{\infty} \leq X, \
		L_{\infty}({\bf x }) \leq c X^{-\lambda_{\infty}}\}.
\]
	The hypothesis is that $C$ contains no non-zero points of
	the lattice $\Lambda$. Therefore the first minimum
	$\tau_1 := \tau_1(C,\Lambda)$ of $\Lambda$,
	with respect to $C$ is $> 1$.

	By Theorem VI, p.219 in \cite{CASS} (Mahler's Duality Theorem),
	we have that
\begin{equation}\label{RP3:P:1:5}
	\tau_1 \tau_{n+1}^* \leq (n+1)!,
\end{equation}
	where $\tau_{n+1}^* := \tau_{n+1}(C^*,\Lambda^*)$
	is the last minimum of the dual lattice $\Lambda^*$,
	with respect to the dual convex body $C^*$.
	So, by (\ref{RP3:P:1:5}),
	we have that $\tau_{n+1}^* \leq (n+1)!$ and thus there exist
	$n+1$ linearly independent points
	${\bf y}_1,\ldots,{\bf y}_{n+1}$ with
\begin{equation}\label{RP3:P:1:5-1}
	{\bf y}_1,\ldots,{\bf y}_{n+1} \in \Lambda^* \cap (n+1)!~C^*.
\end{equation}
	WLOG we may assume that $0<c<1$.
	Applying Lemma \ref{RP3:L:4} with
	$\delta_{\infty} = c X^{-\lambda_{\infty}}$,
	we find that there exist a constant $c_1(n,\xi_{\infty})>0$ such that
\begin{equation}\label{RP3:P:1:6}
	C^*
		\subseteq
		\{
		{\bf y} \in \mathbb{R}^{n+1} \ |  \quad
			\|{\bf y}\|_{\infty}
			\leq c_1 c^{-1} X^{\lambda_{\infty}},
			|\langle{\bf y},{\bf t}_{\infty}\rangle|_{\infty}
			\leq c_1 X^{-1}
		\}
\end{equation}
	By Lemma \ref{RP3:L:2} with
	$\delta_{p} = cX^{-\lambda_{p}}$ for each ${p} \in \mathcal{S}$,
	there exist integers $a, b >0$ with
\begin{equation}\label{RP3:P:1:7}
 \begin{aligned}
	|b|_{\infty}
		&\leq \prod_{{p} \in \mathcal{S}}pc^{-1}X^{\lambda_{p}}
		\ll c^{-|\mathcal{S}|} X^{\sum_{{p} \in \mathcal{S}}\lambda_{p}}
		=  c^{-|\mathcal{S}|} X^{\lambda-\lambda_{\infty}},
	\\
	\quad |b|_{p} &< cX^{-\lambda_{p}} \quad \forall {p} \in \mathcal{S},
 \end{aligned}
\end{equation}
	and $a$ depending only on $\bar \xi$ and $\mathcal{S}$,
	such that
\begin{equation}\label{RP3:P:1:8}
	\Lambda^* \subseteq
	\{ {\bf y} \in \mathbb{Q}^{n+1}  \ | \quad
	| \langle {\bf y},{\bf t}_{p} \rangle |_{p}  \leq |a|_{p}^{-1}
	\quad \forall {p} \in \mathcal{S} \}
	\ \text{ and } \
	ab \Lambda^*  \subseteq \mathbb{Z}^{n+1}.
\end{equation}
	Put ${\bf x}_i = a b {\bf y}_i$ for $i = 1,\ldots,n+1$.
	Since $\tau_{n+1}^* \leq(n+1)!$~, the relations
	(\ref{RP3:P:1:5-1})-(\ref{RP3:P:1:8}),
	imply that, for each $i = 1,\ldots,n+1$, we have
	${\bf x}_i \in \mathbb{Z}^{n+1}$ and
\begin{align*}
	&\| {\bf x}_i \|_{\infty} =
	|a|_{\infty}|b|_{\infty}\| {\bf y}_i \|_{\infty}
	\ll c^{-|\mathcal{S}|-1}
	X^{\lambda},
	\\
	&| \langle {\bf x}_i,{\bf t}_{\infty} \rangle |_{\infty} =
	|a|_{\infty}|b|_{\infty}| \langle {\bf y}_i,{\bf t}_{\infty} \rangle |_{\infty}
	\ll c^{-|\mathcal{S}|} X^{\lambda - \lambda_{\infty}-1},
	\\
	&| \langle {\bf x}_i,{\bf t}_{p} \rangle |_{p}  =
	|a|_{p}|b|_{p}| \langle {\bf y}_i,{\bf t}_{p} \rangle |_{p}
	\leq c X^{-\lambda_{p}}
	\quad \forall {p} \in \mathcal{S}.
\end{align*}
\end{proof}
	The following is the main result of this chapter.
	It is a result of simultaneous approximation
	to real and p-adic numbers $\eta_{\nu}~(\nu \in \mathcal{S}\cup\{\infty\})$
	by values of a single polynomial of $\mathbb{Z}[T]$
	evaluated at the points $\xi_{\nu}~(\nu \in \mathcal{S}\cup\{\infty\})$.
	We will obtain the result stated in the introduction
	as a corollary of it.

	Recall that we fixed $\bar \lambda = (\lambda_{\infty},(\lambda_{p})_{{p}
	\in \mathcal{S}})
	\in [-1,\infty) \times \mathbb{R}_{>0}^{|\mathcal{S}|}$
	with the property that the sum $\lambda$ of its coordinates is positive
	( see (\ref{RP3:lambda}) ).
\begin{theorem} \label{RP3:T:1}
	Let $c>0$  be a real number, let
	$(\eta_{\infty},(\eta_{p})_{{p} \in \mathcal{S}})
	\in \mathbb{R} \times \prod_{{p} \in \mathcal{S}} \mathbb{Z}_{p}$ and,
	for each ${p} \in \mathcal{S}$,
	let $\rho_p \in  \mathbb{Z}_{p}$ with
	$0<|\rho_p|_p \leq \|{\bf t}_{p}\|_{p}^{-1}$.
	Suppose that the inequalities
\begin{equation}\label{RP3:T:1:1}
 \begin{aligned}
	&\| {\bf x} \|_{\infty} \leq X,
	\\
	&L_{\infty}({\bf x }) \leq c X^{-\lambda_{\infty}},
	\\
	&L_{p}({\bf x }) \leq c X^{-\lambda_{p}} \
	\forall {p} \in \mathcal{S},
 \end{aligned}
\end{equation}
	have no non-zero solution ${\bf x} \in \mathbb{Z}^{n+1}$
	for arbitrarily large values of $X$.
	Then there exist infinitely many non-zero polynomials
	$P(T) \in \mathbb{Z}[T]_{\leq n}$ satisfying
\begin{equation}\label{RP3:T:1:2}
 \begin{aligned}
	&|P(\xi_{\infty}) + \eta_{\infty}|_{\infty} \sim
	H(P)^{(\lambda-\lambda_{\infty}-1)/\lambda}
	&\text{and} &
	\quad
	|P'(\xi_{\infty})|_{\infty} \sim H(P),
	\\
	&|P(\xi_{p}) + \eta_{p}|_{p} \sim
	H(P)^{-\lambda_{p}/\lambda}
	&\text{and} &
	\quad
	|P'(\xi_{p})|_{p}  = |\rho_p|_p \quad
	\forall {p} \in \mathcal{S},
 \end{aligned}
\end{equation}
	where the implied constants depend only on $\bar \xi$, $\bar \eta$,
	$c$, $n$ and $\mathcal{S}$.
\end{theorem}

	Note, if $\eta_p \neq 0$ for ${p} \in \mathcal{S}$,
	then the third relation in (\ref{RP3:T:1:2})
	implies that $\big| P(\xi_p)\big|_p = |\eta_p|_p$
	for each $P$ with $H(P)$ sufficiently large.
	This requires that $|\eta_p|_p \leq \|{\bf t}_{p}\|_{p} $.
	Here we ask the stronger condition
	that $|\eta_p|_p \leq 1$ for each ${p} \in \mathcal{S}$.
\ \\ \\
\begin{proof}
	Fix a choice of $X$ for which (\ref{RP3:T:1:1})
	has no non-zero solution in $\mathbb{Z}^{n+1}$.
	By Proposition \ref{RP3:P:1} there exist $n+1$ linearly independent points
	${\bf x}_1,\ldots,{\bf x}_{n+1} \in \mathbb{Z}^{n+1}$
	and a constant $c_1(n,c,\bar {\bf \xi}) > 0$ such that upon writing
	${\bf x}_i = ( x_0^{(i)}, \ldots,x_{n}^{(i)})$, we have
\begin{equation}\label{RP3:T:1:3}
 \begin{aligned}
	&\| {\bf x}_i \|_{\infty} \leq c_1
	X^{\lambda},
	\\
	&| x_{n}^{(i)} \xi_{\infty}^{n} + \ldots  + x_0^{(i)} |_{\infty} \leq c_1
	X^{\lambda-\lambda_{\infty}-1},
	\\
	&| x_{n}^{(i)} \xi_{p}^{n} + \ldots  + x_0^{(i)} |_{p}
	\leq c_1 X^{-\lambda_{p}}
	\quad \forall {p} \in \mathcal{S},
 \end{aligned}
\end{equation}
	for $i = 1,\ldots,n+1$.
	These points determine
	linearly independent polynomials
	$P_1,\ldots,P_{n+1}$ of $\mathbb{Z}[T]_{\leq n}$ given by
\begin{equation} \label{RP3:T:1:4}
	P_i(T) : = \sum_{m = 0}^{n}x_{m}^{(i)} T^m \quad (1 \leq i \leq n+1).
\end{equation}
	With this notation, the condition
	(\ref{RP3:T:1:3}) can be rewritten in the form
\begin{equation}\label{RP3:T:1:5}
 \begin{aligned}
	&H(P_i) \leq c_1
	X^{\lambda},
	\\
	&| P_i(\xi_{\infty}) |_{\infty} \leq c_1
	X^{\lambda-\lambda_{\infty} -1},
	\\
	&| P_i(\xi_{p}) |_{p} \leq c_1 X^{-\lambda_{p}}
	\quad \forall {p} \in \mathcal{S},
 \end{aligned}
\end{equation}
	for $i = 1,\ldots,n+1$.
	We introduce auxiliary polynomials
\begin{equation}\label{RP3:T:1:6}
	R_{\nu}(T) := - \eta_{\nu} + \epsilon_{\nu} + \rho_{\nu}(T-\xi_{\nu})
	\quad \text{for each} \quad {\nu} \in \{\infty\}\cup\mathcal{S},
\end{equation}
	where $\epsilon_{p} \in \mathbb{Q}_{p}$
	for each ${p} \in \mathcal{S}$ and
	$\epsilon_{\infty}, \rho_{\infty} \in \mathbb{R}$
	will be determined later.
	Since $P_1,\ldots,P_{n+1}$ are linearly independent over $\mathbb{Q}$,
	there exist numbers $\theta_{{p},1},\ldots,\theta_{{p},n+1} \in
	\mathbb{Q}_{p}$ for ${p} \in \mathcal{S}$
	and $\theta_{{\infty},1},\ldots,\theta_{{\infty},n+1} \in
	\mathbb{R}$ such that
\begin{equation}\label{RP3:T:1:7}
	R_{\nu}(T) = \sum_{i = 1}^{n+1} \theta_{{\nu},i} P_i(T)
	\ \text{ for } \ {\nu} \in \{\infty\}\cup\mathcal{S}.
\end{equation}
	It follows from (\ref{RP3:T:1:6}) and (\ref{RP3:T:1:7}) that
	for ${\nu} \in \{\infty\}\cup\mathcal{S}$, we have
\begin{equation}\label{RP3:T:1:8}
	\sum_{i = 1}^{n+1} \theta_{{\nu},i} x_m^{(i)} = 0  \quad (2 \leq m \leq n).
\end{equation}
	Choose a real number $\epsilon>0$ with
	$\epsilon< \min_{p \in \mathcal{S}} \{ \|{\bf t}_p\|_p^{-1} |\rho_p|_p\}$.
	By Lemma \ref{RP3:L:1},  for each $i = 1,\ldots,n+1$, there exist numbers
	$r_{i} \in \mathbb{Q}$ such that
\begin{equation}\label{RP3:T:1:8:1}
 \begin{aligned}
	&| r_{i} - \theta_{\infty,i} |_{\infty} \leq
	\epsilon^{-|\mathcal{S}|}\prod_{{p} \in \mathcal{S}}p,
	\\
	&| r_{i} - \theta_{{p},i} |_{p} \leq \epsilon
	\quad \forall {p} \in \mathcal{S},
	\\
	&| r_{i} |_{q} \leq 1
	\quad \forall {q} \notin \mathcal{S}.
 \end{aligned}
\end{equation}
	Put $N = \epsilon^{-|\mathcal{S}|}\prod_{{p} \in \mathcal{S}}p$.
	Now, we define a polynomial
	$P(T) \in \mathbb{Q}[T]_{\leq n}$ by
\begin{equation}\label{RP3:T:1:9}
	P(T) : = \sum_{i = 1}^{n+1} r_{i} P_i(T) =
	\sum_{m = 0}^{n} x_{m} T^m,
\end{equation}
	where
\[
	x_{m} = \sum_{i = 1}^{n+1} r_{i} x_m^{(i)}  \quad (0 \leq m \leq n).
\]
	Using (\ref{RP3:T:1:5}) - (\ref{RP3:T:1:7})
	and (\ref{RP3:T:1:9}), we find that
\begin{align*}
	| P(\xi_{\infty}) + \eta_{\infty} - \epsilon_{\infty} |_{\infty}
	&= | P(\xi_{\infty}) - R_{\infty}(\xi_{\infty}) |_{\infty}
	\\
	&\leq (n+1)N\max_{1 \leq i \leq n+1} \{ | P_i(\xi_{\infty}) |_{\infty} \}
	\\
	&\leq (n+1) Nc_1 X^{\lambda-\lambda_{\infty}-1},
	\\
	| P'(\xi_{\infty}) - \rho_{\infty} |_{\infty}  &=
		| P'(\xi_{\infty}) - R_{\infty}'(\xi_{\infty}) |_{\infty}
		\\
		&<
		(n+1)N\max_{1 \leq i \leq n+1} \{ | P_i'(\xi_{\infty}) |_{\infty} \}
		\\
		&\leq
		(n+1)Nn\|{\bf t}_{\infty}\|_{\infty}
		\max_{1 \leq i \leq n+1} H(P_i)
		\\
		&\leq
		(n+1)Nc_2X^{\lambda},
\end{align*}
	where $c_2 = c_1n\|{\bf t}_{\infty}\|_{\infty}$.
	Upon choosing
\[
	\epsilon_{\infty} = 2 (n+1) Nc_1 X^{\lambda-\lambda_{\infty}-1}
	\ \text{ and } \
	\rho_{\infty} = 2 (n+1)Nc_2X^{\lambda},
\]
	it follows that
\begin{equation}\label{RP3:T:1:9:0}
 \begin{gathered}
	(n+1) Nc_1 X^{\lambda-\lambda_{\infty}-1} \leq
	| P(\xi_{\infty}) + \eta_{\infty} |_{\infty}
	\leq  3(n+1) Nc_1 X^{\lambda-\lambda_{\infty}-1},
	\\
	(n+1)Nc_2X^{\lambda} \leq
	| P'(\xi_{\infty})|_{\infty}  \leq
	3(n+1)Nc_2X^{\lambda}.
 \end{gathered}
\end{equation}
	Fix ${p} \in \mathcal{S}$.
	Since $\|{\bf t}_{p}\|_{p} \geq1$, we have
	$\epsilon < \|{\bf t}_{p}\|_{p}^{-1}|\rho_p|_p
	\leq \|{\bf t}_{p}\|_{p}^{-2}\leq 1$.
	Using (\ref{RP3:T:1:5}) - (\ref{RP3:T:1:7})
	and (\ref{RP3:T:1:9}) again,
	we get
\begin{align*}
	| P(\xi_{p}) + \eta_{p} - \epsilon_{p} |_{p}  &=
		| P(\xi_{p}) - R_{p}(\xi) |_{p}
		\\
		&\leq
		\epsilon\max_{1 \leq i \leq n+1} \{ | P_i(\xi_{p}) |_{p} \}
		< c_1 X^{-\lambda_{p}},
	\\
	| P'(\xi_{p}) - \rho_{p} |_{p}  &=
		| P'(\xi_{p}) - R_{p}'(\xi_{p}) |_{p}
		\\
		&\leq
		\epsilon \max_{1 \leq i \leq n+1} \{ | P_i'(\xi_{p}) |_{p} \}
		\leq
		\epsilon \|{\bf t}_{p}\|_{p}.
\end{align*}
	Let $k_{p}$ be the integer for which
\[
	{p}^{k_{p}} \leq c_1^{-1} X^{\lambda_{p}} < {p}^{k_{p}+1}.
\]
	Choose
	$\epsilon_{p} = {p}^{k_{p}}$, so
	that $c_1X^{-\lambda_{p}} \leq |\epsilon_p|_p$.
	Since $\epsilon
	< \min_{p \in \mathcal{S}} \{ \|{\bf t}_p\|_p^{-1} |\rho_p|_p\}$,
	we note that
	$\epsilon\|{\bf t}_{p}\|_{p} < |\rho_{p}|_p$.
	So, we obtain
\begin{equation}\label{RP3:T:1:11}
\begin{gathered}
	c_1 X^{-\lambda_{p}} \leq |\epsilon_p|_p
	=
	| P(\xi_{p}) + \eta_{p} |_{p} < {p} c_1 X^{-\lambda_{p}},
	\\
 	|P'(\xi_{p})|_{p} = |\rho_{p}|_p.
\end{gathered}
\end{equation}
	To prove that $P(T) \in \mathbb{Z}[T]_{\leq n}$, we need to show
	that all the coefficients $x_{m}$ in (\ref{RP3:T:1:9})
	are integers for $m = 0,\ldots,n$.
	By the construction of $P(T)$ and by the last relation in
	(\ref{RP3:T:1:8:1}), for $m = 0,\ldots,n$, we have that
	$|x_{m}|_{q} \leq 1$ for each $q \notin \mathcal{S}$.
	By (\ref{RP3:T:1:8}), we find that
\[
	| x_{m} |_{p} =
	\Big| \sum_{i = 1}^{n+1} (r_{i} - \theta_{{p},i}) x_{m}^{(i)}
	\Big|_{p} \leq
	\epsilon < 1
	\ \ \text{ for each } \ {p} \in \mathcal{S}  \quad (2 \leq m \leq n).
\]
	This implies that $x_{m} \in \mathbb{Z}$ for $m = 2,\ldots,n$.
	Since
\[
	|x_1 - P'(\xi_{p})|_p = \Big|\sum_{m=2}^{n}mx_m\xi_{p}^{m-1}\Big|_p
	\leq \|{\bf t}_{p}\|_{p} \max_{2\leq m \leq n}|x_m|_p
	\leq  \epsilon \|{\bf t}_{p}\|_{p} < |\rho_{p}|_p
	\leq \|{\bf t}_{p}\|_{p}^{-1} \leq 1,
\]
	the second inequality in (\ref{RP3:T:1:11}) gives
\[
	|x_1|_p = |\rho_{p}|_p \leq \|{\bf t}_{p}\|_{p}^{-1} \leq 1,
\]
	and so $x_{1} \in \mathbb{Z}$.
	Similarly, we find that
\begin{align*}
	|x_0 - P(\xi_{p})|_p & = \Big |\sum_{m=1}^{n}x_m\xi_{p}^{m} \Big |_p
	\leq \|{\bf t}_{p}\|_{p} \max_{1\leq m \leq n}|x_m|_p
	=  \|{\bf t}_{p}\|_{p}
	\max\{|x_1|_p,
	\max_{2\leq m \leq n}|x_m|_p\}
	\\
	&\leq
	\|{\bf t}_{p}\|_{p} \max\{|\rho_{p}|_p,\epsilon\}
	= \|{\bf t}_{p}\|_{p}|\rho_{p}|_p \leq  1.
\end{align*}
	Since $\eta_p \in \mathbb{Z}_p$,
	the first inequality in (\ref{RP3:T:1:11}) gives
	$|P(\xi_{p})|_p \leq |\eta_p|_p \leq 1$ assuming, as we may,
	that $X$ is sufficiently large, and then
\[
	|x_0|_p \leq \max\{ 1,|P(\xi_{p})|_p \} \leq 1.
\]
	So, we get $x_{0} \in \mathbb{Z}$
	and conclude that $P(T) \in \mathbb{Z}[T]_{\leq n}$.
	Moreover, since
\begin{gather*}
	H(R_{\infty}) \ll X^{\lambda},
	\\
	H(P) \leq H(R_{\infty}) + H(P - R_{\infty}),
	\\
	H(P - R_{\infty}) \leq
	(n+1) N \max\{ H(P_i)\}  \ll
	X^{\lambda},
\end{gather*}
	we get $H(P) \ll X^{\lambda}$.
	Since $| P'(\xi_{\infty}) |_{\infty} \ll H(P)$, the second
	estimate in (\ref{RP3:T:1:9:0}) gives $H(P) \gg X^{\lambda}$ and so
\begin{equation}\label{RP3:T:1:9:1}
	H(P) \sim X^{\lambda}.
\end{equation}
	Finally, from (\ref{RP3:T:1:9:0}) - (\ref{RP3:T:1:9:1}),
	we get (\ref{RP3:T:1:2}).
\end{proof}

	We are now ready to prove the result
	on simultaneous approximation stated in the introduction.
\begin{theorem} \label{RP3:C:1}
	Let $c>0$  be a real number and let
	$R(T)$ be a polynomial in $\mathbb{Z}[T]$.
	Suppose that $R(\xi_p) \in \mathbb{Z}_{p}$ for each ${p} \in \mathcal{S}$.
	Suppose also that the inequalities
\begin{equation}\label{RP3:C:1:0}
 \begin{aligned}
	&\| {\bf x} \|_{\infty} \leq X,
	\\
	&L_{\infty}({\bf x }) \leq c X^{-\lambda_{\infty}},
	\\
	&L_{p}({\bf x }) \leq c X^{-\lambda_{p}} \
	\forall {p} \in \mathcal{S},
 \end{aligned}
\end{equation}
	have no non-zero solution ${\bf x} \in \mathbb{Z}^{n+1}$
	for arbitrarily large values of $X$.
	Then there exist infinitely many polynomials $F(T) \in \mathbb{Z}[T]$ with
	the following properties:
	\\
	$(i)$ \ $\deg(R - F) \leq n$,
	\\
	$(ii)$ if $\lambda_{\infty}>-1$,
		there exists a root $\alpha_{\infty}$ of $F$
	in $\mathbb{R}$ such that
\begin{equation}\label{RP3:C:1:1-0}
	| \xi_{\infty} - \alpha_{\infty} |_{\infty}
		\ll H(F)^{-(\lambda_{\infty}+1)/\lambda},
\end{equation}
	$(iii)$ for each ${p} \in \mathcal{S}$,
		there exists a root $\alpha_{p}$ of $F$
		in ${\mathbb{Q}}_p$ such that
\begin{equation}\label{RP3:C:1:1}
 	| \xi_p - \alpha_p |_p \ll H(F)^{-\lambda_{p}/\lambda}.
\end{equation}
	Moreover,
	for each ${p} \in \mathcal{S}$
	such that $\xi_{p} \in \mathbb{Z}_{p}$,
	we can choose $\alpha_{p} \in \mathbb{Z}_{p}$.
\\
	All the implied constants depend
	only on $\bar \xi$, $R$, $n$ and $c$.
\end{theorem}
\begin{proof}
	Put $\eta_{\nu} = R(\xi_{\nu})$ for each
	$\nu \in \{\infty\} \cup \mathcal{S}$.
	For each ${p} \in \mathcal{S}$,
	we have $\eta_{p} = R(\xi_{p}) \in \mathbb{Z}_p$.
	Choose $\rho_p \in \mathbb{Z}_{p}$
	satisfying the inequalities
	$0 < |\rho_p|_p \leq \|{\bf t}_{p}\|_{p}^{-1}$
	and $|\rho_p|_p \neq |R'(\xi_{p})|_p $.
	Then, by Theorem \ref{RP3:T:1}, there exist infinitely many non-zero
	polynomials $P(T) \in \mathbb{Z}[T]_{\leq n}$ satisfying
\begin{equation}\label{RP3:C:1:2}
 \begin{aligned}
	&|P(\xi_{\infty}) + R(\xi_{\infty})|_{\infty} \sim
	H(P)^{(\lambda-\lambda_{\infty} -1)/\lambda}
	&\text{and} &\quad
	|P'(\xi_{\infty})|_{\infty} \sim H(P),
	\\
	&|P(\xi_{p}) + R(\xi_{p})|_{p} \sim
	H(P)^{-\lambda_{p}/\lambda}
	& \text{and} &\quad
	|P'(\xi_{p})|_{p} = |\rho_p|_p \quad
	\forall {p} \in \mathcal{S},
 \end{aligned}
\end{equation}
	where the implied constants depend only on $\bar \xi$, $R$,
	$c$, $n$ and $\mathcal{S}$.
	For each of these polynomials $P$, we show that
	$F = P + R$ satisfies all the required properties.
	This is clear for Part $(i)$ of the theorem.
	To prove Part $(iii)$,
	we fix ${p} \in \mathcal{S}$ and use
	Corollary 1, p.~51 in \cite{CASS-LOC}, which says that for any
	$\xi^{*} \in {\mathbb{Z}_p}$ and $F^{*}(T) \in \mathbb{Z}_p[T]$
	with
\begin{equation}\label{RP3:C:1:4-00}
	|F^{*}(\xi^{*})|_p < |(F^{*})'(\xi^{*})|_p^2,
\end{equation}
	there is a root $\alpha^{*}$ of $F^{*}$  in $\mathbb{Z}_{p}$
	satisfying the inequality
\begin{equation}\label{RP3:C:1:4-0}
	|\xi^{*} - \alpha^{*}|_{p}
	\leq \frac{|F^{*}(\xi^{*})|_{p}}{|(F^{*})'(\xi^{*})|_{p}}.
\end{equation}
	To use this fact, we choose a positive integer $d$,
	such that $d \xi_p \in {\mathbb{Z}_p}$, and define
	the number $\xi^{*} = d \xi_p$ in ${\mathbb{Z}_p}$
	and the polynomial
	$F^{*}(T) = d^{m}F(d^{-1}T) \in \mathbb{Z}[T]$, where $m = \deg(F)$.
	Let us check that the condition (\ref{RP3:C:1:4-00}) holds.
	For this purpose, we first note that
\begin{equation}\label{RP3:C:1:3-1}
 \begin{gathered}
	F^{*}(\xi^{*}) =d^{m} F(d^{-1}\xi^{*}) = d^{m}F(\xi_p),
	\\
	(F^{*})'(\xi^{*}) = d^{m-1}F'(d^{-1}\xi^{*}) = d^{m-1}F'(\xi_p).
 \end{gathered}
\end{equation}
	By the last relation in (\ref{RP3:C:1:2}) and by
	the fact that $|\rho_p|_p \neq |R'(\xi_{p})|_p $, we have that
	$|P'(\xi_{p})|_{p} \neq |R'(\xi_{p})|_p $ and therefore
\begin{equation}\label{RP3:C:1:3}
 \begin{aligned}
 	|F'(\xi_p)|_p = |P'(\xi_p) + R'(\xi_p)|_p
	& = \max
	\{
		|P'(\xi_{p})|_{p},|R'(\xi_{p})|_p
	\}
	\\
	&= \max
	\{
		|\rho_p|_p,|R'(\xi_{p})|_p
	\}.
 \end{aligned}
\end{equation}
	Note that by the third relation in
	(\ref{RP3:C:1:2}) and by (\ref{RP3:C:1:3}), the inequality
\[
	|F(\xi_p)|_p < |d|_p^{m-2}|F'(\xi_p)|_p^2,
\]
	holds assuming, as we may, that $H(P)$ is sufficiently large.
	Combining this with (\ref{RP3:C:1:3-1}), we find that
\[
	|F^{*}(\xi^{*})|_p = |d|_p^{m}|F(\xi_p)|_p
	<
	|d|_p^{2(m-1)}|F'(\xi_p)|_p^2
	=
	| d^{m-1}F'(\xi_p) |_{p}^2
	=
	|(F^{*})'(\xi^{*})|_p^2,
\]
	which means that the condition (\ref{RP3:C:1:4-00})
	is satisfied.
	So, there exists a root $\alpha^*$ of $F^*$ in $\mathbb{Z}_{p}$ with
\[
	|d  \xi_{p} - \alpha^{*}|_{p}
	=
	|\xi^{*} - \alpha^{*}|_{p}
	\leq \frac{|F^{*}(\xi^{*})|_{p}}{|(F^{*})'(\xi^{*})|_{p}}
	=
	\frac{|d|_p |F(\xi_{p})|_{p}}{|F'(\xi_{p})|_{p}}.
\]
	Then $\alpha_{p} = d^{-1}\alpha^{*}$ is a root of $F$
	in $\mathbb{Q}_{p}$ satisfying the inequality
\[
	|\xi_{p} - \alpha_{p}|_{p}
	\leq
	\frac{|F(\xi_{p})|_{p}}{|F'(\xi_{p})|_{p}}.
\]
	By (\ref{RP3:C:1:2}) and (\ref{RP3:C:1:3}), we conclude that
\[
	|\xi_p - \alpha_p|_p
	\ll H(P)^{-\lambda_{p}/\lambda}
	\sim H(F)^{-\lambda_{p}/\lambda}
	\quad \forall {p} \in \mathcal{S}.
\]
	In particular, if $\xi_{p} \in \mathbb{Z}_{p}$
	and if $H(F)$ is sufficiently large,
	the root $\alpha_{p}$ is in $\mathbb{Z}_{p}$.

	To prove Part $(ii)$, we first note that using (\ref{RP3:C:1:2}) and
	the fact that $R(T)$ is fixed, for these polynomials $F$, we get
\begin{equation}\label{RP3:C:1:5:0}
 \begin{aligned}
	&|F(\xi_{\infty})|_{\infty}
	\sim
	H(P)^{(\lambda-\lambda_{\infty} -1)/\lambda}
	\sim H(F)^{1-(\lambda_{\infty} + 1)/\lambda},
	\\
	&|F'(\xi_{\infty})|_{\infty} \sim
	|P'(\xi_{\infty})|_{\infty} \sim H(P) \sim H(F).
 \end{aligned}
\end{equation}
	Fix any such polynomial $F$. For any real $\alpha$
	with $|\xi_{\infty} - \alpha|_{\infty}\leq 1$, we have
\begin{equation}\label{RP3:C:1:5:2}
 	F(\alpha) = F(\xi_{\infty})
		+ (\alpha - \xi_{\infty})F'(\xi_{\infty})
		+ (\alpha - \xi_{\infty})^2 M,
\end{equation}
	where $M$ is a real number with
\[
	|M|_{\infty} \leq c_1 H(F),
\]
	for some constant $c_1>0$ depending only on $\xi_{\infty}$
	and the degree of $F$.
	Choosing
\[
	\alpha = \xi_{\infty} - 2\frac{F(\xi_{\infty})}{F'(\xi_{\infty})},
\]
	and using (\ref{RP3:C:1:5:0}), we find that
\begin{equation}\label{RP3:C:1:5:3}
	|\alpha - \xi_{\infty}|_{\infty}
	\ll H(F)^{-(\lambda_{\infty} + 1)/\lambda}.
\end{equation}
	Because of our choice of $\alpha$, we find using
	(\ref{RP3:C:1:5:2}) and (\ref{RP3:C:1:5:3}) that
\[
	|F(\alpha)+F(\xi_{\infty})|_{\infty}
		=
		\big|(\alpha - \xi_{\infty})^2M\big|_{\infty}
		\ll
		H(F)^{1-2(\lambda_{\infty} + 1)/\lambda}.
\]
	Since $\lambda_{\infty}>-1$, by the first
	relation in (\ref{RP3:C:1:5:0}), we get
\[
	|F(\alpha)+F(\xi_{\infty})|_{\infty}
	<
	|F(\xi_{\infty})|_{\infty},
\]
	if $H(F)$ is sufficiently large.
	This implies that $F(\xi_{\infty})$ and $F(\alpha)$ have opposite signs,
	which means that $F$ has a real root $\alpha_{\infty}$ located
	between $\xi_{\infty}$ and $\alpha$.
	This real root $\alpha_{\infty}$ satisfies the following inequality
\[
	| \xi_{\infty} - \alpha_{\infty} |_{\infty}
		\leq
		| \xi_{\infty} - \alpha |_{\infty}
		\ll H(F)^{-(\lambda_{\infty}+1)/\lambda}.
\]
\end{proof}

%% file: II-02-duality-real.tex
\section{Real case}

	Applying Theorem \ref{RP3:T:1}
	with $S = \emptyset$, we obtain the following statement.

\begin{theorem} \label{PT1}
	Let $n \geq 1$ be an integer and let $c$, $\lambda_{\infty}$,
	$\xi_{\infty}$ and $\eta_{\infty}$ be real numbers with $c > 0$ and
	$\lambda_{\infty} > 0$.
	Suppose that the inequalities
\begin{equation}\label{PT_1}
	\max_{0\leq l \leq n} |x_l|_{\infty} \leq X
	\ \text{ and } \
	\max_{0\leq l \leq n} |x_l - x_0 \xi_{\infty}^l|_{\infty}
		\leq c X^{-\lambda_{\infty}},
\end{equation}
	have no non-zero solution ${\bf x} \in \mathbb{Z}^{n+1}$
	for arbitrarily large values of $X$.
	Then, there exist infinitely
	many non-zero polynomials $P(T) \in \mathbb{Z}[T]_{\leq n}$ such that
\begin{equation}\label{PT_2}
	|P(\xi_{\infty}) + \eta_{\infty}| \sim H(P)^{- 1/\lambda_{\infty}} \quad \text{and} \quad |P'(\xi_{\infty})| \sim H(P),
\end{equation}
	where the implied constants depend only on $c$, $\xi_{\infty}$,
	$\eta_{\infty}$ and $n$.
\end{theorem}

	Applying similarly Theorem \ref{RP3:C:1}
	with $S = \emptyset$, we obtain the following statement,
	which contains the result shown
	by H.~{\sc Davenport} and W.M.~{\sc Schmidt}
	in Lemma 1 of \cite{DS}.

\begin{theorem} \label{PT2}
	Let $n$, $c$, $\lambda_{\infty}$ and
	$\xi_{\infty}$ be as in Theorem \ref{PT1}.
	Let $R(T)$ be a polynomial in $\mathbb{Z}[T]$.
	Suppose again that the inequalities (\ref{PT_1})
	have no non-zero solution ${\bf x} \in \mathbb{Z}^{n+1}$
	for arbitrarily large values of $X$.
	Then there exist infinitely many polynomials $F(T) \in \mathbb{Z}[T]$ with
	the following properties:
	\\
	$(i)$ \ $\deg(R - F) \leq n$,
	\\
	$(ii)$ there exists a real root $\alpha_{\infty}$ of $F$, such that
\begin{equation}\label{PT2:2}
	| \xi_{\infty} - \alpha_{\infty} |_{\infty}
		\ll H(\alpha_{\infty})^{-1- 1/\lambda_{\infty}},
\end{equation}
	where the implied constants depend only on $c$, $\xi_{\infty}$,
	$R$ and $n$.
\end{theorem}

	For any real number $\beta$,
	denote by $\lceil \beta  \rceil$
	the smallest integer $\geq \beta$.
	For each index $n \geq 1$, we define
\begin{equation}\label{PT2:2:0}
	\mu_n : =
	\left \{
	\begin{array}{llll}
		\gamma             &  if \quad n = 2,
		\\{}
		\lceil n/2 \rceil &  if \quad n \neq 2.
	\end{array}
	\right.
\end{equation}
	Refining work of H.~{\sc Davenport} and W.M.~{\sc Schmidt}
	(see Theorem 1 of \cite{DS}), M.~{\sc Laurent} showed in \cite{La.1}
	that for $n\geq2$ and any $\xi \in \mathbb{R}$
	which is not algebraic
	of degree $\leq \lceil (n-1)/2 \rceil$ there exist
	infinitely many algebraic integers $\alpha$ of degree $n$, such that
\[
	0<| \xi - \alpha |_{\infty} \ll H(\alpha)^{-\lceil (n+1)/2 \rceil}.
\]
	The following corollary contains this result
	for the choice $R(T) = T^{n+1}$.
\begin{corollary} \label{PC1}
	Let $n \geq 1$ be an integer and
	$\xi_{\infty}$ be a real number which is not algebraic
	of degree $\leq \lceil \mu_n \rceil$ over $\mathbb{Q}$.
	Let $R(T)$ be an arbitrary polynomial in $\mathbb{Z}[T]$.

	There exist infinitely many real algebraic
	numbers $\alpha_{\infty}$, which are roots of polynomials
	$F(T) \in \mathbb{Z}[T]$ with $\deg(R - F) \leq n$, satisfying
\begin{equation}\label{PC_2}
	| \xi_{\infty} - \alpha_{\infty} |_{\infty} \ll H(\alpha)^{-\mu_n-1},
\end{equation}
	with implied constants depending only on $\xi_{\infty}$, $R$ and $n$.
\end{corollary}
\begin{proof}
	Since $\xi_{\infty}$ is not algebraic of degree $\leq \lceil \mu_n \rceil$,
	the hypotheses of Theorem \ref{PT2} are
	fulfilled with $\lambda_{\infty} = 1/\mu_n$.
	For $n = 2$, this follows from Theorem 1a
	of \cite{DS}.
	For $n \neq 2$, this follows  from the main
	result of \cite{La.1} (which refines Theorem 2a of \cite{DS}).
\end{proof}

%% file: II-03-duality-padic.tex
\section{$P$-adic case.}

	Let $n \geq 1$ be an integer, $p$ be a prime number
	and let $\xi_p \in \mathbb{Q}_p$.
	Applying Theorem \ref{RP3:T:1} with $\mathcal{S} = \{ p \}$
	and $\lambda_{\infty} = -1$, we have the following statement.

\begin{theorem} \label{IV:T:1}
	Let $c$, $\lambda_p$ be positive real numbers and
	let $\eta_p \in \mathbb{Z}_p$. Let $\rho_p \in  \mathbb{Z}_{p}$ be with
	$0<|\rho_p|_p \leq \|{\bf t}_{p}\|_{p}^{-1}$.
	Suppose that the inequalities
\begin{equation}\label{IV:T:1:1}
	\max_{0\leq l \leq n} |x_l|_{\infty} \leq X
	\ \text{ and } \
	\max_{0\leq l \leq n} |x_l - x_0 \xi_{p}^l|_{p}
		\leq c X^{-\lambda_{p}},
\end{equation}
	have no non-zero solution ${\bf x} \in \mathbb{Z}^{n+1}$
	for arbitrarily large values of $X$.
	Then there exist infinitely many non-zero polynomials
	$P(T) \in \mathbb{Z}[T]_{\leq n}$ satisfying
\begin{equation}\label{IV:T:1:2}
	|P(\xi_p) + \eta_p|_p \ll H(P)^{-\lambda_p/(\lambda_p-1)}
	\quad \text{and} \quad |P'(\xi_p)|_p = |\rho_p|_p,
\end{equation}
	where the implied constants depend only on $\xi_p$, $\eta_p$, $c$, $n$ and $p$.
\end{theorem}

	Applying Theorem \ref{RP3:C:1} with $\mathcal{S} = \{ p \}$
	and $\lambda_{\infty} = -1$, we obtain the following statement.

\begin{theorem} \label{IV:C:1}
	Let $c>0$  be a real number and let
	$R(T)$ be a polynomial in $\mathbb{Z}[T]$.
	Suppose $R(\xi_p) \in \mathbb{Z}_{p}$.
	Suppose also that the inequalities
\begin{equation}\label{IV:C:1:0}
	\max_{0\leq l \leq n} |x_l|_{\infty} \leq X
	\ \text{ and } \
	\max_{0\leq l \leq n} |x_l - x_0 \xi_{p}^l|_{p}
		\leq c X^{-\lambda_{p}},
\end{equation}
	have no non-zero solution ${\bf x} \in \mathbb{Z}^{n+1}$
	for arbitrarily large values of $X$.
	There exist infinitely many algebraic numbers $\alpha_p$
	in ${{\mathbb{Q}}_p}$, which are roots of polynomials
	$F(T) \in \mathbb{Z}[T]$ with $\deg(R - F) \leq n$, satisfying
\begin{equation}\label{IV:C:1:1}
	| \xi_p - \alpha_p |_p \ll H(\alpha_p)^{-\lambda_p/(\lambda_p-1)}.
\end{equation}
	Moreover, if $\xi_{p} \in \mathbb{Z}_{p}$, then
	these algebraic numbers $\alpha_p$ can be taken in ${\mathbb{Z}}_p$.
	All the implied constants depend only on $\xi_p$, $R$ and $n$.
\end{theorem}

	The following corollary follows from Proposition \ref{I:3:L:8} 
	and Theorem \ref{IV:C:1} with $n=2$ and $\lambda_{p} = \gamma$.
	In the case where $R(T) = T^{3}$, it contains the result 
	of approximation to $p$-adic numbers by cubic algebraic integers
	established by O~{\sc Teuli\'{e}} in Theorem 1 of \cite{Teu}. 
\begin{corollary} \label{IV:C:2}
	Let $\xi_p \in \mathbb{Z}_p$ and
	let $R(T)$ be a polynomial in $\mathbb{Z}[T]$.
	Suppose that $[\mathbb{Q}(\xi_p):\mathbb{Q}]> 2$.
	Then there exist infinitely many algebraic numbers $\alpha_p$
	in $\mathbb{Z}_p$,
	which are roots of polynomials $F(T) \in \mathbb{Z}[T]$
	with $\deg(R - F) \leq 2$, satisfying
\begin{equation}\label{IV:C:2:1}
	| \xi_p - \alpha_p |_p \ll H(\alpha_p)^{-\gamma^2},
\end{equation}
	with implied constants depending only on $\xi_p$ and $R$.
\end{corollary}

%% file: I-extremal-real-preliminary.tex

	Let $\xi$ be a non-quadratic real number, and let
	$\gamma = (1 + \sqrt{5})/{2}$ denote the golden ratio.
	Applying Corollary \ref{PC1} with $n=2$, shows
	that for any given polynomial $R \in \mathbb{Z}[T]$
	there are infinitely many algebraic numbers $\alpha$ which are roots
	of polynomials $F \in \mathbb{Z}[T]$ satisfying
\begin{equation}\label{intr1}
	\deg(R - F) \leq 2
	\ \text{ and } \
	|\xi - \alpha| \leq cH(\alpha)^{- \gamma^2},
\end{equation}
	for an appropriate constant $c > 0$ depending only on $\xi$ and $R$.
	Recall that $H(\alpha)$ stands for the \emph{height} of $\alpha$. It is
	defined as the height of its
	minimal polynomial over $\mathbb{Z}$, where the height $H(P)$
	of a polynomial $P \in \mathbb{R}[T]$
	is the maximum of the absolute values of its coefficients.
	The goal of this chapter is to provide a partial converse
	to this statement for a certain class of real numbers
	defined below.
	In particular, we will show that the exponent $\gamma^2$
	in (\ref{intr1}) cannot be improved when $R$ has degree $3$ or $4$.
\section{Preliminaries}\label{C3:S1}

	Recall from the introduction that a real number $\xi$
	is called \emph{extremal}, if $[\mathbb{Q}(\xi):\mathbb{Q}]>2$ and,
	for an appropriate constant $c = c(\xi) > 0$, the inequalities
\begin{equation}\label{int1}
	|x_0| \leq X, \quad |x_0 \xi - x_1| \leq c X^{ - 1/{\gamma}},
		\quad |x_0 \xi^2 - x_2| \leq c X^{ - 1/{\gamma}},
\end{equation}
	have a non-zero solution $(x_0,x_1,x_2) \in \mathbb{Z}^3$ for each real
	number $X \geq 1$.
	The existence of such numbers is proved in \cite{ARNCAI.1} and
	Theorem 5.1 of \cite{ARNCAI.1} provides the following criterion.
	A real number $\xi$ is extremal
	if and only if there exist an unbounded sequence of positive integers
	$(X_k)_{k \geq 1}$ and a sequence of
	points $( {\bf x}_{k})_{k \geq 1}$ in $\mathbb{Z}^3$ with
\begin{equation}\label{int2}
\begin{gathered}
	X_{k+1} \sim X_{k}^{\gamma}, \quad \parallel {{\bf  x}_k } \parallel \sim X_{k}, \quad
	\max\{|x_{k,0} \xi - x_{k,1}|,|x_{k,0} \xi^2 - x_{k,2}|\} \ll X_{k}^{-1},
	\\
	|\det({{\bf  x}_k })| \sim 1 \quad \text{and} \quad |\det({{\bf  x}_k },{{\bf  x}_{k+1} },{{\bf  x}_{k+2} })| \sim 1,
\end{gathered}
\end{equation}
	where for a point ${\bf  x} = (x_{0}, x_{1}, x_{2})$ we write $\det({\bf  x}) = x_0x_2 - x_1^2$, $\parallel {\bf  x } \parallel = \max\{|x_0|,|x_1|,|x_2|\}$
	and where $X \sim Y$ means $Y \ll X \ll Y$.
	Note that this also follows from Proposition \ref{RP5:P:1}
	by taking $\lambda_{\infty} = 1/{\gamma}$.
	
	Let
	$M =
	\left(
	\begin{matrix}
	a      &  b
	\\
	c      &  d
	\end{matrix}
	\right ) {\in \GL_2(\mathbb{Z})}$.
	As in \cite{ARNCAI.2}, we denote by $\mathcal{E}(M)$ the set
	of \emph{extremal} real numbers $\xi$ whose
	corresponding sequence of integer points $( {\bf x}_{k})_{k \geq 1}$,
	viewed as symmetric matrices
\[
	{ {\bf x}_k = }
	\left (
	\begin{matrix}
	x_{k,0}      &  x_{k,1}
	\\
	x_{k,1}      &  x_{k,2}
	\end{matrix}
	\right ),
\]
	belongs to ${\GL_2(\mathbb{Z})}$ and satisfies the recurrence formula
\begin{equation}\label{AQ1}
 	{\bf x}_{k+1} =  {\bf x}_{k} S_{k} {\bf x}_{k-1},
	\ \text{ where } \
	{S_k = }
	\left \{
	\begin{matrix}
	M      & \text{\small{if \quad k \quad is even}},
	\\
	{}^t M      & \text{\small{if \quad k \quad is odd}}.
	\end{matrix}
	\right.
\end{equation}
	Taking the transpose of this identity, using the fact that
	${\bf x}_{j}$ is symmetric matrix for each $j\geq1$, we get
\begin{equation}\label{AQ1:1}
 	{\bf x}_{k+1} =  {\bf x}_{k-1} S_{k-1} {\bf x}_{k}.
\end{equation}
	On the basis of Cayley-Hamilton's theorem and (\ref{AQ1}) the following identities are
	established in \cite{ARNCAI.2} (see Lemma 2.5, p.1084)
\begin{equation}\label{AQ2}
	 {\bf x}_{k+2} = \Tr( {\bf x}_{k} S_{k}) {\bf x}_{k+1} - \det( {\bf x}_{k} S_{k}) {\bf x}_{k-1},
\end{equation}
\begin{equation}\label{AQ3}
 	{\bf x}_{k}J {\bf x}_{k+1} = \det( {\bf x}_{k})JS_{k} {\bf x}_{k-1},
 	\ \text{ where } \ 
 	J =
	\left(
	\begin{matrix}
	0      &  1
	\\
	-1      &  0
	\end{matrix}
	\right ).
\end{equation}
	From now on, we fix an integer $a>0$
	and an element $\xi$ from the set $\mathcal{E}_a$ of extremal
	real numbers attached to the matrix
	$M=
	\left (
	\begin{matrix}
	a    & 1
	\\
	-1   & 0
	\end{matrix}
	\right )
	$.
	In this case we have
\[
	S_k =
	\left (
	\begin{matrix}
	a    & (-1)^k
	\\
	-(-1)^k   & 0
	\end{matrix}
	\right )
	=
	A + (-1)^k J , \ \ \text{where} \
	A = \left (
	\begin{matrix}
	a    & 0
	\\
	0   & 0
	\end{matrix}
	\right ).
\]
	Put
\[
\begin{tabular}{ | c | }
	\hline
	$\epsilon_{k} = \det( {\bf x}_{k} )$
	\\
	\hline
\end{tabular}
\]
	For each $k\geq1$, we have $\epsilon_k \in \{-1,1\}$,
	since ${\bf x}_{k} \in \GL_2(\mathbb{Z})$.
	Moreover, since $\det(S_{k}) = 1$, we have the following identities
\begin{equation}\label{AQ6:0:0}
 \begin{aligned}
	\epsilon_{k+2} &= \det({\bf x}_{k+1} S_{k+1} {\bf x}_{k})
		= \det({\bf x}_{k+1})\det( {\bf x}_{k})
		= \epsilon_{k+1}\epsilon_{k},
	\\
	\epsilon_{k+3} &= \epsilon_{k+2}\epsilon_{k+1}
		= \epsilon_{k+1}^2\epsilon_{k}
		= \epsilon_{k},
	\\
	\epsilon_{k+3}\epsilon_{k} &= \epsilon_{k}^2 = 1.
 \end{aligned}
\end{equation}
	Since $\xi \in \mathcal{E}_a$,
	the identities (\ref{AQ2}) and (\ref{AQ3}) can
	be rewritten in the form
\begin{equation}\label{AQ6:0}
	{\bf x}_{k+2} = a x_{k,0}  {\bf x}_{k+1} - \epsilon_{k}  {\bf x}_{k-1},
\end{equation}
	and
\begin{equation}\label{AQ6:1}
 \begin{aligned}
	x_{k,0}x_{k+1,1} &= x_{k,1}x_{k+1,0} - \epsilon_k(-1)^{k} x_{k-1,0},
	\\
	x_{k,1}x_{k+1,2} &= x_{k,2}x_{k+1,1} - \epsilon_k( a x_{k-1,1} + (-1)^{k} x_{k-1,2} ),
	\\
	x_{k,0}x_{k+1,2} &= x_{k,1}x_{k+1,1} - \epsilon_k(-1)^{k} x_{k-1,1},
	\\
	x_{k,1}x_{k+1,1} &= x_{k,2}x_{k+1,0} - \epsilon_k( a x_{k-1,0} + (-1)^{k} x_{k-1,1} ).
 \end{aligned}
\end{equation}
	The following identity is the sum of two last identities in (\ref{AQ6:1}).
\begin{equation}\label{AQ6:2}
	x_{k,0}x_{k+1,2} = x_{k,2}x_{k+1,0} - \epsilon_k( a x_{k-1,0} + 2(-1)^{k} x_{k-1,1}),
\end{equation}
	These identities precise the formulas of Lemma 2.5, p.1084 of \cite{ARNCAI.2}.
	Using the formula
\[
	{\bf w} J  {\bf w} J = J  {\bf w} J  {\bf w} = -\det( {\bf w})I,
\]
	valid for any symmetric $2\times2$ matrix ${\bf w}$
	( see p.46 of \cite{ARNCAI.1}),  we find that
\[
	 {\bf x}_{k}J  {\bf x}_{k}J = -\det(  {\bf x}_{k})I
	 = -\epsilon_k I
\]
	and so
\[
	J  {\bf x}_{k}J = -\epsilon_k  {\bf x}_{k}{}^{-1}.
\]
	Since $J^{-1} = - J$ it follows that
	$  {\bf x}_{k}J = \epsilon_k J   {\bf x}_{k}{}^{-1}$.
	Multiplying the last formula by
	$  {\bf x}_{k+2}$ on the right and applying (\ref{AQ1:1}),
	we deduce that
\[
	{\bf x}_{k}J  {\bf x}_{k+2}
	= \epsilon_k J  {\bf x}_{k}{}^{-1} {\bf x}_{k+2}
	= \epsilon_k J  {\bf x}_{k}^{-1} {\bf x}_{k+1} S_{k+1}  {\bf x}_{k}
	= \epsilon_k J  {\bf x}_{k}^{-1} {\bf x}_{k} S_k  {\bf x}_{k+1},
\]
	and so
\begin{align}\label{A:I:3}
  	{\bf x}_{k}J  {\bf x}_{k+2} = \epsilon_k JS_{k}  {\bf x}_{k+1}.
\end{align}
	This is the same identity as in (\ref{AQ3})
	with subscripts $k+1$ and $k-1$ replaced by
	$k+2$ and $k+1$. So, similarly as in (\ref{AQ6:1}),
	the identity (\ref{A:I:3}) can be rewritten in the form
\begin{equation}\label{A:I:5}
  \begin{aligned}
	x_{k,0}x_{k+2,1} &= x_{k,1}x_{k+2,0} - \epsilon_k(-1)^{k} x_{k+1,0},
	\\
	x_{k,0}x_{k+2,2} &= x_{k,1}x_{k+2,1} - \epsilon_k(-1)^{k} x_{k+1,1},
	\\
	x_{k,1}x_{k+2,1} &= x_{k,2}x_{k+2,0} - \epsilon_k( a x_{k+1,0} + (-1)^{k} x_{k+1,1} ),
	\\
	x_{k,1}x_{k+2,2} &= x_{k,2}x_{k+2,1} - \epsilon_k( a x_{k+1,1} + (-1)^{k} x_{k+1,2} ).
  \end{aligned}
\end{equation}
	The following identity is the sum of
	the second and the third identities in (\ref{A:I:5}).
\begin{equation}\label{A:I:5:1}
	x_{k,0}x_{k+2,2} = x_{k,2}x_{k+2,0} - \epsilon_k( a x_{k+1,0} + 2(-1)^{k} x_{k+1,1}).
\end{equation}

	We now derive an identity involving ${\bf x}_{k} J {\bf x}_{k+4}$.
	We have
\begin{align*}
  {\bf x}_{k}J  {\bf x}_{k+4}
	& = {\bf x}_{k} J {\bf x}_{k+2} S_{k+2} {\bf x}_{k+3} &
	\\
	& = \epsilon_k J S_{k}  {\bf x}_{k+1} S_{k+2} {\bf x}_{k+3} & _{\leftarrow \ \text{by} \ (\ref{A:I:3})}
	\\
	& = \epsilon_k J S_{k}  {\bf x}_{k+1} \big ( A +  (-1)^k J  \big ) {\bf x}_{k+3} &
	\\
	& = \epsilon_k J S_{k} \big ( {\bf x}_{k+1} A {\bf x}_{k+3} +  (-1)^k {\bf x}_{k+1} J  {\bf x}_{k+3} \big )
	\\
	& = \epsilon_k J S_{k} \big ( {\bf x}_{k+1} A {\bf x}_{k+3} +  (-1)^k \epsilon_{k+1} J S_{k+1} {\bf x}_{k+2} \big )  & _{\leftarrow \ \text{by} \ (\ref{A:I:3})}
	\\
	& = \epsilon_k \big ( J S_{k}  {\bf x}_{k+1} A {\bf x}_{k+3} +  (-1)^k \epsilon_{k+1} J S_{k} J S_{k+1} {\bf x}_{k+2} \big )
	\\
	& = \epsilon_k \big ( J S_{k}  {\bf x}_{k+1} A {\bf x}_{k+3} - (-1)^k \epsilon_{k+1}  {\bf x}_{k+2} \big )  & _{\leftarrow \ \text{since} \ J S_{k} J S_{k+1} = - I \ },
\end{align*}
	which gives a new identity
\[
  	{\bf x}_{k}J  {\bf x}_{k+4}
 	= \epsilon_k \big ( J S_{k}  {\bf x}_{k+1} A {\bf x}_{k+3} - (-1)^k \epsilon_{k+1}  {\bf x}_{k+2} \big ).
\]
	Let us write this identity in an explicit form
\begin{align*}
	&\left (
	\begin{matrix}
		-x_{k,1}x_{k+4,0} + x_{k,0}x_{k+4,1}      &  -x_{k,1}x_{k+4,1} + x_{k,0}x_{k+4,2}
		\\
		-x_{k,2}x_{k+4,0} + x_{k,1}x_{k+2,1}      &  -x_{k,2}x_{k+4,1} + x_{k,1}x_{k+4,2}
	\end{matrix}
	\right )
	\\
	&
	 \quad \quad
	 =- \epsilon_k a
	\left (
	\begin{matrix}
	(-1)^k x_{k+1,0} x_{k+3,0}       &  (-1)^k x_{k+1,0} x_{k+3,1}
	\vspace{1mm}\\
	( a x_{k+1,0} + (-1)^k x_{k+1,1} ) x_{k+3,0}       &  ( a x_{k+1,0} + (-1)^k x_{k+1,1} ) x_{k+3,1}
	\end{matrix}
	\right )
	\\
	& \quad \quad \quad
	- (-1)^k \epsilon_{k+2}  {\bf x}_{k+2} .
\end{align*}
	Taking the corresponding elements in the positions
	$(1,2)$ and $(2,2)$, and using (\ref{AQ6:0:0}),
	this gives the identities
\begin{equation}\label{A:I:8}
	x_{k,0}x_{k+4,2} =  x_{k,1}x_{k+4,1} - \epsilon_k (-1)^k ( a x_{k+1,0}x_{k+3,1} + \epsilon_{k+1} x_{k+2,1}),
\end{equation}
\begin{equation}\label{A:I:9}
	x_{k,1}x_{k+4,2} =  x_{k,2}x_{k+4,1} - \epsilon_k \big ( a ( a x_{k+1,0} + (-1)^k x_{k+1,1})x_{k+3,1} + (-1)^k \epsilon_{k+1} x_{k+2,2} \big ).
\end{equation}

	Now, we introduce some notation and prove an auxiliary lemma.
	For each $k \geq 1$ we put
\begin{equation}\nonumber
	a_{k} =
	\left |
	\begin{matrix}
	x_{k,0}    & x_{k,1}
	\\
	x_{k+1,0}   & x_{k+1,1}
	\end{matrix}
	\right |
	, \
	b_{k} = -
	\left |
	\begin{matrix}
	x_{k,0}    & x_{k,2}
	\\
	x_{k+1,0}   & x_{k+1,2}
	\end{matrix}
	\right |
	, \
	c_{k} =
	\left |
	\begin{matrix}
	x_{k,1}    & x_{k,2}
	\\
	x_{k+1,1}   & x_{k+1,2}
	\end{matrix}
	\right |.
\end{equation}
	Using the formulas (\ref{AQ6:0}) - (\ref{AQ6:2}), we find that
\begin{equation}\label{AQ6-000}
	\left \{
	\begin{aligned}
		a_k &=  \epsilon_k(-1)^{k+1} x_{k-1,0},
		\\
		b_k &=  \epsilon_k( a x_{k-1,0} + 2(-1)^{k} x_{k-1,1}),
		\\
		c_k &=  \epsilon_k( -a x_{k-1,1} - (-1)^{k} x_{k-1,2}).
	\end{aligned}
	\right.{}
\end{equation}
	Using this, we find that
\begin{align*}
	\det( {\bf x}_{k-1},  {\bf x}_{k}, {\bf x}_{k+1})
	&=
	a_k x_{k-1,2} + b_k x_{k-1,1} + c_k x_{k-1,0}
	\\
	&=\epsilon_k(-1)^{k+1}(
	x_{k-1,0}x_{k-1,2}-2x_{k-1,1}^2 + x_{k-1,2}x_{k-1,0}
	)
	\\
	& =
	2\epsilon_{k+1}(-1)^{k+1}.
\end{align*}
\vspace{3mm}
\begin{lemma} \label{ALT2}
Let $\xi \in \mathcal{E}_a$ with $a \in \mathbb{Z}_{>0}$. For any $k \geq 1$,
\\
$(i)$ the $\gcd$ of $a_{k}$ and $b_{k}x_{k+1,2} \ + \ c_{k}x_{k+1,1}$ divides $2$,
\\
$(ii)$ the $\gcd$ of $a_{k}$ and $b_{k}$ divides $2$.
\end{lemma}
\begin{proof}
	Since $\det( {\bf x}_{k+1}) =
	\left |
	\begin{matrix}
		x_{k+1,0}    & x_{k+1,1}
		\\
		x_{k+1,1}   & x_{k+1,2}
	\end{matrix}
	\right | = \pm 1$, we have
\begin{equation}\label{AL1}
	\gcd(b_{k}x_{k+1,1} \ + \ c_{k}x_{k+1,0}, \ b_{k}x_{k+1,2} \ + \ c_{k}x_{k+1,1}) = \gcd(b_k, c_k).
\end{equation}
	It follows from $\det( {\bf x}_{k+1},  {\bf x}_{k},  {\bf x}_{k+1}) = 0$ that
\[
	a_{k} x_{k+1,2} + b_{k} x_{k+1,1} + c_{k}x_{k+1,0} = 0.
\]
	We notice that for $k \gg 1$, we have $x_{k,0} \neq 0 $
	by (\ref{int2}) and the fact that the sequence  $(X_k)_{k \geq 1}$ is unbounded. By (\ref{AQ6-000}), this implies
	that $a_k \neq 0$. This gives
\[
	\ a_{k} \ | \ b_{k}x_{k+1,1} \ + \ c_{k}x_{k+1,0}.
\]
 	From this relation and (\ref{AL1}) we deduce that
\begin{equation}\label{AL2}
	\gcd(a_k,b_{k}x_{k+1,2} \ + \ c_{k}x_{k+1,1}) \ | \ \gcd(a_k, b_k, c_k). 
\end{equation}
	Furthermore, by definition of the determinant we have $a_k x_{k-1,2}+b_k x_{k-1,1}+c_k x_{k-1,0} = \det( {\bf x}_{k-1},  {\bf x}_{k},  {\bf x}_{k+1}) = 2\epsilon_{k+1}(-1)^{k+1}$ and therefore
\begin{equation}\label{AL3}
	\gcd(a_k, b_k, c_k) \ | \ 2.
\end{equation}
	The assertion (i) follows by combining  (\ref{AL2}) and (\ref{AL3}).
	The assertion (ii) follows from the formulas for $a_k$ and $b_k$ given by (\ref{AQ6-000})
	and the fact that
\[
	\gcd(x_{k-1,0}, x_{k-1,1}) = 1.
\]
\end{proof}

%% file: paper.tex
\section{Approximation to extremal real numbers by algebraic numbers
		of degree at most $4$}

	Here we show that $(3+\sqrt{5})/2$ is the optimal exponent of approximation
	to transcendental real numbers by algebraic numbers of degree at most
	$4$ with bounded denominator and trace.

	Let the notation be as in \S \ref{C3:S1}.
	This means that we fixed a choice of a positive integer $a$, an extremal real
	number $\xi \in \mathcal{E}_a$,
	and corresponding sequences $(X_k)_{k \geq 1}$ in $\mathbb{Z}_{>0}$
	and $({\bf x}_k)_{k \geq 1}$ in $\mathbb{Z}^3$
	as in (\ref{int2}). For any integer $n \geq 0$, we denote by
	$\mathbb{Z}[T]_{\leq n}$ the set of polynomials of $\mathbb{Z}[T]$
	of degree at most $n$, and for any real number $\beta$,
	we denote by $\{ \beta \}$ the distance from $\beta$ to its
	closest integer. In the computations below, we will often use
	the fact that for any $\beta, \beta' \in \mathbb{R}$, we have
\begin{equation}\label{beta-0}
	\big |
		\{\beta\} - \{\beta'\}
	\big |
		\leq
		\min
		\big(
			\{\beta + \beta'\},
			\{\beta - \beta'\}
		\big)
\end{equation}
	The main result of this section is the following statement which
	implies Theorem \ref{introduction:T:3} of the introduction.

\begin{theorem}\label{ALT0}
There exists a constant $c > 0$ such that
for any $R \in \mathbb{Z}[T]$ of degree $3$ or $4$ and any $P \in \mathbb{Z}[T]_{\leq 2}$ with $P \neq 0$,  we have
 \begin{equation}\label{AL0}
 	| R(\xi) + P(\xi) |  \geq  c H(R)^{- 2 \gamma^9} H(P) ^{- \gamma },
 \end{equation}
 and moreover
 \begin{equation}\label{ALT0_01}
 	| R(\xi) |  \geq  c  H(R)^{ - 1 - \gamma^5}.
 \end{equation}
\end{theorem}

	Before going into the proof, we mention the following corollary
	which provides a measure of
	approximation to the elements of $\mathcal{E}_a$ by algebraic
	numbers of degree $\leq 4$.
\begin{corollary}\label{ALT0C}
	There exists a constant $c = c(\xi) > 0$ with the following properties.
	For any algebraic number $\alpha$
	of degree at most 4 we have
\begin{equation}\label{AL0C1_1}
 	| \xi - \alpha | \geq c H(\alpha)^{-5\gamma^2}.
\end{equation}
	Moreover, if $\deg(\alpha) \leq 3$ and the denominator of $\alpha$ is
	bounded above by some real
	number $B > 0$, then we have
\begin{equation}\label{AL0C1_2_1}
 	| \xi - \alpha | \geq c B^{- 6 \gamma^9} H(\alpha)^{-\gamma^2}.
\end{equation}
	If $\deg(\alpha) = 4$, the denominator of $\alpha$ and
	the absolute value of its trace
	are bounded above by $B$, then we have
\begin{equation}\label{AL0C1_2}
 	| \xi - \alpha | \geq c B^{- 10 \gamma^9} H(\alpha)^{-\gamma^2}.
\end{equation}
\end{corollary}
\begin{proof}
	Let $\alpha$ be an algebraic number of degree at most 4.
	As in \cite{ARNCAI.1}, Proposition $9.1$, define $Q(T) \in \mathbb{Z}[T]_{\leq 4}$ to
	be its minimal polynomial or the product of it by some appropriate power of $T$,
	making $Q(T)$ of degree $3$, if it is not of degree $4$ originally.
	Since $ H(Q) = H(\alpha)$, the second part of Theorem \ref{ALT0} leads to
\[
	|\xi - \alpha| \gg H(Q)^{-1} | Q(\xi) |
	\gg H(Q)^{- 2 - \gamma^5} = H(\alpha)^{- 5\gamma^2}.
\]
	Now, suppose that $\deg(\alpha) \leq 3$ and that the denominator
	$\den(\alpha)$ of $\alpha$ is bounded above
	by some real number $B > 0$. Write $Q(T) = a_0 T^3 + a_1 T^2 + a_2 T + a_3$.
	We have $| a_0 | \leq B^3$ since $a_0$ divides $\den(\alpha)^3$. So $Q(T)$ can be written
	as a sum $Q = R + P$, where $R(T) = a_0 T^3\in \mathbb{Z}[T]$ has degree $3$ and height $H(R) \leq B^3$,
	and where $P(T) = a_1 T^2 + a_2 T + a_3 \in \mathbb{Z}[T]_{\leq 2}$ satisfies $H(P) \leq H(Q) = H(\alpha)$.
	Since $P \neq 0$, then the inequality (\ref{AL0}) of Theorem \ref{ALT0} gives
\[
	|\xi - \alpha| \gg H(\alpha)^{-1} | R(\xi) + P(\xi) |
	\gg H(\alpha)^{-1}  H(R)^{- 2 \gamma^9} H(P)^{- \gamma}
	\gg B^{- 6 \gamma^9} H(\alpha)^{- \gamma^2}.
\]
	Finally, suppose that $\deg(\alpha) = 4$ and that the denominator $\den(\alpha)$ of $\alpha$ and the
	absolute value of its trace $| \Tr(\alpha) |$ are bounded above by some real number $B > 0$.
	Write $Q(T) = a_0 T^4 + a_1 T^3 + a_2 T^2 + a_3 T + a_4$.
	We have $| a_0 | \leq B^4$ since $a_0$ divides $\den(\alpha)^4$ and $| a_1 | \leq B^5$ since
	$| \Tr(\alpha) | = | a_1 / a_0 | \leq B$. So $Q(T)$ can be written as a sum $Q = R + P$, where
	$R(T) = a_0 T^4 + a_1 T^3 \in \mathbb{Z}[T]$ has degree $4$ and height $H(R) \leq B^5$, and where
	$P(T) = a_2 T^2 + a_3 T + a_4 \in \mathbb{Z}[T]_{\leq 2}$ satisfies $H(P) \leq H(Q) = H(\alpha)$.
	Since $P \neq 0$, then the inequality (\ref{AL0}) of Theorem \ref{ALT0} gives
\[
	|\xi - \alpha| \gg H(\alpha)^{-1} | R(\xi) + P(\xi) |
	\gg H(\alpha)^{-1} H(R)^{- 2 \gamma^9} H(P)^{- \gamma}
	\gg B^{- 10 \gamma^9} H(\alpha)^{- \gamma^2}.
\]
\end{proof}

	Recall that $\{\beta\}$ denotes the distance from
	a real number $\beta$ to its closest integer.
	To prove the main estimate of Theorem \ref{ALT0} we need
	a lower bound for $\{ x_{k,0} R(\xi) \}$
	where $k$ is an arbitrary large positive integer.
	The next proposition implies that the sequence
	$\{ x_{k,0} R(\xi) \}$ tends to a limit as $k$
	tends to infinity in a congruence class modulo $3$
	or modulo $6$ if the polynomial $R$ is of degree of at most $3$
	or at most $4$, respectively.
\begin{proposition} \label{AQT1}
	There exists a constant $c > 0$ such that for any polynomial
	$R \in \mathbb{Z}[T]_{\leq 4}$ and any integer $k \geq 1$, we have
\begin{equation}\label{AQ7}
	| \ \{ x_{k+6,0}R(\xi) \} - \{ x_{k,0} R(\xi) \} \ | \leq c H(R) X^{-1}_{k}.
\end{equation}
	Moreover, if $\deg(R)\leq3$, we have
\begin{equation}\label{AQ7:1}
	| \ \{ x_{k+3,0}R(\xi) \} - \{ x_{k,0} R(\xi) \} \ | \leq c H(R) X^{-1}_{k}.
\end{equation}
\end{proposition}
\begin{proof}
	Suppose that $R \in \mathbb{Z}[T]_{\leq 4}$.
	Using (\ref{beta-0}) and the fact that $\epsilon_{k+1} \in \{-1,1\}$
	(see sec.~3.1),
	for each $k \in \mathbb{Z}_{>0}$, we find that
\begin{equation}\label{AQT1P0}
	\begin{aligned}
		| \ \{ x_{k+6,0} R(\xi) \} - \{ x_{k,0} R(\xi) \} \ |
		& \leq \{ (x_{k+6,0} - x_{k,0}) R(\xi) \}
		\\
		& \leq H(R) \sum_{n = 0}^{4} \{ (x_{k+6,0} - x_{k,0}) \xi^n \},
	\end{aligned}
\end{equation}
	and if $\deg(R)\leq3$, we get
\begin{equation}\label{AQT1P0:1}
	\begin{aligned}
		| \ \{ x_{k+3,0} R(\xi) \} - \{ x_{k,0} R(\xi) \} \ |
		& \leq \{ (x_{k+3,0} + \epsilon_{k+1}x_{k,0}) R(\xi) \}
		\\
		& \leq H(R) \sum_{n = 0}^{3} \{ (x_{k+3,0} + \epsilon_{k+1}x_{k,0}) \xi^n \}.
	\end{aligned}
\end{equation}
	To prove (\ref{AQ7}), it suffices to show that
	$\{ (x_{k+6,0} - x_{k,0})\xi^n \} = O(X_k^{-1})$ for $n = 0,1,...,4$.
	To prove (\ref{AQ7:1}), we use the fact that $\epsilon_{k+1} = \pm1$
	and so, it suffices to show that
	$\{ (x_{k+3,0} + \epsilon_{k+1}x_{k,0})\xi^n \}
	= O(X_k^{-1})$ for $n = 0,1,...,3$.
	If $n = 0,1,2$, this is clear since by (\ref{int2}) we have
\begin{equation}\label{AQT1P1}
	x_{k,0} \xi^n  = x_{k,n} + O(X_k^{-1}).
\end{equation}
	For $n \geq 3$, by applying successively (\ref{AQT1P1})
	and the identity (\ref{AQ6:0}) we find
\begin{equation}\label{AQT1P2}
\begin{aligned}
 	(x_{k+3,0} + \epsilon_{k+1} x_{k,0})\xi^n
	&= (x_{k+3,2} + \epsilon_{k+1} x_{k,2})\xi^{n-2} + O(X_k^{-1})
 	\\
 	& = a x_{k+1,0}x_{k+2,2}\xi^{n-2} + O(X_k^{-1}).
\end{aligned}
\end{equation}
	On the other hand, the identity (\ref{AQ6:2}) gives
\[
 	x_{k+1,0}x_{k+2,2} = x_{k+1,2}x_{k+2,0}
	- \epsilon_{k+1}( a x_{k,0} + 2(-1)^{k+1} x_{k,1}).
\]
	It follows from this and (\ref{AQT1P1}) that
\begin{equation}\label{AQT1P2-1}
 	x_{k+1,0}x_{k+2,2}\xi^{n-2}
	=  ( x_{k+1,2}x_{k+2,1} - \epsilon_{k+1}( a x_{k,1} + 2(-1)^{k+1} x_{k,2}))\xi^{n-3} + O(X^{-1}_{k}).
\end{equation}
	Combining (\ref{AQT1P2}) and (\ref{AQT1P2-1}) we deduce that
\begin{equation}\label{AQT1P2-2}
	(x_{k+3,0} + \epsilon_{k+1} x_{k,0})\xi^n =  n_k\xi^{n-3} + O(X^{-1}_{k}),
\end{equation}
	where $n_k$ is the integer $a ( x_{k+1,2}x_{k+2,1} - \epsilon_{k+1} a x_{k,1} + \epsilon_{k+1} 2 a (-1)^{k} x_{k,2})$.
	So, using (\ref{AQT1P0:1}), (\ref{AQT1P1}) and (\ref{AQT1P2-2}),
	we complete the proof of (\ref{AQ7:1}).

	Furthermore, for $n \geq 4$, we have
\begin{align*}
	n_k\xi^{n-3} &= a ( x_{k+1,2}x_{k+2,1} - \epsilon_{k+1} a x_{k,1})\xi^{n-3} + 2 a \epsilon_{k+1} (-1)^{k} x_{k,2} \xi^{n-3}
	\\
	&= m_k \xi^{n-4} + 2 a \epsilon_{k+1} (-1)^{k} x_{k,2} \xi^{n-3} + O(X^{-1}_{k}),
\end{align*}
	where $m_k$ is an integer. We conclude that
\begin{align*}
	(x_{k+3,0} + \epsilon_{k+1} x_{k,0})\xi^3 & =  n_k + O(X^{-1}_{k}),
	\\
	(x_{k+3,0} + \epsilon_{k+1} x_{k,0})\xi^4 & =  m_k + 2 a \epsilon_{k+1} (-1)^{k} x_{k,2} \xi + O(X^{-1}_{k}).
\end{align*}
	Since $\epsilon_{k+4} = \epsilon_{k+1}$, it follows from these two formulas that
\begin{align*}
	(x_{k+6,0} - x_{k,0})\xi^3 & = (x_{k+6,0} + \epsilon_{k+4} x_{k+3,0})\xi^3 - \epsilon_{k+4}(x_{k+3,0} + \epsilon_{k+1} x_{k,0}) \xi^3
	\\
	& = n_{k+3}-\epsilon_{k+4}n_{k} + O(X^{-1}_{k}),
	\\
	(x_{k+6,0} - x_{k,0})\xi^4 & = (x_{k+6,0} + \epsilon_{k+4} x_{k+3,0})\xi^4 - \epsilon_{k+4}(x_{k+3,0} + \epsilon_{k+1} x_{k,0}) \xi^4
	\\
	& = m_{k+3}-\epsilon_{k+4}m_{k} - 2 a \epsilon_{k+4} (-1)^{k} ( x_{k+3,0} + \epsilon_{k+1} x_{k,0})\xi^3  + O(X^{-1}_{k})
	\\
	& = m_{k+3}-\epsilon_{k+4}m_{k} - 2 a \epsilon_{k+4} (-1)^{k} n_{k}  + O(X^{-1}_{k}),
\end{align*}
	and therefore, we have
	$\{ (x_{k+6,0} - x_{k,0})\xi^n \} = O(X_k^{-1})$ for $n = 3,4$,
	which together with (\ref{AQT1P0}) and (\ref{AQT1P1}),
	we completes the proof of (\ref{AQ7})
\end{proof}
\begin{corollary} \label{AQQT1}
	Suppose that $R \in \mathbb{Z}[T]_{\leq 4}$.
	Then the sequence $\Big (\{ x_{k,0}R(\xi) \} \Big)_{k\geq1}$
	has at most $6$ accumulation points. More precisely, for
	each $l=0,1,\ldots,5$, $\{ x_{l+6i,0}R(\xi) \}$ tends to a limit
	$\eta_{l}(R)$ as $i$ tends to infinity.

	Moreover, if $\deg(R)\leq3$, then
	$\Big (\{ x_{k,0}R(\xi) \} \Big)_{k\geq1}$
	has at most $3$ accumulation points. More precisely, for
	each $l=0,1,2$, $\{ x_{l+3i,0}R(\xi) \}$ tends to a limit
	$\delta_{l}(R)$ as $i$ tends to infinity.
\end{corollary}
\begin{proof}
	Since $X_k$ tends to infinity faster than any geometric sequence,
	the inequality (\ref{AQ7}) of Proposition \ref{AQT1} implies that
	$\Big (\{ x_{l+6i,0}R(\xi) \} \Big)_{i\geq1}$ is a Cauchy
	sequence for each $l=0,1,\ldots,5$. Similarly, if $\deg(R)\leq3$,
	the inequality (\ref{AQ7:1}) implies that
	$\Big (\{ x_{l+3i,0}R(\xi) \} \Big)_{i\geq1}$ is a Cauchy
	sequence for each $l=0,1,2$.
\end{proof}
	The next proposition provides a rough lower bound for the numbers
	$\{ x_{k,0}R(\xi) \}$.
\begin{proposition} \label{ALT1}
	There exists a constant $c = c(\xi) > 0$
	such that for any $k \geq 1$ and any non-zero
	polynomial $R \in \mathbb{Z}[T]$ of degree $3$ or $4$ with
	$H(R) \leq c X_k^{{1}/{\gamma^3}}$ we have
\begin{equation}\label{AL1-0}
	\{ x_{k,0}R(\xi) \} \geq c X_{k}^{-{2}/{\gamma^2}}.
\end{equation}
\end{proposition}
\begin{proof}
	Let $R(T) = pT^4 + qT^3 + rT^2 + s T + t$ be a
	polynomial of $\mathbb{Z}[T]$ of degree $3$ or $4$.
	For our purposes, we construct a sequence of
	polynomials $(P_k)_{k \geq 1}$ in $\mathbb{Z}[T]$
	of the same degree by putting
\begin{equation} \label{AL9}
	P_k(T) = \big ( p a_k T^2 + ( q a_k - p b_k) T \big ) Q_k(T) =
		a_k^2R(T) + B_k T^2 + C_k T + D_k,
\end{equation}
	where $a_k$, $b_k$, $c_k$ are the integers defined in \S 2,
	$Q_k(T) = a_k T^2 + b_k T + c_k$ and
\[
    \begin{aligned}
		B_k & = p a_k c_k + (q a_k - p b_k)b_k - r a_k^2,
		\\
		C_k & = (q a_k - p b_k) c_k - s a_k^2,
		\\
		D_k & = - t a_k^2.
    \end{aligned}
\]
	By the virtue of the estimates $H(Q_k) \sim X_k^{1/{\gamma}}$
	and $|Q_k(\xi)| \sim X_k^{- \gamma^2}$
	(see Proposition 8.1 of \cite{ARNCAI.1}), we have
\begin{equation} \label{AL10}
\begin{gathered}
	| B_k |, | C_k |, | D_k |, H(P_k) \ll H(R) H(Q_k)^2 \ll H(R) X_k^{{2}/{\gamma}},
	\\
	| P_k(\xi) | \ll H(R) H(Q_k) |Q_k(\xi)| \ll H(R) X_k^{-2}.
\end{gathered}
\end{equation}
Consider the integer
\[
	N_k = a_k^2 x_{k+1,R} + B_k x_{k+1,2} + C_k x_{k+1,1} + D_k x_{k+1,0},
\]
where $x_{k+1,R}$ denotes the closest integer to $x_{k+1,0} R(\xi)$. From (\ref{AL9}), we get
\[
	N_k = a_k^2 (x_{k+1,R} - x_{k+1,0} R(\xi) ) + B_k (x_{k+1,2}
	    - x_{k+1,0}\xi^2) + C_k (x_{k+1,1} - x_{k+1,0}\xi) + x_{k+1,0}P_k(\xi).
\]
By (\ref{AL10}) it follows that there exists a constant $c_1 > 0$, such that for all $k\geq 1$
\begin{equation}\label{AL11}
	| N_k | \leq c_1 \Big ( X_k^{{2}/{\gamma}} \big ( \{x_{k+1,0}R(\xi) \} + H(R) X_{k+1}^{-1} \big ) + H(R) X_{k+1} X_k^{-2} \Big ).
\end{equation}
	We now provide a condition on $R$ that ensures $ N_k \neq 0$.
	If $\deg(R) = 4$ then $p \neq 0$ and we find
\[
	N_k \equiv - p b_k ( b_k x_{k+1,2} + c_k x_{k+1,1}) \mod a_k.
\]
	If $\deg(R) = 3$ we have $p = 0$ and $q \neq 0$, and then
\[
	N_k \equiv q a_k ( b_k x_{k+1,2} + c_k x_{k+1,1}) \mod a_k^2.
\]
	If $N_k = 0$ then it follows from Lemma \ref{ALT2}
	that $a_k$ divides $4p$ or $4q$. Hence we have $\big| a_k \big| \ll H(R)$
	which implies
\[
	X_{k}^{1/\gamma} \ll H(R).
\]
	Whence we deduce that if $H(R) \leq c_2 X_{k}^{1/\gamma}$ for an appropriate constant $c_2 > 0$ then $| N_k | \geq 1$.

	If we furthermore assume
\begin{equation}\label{AL13}
	H(R) (X_k^{{2}/{\gamma}} X_{k+1}^{-1} + X_{k+1} X_k^{-2}  ) \leq \cfrac{1}{2 c_1},
\end{equation}
	then the inequality (\ref{AL11}) together with $| N_k | \geq 1$ implies
\begin{equation}\label{AL13-1}
	\{ x_{k+1,0}R(\xi) \} \geq \cfrac{1}{2 c_1} X_k^{-{2}/{\gamma}}.
\end{equation}
	Since $X_k^{{2}/{\gamma}} X_{k+1}^{-1} \sim X_{k+1} X_k^{-2} \sim X_k^{- 1/{\gamma^2}}$ the condition (\ref{AL13})
	is satisfied if $H(R) \leq c_3 X_k^{{1}/{\gamma^2}}$ for an appropriate constant $c_3 > 0$. Assuming $c_3 \leq c_2$,
	we conclude that (\ref{AL13-1}) holds whenever $H(R) \leq c_3 X_k^{{1}/{\gamma^2}}$. The conclusion follows.
\end{proof}

	By combining the above proposition with the preceding one,
	we obtain a better bound for $\{ x_{l,0}R(\xi) \}$ when $l$
	is a large integer.
\begin{corollary} \label{ALC2}
	There exists a constant $c = c(\xi) > 0$
	such that for any
	$l, k \in \mathbb{Z}_{>0}$ with $l \equiv k \mod 6$ and $l \geq k \geq 1$,
	and for any non-zero polynomial $R \in \mathbb{Z}[T]$ of degree $3$ or $4$
	with $H(R) \leq c X_{k}^{{1}/{\gamma^3}}$ we have
\begin{equation}\label{ALC2_1}
 	\{ x_{l,0}R(\xi) \} \geq c X_{k}^{- {2}/{\gamma^2}}.
\end{equation}
\end{corollary}
\begin{proof}
	Let $k$ and $l$ be positive integers with $l \equiv k \mod 6$ and $l \geq k$.
	Since the sequence $( X_k )_{k \geq 1}$ grows at least geometrically, then by Proposition \ref{AQT1}
	there exists a constant $c_1 = c_1(\xi)  \geq 1$ such that
\[
	| \ \{ x_{l,0}R(\xi) \} - \{ x_{k,0} R(\xi) \} \ | \leq c_1 H(R) X^{-1}_{k},
\]
	for any polynomial $R \in \mathbb{Z}[T]$ of degree at most $4$. By Proposition \ref{ALT1} there
	exists a constant $c_2 = c_2(\xi) > 0$ such that $\{ x_{k,0}R(\xi) \} \geq c_2 X_{k}^{- 2/{\gamma^2}}$
	if $R \in \mathbb{Z}[T]$ has degree $3$ or $4$ and $H(R) \leq c_2 X_k^{1/{\gamma^3}}$.
	Suppose that $R \in \mathbb{Z}[T]$ has degree $3$ or $4$ and that $H(R) \leq \cfrac{c_2}{2 c_1} X_k^{{1}/{\gamma^3}}$.
	Then by combining these estimates we find (using $1/\gamma^3 - 1 = -2/\gamma^2$)
\[
	\{ x_{l,0}R(\xi) \}
		\geq \{ x_{k,0}R(\xi) \} - c_1 H(R) X_k^{-1}
		\geq \frac{c_2}{2} X_{k}^{- {2}/{\gamma^2}}
		\geq \frac{c_2}{2 c_1} X_{k}^{- {2}/{\gamma^2}}.
\]
\end{proof}

	In particular, the above corollary shows that the real numbers
	$\delta_l(R)$ and $\eta_l(R)$, defined in Corollary \ref{AQQT1},
	are all non-zero, for any
	polynomial $R \in \mathbb{Z}[T]$ of degree $3$ or $4$.
	Now we can proceed with the proof of the main Theorem \ref{ALT0}.
\\
\\
\begin{proof}[Proof of Theorem \ref{ALT0}]

	Let $R(T)$ and $P(T)$ be as in the statement of the theorem.
	Consider the following identity
\[
 	x_{l,0}R(\xi) =  x_{l,0} \big ( R(\xi) + P(\xi) \big )  - x_{l,0} P(\xi).
\]
	Since $\{x_{l,0}\xi^2\} \ll X_l^{-1}$, $\{x_{l,0}\xi\} \ll X_l^{-1}$,
	there exists a constant $c_1 > 0$ such that for all $l \geq 1$ we have
 \begin{equation}\label{ALT0_1}
 	c_1 \{ x_{l,0}R(\xi) \} \leq  X_l | R(\xi) + P(\xi) | + H(P) X_l^{-1}.
 \end{equation}
	In order to obtain a lower bound for $| R(\xi) + P(\xi) |$ we need
	a lower bound for $\{ x_{l,0}R(\xi) \}$ and an upper bound for $H(P)$.
	Denote by $c_2$ the constant $c$ of Corollary \ref{ALC2}, and let $k$ be the smallest integer such
	that $H(R) \leq c_2 X_k^{1/{\gamma^3}}$. It follows by Corollary \ref{ALC2} that
	$\{ x_{l,0}R(\xi) \} \geq c_2 X_{k}^{- {2}/{\gamma^2}}$ if $l \equiv k \mod 6$ and $l \geq k$.
	Since every integer $l \geq k$ is congruent modulo $6$ to some integer in $[k,k+5]$,
	we deduce that for all $l \geq k$, we have
\begin{equation}\label{ALT0_3}
  \{ x_{l,0}R(\xi) \} \geq c_2 X_{k + 5}^{- {2}/{\gamma^2}}.
\end{equation}
	Choose $l$ to be the smallest integer with $l \geq k$ such that
\begin{equation}\label{ALT0_4}
	H(P) \leq \frac{1}{2} c_1 c_2 X_l X_{k + 5}^{- 2/{\gamma^2}}.
\end{equation}
	It follows from (\ref{ALT0_1}) and (\ref{ALT0_4}) that
\[
 	X_l | R(\xi) + P(\xi) | \geq \frac{1}{2} c_1 c_2 X_{k + 5}^{- 2/{\gamma^2}}.
\]
	The choice of $k$ and $l$ implies $H(R) \gg X_k^{1/{\gamma^4}}$ and
	$H(P) \gg X_l^{1/\gamma} X_{k + 5}^{- 2/{\gamma^2}}$.
	So, we get
\[
	X_l  \ll H(P)^{\gamma} X_{k + 5}^{2/{\gamma}}
	\ \text{ and } \
	X_{k + 5} \ll X_k^{\gamma^5} \ll H(R)^{\gamma^9},
\]
	and these estimates lead to
\[
 	| R(\xi) + P(\xi) | 	\gg  X_l^{-1} X_{k + 5}^{- 2/{\gamma^2}}
				\gg  H(P)^{- \gamma} X_{k + 5}^{- 2}
				\gg  H(P)^{- \gamma } H(R)^{- 2 \gamma^9}.
\]

	In the case where $P(T) = 0$ the inequality (\ref{ALT0_1}) becomes
\begin{equation}\label{ALT0_1_1}
 	| R(\xi) | \geq c_1 X_l^{-1} \{ x_{l,0}R(\xi) \}, \quad \text{for any} \quad l \geq 1.
\end{equation}
	By Proposition \ref{ALT1} there exists a constant $c_3 = c_3(\xi) > 0$ such that
	$\{ x_{l,0}R(\xi) \} \geq c_3 X_l^{- {2}/{\gamma^2}}$ if $H(R) \leq c_3 X_l^{1/{\gamma^3}}$.
	In this case we define $l$ to be the smallest positive integer such that  $H(R) \leq c_3 X_l^{1/{\gamma^3}}$.
	By the choice of $l$ we have $H(R) \gg X_l^{1/{\gamma^4}}$ and so (\ref{ALT0_1_1}) implies
\[
 	| R(\xi) |
		\gg  X_l^{-1 - {2}/{\gamma^2}}
		\gg  H(R)^{ - \gamma^4 - 2 \gamma^2} =  H(R)^{ - 1 - \gamma^5}.
\]

\end{proof}

%% file: period.tex
\section{Accumulation points}
\subsection{Proof of Theorem \ref{ALT0} revisited}
	Let $\xi$ be an extremal real number and
	let $({\bf x}_{k})_{\geq1}$ be the sequence of points in
	$\mathbb{Z}^3$ attached to $\xi$ as in \S \ref{C3:S1}.
	For any real number $\eta$, we define
\[
	\theta_{\xi}(\eta) = {\lim\inf}_{k\rightarrow\infty}\{x_{k,0}\eta\}.
\]
	With this notation it follows from Corollary \ref{ALC2}
	that for any fixed choice of a positive integer $a$,
	an extremal real number $\xi \in \mathcal{E}_a$
	and a non-zero polynomial
	$R \in \mathbb{Z}[T]$ of degree $3$ or $4$,
	we have $\theta_{\xi}(R(\xi))>0$.
	Then this and the inequality (\ref{ALT0_1}) in the
	proof of Theorem \ref{ALT0} imply that
\[
	| R(\xi) + P(\xi) | \gg  H(P)^{- \gamma },
\]
	for any non-zero $P \in \mathbb{Z}[T]$ of degree $\leq 2$,
	which in turn implies that
\[
	|\xi - \alpha| \gg H(\alpha)^{- \gamma^2}
\]
	for any root $\alpha$ of a polynomial of the form $R(T) + P(T)$
	with $P \in \mathbb{Z}[T]_{\leq 2}$, where
	the implied constants depend only on $R$ and $\xi$.
	The next theorem implements this argument in a
	more general context.
\begin{theorem} \label{accum:T1}
	Suppose that $\theta_{\xi}(\eta) \neq 0$.
	Then, for any non-zero polynomial
	$P(T) \in \mathbb{Z}[T]_{\leq 2}$, we have
\begin{equation}\label{accum:0}
 	| P(\xi) + \eta | \gg H(P)^{-\gamma},
\end{equation}
	where the implied constant depends only on $\xi$ and $\eta$.
\end{theorem}
\begin{proof}
	Fix a polynomial $P\in \mathbb{Z}[T]_{\leq 2}$.
	For each $k\geq1$, we have (same as in the proof of Theorem \ref{ALT0})
\[
 	c_1 \{ x_{k,0}\eta \} \leq  X_k | P(\xi) + \eta | + H(P) X_k^{-1},
\]
	for some  $c_1 = c_1(\xi)>0$.
	Since $\theta_{\xi}(\eta) \neq 0$, there exists a constant
	$c_2 = c_2(\xi,\eta)>0$ and some $k_0\geq1$, such that
\begin{equation}\label{accum:1}
 	c_2 \leq  X_k | P(\xi) + \eta | + H(P) X_k^{-1},
\end{equation}
	for each $k\geq k_0$.
	Let $k$ be the smallest index such that
\[
	H(P) \leq \frac{c_2}{2}X_k.
\]
	Assuming that the height $H(P)$ is sufficiently large,
	we have $k\geq k_0 + 1$ and
\[
	H(P) > \frac{c_2}{2}X_{k-1}.
\]
	Using this and the fact that $X_{i-1} \sim X_{i}^{1/\gamma}$,
	it follows from (\ref{accum:1}) that
\[
 	| P(\xi) + \eta |
	\geq
	\frac{c_2}{2}X_k^{-1}
	\gg
	H(P)^{-\gamma},
\]
	where the implied constant depends only on $\xi$ and $\eta$.
\end{proof}
%
%
\subsection{Properties of the accumulation points}
	Fix any $R \in \mathbb{Z}[T]$ of degree $3$ or $4$.
	As in Section 3.2 we fix a choice of a positive integer $a$,
	an extremal real number $\xi \in \mathcal{E}_a$ and
	corresponding sequences $(X_k)_{k \geq 1}$
	and $({\bf x}_k)_{k \geq 1}$ satisfying (\ref{int2}).

	Recall that the proof of the fact $\theta_{\xi}(R(\xi))>0$
	uses two arguments.
	Firstly, Corollary \ref{AQQT1} show that the sequence
		$\Big(\{ x_{k,0} R(\xi)\}\Big)_{k\geq1}$ has at most
		$6$ accumulation points $(\eta_l(R))_{0\leq l\leq5}$,
		reducing to at most $3$ accumulation points
		$(\delta_l(R))_{0\leq l\leq2}$ if $\deg(R)=3$.
	Secondly, Corollary \ref{ALC2}
		implies that $\eta_l(R)>0$ for each $l=0,\ldots,5$.

	Here, we give a new proof of the fact $\theta_{\xi}(R(\xi))>0$
	by showing that
	$\delta_l(R) \notin \bar{\mathbb{Q}}$ for each $l=0,1,2$
	if $\deg(R)=3$ and
	$\eta_l(R) \notin \mathbb{Q}$ for each $l=0,\ldots,5$
	if $3\leq\deg(R)\leq4$.
\begin{proposition} \label{P:DEG3}
	Suppose that $R(T) \in \mathbb{Z}[T]$ with $\deg(R) = 3$
	and let $l \in \{0,1,2\}$.

	(i) For any index $k \geq 1$ with $ k \equiv l+1 \mod 3 $
	there exists an integer $y \neq 0$
	with $\gcd(y,x_{k,0})\sim1$, such that
\begin{equation}\label{P:DEG3:1}
	\Big | \delta_l(R) - \frac{y}{x_{k,0}} \Big | \ll  X_k^{- \gamma^2}.
\end{equation}
	In particular, $y/{x_{k,0}}$ is a convergent of $\delta_l(R)$
	with denominator $\sim x_{k,0}$,
	for all $k$ sufficiently large.

	(ii) For any index $k \geq 1$ with $ k \equiv l+2 \mod 3 $
	there exists an integer $z \neq 0$ with $\gcd(z,x_{k,0})\sim1$,
	such that
\begin{equation}\label{P:DEG3:2}
	\Big | \delta_l(R) - \frac{z}{x_{k,0}} \Big | \ll  X_k^{- \gamma^2 - 1}.
\end{equation}
	In particular, $z/{x_{k,0}}$ is a convergent of $\delta_l(R)$
	with denominator $\sim x_{k,0}$,
	for all $k$ sufficiently large.

	(iii) Conversely, there exists a constant $c = c(\xi,R,l)>0$
	such that, for each convergent of
	$\delta_l(R)$, with sufficiently large denominator $q$, there
	exists an integer $k\geq1$ with
	$k \nequiv l \mod 3$ and $ c x_{k,0} \leq q \leq x_{k,0}$.

	(iv) $\delta_l(R) \notin \bar{\mathbb{Q}}$.
\end{proposition}
\begin{proof}
	Write $R(T)$ in the form
\[
	R(T) = g T^3 + Q(T),
\]
	where $\deg(Q) \leq 2$.
	For the proof of Part (i), we use the second
	identity in (\ref{A:I:5}) that gives
\begin{equation}\label{P:DEG3:3}
  \begin{aligned}
	x_{k,0}x_{k+2,2}\xi
	& = x_{k,1}x_{k+2,1}\xi	- \epsilon_k (-1)^{k} x_{k+1,1} \xi
	\\
	&= A_k + \bigo(X_{k+1}^{-1}),
  \end{aligned}
\end{equation}
	where $A_k = x_{k,1}x_{k+2,2} - \epsilon_k (-1)^{k} x_{k+1,2}$.
	Since $x_{k+2,0}Q(\xi) = (\text{integer}) + \bigo(X_{k+2}^{-1})$, we get
\begin{equation}\label{P:DEG3:3:0}
 \begin{aligned}
	x_{k+2,0}R(\xi)
		&=  g x_{k+2,0}\xi^3 + (\text{integer}) + \bigo(X_{k+2}^{-1})
		\\
		&=  g x_{k+2,2}\xi + (\text{integer}) + \bigo(X_{k+2}^{-1}).
 \end{aligned}
\end{equation}
	Also, by (\ref{P:DEG3:3}), we have
	$ x_{k+2,2}\xi
	= {A_k}/{x_{k,0}} + \bigo(X_{k+2}^{-1})$, which together with
	(\ref{P:DEG3:3:0}), implies that
\[
	x_{k+2,0}R(\xi)
	= \frac {g A_k}{x_{k,0}} + (\text{integer}) + \bigo(X_{k+2}^{-1}).
\]
	By this, we obtain
\begin{equation}\label{P:DEG3:4}
	\{ x_{k+2,0}R(\xi) \}
	= \Big\{ \frac {g A_k}{x_{k,0}}\Big\} + \bigo(X_{k+2}^{-1}).
\end{equation}
	Since $k+2 \equiv l \mod 3$, by the inequality (\ref{AQ7:1})
	and by Corollary \ref{AQQT1}, we have
\begin{equation}\label{P:DEG3:5}
	\delta_l(R) =
	 \{ x_{k+2,0}R(\xi) \} + \bigo (X_{k+2}^{-1}),
\end{equation}
	Let $B_k$ denotes the closest integer to ${g A_k}/{x_{k,0}}$.
	By (\ref{P:DEG3:5}) and (\ref{P:DEG3:4}), we deduce
\[
	\delta_l(R)
	= \left|\cfrac { g A_k - B_k x_{k,0}  }{x_{k,0}} \right|
	+  \bigo(X_{k}^{- \gamma^2}).
\]
	Let $y$ be the numerator of the fraction on the right.
	To complete the proof of Part (i), it remains only to
	show that  $\gcd (y,x_{k,0})$ is bounded above and non-zero.
	Since $\gcd (y,x_{k,0}) = \gcd (x_{k,0}, g A_k) $
	divides $ g \gcd (x_{k,0}, A_k)$ it suffices to show
	that $\gcd (x_{k,0}, A_k) \ | \ 2$.
	For this, recall from (\ref{AQ6-000}) that
\begin{equation}\label{P:DEG3:6}
	a_{k+1} = x_{k+1,0}x_{k+2,1} - x_{k+1,1}x_{k+2,0}  = \epsilon_{k+1}(-1)^{k} x_{k,0}.
\end{equation}
	Furthermore, we will show below that
\begin{equation}\label{P:DEG3:7}
	A_k = \epsilon_{k+1}(-1)^{k+1}( b_{k+1}x_{k+2,2} + c_{k+1}x_{k+2,1} ).
\end{equation}
	If we accept this result, then by
	Lemma \ref{ALT2} it follows from (\ref{P:DEG3:6})
	and (\ref{P:DEG3:7}) that
\[
	\gcd( x_{k,0}, A_k)
	=
	\gcd( a_{k+1}, b_{k+1}x_{k+2,2}+c_{k+1}x_{k+2,1})
	\ | \ 2.
\]
	Now, it remains to prove (\ref{P:DEG3:7}).
	First we consider $b_{k+1}x_{k+2,2} + c_{k+1}x_{k+2,1}$
	and replace $b_{k+1}$ and $c_{k+1}$ by
	their expressions given in (\ref{AQ6-000})
\begin{equation}\label{P:DEG3:8}
 \begin{aligned}
	& b_{k+1}x_{k+2,2} + c_{k+1}x_{k+2,1}
	\\
	& \qquad = x_{k+2,2} ( \epsilon_{k+1}(a x_{k,0} + 2(-1)^{k+1}x_{k,1}) ) + x_{k+2,1}(-\epsilon_{k+1}(a x_{k,1} + (-1)^{k+1} x_{k,2}))
	\\
	& \qquad = \epsilon_{k+1}  a x_{k,0} x_{k+2,2}  - \epsilon_{k+1} \big ( (-1)^{k}x_{k,1} x_{k+2,2}
	\\
	& \qquad \qquad \qquad \qquad \qquad \ + a x_{k,1} x_{k+2,1} + (-1)^{k} ( x_{k,1} x_{k+2,2} - x_{k,2} x_{k+2,1} ) \big )
 \end{aligned}
\end{equation}
	Now, let us compute
	$a x_{k,1} x_{k+2,1} + (-1)^{k} ( x_{k,1} x_{k+2,2} - x_{k,2} x_{k+2,1} )$
	separately
\begin{equation}\label{P:DEG3:9}
  \begin{aligned}
	a x_{k,1} x_{k+2,1} & + (-1)^{k} ( x_{k,1} x_{k+2,2} - x_{k,2} x_{k+2,1} ) &
	\\
	& = a x_{k,1} x_{k+2,1} + (-1)^{k} ( - \epsilon_k( a x_{k+1,1} + (-1)^{k} x_{k+1,2} ))  & _{\text{by} \ (\ref{A:I:5})_4}
	\\
	&= a (x_{k,1} x_{k+2,1} - (-1)^{k} \epsilon_k  x_{k+1,1}) - \epsilon_k x_{k+1,2} &
	\\
	&= a x_{k,0} x_{k+2,2} - \epsilon_k x_{k+1,2}  & _{\text{by} \ (\ref{A:I:5})_2}
  \end{aligned}
\end{equation}
	Finally, it follows from (\ref{P:DEG3:8}) and (\ref{P:DEG3:9}) that
\begin{align*}
	b_{k+1}x_{k+2,2} + c_{k+1}x_{k+2,1}
	& = \epsilon_{k+1}  a x_{k,0} x_{k+2,2}
	- \epsilon_{k+1} \big ( (-1)^{k}x_{k,1} x_{k+2,2} + a x_{k,0} x_{k+2,2} - \epsilon_k x_{k+1,2} \big )
	\\
	& = \epsilon_{k+1} (-1)^{k+1} \big ( x_{k,1} x_{k+2,2} - \epsilon_k (-1)^{k} x_{k+1,2}  \big )
	\\
	& = \epsilon_{k+1} (-1)^{k+1} A_k,
\end{align*}
	and this completes the proof of Part (i).

	For the proof of Part (ii) we multiply the identity
	(\ref{A:I:8}) by $\xi$ and obtain
\begin{equation}\label{P:DEG3:10}
	x_{k,0}x_{k+4,2}\xi = E_k + O(X_{k+2}^{-1}),
\end{equation}
	where $E_k = x_{k,1}x_{k+4,2} - \epsilon_k (-1)^k ( a x_{k+1,0}x_{k+3,2}
	+ \epsilon_{k+1} x_{k+2,2})$.
	Using the fact that $x_{k+4,0}R(\xi) = g x_{k+4,2}\xi
	+ (\text{integer}) + O(X_{k+4}^{-1})$,
	we deduce from (\ref{P:DEG3:10}) that
\begin{equation}\label{P:DEG3:11}
	\{ x_{k+4,0}R(\xi) \} =
	 \Big | \frac{g E_k - F_k x_{k,0}}{x_{k,0}} \Big |
	 + \bigo(  X_{k}^{- \gamma^2 - 1}),
\end{equation}
	where $F_k$ is the closest integer to $g E_k/x_{k,0}$.
	Since $k+4 \equiv l \mod 3$, by the inequality (\ref{AQ7})
	and by Corollary \ref{AQQT1}, we have
\[
	\delta_l(R)
	= \{ x_{k+4,0}R(\xi) \} + \bigo(X_{k+4}^{-1}).
\]
	From this and (\ref{P:DEG3:11}),
	for any $l = 0, 1 ,2$ and $k \geq 1$
	with $k+4 \equiv l \mod 3$, we get
\begin{equation}\label{P:DEG3:12}
	\delta_l(R)
	= \left | \frac{g E_k - F_k x_{k,0}}{x_{k,0}} \right |
	+  \bigo(X_{k}^{- \gamma^2 - 1}).
\end{equation}
	Let $z$ be the numerator of the fraction on the right.
	We claim that
\[
	\gcd(x_{k,0},z) = \gcd(x_{k,0},g E_k) \ | \ 2 g.
\]
	If we accept this claim, then (\ref{P:DEG3:12})
	shows that, for each $k$ sufficiently large,
	$z/x_{k,0}$ is a convergent of $\delta_l(R)$
	with denominator $\sim x_{k,0}$.
	To prove this claim, we first note that
\begin{align*}
	E_k
	& = x_{k,1}(a x_{k+2,0} x_{k+3,2}
		- \epsilon_{k+2}x_{k+1,2}) - \epsilon_k (-1)^k ( a x_{k+1,0}x_{k+3,2}
		+ \epsilon_{k+1} x_{k+2,2})
	& _{\text{by} \ (\ref{AQ6:0})}
	\\
	& = a x_{k+3,2} (x_{k,1}x_{k+2,0} - \epsilon_k (-1)^k x_{k+1,0})  -  \epsilon_{k+2}(x_{k,1}x_{k+1,2}  + (-1)^k x_{k+2,2})
	\\
	& = a x_{k+3,2} x_{k+2,1} x_{k,0}  -  \epsilon_{k+2}(x_{k,1}x_{k+1,2}  + (-1)^k x_{k+2,2}).
	& _{\text{by} \ (\ref{A:I:5})_1}
\end{align*}
	From this we deduce that
\begin{equation}\label{P:DEG3:13}
	E_k  \equiv  -  \epsilon_{k+2}T_k \mod x_{k,0},
\end{equation}
	where $T_k = x_{k,1}x_{k+1,2}  + (-1)^k x_{k+2,2}$. Now we consider
\begin{equation}\label{P:DEG3:14}
	T_k x_{k,1} = x_{k,1}^2 x_{k+1,2}  + (-1)^k x_{k+2,2}x_{k,1}.
\end{equation}
	Since $A_k = x_{k,1}x_{k+2,2} - \epsilon_k (-1)^{k} x_{k+1,2}$, then
	$x_{k,1}x_{k+2,2} = A_k + \epsilon_k (-1)^{k} x_{k+1,2}$.
	Also, from $\det({\bf x}_k) = \epsilon_k$ we have $x_{k,1}^2 = x_{k,0}x_{k,2} - \epsilon_k$.
	From these two equalities and (\ref{P:DEG3:14}) we obtain
\[
 \begin{aligned}
	T_k x_{k,1}
	& = (x_{k,0}x_{k,2} - \epsilon_k) x_{k+1,2}  + (-1)^k (A_k + \epsilon_k (-1)^{k} x_{k+1,2})
	\\
	& = x_{k,0}x_{k,2}x_{k+1,2} - \epsilon_k x_{k+1,2}  + (-1)^k A_k + \epsilon_k x_{k+1,2}
	\\
	& = x_{k,0}x_{k,2}x_{k+1,2}  + (-1)^k A_k,
 \end{aligned}
\]
	and therefore
\[
	T_k x_{k,1} \equiv (-1)^k A_k \mod x_{k,0}.
\]
It follows that
\[
	\gcd(x_{k,0},T_k x_{k,1}) =  \gcd(x_{k,0},A_k),
\]
	and since $\gcd(x_{k,0},x_{k,1}) = 1$, we deduce
\[
	\gcd(x_{k,0},T_k) =  \gcd(x_{k,0},A_k).
\]
	In Part (i) we have shown that $\gcd(x_{k,0},A_k)$ divides $2$,
	which implies
	that $\gcd(x_{k,0},T_k)$ also divides $2$, and from (\ref{P:DEG3:13})
	we deduce that
	$\gcd(x_{k,0},E_k) \ | \ 2$. Since $\gcd(x_{k,0},g E_k)$ divides
	$g \gcd(x_{k,0},E_k)$, we finally obtain that
	$\gcd(x_{k,0},g E_k) \ | \ 2g$, which proves the claim.

	For the proof of Part (iii) we use properties of continued fractions.
	We know that for any $\alpha \in \mathbb{R} \backslash \mathbb{Q}$
	with convergents $(p_m/q_m)_{m \geq 1}$, we have
\begin{equation}\label{P:DEG3:15}
	\big | \alpha - p_m/q_m \big | \sim (q_m q_{m+1})^{-1}.
\end{equation}
	Let $p/q$ be a convergent of $\delta_l(R)$ with a sufficiently 
	large denominator $q$.
	Then there exists an integer $k$ such that  $x_{k,0} < q \leq x_{k+1,0} $,
	and we have three cases.

	If $k \equiv l \mod 3$, then $k-1 \equiv l+2 \mod 3$ and there exist consecutive
	convergents $p_s/q_s$ and $p_{s+1}/q_{s+1}$ of $\delta_l(R)$ with $p_s/q_s = z/x_{k-1,0}$
	for some integer $z$, which by (\ref{P:DEG3:2}) satisfies
\[
	\big | \delta_l(R) - z/x_{k-1,0} \big | \ll  X_{k-1}^{- \gamma^2 - 1}.
\]
	By this and (\ref{P:DEG3:15}), we have
\[
	(q_s q_{s+1})^{-1} \ll \big | \delta_l(R) - z/x_{k-1,0} \big | \ll X_{k-1}^{- \gamma^2 - 1}.
\]
	This inequality and the fact that $q_s \leq x_{k-1,0}$
	imply $ x_{k+1,0} \ll q_{s+1} $.
	Since $q_s$ and $q_{s+1}$ are denominators of consecutive convergents and
	$q_s < q$, we also have $ q_{s+1} \leq q $. Whence the inequalities $ x_{k+1,0} \ll q_{s+1} $
	and $q \leq x_{k+1,0}$ imply $ x_{k+1,0} \ll q \leq x_{k+1,0} $.

	If $k \equiv l+1 \mod 3$, there exist consecutive convergents
	$p_t/q_t$ and $p_{t+1}/q_{t+1}$ of $\delta_l(R)$, with $p_t/q_t = y/x_{k,0}$
	for some integer $y$ satisfying (\ref{P:DEG3:1}). It follows from (\ref{P:DEG3:1})
	and (\ref{P:DEG3:15}) that
\[
	(q_t q_{t+1})^{-1} \ll \big | \delta_l(R) - y/x_{k,0} \big | \ll X_{k}^{-\gamma^2}.
\]
	This inequality and the fact that $q_t \leq x_{k,0}$
	imply $ x_{k+1,0} \ll q_{t+1} $.
	Since $q_t$ and $q_{t+1}$ are denominators of consecutive convergents and
	$q_t < q$, we also have $ q_{t+1} \leq q $. Whence the inequalities $ x_{k+1,0} \ll q_{t+1} $
	and $q \leq x_{k+1,0}$ imply $ x_{k+1,0} \ll q \leq x_{k+1,0} $.

	If $k \equiv l+2 \mod 3$, there exist consecutive convergents
	$p_u/q_u$ and $p_{u+1}/q_{u+1}$ of $\delta_l(R)$,
	with $p_u/q_u = z/x_{k,0}$ for some integer $z$ satisfying (\ref{P:DEG3:2}).
	It follows from (\ref{P:DEG3:2}) and (\ref{P:DEG3:15}) that
\[
	(q_u q_{u+1} )^{-1}
	\ll \big | \delta_l(R) - z/x_{k,0} \big | \ll X_{k}^{-\gamma^2-1}.
\]
	This inequality and the fact that $q_u \leq x_{k,0}$ imply $ x_{k+2,0} \ll q_{u+1} $.
	Since $q_u$ and $q_{u+1}$ are denominators of consecutive convergents and $q_u < q$,
	we deduce that $ x_{k+2,0} \ll q_{u+1} \leq q \leq x_{k+1,0} $, which is impossible
	for $k$ and $q$ sufficiently large.
	From this we conclude that there is no convergent with a denominator $q$
	satisfying the inequality $x_{k,0} < q \leq x_{k+1,0} $ for such $k$.

	Part (iv) is a consequence of Roth's Theorem together
	with Parts (i) and  (ii).
\end{proof}
\begin{remark}\label{accum:rem:1}
	Part (ii) of Proposition \ref{P:DEG3} implies that
	$w_1^*\big(\delta_l(R)\big) \geq \gamma^2$
	for $l = 0,1,2$,
	where $w^*$ is Koksma's classical 
	exponent of approximation introduced in \S \ref{C1:S4:SS2}.
\end{remark}
	Now, we prove that for any non-zero
	$R(T) \in \mathbb{Z}[T]$ of degree $4$
	the accumulations points
	$\big(\eta_l(R)\big)_{0\leq l\leq 5}$
	are irrational and so, they are non-zero.
\begin{proposition} \label{P:DEG4}
	Suppose that $R(T) \in \mathbb{Z}[T]$ with $\deg(R) = 4$
	and let $l \in \{0,\ldots,5\}$.

	(i) For any sufficiently large index
	$k \geq 1$ with $ k \equiv l+4 \mod 6 $
	there exists an integer $y \neq 0$ such that
\begin{equation}\label{P:DEG4:1}
	\Big | \eta_l(R) - \frac{y}{x_{k,0}x_{k-1,0}} \Big | \ll  X_k^{- \gamma^2}.
\end{equation}

	(ii) For any sufficiently large index
	$k \geq 1$ with $ k \equiv l+2 \mod 6 $
	there exists an integer $z \neq 0$ such that
\begin{equation}\label{P:DEG4:2}
	\Big | \eta_l(R) - \frac{z}{x_{k,0}^2} \Big | \ll  X_k^{- \gamma^2-1}.
\end{equation}

	(iii)  $\eta_l(R) \notin {\mathbb{Q}}$.
\end{proposition}
\begin{proof}
	We can write $R(T)$ in the form
\[
	 R(T) = f T^4 + g T^3 + Q(T),
\]
	where $\deg(Q) \leq 2$. For the proof of Part (i)
	we use (\ref{P:DEG3:3}), which gives
\begin{align*}
	x_{k,0}x_{k+2,2}\xi^2 & = x_{k,1}x_{k+2,2}\xi
	- \epsilon_k (-1)^{k} x_{k+1,2} \xi + \bigo(X_{k+1}^{-1})
	\\
	& = \big ( x_{k,2}x_{k+2,1} - \epsilon_k (a x_{k+1,1}
	+ (-1)^k x_{k+1,2} ) \big ) \xi
	\\
	& \quad\quad\quad\quad\quad\quad
	- \epsilon_k (-1)^{k} x_{k+1,2} \xi + \bigo(X_{k+1}^{-1})
	\ _{\text{by} \ (\ref{A:I:5})_4}
	\\
	& = ( x_{k,2}x_{k+2,2} - \epsilon_k a x_{k+1,2} )
	- 2 \epsilon_k (-1)^{k} x_{k+1,2} \xi + \bigo(X_{k+1}^{-1})
\end{align*}
	Multiplying both sides of the above equality by $x_{k-1,0}$ we obtain
\[
	x_{k-1,0}x_{k,0}x_{k+2,2}\xi^2 =
	x_{k-1,0} ( x_{k,2}x_{k+2,2} - \epsilon_k a x_{k+1,2} )
	- 2 \epsilon_k (-1)^{k} x_{k-1,0}x_{k+1,2} \xi + \bigo(X_k^{-1}).
\]
	By (\ref{P:DEG3:3}) we have $x_{k-1,0} x_{k+1,2} \xi = A_{k-1}
	+ \bigo(X_{k}^{-1})$, and hence
\[
	x_{k-1,0}x_{k,0}x_{k+2,2}\xi^2 = C_k + \bigo(X_{k}^{-1}),
\]
	where $C_k = x_{k-1,0} ( x_{k,2}x_{k+2,2} - \epsilon_k a x_{k+1,2} ) - 2 \epsilon_k (-1)^{k} A_{k-1}$.
	By this and (\ref{P:DEG3:3}), we get
\begin{equation}\label{P:DEG4:3}
	f x_{k-1,0}x_{k,0}x_{k+2,2}\xi^2 + g x_{k-1,0}x_{k,0}x_{k+2,2}\xi = f C_k + g x_{k-1,0} A_k + O(X_{k}^{-1}).
\end{equation}
	From this and $x_{k+2,0}Q(\xi) = (\text{integer})
	+ \bigo(X_{k+2}^{-1})$, we have
\[
	x_{k+2,0}R(\xi) = f x_{k+2,2}\xi^2
	+ g x_{k+2,2}\xi + (\text{integer}) + \bigo(X_{k+2}^{-1}),
\]
	which leads to
\begin{equation}\label{P:DEG4:4}
	\{ x_{k+2,0}R(\xi) \}
	= \{ f x_{k+2,2}\xi^2 + g x_{k+2,2}\xi \} + \bigo(X_{k+2}^{-1}).
\end{equation}
	By (\ref{P:DEG4:4}) and (\ref{P:DEG4:3}), we get
\begin{equation}\label{P:DEG4:5}
	\{ x_{k+2,0}R(\xi) \} =
	 \left | \frac {f C_k  + g x_{k-1,0} A_k
	 - D_k x_{k-1,0}x_{k,0}}{x_{k-1,0}x_{k,0}} \right | + \bigo(X_{k+2}^{-1}),
\end{equation}
	where $D_k$ denotes the closest integer
	to $({f C_k  + g x_{k-1,0} A_k})/({x_{k-1,0}x_{k,0}})$.
	By the inequality (\ref{AQ7})
	and by Corollary \ref{AQQT1}, for any $l = 0,\ldots,5$ and $k \geq 1$
	with $k+2 \equiv l \mod 6$, we have
\[
	\eta_l(R) = \{ x_{k+2,0}R(\xi) \} + \bigo(X_{k+2}^{-1}),
\]
	and by (\ref{P:DEG4:5}) we deduce
\[
	\eta_l(R) =
	 \left | \frac {f C_k  + g x_{k-1,0} A_k
	 - D_k x_{k-1,0}x_{k,0}}{x_{k-1,0}x_{k,0}} \right |
	 + \bigo(X_{k}^{- \gamma^2}).
\]
	We notice that if $x_{k-1,0}$ divides $f C_k$ then it divides $f A_{k-1}$
	and, since $\gcd (x_{k-1,0}, A_{k-1})$ divides $2$, we conclude that
	$x_{k-1,0}$ divides $2f$.
	This is impossible if $k$ is sufficiently large and thus for
	large $k$ the numerator in the fraction above is non-zero.

	For the proof of Part (ii) we multiply the identity (\ref{A:I:8}) by $\xi^2$ and obtain
\begin{equation}\label{P:DEG4:6}
	x_{k,0}x_{k+4,2}\xi^2 = x_{k,1}x_{k+4,2}\xi - \epsilon_k (-1)^k ( a x_{k+1,0}x_{k+3,2}\xi + \epsilon_{k+1} x_{k+2,2}\xi) + O(X_{k+2}^{-1}).
\end{equation}
	Replacing $k$ with $k+1$ in the second identity of (\ref{A:I:5}) and multiplying it by $\xi$, we have
\[
	x_{k+1,0}x_{k+3,2} \xi = x_{k+1,1}x_{k+3,2} + \epsilon_{k+1} (-1)^{k} x_{k+2,2} + O(X_{k+2}^{-1}).
\]
	Applying this and (\ref{A:I:9}) to (\ref{P:DEG4:6}) we get
\[
  \begin{aligned}
  x_{k,0}x_{k+4,2}\xi^2
	& = x_{k,2}x_{k+4,2} - \epsilon_k a ( a x_{k+1,0}x_{k+3,2} + 2(-1)^k x_{k+1,1}x_{k+3,2} + \epsilon_{k+1} x_{k+2,2})
	\\
	& - 2 (-1)^k \epsilon_{k+2} x_{k+2,2}\xi + O(X_{k+2}^{-1}).
  \end{aligned}
\]
	Multiplying this by $x_{k,0}$ and applying (\ref{P:DEG3:3}),
	in order to replace
	$x_{k,0}x_{k+2,2}\xi$, we obtain
\[
	x_{k,0}^2 x_{k+4,2}\xi^2 = x_{k,0}G_k - 2 (-1)^k \epsilon_{k+2} A_k + O(X_{k+1}^{-1}),
\]
	where $G_k = x_{k,2}x_{k+4,2} - \epsilon_k a ( a x_{k+1,0}x_{k+3,2} + 2(-1)^k x_{k+1,1}x_{k+3,2}
	+ \epsilon_{k+1} x_{k+2,2})$ and $A_k$ is defined as in (\ref{P:DEG3:7}).
	By this and (\ref{P:DEG3:10}), we have
\begin{equation}\label{PERT1Q11}
	f x_{k,0}^2 x_{k+4,2} \xi^2 + g x_{k,0}^2 x_{k+4,2} \xi = f N_k + g x_{k,0} E_k  + O(X_{k+1}^{-1}),
\end{equation}
	where $N_k = x_{k,0}G_k - 2 (-1)^k \epsilon_{k+2} A_k$.
	Since $x_{k+4,0} Q(\xi) = (\text{integer}) + O(X_{k+4}^{-1})$, we have
\[
	x_{k+4,0}R(\xi) = f x_{k+4,2} \xi^2 + g x_{k+4,2} \xi + (\text{integer}) + O(X_{k+4}^{-1}),
\]
	which leads to
\[
	\{ x_{k+4,0}R(\xi) \}
	= \{ f x_{k+4,2} \xi^2 + g x_{k+4,2} \xi \}
	+ \bigo(X_{k+4}^{-1}).
\]
	By this and (\ref{PERT1Q11}) we deduce
\begin{equation}\label{PERT1Q12}
	\{ x_{k+4,0}R(\xi) \}
	= \frac {| f N_k + g x_{k,0} E_k - L_k x_{k,0}^2|}{x_{k,0}^2}
	+ \bigo(X_{k}^{-\gamma^2-1}),
\end{equation}
	where $L_k$ is the closest integer to $(f N_k + g x_{k,0} E_k)/x_{k,0}^2$.
	Also, by the inequality (\ref{AQ7})
	and by Corollary \ref{AQQT1},
	it follows that for any $l = 0,\ldots,5$ and $k \geq 1$
	with $k+4 \equiv l \mod 6$ we have
\[
	 \eta_l(R) = \{ x_{k+4,0}R(\xi) \} + \bigo(X_{k+4}^{-1}),
\]
	whence from (\ref{PERT1Q12}) we obtain
\[
	\eta_l(R)
	= \frac {| f N_k + g x_{k,0} E_k - L_k x_{k,0}^2 |}{x_{k,0}^2}
	+  \bigo(X_{k}^{-\gamma^2-1}).
\]
	Finally, since
	$\gcd(x_{k,0},f N_k) \ \big | \ f \gcd(x_{k,0},N_k)\ \big | \ f \gcd(x_{k,0},2 A_k)\ \big | \ 2 f \gcd(x_{k,0},A_k)$,
	and since $\gcd(x_{k,0},A_k) \ \big | \ 2 $, we deduce that the numerator in the
	fraction above is not zero for $k \gg 1$.

	Part (iii) follows from Parts (i) and  (ii).
\end{proof}
